\documentclass[12pt]{article}
\usepackage{palatino}

\usepackage{amsmath,xcolor}
\definecolor{ao(english)}{rgb}{0.0, 0.5, 0.0}
\usepackage{hyperref}
\hypersetup{
	colorlinks=true,
	linkcolor=ao(english),
	citecolor=ao(english),
	filecolor=magenta,      
	urlcolor=cyan,
	linktoc=page
}
\usepackage{amsthm}
\usepackage{amssymb}
\usepackage{enumitem}
\usepackage{placeins}
\usepackage{graphicx}
\usepackage{mathrsfs}
\usepackage{epstopdf}
\usepackage{xcolor}
\usepackage[longnamesfirst]{natbib}
\bibpunct[ ]{(}{)}{,}{a}{}{,}

\theoremstyle{plain}
\newtheorem{theorem}{Theorem}[section]
\newtheorem{corollary}[theorem]{Corollary}
 
\newtheorem{prop}[theorem]{Proposition}
\newtheorem{lemma}[theorem]{Lemma}

\theoremstyle{definition}
\newtheorem{definition}[theorem]{Definition}
\newtheorem{remark}[theorem]{Remark}
\newtheorem{assumptions}[theorem]{Assumptions}\newtheorem{tightness-condition}[theorem]{Tightness condition}
\newtheorem{example}[theorem]{Example}

\newcommand{\N}{\mathbb{N}}

\newcommand{\prob}{\mathbb{P}}
\newcommand{\oo}{{\mathfrak o}}
\newcommand{\p}{\mathbb{P}}
\newcommand{\E}{\mathbb{E}}
\newcommand{\Rd}{\mathbb{R}^4_\uparrow}

\newcommand{\expt}{\mathbb{E}}
\newcommand{\indic}{\mathbf{1}}

\newcommand{\floor}[1]{{\left\lfloor #1 \right\rfloor}}
\newcommand{\ceil}[1]{{\left\lceil #1 \right\rceil}}
\newcommand{\CDF}{\operatorname{CDF}}

\newcommand{\bigslant}[2]{{\raisebox{.2em}{$#1$}\left/\raisebox{-.2em}{$#2$}\right.}}

\newcommand{\sset}{\subset}

\newcommand{\la}{\lambda}
\newcommand{\al}{\alpha}
\newcommand{\Om}{\Omega}

\newcommand{\mathand}{\;\text{and}\;}

\newcommand{\mathas}{\;\text{as}\;}

\newcommand{\ga}{\gamma}
\newcommand{\Ga}{\Gamma}
\newcommand{\ep}{\epsilon}
\newcommand{\om}{\omega}
\newcommand{\ka}{\kappa}
\newcommand{\de}{\delta}
\newcommand{\be}{\beta}
\newcommand{\De}{\Delta}
\newcommand{\sig}{\sigma}
\newcommand{\eps}{\epsilon}

\newcommand{\del}{\partial}

\newcommand{\scrE}{\mathcal{E}}
\newcommand{\scrA}{\mathfrak{A}}
\newcommand{\scrD}{\mathcal{D}}

\newcommand{\scrK}{\mathcal{K}}

\newcommand{\scrL}{\mathcal{L}}
\newcommand{\scrH}{\mathcal{H}}
\newcommand{\scrS}{\mathcal{S}}

\newcommand{\scrI}{\mathcal{I}}

\newcommand{\close}[1]{\mkern 1.5mu\overline{\mkern-1.5mu#1\mkern-1.5mu}\mkern 1.5mu}

\newcommand{\Z}{\mathbb{Z}}

\newcommand{\Q}{\mathbb{Q}}
\newcommand{\R}{\mathbb{R}}

\newcommand{\eqd}{\stackrel{d}{=}}

\newcommand{\cvgd}{\stackrel{d}{\to}}

\newcommand{\cvgp}{\stackrel{\prob}{\to}}
\newcommand{\X}{\times}

\newcommand{\cvgup}{\uparrow}
\newcommand{\cvgdown}{\downarrow}

\newcommand{\smin}{\setminus}
\newcommand{\lf}{\left}
\newcommand{\rg}{\right}

\DeclareMathOperator{\Int}{int}

\newcommand{\bp}{\mathbf{p}}
\newcommand{\bq}{\mathbf{q}}

\newcommand{\Rh}{\R^4_\circ}

\numberwithin{equation}{section}
\title{The scaling limit of the longest increasing subsequence}
\author{Duncan Dauvergne \and B\'alint Vir\'ag}
\date{October 2022}

\begin{document}
	
	\maketitle
	
	\begin{abstract}
		We provide a framework for proving convergence to the directed landscape, the central object in the Kardar-Parisi-Zhang universality class. 
		For last passage models, we show that {Skorokhod convergence in law} to the Airy line ensemble implies convergence to the Airy sheet. In i.i.d.\ environments, we show that Airy sheet convergence implies convergence of distances and geodesics to their counterparts in the directed landscape. Our results imply convergence of classical last passage models and {coupled TASEP.} 
		Our framework is built on the notion of a directed metric, a generalization of metrics which behaves better under limits. 
		
		As a consequence of our results, we show
that the scaled longest increasing subsequence in a uniform random permutation converges to
the directed geodesic. 
	\end{abstract}

	\begin{center}	\includegraphics[scale=0.8
		]{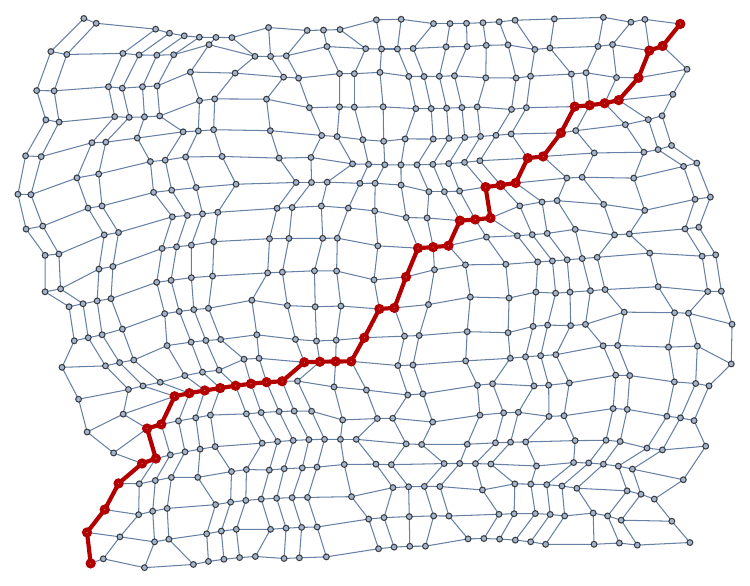}		
	\end{center}		
	
	\tableofcontents	
	
	\section{Introduction}
	
	\subsection{Random plane geometry}
	
	The goal of this paper is to set up a framework for proving convergence of random  directed plane geometry models to a universal limit,  the directed landscape. 
 
	In these models, we have a random ``distance'' function $d$ defined on the plane or a periodic subset, such as  a lattice. 
	A simple example is the model on the front page, which is equivalent to the  Sepp\"al\"ainen-Johansson model, see Example \ref{Ex:SJ}. For this picture, each horizontal edge in $\mathbb Z^2$ is given an i.i.d.\ length $1$ or $2$ and all vertical edges have length $1$. {Subtracting $1$ from every edge gives the weights in Example \ref{Ex:SJ}; this does not change geodesics, since all directed paths between fixed endpoints have the same number of edges.}

	In such models, how far should a point $p$ be from the origin for $d(0,p)$ to have standard deviation roughly $\sigma$? 
	The empirical answer is order $\sigma^3$, and histograms of $d(0, p)$ show that the distribution is not normal. This suggests that the noise in the random plane geometry has a nontrivial  structure. 

	To understand this structure, it helps to consider a simpler deterministic setting. In Euclidean geometry, paths from $0$ to a point at distance $\sig^3$ from $0$ whose length is $\sigma^3+\sigma$ deviate from a straight line by distance $O(\sigma^2)$. {This suggests that} the noise in our random metric has nontrivial correlations in a scaling window of size $\sigma^3$ in the main direction and $\sigma^2$ in the transversal direction. {For rigorous relations between fluctuation and transversal exponents in first passage percolation, under assumptions on their existence, see \cite*{chatterjee2013universal,auffinger2014simplified}.} To make this more explicit, given a direction $v_0$, there are vectors $v$ {parallel to $v_0$} and $u$ and a scalar $\mu$ so that as $\sigma\to \infty$ we expect 
	\begin{equation}\label{E:metric}
	d(\sigma^2ux + \sigma^3 v s, \sigma^2uy +\sigma^3vt)=(t-s)\sigma^3\mu - \sigma \scrL (x,s;y,t)+o(\sigma).
	\end{equation}
	Here $\scrL$ is a limiting universal random plane geometry called the {\bf directed landscape}, constructed in \cite*{DOV}. The negative sign in front of $\sig$ in \eqref{E:metric} is a convention.
	
	In \cite*{DOV}, it was shown that one particular random metric model -- Brownian last passage percolation -- converges to the directed landscape. This paper {can be thought of} as a followup. We give a brief summary of the new results.
	
\subsection{New results}

We develop a general framework for the convergence of last passage models to the directed landscape, and apply this framework to several classical models. The following main contributions are described in more detail in the later part of the introduction. 

\smallskip 

\begin{itemize}[nosep]
\item We develop the notion of directed metric spaces, of which usual metrics, last passage percolation models, and the directed landscape are all examples. We formalize the notion of geodesics and establish hypograph and graph topologies for studying convergence. 

\item As a result of this framework, convergence of geodesics will follow from convergence of the corresponding directed metric spaces in either topology.
	
\item We introduce a maximal inequality for last passage percolation that allows for simple tightness conditions in the hypograph topology. For tightness in the stronger graph topology and for stronger exponential tightness, we give a general chaining argument that will apply to all models.

\item We use the RSK isometry to prove that convergence to the Airy sheet for last passage models follows from convergence to the Airy line ensemble without any additional assumptions. For i.i.d. lattice last passage models, convergence to the Airy line ensemble also implies convergence to the directed landscape and convergence of geodesics.

\item We use our framework to prove convergence to the directed landscape for geometric and exponential last passage percolation  (this part of the present work is used in \cite{busani2022stationary} and \cite{ganguly2022fractal}), Poisson last passage percolation across lines and in the plane (used in \cite{dauvergne2022non}).
 \item We prove that coupled TASEPs (totally asymmetric exclusion processes) started from multiple initial conditions converge to  KPZ fixed points coupled via the directed landscape. The coupling of TASEPs started from different initial conditions used throughout this paper is the one inherited from a common underlying exponential last passage percolation environment (as opposed to the so-called basic coupling of the interacting particle systems).
 
 To do so, we give a general framework for moving from convergence of directed metrics to convergence of metric interfaces. TASEP height functions are interfaces in the exponential last passage metric. This part of the present work has been used in \cite{rahman2021infinite} to show convergence of second class particle trajectories.
 \item We give a new definition for the Airy sheet in terms of the Airy line ensemble, resolving Conjecture 14.1 from \cite*{DOV}.
	\item We establish new symmetries for the Airy sheet and the Airy line ensemble.
	\item  We prove that the shape of the longest increasing subsequence in a random permutation converges to the directed geodesic. 
\end{itemize}
	\subsection{Directed metric spaces}

	We expect $\scrL$ to be the limit of most random metrics in the plane built out of independent noise. However, $\scrL$ itself is not a metric, and neither are the last passage models that are the main focus of this paper. Because of this, we need a generalized notion of metrics.
	
	A {\bf directed metric of positive sign} on a set $S$ is a function $d:S^2\to \R \cup \{\infty\}$ satisfying 
	\begin{align*}
	d(p,p)&=0, \qquad \qquad p\in S,\\
	d(p,q)+d(q,r)&\ge d(p,r), \qquad p,q,r\in S.
	\end{align*}
	Unlike with ordinary metrics, directed metrics can be asymmetric and may take negative values.  We say that $d$ is a {\bf directed metric of negative sign} if $-d$ is a directed metric of positive sign.
	
	All ordinary metrics are directed metrics of positive sign.  
	A familiar example of a directed metric of negative sign is the $L^p$-norm on $\R^d$ for $p\in(0,1)$: set $\|x\|_p^p=\sum x_i^p$ when all $x_i\ge 0$ and $-\infty$ otherwise, and let $d(x, y) = \|x-y\|_p$.
	
	The notion of a directed metric unifies the usual notions of metric, distance in directed graphs, and last passage percolation. In addition, directed metrics are closed under the deterministic shifts used in the approximation \eqref{E:metric}, while ordinary metrics are not.  Properties of directed metrics are studied in Section \ref{S:directed-metrics}. 
	
	\subsubsection{Classical models and other examples}
	\label{SS:classical}
	
	The following are selected examples from Section \ref{S:metric-examples}. Most of these are integrable classical models whose fine scaling behaviour is accessible.
	To construct these examples, we introduce the useful notion of an {\bf induced metric}. In the positive sign setting, given a function $d_0:A \to \R\cup \{\infty\}$ for some subset $A\subset S^2$, the induced metric of positive sign is the {pointwise} supremum of all directed metrics on $S$ with $d|_A \le d_0$, {provided this collection is nonempty,} see Definition \ref{D:induced-dm}. Induced metrics of negative sign are defined the same way, with the inequality reversed and supremum replaced by infimum.
	
	\begin{example}\label{Ex:graphs-intro}
		Let $G$ be an undirected graph, and let $d_0(x, y) = 1$ whenever $(x, y)$ is an edge. The metric of positive sign on $G$ induced by $d_0$ is the usual graph distance. 
		
		This generalizes to weighted, directed graphs: when $d_0$ is defined on all directed edges to be the edge weight, the induced metric is the weighted directed graph distance. 
	\end{example}
	
	\begin{example}[Sepp\"al\"ainen-Johansson]\label{Ex:SJ} Continuing Example \ref{Ex:graphs-intro}, let $G$ be the square lattice $\mathbb Z^2$ where $(x, y)$ is a directed edge whenever $y$ is directly north or directly east of $x$. Let $d_0=0$ on all north edges, and let $d_0$ be $0$ or $1$ on east edges according to independent coin tosses{, with probability $p$ of weight $1$}. The resulting induced metric of positive sign on $\Z^2$ is a random directed metric known as the Sepp\"al\"ainen-Johansson model with parameter $p$, see Figure \ref{fig:sepp-joh}. This model was introduced by \cite*{seppalainen1998exact} and further studied by \cite*{johansson2001discrete}. 
	\end{example}
	
	\begin{example}[Geometric and exponential last passage]\label{Ex:geom-exp} Let $G = (V, E)$ be a directed graph, and define $d_0((x, y), (y, z))=1$ for any pair of directed edges $(x,y),(y,z)$. The induced metric $d$ on $E$ of negative sign, {when it exists,} assigns to each pair $e,f$ of directed edges the length of the longest walk starting with $e$ and ending with $f$, as measured by the number of intermediate vertices. 
		
		If we set $d_0((x,y),(y,z))=w(y)$ for all directed edges for some weight function $w$ on $V$, the induced metric corresponds to last passage percolation.
		
		When $G=\mathbb Z^2$ with edges directed north or east as in Example \ref{Ex:SJ}, and the weights $w(y)$ are i.i.d. and nonnegative, we call the corresponding model {\bf i.i.d.\ lattice last passage percolation}. This model is integrable when the weights $w(y)$ are {\bf exponential} or {\bf geometric}. Throughout the paper, we assume all exponential random variables are mean $1$, and that geometric random variables have mean $\ga \in (0, \infty)$ and are supported on $\{0, 1, 2, \dots\}$.
		
		This distance function can be extended to be defined on $V \cup E$ rather than just on $E$.
		This is done by instead looking at the metric induced by the function given by $d_0((x,y),y)=w(y)$ and  $d_0(x,(x,y))=0$, see Example \ref{Ex:graphs}.
	\end{example}
	
	For studying Examples \ref{Ex:SJ} and \ref{Ex:geom-exp}, our labelling of the points in $\Z^2$ will be as in Figure \ref{fig:sepp-joh}, where indices increase as we move south and east. There are natural reasons for making this unusual choice that are related to the combinatorics of last passage percolation, see Section \ref{S:lpp-cadlag}. 
	
	\begin{figure}
		\centering
		\includegraphics[scale=0.8]{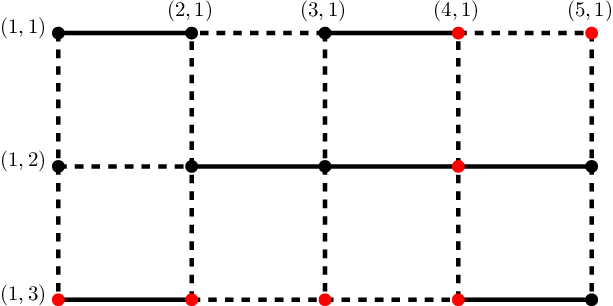}
		\caption{An instance of the Sepp\"al\"ainen-Johansson model. The solid edges have weight $1$ and the dashed edges have weight $0$. The path of red vertices is a geodesic from $(1, 3)$ to $(5, 1)$.}
			\label{fig:sepp-joh}
	\end{figure}
	
	\begin{example}[Poisson last passage]\label{Ex:Poi} Let $P\subset \mathbb R^2$ be a discrete set of points, and for $a, c \in \R^2$ with $a_1\le c_1,a_2\le c_2$, let  
		$$d_0(a,c)= {\mathbf 1} \Big\{[a_1,c_1]\times [a_2,c_2] \mbox{ contains a point of $P$ different from }a\Big\}.
		$$
		When $P$ is {a rate-$1$} Poisson point process, the induced metric of negative sign on $\mathbb R^2$ is Poisson last passage percolation, see Definition \ref{D:planar-poisson}.   
	\end{example}
	
	\begin{example}[Semidiscrete metrics, Poisson lines, Brownian last passage percolation]\label{Ex:Brown}
		Let $f_i:\R\to \R,i\in \Z$ be a sequence of cadlag functions with no negative jumps. For $(x, m), (y, n) \in \R \X \Z$, let
		$$
		d_0((x,m),(y,n))=\begin{cases}
		f_m(y)-f_m({x}), \qquad &m=n, x<y\\
		f_n(x)-f_n(x^-) \qquad &m=n+1, x=y
		\end{cases}
		$$
		We call the directed metric of negative sign on $\R \X \Z$ induced by $d_0$ the {\bf semidiscrete metric} of $f$. 
		
			When the $f_i$ are independent Brownian motions, the semidiscrete metric of negative sign is Brownian last passage percolation, see Section \ref{SS:Brownian-LP}. 
		When the $f_i$ are independent Poisson counting processes, the semidiscrete metric is called last passage percolation across Poisson lines, see Section \ref{SS:Poisson lines}. It is equivalent to the Hammersley process on the lattice, see \cite{seppalainen1996hydrodynamic} Section 3.1, and \cite{ferrari2005multiclass}.
	\end{example}

	All examples introduced so far can alternately be described by optimizing over paths. For example, in the setting of Example \ref{Ex:Brown}, let $x \le y$ and $n \ge m$. A {\bf path} $\pi$ from  $(x,n)$  to $(y,m)$ is a union of closed intervals
	\begin{equation*}
	[t_i,t_{i-1}]\times \{i\}\subset   \mathbb R \times \Z, \qquad i=m,m+1,\ldots,n,
	\qquad x=t_n\le t_{n-1}\le \cdots \le t_m \le  t_{m-1}=y.
	\end{equation*}
	The \textbf{length} of $\pi$ is given by
	\begin{equation}
	\label{E:pif-sumim}
	|\pi|_f =  \sum_{i=m}^{n} f_{i}(t_{i-1}) - f_i(t_{i}^-).
	\end{equation}
	\textbf{Geodesics} in these models are paths that maximize length among paths between the same endpoints. When ${f_n}$ is continuous at $x$, then $d((x,{n}),(y,{m}))$ is the maximal path length.

	\begin{example}
		\label{Ex:landscape-intro}
		The directed landscape $\scrL$ is a random directed metric of negative sign on $\R^2$. It is continuous on the parameter space
		$$
		\Rd = \{(x, s; y, t) \in \R^4 : s < t\},
		$$
		and satisfies $\scrL(u; v) = -\infty$ for all $(u; v) \notin\Rd$ with $u \ne v$.
		Because of this, we can think of $\scrL$ as a {\it spacetime} metric, which only assigns finite distances to points in the right time order. Paths and geodesics in $\scrL$ only move forward in time, and between any pair of points $(p, q) \in \Rd$, {there is almost surely} a unique $pq$-geodesic in $\scrL$.
		The directed landscape is independent on disjoint time strips $\R \X [s_1, t_1], \dots, \R \X [s_{{k}}, t_{{k}}]$ and has natural scale and translation invariance properties. See Definition \ref{D:DL-def} and surrounding discussion for more detail.
	\end{example}

	\subsection{The main theorem for classical models}
	\label{SS:classical-intro}
	
	In the classical integrable last passage models above (Examples \ref{Ex:SJ}, \ref{Ex:geom-exp}, \ref{Ex:Poi}, and Poisson lines and Brownian last passage percolation in Example \ref{Ex:Brown}), we can establish convergence to the directed landscape.
	The following table gives scaling parameters, which are built from a direction $\rho \in (0, \infty)$. {Here $\|v\|_d$ denotes the deterministic limit shape.}
	For the geometric distribution with mean $\gamma$, set $\bar \gamma=\sqrt{\gamma(\gamma+1)}$. For the Sepp\"al\"ainen-Johansson model with Bernoulli-$p$ variables, we set $\lambda=p/(1-p)$, and additionally assume that $\rho > 1/\lambda$. {This is the curved part of the limit shape; outside this region the Tracy-Widom scaling does not apply.} 
	\[\def\arraystretch{2}\setlength{\arraycolsep}{3pt}
	\begin{array}{l|cc|ccc}
	&  \|(x,-y)\|_d  &  \chi^3  &  \alpha  &  \beta  &  \chi/\tau^2 \\ \hline
	\text{S-J} & \frac{{[(\sqrt{x \la} - \sqrt{y})^+]^2}}{\lambda +1} &
	-\frac{\sqrt{\la}(\sqrt{\la \rho} - 1)^2(\sqrt{\rho} + \sqrt{\la})^2}{\sqrt{\rho}(\la+1)^3} &
	\frac{(\sqrt{\rho\la}-1)^2}{\lambda +1} &
	\frac{\la-\sqrt{\lambda/\rho
	}}{\lambda +1} &
	\frac{-\sqrt{\lambda }}{4(
		\lambda +1) \rho ^{3/2}}\\
	\text{Geometric} & (x+y)
	\gamma +2 \sqrt{x y}\bar \gamma &
	\frac{ \bar \gamma\left(\bar \ga(1 + \rho) +  (2 \gamma +1)\sqrt{\rho}
		\right)^2}{\sqrt{\rho} } & 
	\gamma(\rho+1) +2\bar \gamma 
	\sqrt{\rho}  & \gamma
	+\frac{\bar \gamma }{\sqrt{\rho }} &
	\frac{\bar \gamma }{4\rho ^{3/2}} \\
	\text{Exponential
	} & (
	\sqrt{x}+\sqrt{y})^2 &
	\frac{\left(\sqrt{\rho
		}+1\right)^4}{\sqrt{\rho }}
	& (\sqrt{\rho
	}+1)^2 &
	1+\frac{1}{\sqrt{\rho }} &
	\frac{1}{4 \rho ^{3/2}} \\
	\text{Poisson} & \|(x, y)\|_d = 2 \sqrt{xy}	& \sqrt{\rho } & 2
	\sqrt{\rho } &
	\frac{1}{\sqrt{\rho }} &
	\frac{1}{4 \rho ^{3/2}} \\
	\text{Brownian} & 2 \sqrt{x
		y} & \rho ^{3/2} & 2
	\sqrt{\rho } &
	\frac{1}{\sqrt{\rho }} &
	\frac{1}{4 \rho ^{3/2}} \\
	\text{Poisson lines} & x + 2 \sqrt{x y}
	&\sqrt{\rho}(1 + \sqrt{\rho})^2 & \rho + 2
	\sqrt{\rho } &
	1 + \frac{1}{\sqrt{\rho }} &
	\frac{1}{4 \rho ^{3/2}} \\
	\end{array}
	\]
	The table above uniquely determines the scaling parameters $\chi, \alpha, \beta$  and $\tau > 0$.  The value $\chi$ is positive for all models except the Sepp\"al\"ainen-Johansson model,  the only directed metric of positive sign in the table.
	
	\begin{theorem}\label{T:intro-main}
		Given $\rho$, consider either geometric, exponential, or Brownian last passage percolation with
		$\chi, \tau, \alpha, \beta$ in the table above.
		Let $v=(\rho,-1), u=(\tau,0)$. Then for any sequence of $\sigma \to \infty$, there is a coupling of identically distributed copies $d_\sigma$ of $d$ and the directed landscape $\scrL$ so that 
		\begin{equation}
		\label{E:d-sigma-intro}
		d_\sigma(x\sigma^2 u + s \sigma^3 v, y \sigma^2 u + t \sigma^3 v) = \alpha\sigma^3 (t-s)  + \beta \tau  \sigma^2(y-x)  +  \chi \sigma (\scrL +o_\sigma)(x,s;y,t), 
		\end{equation}
		where {$f^\pm=\max(\pm f,0)$ and} the random function $o_\sigma$ is small in the sense that for any compact set $K \sset \Rd$,  $\sup_K |o_\sigma|\to 0$ a.s.\ and there is $c>0$ so that with $\mathfrak{l}, \mathfrak{r}$ as in the table below, 
		{the centered passage fields satisfy}
		\begin{equation}
		\label{E:scaled-field-tails}
		{\limsup_{\sigma\to\infty}\E\exp\Big(c\sup_K\big[((\scrL+o_\sigma)^-)^{\mathfrak l}+((\scrL+o_\sigma)^+)^{\mathfrak r}\big]\Big)<\infty.}
		\end{equation}
		{Moreover, }
		\begin{equation}
		\label{E:osig-cvg}
		\E\exp \Big(c\sup_K {[}(o_\sigma^-)^{{\mathfrak l\wedge(3/2)}}+(o_\sigma^+)^{\mathfrak r}{]}\Big)\to  1.		\end{equation} {The same result holds for Poisson last passage percolation,} with $v = (\rho, 1)$ instead of $(\rho, -1)$.  
		For the Sepp\"al\"ainen-Johansson model and the Poisson line model, {the analogous convergence holds in the weaker hypograph topology, without the exponential moment conclusions}, see Sections \ref{SSS:intro-hypograph}, \ref{SS:Poisson lines} and \ref{SS:SJ} for details. 
	\end{theorem}
	\[\begin{array}{l|c|c}
	&  \text{Left tail exponent }\mathfrak l  & \text{Right tail exponent } \mathfrak r  \\ \hline
	\text{Geometric} & 3 & 1 \\
	\text{Exponential
	} & 3 & 1 \\
	\text{Poisson} & 3 & {1}  \\
	\text{Brownian} & 3/4 & 3/2
	\end{array}
	\]
	
	For the five discrete or semi-discrete models above, note that we need to modify the arguments of $d_\sig$ in \eqref{E:d-sigma-intro} to ensure that they lie in the appropriate discrete set, see Section \ref{S:integrable} for details. On the almost sure set where convergence holds in Theorem \ref{T:intro-main}, we have the following strong convergence of geodesics. 
	
	\begin{theorem}\label{T:intro-geod}
		Let $\pi_\sigma$ be the image of an arbitrary geodesic in $d_\sigma$ under the linear map satisfying $\sig^3 v \mapsto (0, 1)$ and $ \sig^2 u \mapsto (1, 0)$.
		Assume that as $\sig \to \infty$, the endpoints of $ \pi_\sigma$ converge to points $p, q$ with $(p,q) \in \Rd$.
		
		Then $\pi_\sigma$ is precompact in the Hausdorff topology, and on the almost sure set where there is a unique $pq$-geodesic $\pi$ in $\mathcal L$, $\pi_\sigma \to \pi$. In the four cases in Theorem {\ref{T:intro-main}} where $o_\sig \to 0$ compactly on $\Rd$, all subsequential limits of $\pi_\sigma$ are $pq$-geodesics in $\scrL$ even if there are multiple $pq$-geodesics.
	\end{theorem}
	
	In Theorem \ref{T:intro-geod} and in the sequel, we say that a sequence of functions $f_n \to f$ \textbf{compactly} if $f_n \to f$ uniformly on compact sets. {The estimate \eqref{E:osig-cvg} implies compact convergence of $o_\sig$ in probability.}
	In the context of Brownian last passage percolation, slightly weaker versions of Theorem \ref{T:intro-main} and Theorem \ref{T:intro-geod} were shown in \cite*{DOV}. We have included the results here for completeness. Theorems \ref{T:intro-main} and \ref{T:intro-geod} are proven in Section \ref{S:integrable}.
	
	In many models, the main direction $v$ in the above scaling can change with $\sigma$ arbitrarily, and we can still conclude hypograph convergence to the directed landscape and convergence of geodesics. See Section \ref{S:integrable} for the corresponding theorems. 
	
	\begin{remark}[Scaling heuristics]
		As $\sigma\to\infty$, 
		$d(\sigma^3p,\sigma^3q)/\sigma^3\to {\|q-p\|_d}$. The function $\|\cdot \|_d:\mathbb R^2 \mapsto \mathbb R\cup \{\infty\}$ ($-\infty$ if the metric $d$ is of negative sign)  is a directed norm: it satisfies the triangle inequality, $\|0\|_d=0$ and  $\|av\|_d=a\|v\|_d$ for $a>0$.  The parameters $\beta, \alpha$ and $\chi/\tau^2$ are determined by the Taylor expansion of the norm near $v$ in the first coordinate direction: 
		\begin{equation}\label{E:Taylor}
		\|v+(x,0)\|_d=\alpha + \beta x - \frac{\chi}{\tau^2} x^2 +o(x^2).
		\end{equation} 
		The magnitude of $\chi$ is given by the standard deviation $$\operatorname {SD}d(0,\sigma^3 v)/\sigma\to  |\chi|\,\operatorname{SD}\mathcal L(0,0;0,1),$$ where $\mathcal L(0,0;0,1)$ has GUE Tracy-Widom law. The sign of $\chi$ is the opposite of the sign of the  directed metric $d$.
	\end{remark}
	\begin{figure}[ht]
		\centering\includegraphics[scale=0.7]{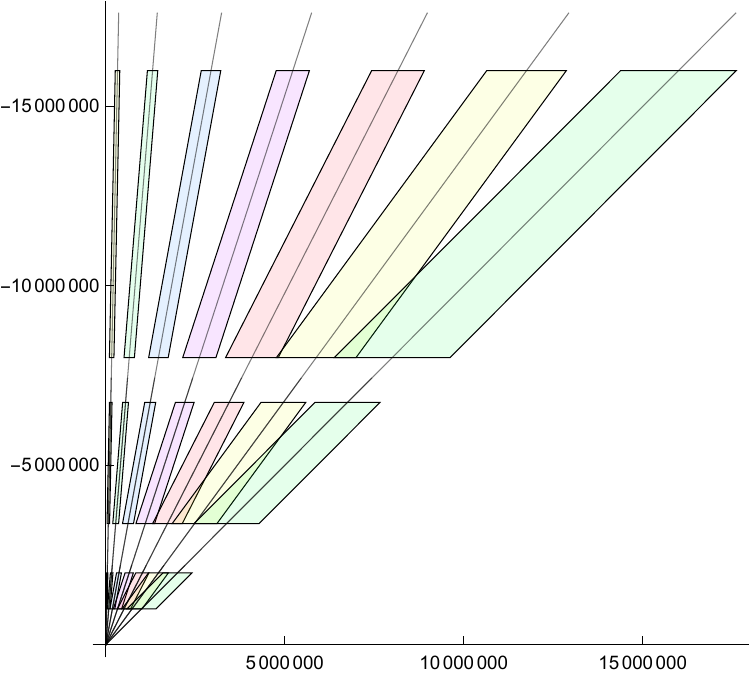}
		\caption{Scaling windows for geometric last passage percolation with mean $\ga = 1$ at slopes $\rho=(i/7)^2$, $i=1\ldots 7$ and $\sigma=100,200,300$.}
	\end{figure}

\begin{remark}[Choice of direction]
		Our choice of  the direction of $u=(\tau,0)$ as the  first coordinate vector $(1,0)$ is convenient for the proof. However, the direction $e$ can be an arbitrary vector linearly independent of  $v$. Then 
    the Taylor expansion of $\|v+xe\|_d$ determines the new parameters $\beta,{\chi/\tau^2}$ as in  \eqref{E:Taylor}, $u=\tau e$, while  $\alpha, \chi$ are unchanged. 
		When $e=(0,1)$, the space direction is vertical, rather than horizontal. 
		When  $e$ is a tangent vector to the $\|\!\cdot \!\|_d$-disk of radius $\|v\|_d$ at $v$, then $\beta=0$, and simpler asymptotics equivalent to \eqref{E:metric} hold.
		
		More generally, let ${v=(v_1,\varepsilon v_2)}$ where $v_1,v_2>0$ and {$\varepsilon=-1$, except for planar Poisson last passage, where $\varepsilon=1$.} Let  $e=(e_1,e_2)$, linearly independent from $v$. Then by a straightforward change-of-variables argument, Theorem \ref{T:intro-main} holds as follows: 
		\begin{equation}
		\label{E:d-sigma-intro2}
		d_\sigma(x\sigma^2 u + s \sigma^3 v, y \sigma^2 u + t \sigma^3 v) = \tilde \alpha\sigma^3 (t-s)  + \tilde \gamma  \sigma^2(y-x)  +  \tilde \chi \sigma (\scrL +o_\sigma)( x,\tilde \omega s; y,\tilde \omega t), 
	\end{equation}
where  $\alpha,\beta,\chi,\tau$ are from the table {with $\rho=v_1/v_2$} and 
$$
u=\tau e, \quad\tilde\alpha =  v_2 \alpha, \quad\tilde\eta={e_1-\varepsilon\frac{v_1}{v_2}e_2},\quad \tilde\gamma={\beta\tilde\eta\tau+\varepsilon\alpha u_2},
\quad \tilde \chi=\chi |\tilde\eta|^{1/2}, \quad \tilde \omega=v_2/|\tilde \eta|^{3/2}.
$$
{The induced maps on scaled coordinates send each compact set into a common compact set for all large $\sigma$, so the marginal tail bounds are preserved. The error estimate follows from compact convergence and these bounds, as in the proof of Theorem \ref{T:geom-lpp-emph}.}

{For clarity, we record some special cases with explicit constants. As $\sigma\to\infty$, each of the following six expressions converges in distribution to $\scrL(x,s;y,t)$, uniformly on compact subsets of $\{(x,s;y,t):0<s<t\}$. The exponential moment conclusions of Theorem \ref{T:intro-main} hold with the same exponents in these coordinates, as does the geodesic convergence of Theorem \ref{T:intro-geod}.

For Poisson last passage of intensity $1$ in up-right coordinates, we use the space--time convention of \cite*{prahofer2002scale}. In these coordinates the Poisson process has intensity $2$, and admissible paths are graphs of $1$-Lipschitz functions of time. Write $\widehat d_{\mathrm{Poisson}}$ for the corresponding last passage metric. Their droplet height is $h(x,t)=\widehat d_{\mathrm{Poisson}}(0,0;x,t)$. The change $(u,v)=(t+x,t-x)$ gives up-right coordinates, in which the Poisson intensity is $1$. The expression is
\begin{equation}\label{E:standard-poisson-scaling}
\frac{\widehat d_{\mathrm{Poisson}}(\sigma^2x,\sigma^3s;\sigma^2y,\sigma^3t)-2\sigma^3(t-s)}{\sigma}.
\end{equation}

Let $d_{\mathrm{Exp}}$ denote last passage with independent exponential weights of mean $1$ on $\mathbb Z_{\ge0}^2$, with edges directed north and east. Coordinates are rounded down. For diagonal and horizontal spatial directions, respectively, the expressions are
\begin{equation}\label{E:standard-exponential-diagonal}
\frac{d_{\mathrm{Exp}}(\sigma^3s+2^{\frac{2}{3}}\sigma^2x,\sigma^3s-2^{\frac{2}{3}}\sigma^2x;\sigma^3t+2^{\frac{2}{3}}\sigma^2y,\sigma^3t-2^{\frac{2}{3}}\sigma^2y)-4\sigma^3(t-s)}{2^{\frac{4}{3}}\sigma},
\end{equation}
\begin{equation}\label{E:standard-exponential-horizontal}
\frac{d_{\mathrm{Exp}}(\sigma^3s+2^{\frac{5}{3}}\sigma^2x,\sigma^3s;\sigma^3t+2^{\frac{5}{3}}\sigma^2y,\sigma^3t)-4\sigma^3(t-s)-2^{\frac{8}{3}}\sigma^2(y-x)}{2^{\frac{4}{3}}\sigma}.
\end{equation}

An equivalent diagonal normalization avoids fractional powers of $2$:
\begin{equation}\label{E:standard-exponential-diagonal-half}
\frac{d_{\mathrm{Exp}}(\tfrac12\sigma^3s+\sigma^2x,\tfrac12\sigma^3s-\sigma^2x;
\tfrac12\sigma^3t+\sigma^2y,\tfrac12\sigma^3t-\sigma^2y)-2\sigma^3(t-s)}{2\sigma}.
\end{equation}

Let $d_{\mathrm{Geom},q}$ denote last passage on the same lattice with independent weights satisfying $\p(W=k)=(1-q)q^k$, $k\ge0$, for fixed $q\in(0,1)$. The mean is $\gamma=q/(1-q)$, and coordinates are again rounded down. Here $\tau$ is the transversal constant from the table at $\rho=1$, and
\begin{equation}\label{E:standard-geometric-transversal}
\tau'=q^{-\frac{1}{6}}(1+\sqrt q)^{\frac{2}{3}},\qquad \tau=2\tau'.
\end{equation}
For diagonal and horizontal spatial directions, respectively, the expressions are
\begingroup
\medmuskip=3mu
\begin{equation}\label{E:standard-geometric-diagonal}
\frac{(1-\sqrt q)d_{\mathrm{Geom},q}(\sigma^3s+\tau'\sigma^2x,\sigma^3s-\tau'\sigma^2x;\sigma^3t+\tau'\sigma^2y,\sigma^3t-\tau'\sigma^2y)-2\sqrt q\,\sigma^3(t-s)}{q^{\frac{1}{6}}(1+\sqrt q)^{\frac{1}{3}}\sigma},
\end{equation}
\endgroup
\begin{equation}\label{E:standard-geometric-horizontal}
\frac{(1-\sqrt q)d_{\mathrm{Geom},q}(\sigma^3s+\tau\sigma^2x,\sigma^3s;\sigma^3t+\tau\sigma^2y,\sigma^3t)-2\sqrt q\,\sigma^3(t-s)-\sqrt q\,\tau\sigma^2(y-x)}{q^{\frac{1}{6}}(1+\sqrt q)^{\frac{1}{3}}\sigma}.
\end{equation}}
	\end{remark}
	\begin{remark}[Possible projects]
		Some of the results above can be improved. The Brownian left tail exponent $3/4$ is not optimal. This is because the directed metric takes negative values, and so more care is needed in the chaining argument, see Section \ref{S:graph-cvg}. For the Poisson lines and Sepp\"al\"ainen-Johansson models, tail bounds are needed to get the stronger version of convergence. This could be done in two ways:  (i) analysis of the contour integral formulas, (ii) by a combinatorial connection to other models where tail bounds are known. 
	\end{remark}
	
	\subsection{Outline of the proof}
	
	{The proof of Theorem \ref{T:intro-main} has three steps.}
	\begin{enumerate}[label={\arabic*.}]
		\item {\emph{Convergence to the Airy line ensemble.} Fix the starting point and the final time, and let the final spatial coordinate vary. The rescaled multi-path last passage values converge to the parabolic Airy line ensemble. This is the input from the companion paper \cite*{DNV}.}
		\item {\emph{Convergence to the Airy sheet.} Keep the initial and final times fixed, but now allow both spatial endpoints to vary. The RSK isometry expresses these passage values in terms of the line ensembles from Step 1. We show that their convergence implies convergence of the two-variable passage field to the Airy sheet $\scrS(x,y)$. This is proved in Section \ref{S:sheet}, following the preliminaries in Sections \ref{S:lpp-cadlag}--\ref{S:ALE-S}.}
		\item {\emph{Convergence to the directed landscape and its geodesics.} Independence on disjoint time intervals and composition of passage values give multi-time convergence to $\scrL$ in Section \ref{S:landscape-cvgfdd}. We then strengthen this to hypograph convergence using a maximal inequality in Section \ref{S:hypo}, and to graph convergence under an additional moment assumption in Section \ref{S:graph-cvg}. The deterministic results of Section \ref{S:geodesics} turn these metric convergence statements into convergence of geodesics.}
	\end{enumerate}
	{Sections \ref{S:cvg-lattice}--\ref{S:integrable} put these steps together for the lattice and integrable models.}
	
	\subsection{RSK}

	The key ingredient for all our proofs is a `Greene's theorem formulation' of the classical Robinson-Schensted-Knuth correspondence (RSK). In this paper we use a version of this correspondence, introduced in \cite*{DNVcadlag}, that unites classical RSK, dual RSK and continuous path RSK. We will also be able to take limits without leaving this universe. 
	
	Let $\scrD^n_0$ be the space of $n$-tuples of cadlag functions from $[0,\infty)\to\mathbb R$ without negative jumps. 
	For continuous $f\in \scrD^n_0$, recalling the notion of path length from \eqref{E:pif-sumim}, we can define 
	$$
	f[(x, m)^k\to (y, n)^k] = \sup \sum_{i=1}^k |\pi_i|_f,
	$$
	where the supremum is over all sets of $k$ paths $\pi_1, \dots, \pi_k$ from $(x, m)$ to $(y, n)$ that are disjoint away from $x$ and $y$. The same definition works for general $f \in \scrD^n_0$ but with a few extra technicalities, see Section \ref{S:multipoint} for details. In our setting the RSK transform maps a function $f\in \scrD^n_0$ to its \textbf{melon} $Wf$ as follows. 
	\begin{definition}\label{D:intro-melon} Let $f[p^0\to q^0]=0$, and for $k \ge 1$, set $f[p\to_{\De k} q] = f[p^k \to q^k] - f[p^{k-1}\to q^{k-1}]$.
		The {\bf RSK} or {\bf melon} map $W:\scrD^n_0\to\scrD^n_0$ is defined by 
		$$
		Wf_k(y)=f[(0, n)\to_{\De k} (y, 1)].
		$$   
	\end{definition}
	The RSK map has two remarkable properties. First, it is an isometry.
	\begin{theorem}[RSK isometry{; see Proposition 3.12(i) in \cite*{DNVcadlag}}]\label{T:intro-RSK-isom} For all $0 \le x\le y$,
		\begin{equation}
		\label{E:isometry}
			f[(x,n)\to (y,1)]=Wf[(x,n)\to (y,1)].
		\end{equation}
	\end{theorem}
	{The more general form of this isometry that we use below is recalled in Proposition \ref{P:W-facts}(i).}
	{Second, certain nice measures} on $\scrD^n_0$ have tractable pushforwards under $W$. For example, independent Brownian motions get mapped to nonintersecting Brownian motions under $W$, see Theorem $7$ in \cite*{o2002representation}.
	Such properties make RSK a perfect tool for the probabilistic analysis of directed geometry problems.

	In particular, we may hope to understand limits of random  directed metrics in the plane  by taking the limits of the RSK isometry \eqref{E:isometry}. This project was  carried out for Brownian last passage percolation in \cite*{DOV}.  One goal of this paper is to make it applicable in a more general setting. 
	
	\subsection{The scaling limit of melons and the Airy sheet}
	
	In integrable settings, the melon side of the RSK isometry is an ensemble of nonintersecting random walks. The natural limit of nonintersecting random walks is an infinite sequence of nonintersecting Brownian motions known as the parabolic Airy line ensemble. This object was first described by \cite*{prahofer2002scale} as the scaling limit of the polynuclear growth model, and realized as a system of nonintersecting locally Brownian functions by \cite*{CH}. 
	The parabolic Airy line ensemble is a random sequence of functions $\scrA = \{\scrA_1> \scrA_2 > \ldots\}$. The process $\mathcal A(x) = \scrA(x)+x^2$ is stationary in $x$, and on every finite interval each $\scrA_i$ is absolutely continuous with respect to Brownian motion. 
	
	In the companion paper \cite*{DNV} it was shown that in many models, the limit of the melon is the parabolic Airy line ensemble. One of the main theorems of the present paper is that convergence to the Airy line ensemble implies convergence of the melon side of the RSK isometry \eqref{E:isometry} to a last passage problem in $\scrA$. The resulting process characterizes the Airy sheet, see Definition \ref{D:airy-sheet-mble} and Figure \ref{fig:Airy sheet} for more details. 

	\begin{figure}
		\centering
		\includegraphics[scale=0.16]{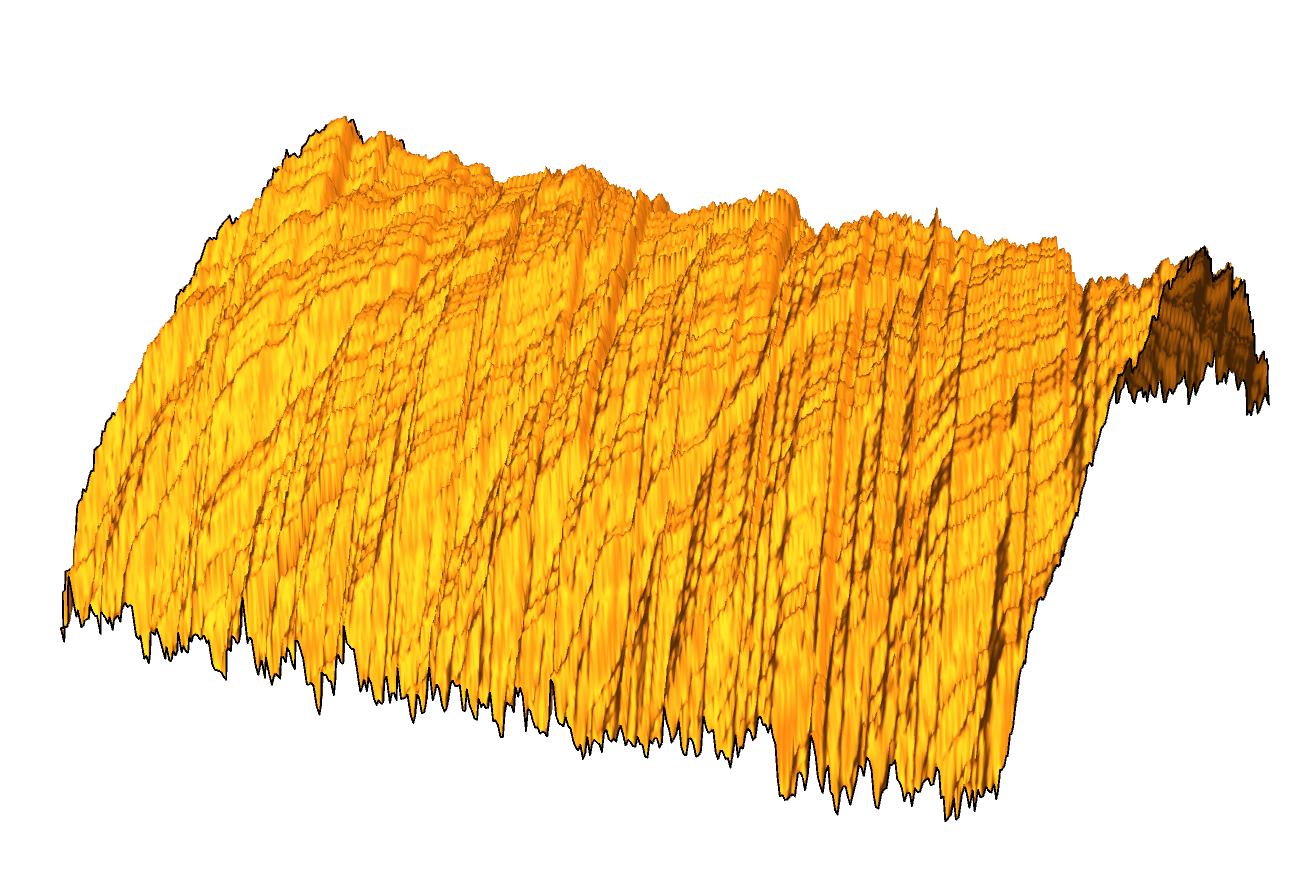}\hspace{5em}
		\includegraphics[scale=0.13]{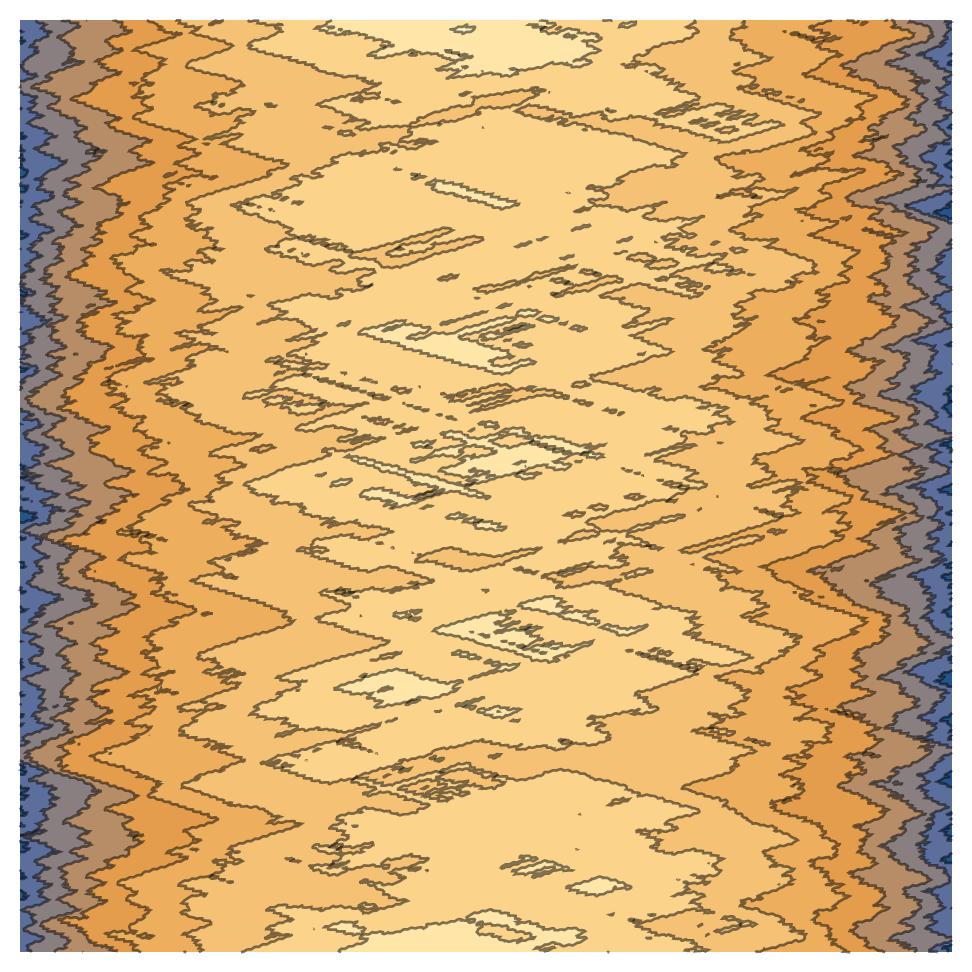}\hspace{2em}
		\caption{The Airy sheet $\mathcal S(x,y)$ function plot and contour plot. The ridge corresponds to $x=y$. See the decay $\mathcal S(x,x+y)\sim -y^2$. }
		\label{fig:Airy sheet}
	\end{figure}
	
	\subsection{General convergence to the Airy sheet}
	 
	\begin{theorem}\label{T:intro-sheet} For each $n$, let $F_n$ be a sequence of $n$ random cadlag functions as in Example \ref{Ex:Brown}.
		Let $a_n\le 0 $, $c_n\in \mathbb R$. If for every $z \in \R$, the sequences of random functions
		\begin{align*}
		F_n[(z + a_n, n) \to_{\Delta k} (z + \cdot, 1)] + c_n, \quad \mathand \quad
		F_n[(z + a_n - \cdot, n)  \to_{\Delta k} (z, 1)] + c_n
		\end{align*}
		{each} converge in law in the product of Skorokhod topologies {over $k$} to the parabolic Airy line ensemble, then in some coupling of the processes $F_n$ and an Airy sheet $\scrS$ we have   
		$$
		F_n[(x+ a_n, n) \to (y, 1)] + c_n \to \scrS(x,y) \qquad \mbox{compactly a.s.}
		$$
	\end{theorem}
	Theorem \ref{T:intro-sheet} is an immediate consequence of Theorem \ref{T:airy-sheet-gen}.
	A key ingredient in the proof of this theorem is  tightness for convergence to  the Airy sheet $\mathcal S$. This follows from tightness of the single-variable marginals, and the fact that both $\mathcal S$ and its prelimits share a useful property: they are cumulative distribution functions of the corresponding {\bf shock measures}, random measures on $\R^2$. See Section \ref{S:sheet}  for details. 
	
	\subsection{From two times to all times to geodesics}
	
	Convergence to the Airy sheet requires no additional assumptions beyond Skorokhod convergence to the Airy line ensemble. 
	
	The {\bf directed landscape} $\scrL$ is defined from independent Airy sheets analogously to how Brownian motion is defined from independent normal distributions, see Definition \ref{D:DL-def} for more details. 
	
	Because of this, in many i.i.d\ models convergence to the Airy sheet immediately implies a weak type of finite-dimensional distribution convergence that we refer to as \textbf{multi-time convergence}. This level of convergence does not imply a strong notion of convergence of geodesics, so in Section \ref{S:topologies} we consider two stronger topologies. In Sections \ref{S:hypo} and \ref{S:graph-cvg} we give sufficient conditions for tightness in these topologies. 
	
	The first topology is essentially compact convergence. Since compact convergence is not a separable topology on discontinuous functions, and our prelimits are often discontinuous, we replace it by a {\bf  graph convergence} in the localized Hausdorff metric. This convergence requires an extra moment assumption, but it implies a very strong form of convergence of geodesics. 
	
	A weaker topology is that of {\bf hypograph convergence}. Translated to the language of functions, a sequence $f_n$ converges to a continuous function $f$ in this topology if $(f_n - f)^+ \to 0$ uniformly on compact sets, and if for every $x$ we can find a sequence $x_n \to x$ such that $f_n(x_n) \to f(x)$, see Section \ref{S:topologies} for more precise statements.
	This topology has the advantage that for i.i.d.\ models, we need no extra assumptions in addition to Airy sheet convergence. It also implies a fairly strong form of geodesic convergence, but it allows for `holes' in the prelimit that geodesics tend to avoid. Graph convergence ensures that such holes do not exist, while the hypograph topology ignores them. 
	
	We give an example of a sequence of last passage models satisfying multi-time convergence but not hypograph convergence at the beginning of Section \ref{S:hypo}, and an example satisfying {hypograph convergence but not graph convergence} in Example \ref{Ex:elpp}.
	
	\subsubsection{Hypograph convergence}
	\label{SSS:intro-hypograph}
	The key ingredient for proving hypograph convergence is a last passage version of Doob's classical {\bf maximal inequality}. It is proven in the setting of semidiscrete metrics of negative sign applied to sequences of independent functions with independent increments, see Example \ref{Ex:Brown}. Lattice last passage models can be embedded as semidiscrete metrics, and hence also fit into this setting.
	\begin{theorem}\label{T:intro-maxineq}
		{Let $d$ be the semidiscrete metric of negative sign of independent cadlag functions with independent increments, as above, and write $d(A,B)=\sup_{a\in A,b\in B}d(a,b)$.}
		Let $A, B \subset \R \X \Z$ be finite unions of bounded product sets. For  $p, q \in \R \X \Z, c \in \R$ and $\ep > 0$ we have
		\begin{equation}
		\label{E:doop-boot-I}
		\; \p\lf( d(p, q) \le c - 2\ep, c < d(A, B) \rg)\le \;\sup_{a \in A} \p(d(p, a) < -\ep) +  \sup_{b \in B} \p \lf(d(b, q) < -\ep\rg).
		\end{equation}
	\end{theorem}
	The maximal inequality is proven as Lemma \ref{L:doob-lpp-boot}. The version in that lemma is slightly stronger than Theorem \ref{T:intro-maxineq} but more technical.
	Theorem \ref{T:intro-maxineq} reduces hypograph convergence for models satisfying Airy sheet convergence to a simple one-dimensional convergence requirement called {\bf mobile Tracy-Widom convergence}. This requirement is easy to check in all cases we study.
	{The next theorem applies to rescaled i.i.d.\ lattice last passage models, embedded as line last passage models as in Section \ref{SS:ind-hypo}.}
	
	\begin{theorem}\label{T:intro-hypo}
        {Let $d_n$ be rescaled i.i.d.\ lattice last passage metrics with nonnegative weights, with the embedding and scaling assumptions of Theorem \ref{T:ind-main}.}
        We have $d_n \cvgd \scrL$ in the hypograph topology provided that: \begin{itemize}
			\item (Mobile Tracy-Widom convergence) For any sequence $u_n \to u\in \Rd$, we have 
			$
			d_n(u_n) \cvgd \scrL(u).
			$
			\item (Airy sheet convergence) For all $s < t \in \R$ we have $d_n(\cdot, s; \cdot ,t)\cvgd \scrL(\cdot, s; \cdot, t)$ in the graph topology.
		\end{itemize}
	\end{theorem}
	See Theorem \ref{T:ind-main} for a stronger version.
	Hypograph convergence deterministically implies convergence of geodesics, see Theorems \ref{T:geod-cvg} and \ref{T:geod-cvg-stronger}. In particular, Theorem \ref{T:geod-cvg-stronger} implies that if $d_n \to \scrL$ almost surely in the hypograph topology and $u_n\to u \in \Rd$, then on the almost sure set where there is a unique $u$-geodesic $\pi$ in $\scrL$, any sequence of $u_n$-geodesics converges to $\pi$ in the Hausdorff topology. 
	
	\subsubsection{Compact convergence}
	
	Graph convergence is a separable substitute to compact convergence. It is equivalent to compact convergence when the limit is continuous, as is the case in our setting. 
	
	{We show compact convergence in the context} of i.i.d.\ lattice last passage environments, Example \ref{Ex:geom-exp}.
	\begin{theorem}\label{T:intro-compact} {Let $d_n$ be the rescaled metrics of a fixed i.i.d.\ nonnegative lattice last passage environment, with fixed parameters $\rho,\al,\be,\chi,\tau>0$ and the embedding, centering and $1$-$2$-$3$ scaling of Theorem \ref{T:lattice-models-i}. Suppose that $d_n\cvgd\scrL$ in the multi-time sense.} Let $U={\{(0,0;x,1):-3\le x \le 3\}}$. If for some $\delta>0$,
		\begin{equation} 
		\label{E:limsup}
		\limsup_{n \to \infty } \,\sup_{u\in U} {\E [(d_n(u)^-)^{5 + \de}]} < \infty,
		\end{equation}
	then in some coupling $d_n\to \scrL$ compactly on $\Rd$. 
	\end{theorem}
	
	Theorem \ref{T:intro-compact} follows from Theorem \ref{T:lattice-models-i}, see Remark \ref{R:intro-cct}. {The constant $3$ accommodates the spatial increments in the chaining argument, see \eqref{E:KnA1-explicit} and the discussion following Theorem \ref{T:lattice-models-i}.} {Theorem \ref{T:intro-compact} shows that $5+ \de$ moments are sufficient for compact convergence.}
	
	Again, compact convergence deterministically implies convergence of geodesics. The type of convergence is stronger than in the hypograph setting. Indeed, Theorem \ref{T:geod-cvg} shows that if  $u_n\to u \in \Rd$ and $d_n(u_n) \to \scrL(u)$ then any sequence of $u_n$-geodesics in $d_n$ is precompact in the Hausdorff topology, and all subsequential limits  are $u$-geodesics in $\scrL$.
	
	One important advantage of compact convergence is that we get Airy sheet limits in all transversal directions. For last passage percolation across rectangles, this typically means that in a suitable scaling, convergence of left-to-right last passage values to the Airy sheet also implies convergence of bottom-to-top last passage values to the Airy sheet. 

	\subsection{The longest increasing subsequence}
	
	\cite*{ulam1961monte} asked what the length $L_n$ of the longest increasing subsequence in a uniform random $n$-element permutation is.
	\cite*{hammersley1972few} initiated the study of this problem in earnest, relating this length to a last passage value and established the asymptotics $L_n \sim c\sqrt{n}$.  \cite*{vershik1977asymptotics} and \cite*{logan1977variational} independently determined the value $c=2$, and \cite*{baik1999distribution} identified the limiting fluctuations as the Tracy-Widom random variable $\scrL(0, 0; 0, 1)$. Our main theorem concerns the {\it shape} of the sequence. 
	
	To state this theorem, recall from Example \ref{Ex:landscape-intro} that there is almost surely a unique geodesic in the directed landscape from $(0,0)$ to $(0, 1)$. This geodesic is of the form $\{(\Pi(t), t) : t \in [0, 1]\}$ where $\Pi:[0, 1] \to \R$ is a continuous function, Figure \ref{fig:directed}.
	
	\begin{figure}
		\centering
		\includegraphics[scale=0.4]{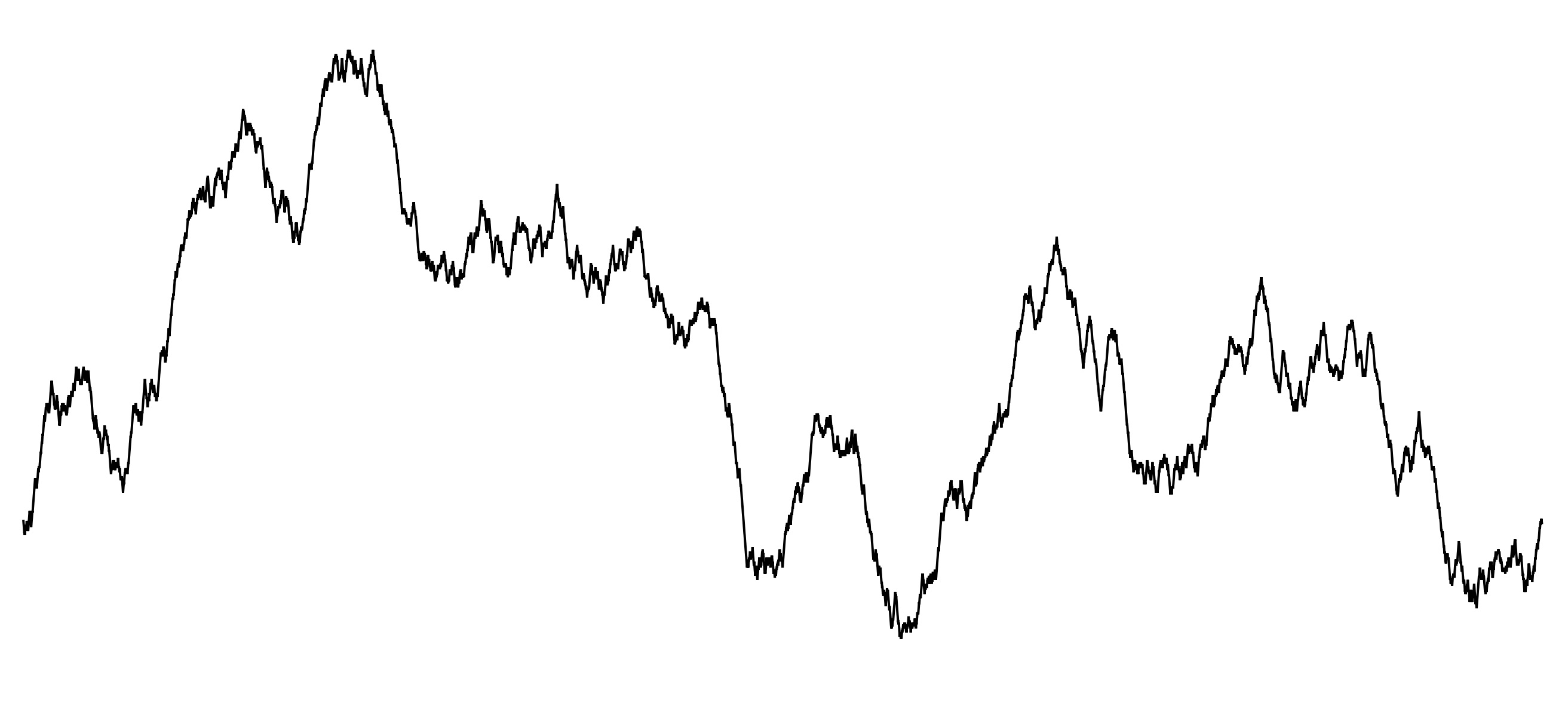}
		\caption{The directed geodesic, the scaling limit of the longest increasing subsequence. }
		\label{fig:directed}
	\end{figure}

	\begin{theorem}\label{T:intro-LIS}
		Let $I_n$ be a longest increasing subsequence in a uniform random permutation of $\{1,\ldots, n\}$. We think of $I_n$ as a function from $\{1, \ldots, L_n\} \to \{1, \ldots, n\}$, and set {$I_n(0)=0$ and} $I_n(i)=n$ for $i> L_n$. Then in some coupling of the $I_n$ and $\scrL$, 
		$$
		I_n(\lfloor 2t\sqrt{n}\rfloor) = nt + {n^{5/6}} \Pi(t) + n^{5/6}\ep_n(t),
		$$
		where the error $\ep_n \to 0$ uniformly almost surely on $[0, 1]$.
	\end{theorem}
	
	Note that since $L_n = 2\sqrt{n} + O(n^{1/6})$, we could alternately replace the $\lfloor 2t\sqrt{n}\rfloor$ in the argument of $I_n$ above by $\floor{L_n t}$.
	Theorem \ref{T:intro-LIS} is proven as part of Theorem \ref{T:subsequence} in Section \ref{SS:LIS}. In that section we also show that two distinct longest increasing subsequences converge to the same geodesic, and also establish a Kolmogorov-Smirnov test for longest increasing subsequences in i.i.d.\ sampling: the limit is the directed geodesic, rather than a Brownian bridge. These theorems are all consequences of convergence of geodesics in Poisson last passage percolation, Theorem \ref{T:intro-geod}.
	
	\subsection{The limit of TASEP}
	\label{SS:TASEP-intro}
	
	For the classical height function representation, let $\mathcal H$ be the set of functions $h:\mathbb Z \mapsto \mathbb Z$ so that $h(0)$ is even, and $|h(i)-h(i-1)|=1$ for all $i$. Continuous-time TASEP $h_t, t \in [0, \infty)$ is a Markov process on $\mathcal H$ in which each local maximum decreases by $2$ independently at rate 1. 
	
	There are several natural couplings of TASEP in which all initial conditions are driven by a same noise. In each, the evolution can be interpreted as a discrete Hamilton-Jacobi equation in a noisy environment. {Next, we construct the scaling limit of the coupling inherited from exponential last passage percolation.}
	
	Let $X$ be an array of independent, mean $1$ exponential random variables indexed by even lattice points, i.e. $x\in \Z^2$ with $x_1 + x_2 \in 2 \Z$. For any initial condition $h_0$, we can define a TASEP evolution $h_t$ started from $h_0$ by letting $X(\kappa,x)$ be the increment between the first time that $h$ has a local maximum of value $\kappa$ at location $x$ (if this ever happens) and the time when $h(x)$ decreases to $\kappa-2$. The evolution of TASEP from any initial condition given $X$ is deterministic.   
	We show that in the scaling of functions   $(\iota_nf)(x)={f}(2n^{2/3}x)/n^{1/3}$, {the limit of coupled TASEP is given by the variational formula with a common directed landscape.} 
	
	The corresponding limit theorem for a single TASEP was shown by \cite*{matetski2016kpz}, Theorem 3.13. We start by stating their theorem, reformulated in our language based on \cite*{nica2020one} and \cite*{DOV}. 
	\begin{theorem}[\cite*{matetski2016kpz}]
		\label{T:matetski-kpz}
		Let $T \sset (0, \infty)$ be finite or countable. Let $h_0: \R \to \R \cup \{-\infty\}$, $h_0\not \equiv -\infty$ be fixed {and upper semicontinuous}. There is a coupling of copies $X_n$ of $X$ and  a directed landscape $\scrL$ so that on an event of probability $1$ for all $t\in T$ the following holds. 
		
		For any deterministic sequence of functions $h_{n,0}\in \mathcal H$ the TASEP evolution $h_{n,t}$ satisfies
		\begin{equation}\label{E:TASEP-limit}
		\iota_n (h_{n,2nt}+nt)(y) \to \sup_{x\in \mathbb R} {[}h_0(x)+\scrL(x,0;y,t){]}
		\end{equation}
		uniformly on compact subsets of $y\in \mathbb R$ as long as
		\begin{itemize}
			\item $ \iota_n h_{n,0}\to  h_0$, in the hypograph topology (see Section \ref{S:topologies}).
			\item there exists $s, a \in \R$ such that for all $n$  large enough, $\iota_n h_{n,0}(x) < a + s |x|$  for all $x$.
		\end{itemize}
	\end{theorem}
	
	While this theorem is about the Markov process of height functions, our strengthening is about the driving noise of the discrete Hamilton-Jacobi equation.
	\begin{theorem}[Scaling limit of coupled TASEP]
		\label{T:intro-TASEP} In the setting of Theorem \ref{T:matetski-kpz}, there is a coupling of $X_n, {\scrL}$ such that almost surely, \eqref{E:TASEP-limit} holds simultaneously for all $h_{n, 0}, h_0$ satisfying the bullet points.
\end{theorem}
	
	In the discrete setting TASEP is the evolution of the boundary of a growing disk in exponential last passage percolation. In the scaling limit, disk boundaries become straight lines, and the shapes of disks become second order effects, captured in the optimization problem \eqref{E:TASEP-limit}. This phenomenon of {\bf interfaces} is explored in Section \ref{S:interfaces}. In Section \ref{S:TASEP}, we use this interface machinery to translate convergence results about exponential last passage percolation into Theorem \ref{T:intro-TASEP}. Section \ref{S:TASEP} also contains a few other TASEP convergence results, and a discussion of discrete-time TASEP.
	
	\subsection{The Airy sheet conjecture}
	
	{The coupling in \cite*{DOV} determines $\scrS(x,y)-\scrS(x,z)$ from $\scrA$ for $x>0$. To recover the whole sheet from the same ensemble, we must show that the corresponding construction in $\scrA(-\cdot)$ gives these differences for $x<0$ in the same Airy sheet. We obtain this using convergence of a symmetric lattice model and a comparison of multi-path last passage values, which identifies the two limiting line ensembles as reflections of one another. Ergodic averaging, as in the original definition, recovers the additive constants, so no additional randomness is needed to construct the whole sheet.} In particular, we affirm Conjecture 14.1 in \cite*{DOV} and show the following.
	\begin{theorem}\label{T:sheet-map} The Airy sheet is an explicit deterministic function of the Airy line ensemble. 
	\end{theorem}
	More precisely, the following construction is equivalent to  Definition \ref{D:airy-sheet}.
	
	\begin{definition}\label{D:intro-sheet}
		\label{D:airy-sheet-mble}
		Let $\scrA$ be a parabolic Airy line ensemble, and let $\scrA^*(x) = \scrA(-x).$ For $(x, y, z) \in \Q^+ \X \Q^2$ let $x_k=(-\sqrt{k/(2x)}, k) $ and let
		\begin{equation*}
		\begin{split}
		\scrS'(x, y, z) &= \scrA[x_k \to (y, 1)]- \scrA[x_k \to (z, 1)],\\
		\scrS'(-x, y, z) &= \scrA^*[x_k \to (-y, 1)]- \scrA^*[x_k \to (-z, 1)],
		\end{split}
		\end{equation*}
		for all  large $k$, as the right hand sides  stabilize when $k \to \infty$. For $(x, y) \in {(\Q \setminus \{0\})} \X \Q$, define
		\begin{equation*}
		\scrS(x, y) = \E\scrA_1(0)+\lim_{n \to \infty} \frac{1}{{2n+1}} \sum_{z=-n}^{n}{[}\scrS'(x, y, z) - (z - x)^2{]}.
		\end{equation*}
		The {\bf Airy sheet} is the unique continuous extension of $\scrS$ to $\mathbb R^2$. 
	\end{definition}
	Theorem \ref{T:sheet-map} is proven in Section \ref{SS:symmetries}. 
	
	\subsection{Symmetries of the directed landscape}
	
	Some symmetries of the directed landscape were not shown in  \cite*{DOV} because the corresponding symmetries were not present in the prelimit. We rectify this in Section \ref{SS:symmetries}.
	
	\begin{prop}\label{P:intro-sym}
		As a function of $x, s, y,$ and $t$, the directed landscape $\scrL$ satisfies
		\begin{equation*}
		\scrL(x, s; y, t) \eqd \scrL(-x, s; -y, t) \eqd \scrL(y, -t; x, -s).
		\end{equation*}
	\end{prop}
	
	\subsection{More related work}
	\label{SS:related}
	While the full construction of the directed landscape is quite recent, several aspects of it and of related KPZ models have been studied previously. We give a partial literature review here, focusing on the works most relevant to the current paper that are not discussed earlier in the introduction.
	For a gentle introduction suitable for a newcomer to the area, see  \cite*{romik2015surprising}. Review articles and books focusing on more recent developments include  \cite*{corwin2016kardar};  \cite*{ferrari2010random};  \cite*{quastel2011introduction};  \cite*{weiss2017reflected};  \cite*{zygouras2018some}.
	
	Marginals of the directed landscape have been studied using exact formulas, including the works of \cite*{baik1999distribution}, \cite*{prahofer2002scale}, and \cite*{matetski2016kpz} discussed above. In addition to these, \cite*{johansson2019multi} and \cite*{liu2019multi} independently found formulas for the joint distribution of $\{\scrL(p, q_i) : i \in \{1, \dots, k\}\}$, building on work of  \cite*{johansson2017two, johansson2018two} and  \cite*{baik2017multi}.
	
	These papers provide a strong integrable framework for understanding $\scrL$ and other KPZ models. There has also been a large amount of research investigating probabilistic and geometric properties of these models. For example, prior to \cite*{DOV}, tightness and regularity estimates for the directed landscape were obtained by \cite*{hammond2018modulus, hammond2017modulus, pimentel2017local}. \citet*{corwin2015renormalization} also gave geometric heuristics predicting many properties of the directed landscape.
	
	{One of the most important} probabilistic contributions to this area is the work of \cite*{CH}, which shows that the parabolic Airy line ensemble satisfies a certain Gibbs resampling property. We used this property in \cite*{DV} to further understand the parabolic Airy line ensemble, and in \cite*{DOV} used this understanding to build the Airy sheet. The works of \cite*{hammond2016brownian, hammond2017modulus, hammond2017patchwork, calvert2019brownian} further refine the understanding of the Airy line ensemble using the Gibbs property.
	
	Since the paper \cite*{DOV} was posted, there has also been significant activity both on understanding the directed landscape itself and on using the directed landscape to better understand the KPZ fixed point. \cite*{bates2019hausdorff} investigated the structure of exceptional points with multiple $\scrL$-geodesics, building on work of \cite*{hammond2017exponents} and \cite*{basu2021fractal}, and \cite*{dauvergne2020three} showed that $\scrL$-geodesics have a deterministic three-halves variation. 
	
	The construction of the Airy sheet in terms of the Airy line ensemble has been used by \cite*{sarkar2020brownian} to prove Brownian absolute continuity of the KPZ fixed point, by \cite*{ganguly2021local} to show that the difference profile in the Airy sheet is related to Brownian local time, and the prelimiting version has been used  by \cite*{corwin2021exceptional} to study exceptional times in the KPZ fixed point. An extension of this construction has also been developed by \cite*{dauvergne2021disjoint} to study disjoint optimizers in $\scrL$ and build a theory of multi-point last passage values in the limit. 
	
	While this paper focuses on understanding limits of random metrics, positive temperature versions of these objects known as random polymers should also converge to the directed landscape. As with random metrics, a small collection of random planar polymer models have an integrable structure related to the geometric RSK correspondence, e.g. see \cite{o2001brownian, seppalainen2012scaling}. Limits of these models yield the continuum random polymer, which is related to the KPZ equation via a Feynman-Kac formula, see \cite{alberts2014intermediate}. While the integrable structure of random polymer models is not as tractable as that of last passage models, \cite{virag2020heat} showed convergence in the random polymer setting to the full KPZ fixed point. Combining this result with the Gibbs ensemble characterization theorem in \cite{dimitrov2021characterization} and tightness results from \cite{corwin2016kpz,dimitrov2021tightness} gives Airy line ensemble convergence for these models.
	
	\subsection{A short guide to the paper}
	
	There is admittedly a lot to unpack in the paper, so here is a short guide to orient the reader. Much of this is also discussed earlier in the introduction.
	
	Sections \ref{S:lpp-cadlag}-\ref{S:sheet} cover convergence to the Airy sheet. Of these, Sections \ref{S:lpp-cadlag} and \ref{S:ALE-S} are brief and preliminary, and Section \ref{S:sheet} contains the novel ideas and proofs. This part of the paper is essentially independent from the remaining sections. 
	
	The bulk of the paper's new ideas are in Sections \ref{S:directed-metrics}-\ref{S:graph-cvg}. Section \ref{S:directed-metrics} builds a theory of general directed metrics and their geodesics, Section \ref{S:metric-examples} elaborates on the examples described earlier and gives precise ways to embed last passage metrics into $\R^2$, and Section \ref{S:topologies} introduces the topologies we use in the paper. The next three sections give conditions for directed metrics to converge to the directed landscape in the sense of multi-time convergence (Section \ref{S:landscape-cvgfdd}), in the hypograph topology (Section \ref{S:hypo}), and in the graph topology (Section \ref{S:graph-cvg}).
	
	Sections \ref{S:cvg-lattice}-\ref{S:consequences} apply the general theory {built up to that point}. Section \ref{S:cvg-lattice} "puts everything together" for i.i.d.\ lattice last passage models. The proofs in this section are routine. Two results of interest in this section are Theorem \ref{T:lattice-models-i} and Corollary \ref{C:Xiidmaster}, which make precise the idea that for an i.i.d.\ lattice last passage model, convergence to the Airy line ensemble implies convergence to the directed landscape. 
	Section \ref{S:integrable} gives a long-form list of precise theorems for all integrable models, elaborating on the theorems in Section \ref{SS:classical-intro}. Treat this section as a dictionary. Section \ref{S:consequences} proves the consequences of our main theorems: Theorems \ref{T:intro-LIS} and \ref{T:sheet-map}  and Proposition \ref{P:intro-sym}.
	
	The final two sections consider random growth models that can be associated to last passage models. As discussed in Section \ref{SS:TASEP-intro}, these two processes are connected by thinking of growth models as interfaces in directed metrics. Section \ref{S:interfaces} builds up a general theory of interface convergence for planar directed metrics, and Section \ref{S:TASEP} applies this theory to TASEP.
	
\section{Preliminaries: percolation across cadlag functions}
\label{S:lpp-cadlag}

In this section we introduce last passage percolation across cadlag functions. Cadlag last passage percolation is a common generalization of lattice last passage percolation with nonnegative weights and line models of last passage percolation. The Airy line ensemble appears as the common limit of all these models, and also fits within the cadlag framework.
Basic combinatorial results about last passage percolation and the RSK correspondence have analogues in the cadlag setting. We use these results as the starting point for this paper. 

A thorough treatment of cadlag last passage percolation is given in \cite*{DNVcadlag}. In that paper, basic combinatorial and probabilistic properties of cadlag RSK (e.g. isometry, bijectivity, measure preservation) are shown from first principles using the framework of Pitman transforms.

\subsection{Last passage percolation across cadlag functions}
\label{SS:cadlag-lpp}
{A function $f$ from} an interval $I \sset \R$ to $\R$ is {\bf cadlag} if $f(y)\to f(x)$ as $y\downarrow x$ in $I$, and the limit as $y\uparrow x$ also exists. This left-sided limit is denoted $f(x^-)$. When $f(x^-) \ne f(x)$, we say that $f(x) - f(x^-)$ is a \textbf{jump} of $f$ and that $x$ is a \textbf{jump location}. Let $\mathcal D^\Z$ be the space of all functions  
$$
f:\R \X \Z \to \R, \qquad (x,i)\mapsto f_i(x),
$$
where each $f_i$ is a cadlag function whose jumps are all positive.
We will often think of $f$ as a sequence of functions $(f_i : i \in \Z)$ and we will refer to $f$ as an environment.

For $(p, q) = (x, n; y, m) \in (\R \X \Z)^2$, we write $p \nearrow q$ if $x \le y$ and $n \ge m$. This directionality is with respect to the coordinates used in Figure \ref{fig:line-defs}, where the line index is increasing as we go down the page. For $p \nearrow q$, a path $\pi$ from $p$ to $q$ is a union of closed intervals 
\begin{equation}
\label{E:intervals}
[t_i,t_{i-1}]\times \{i\}\subset   \mathbb R \times I, \qquad i=m,m+1,\ldots,n,
\end{equation}
where
\begin{equation}
\label{E:jumptimes}
x=t_n\le t_{n-1}\le \cdots \le t_m \le  t_{m-1}=y.
\end{equation}
The points $t_i$ are called the {\bf jump times} of $\pi$. 
For two paths $\pi$ and $\rho$, we say that $\pi$ is {\bf to the left of} $\rho$ if for every $(x,\ell)\in \pi$ there exists $(y,m)\in \rho$ such that $\ell \le m$ and $x\le y$. Equivalently, we say that $\rho$ {\bf is to the right of} $\pi$. 
We say that $\pi$ and $\rho$  are {\bf essentially disjoint} if the set $\pi \cap \rho$ is finite. See Figure \ref{fig:line-defs} for an illustration of these definitions.

\begin{figure}
	\centering
	\includegraphics[scale=1]{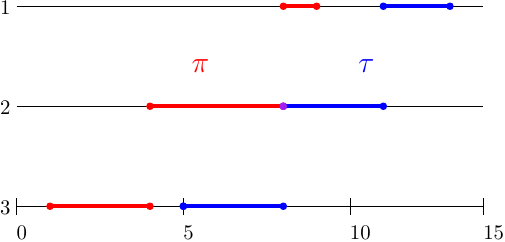}
	\caption{Cadlag paths: here $\pi$ is a path from $(1, 3)$ to $(9, 1)$ with jump times at $1, 4, 8,$ and $9$. The path $\pi$ is to the left of $\tau$ and is essentially disjoint from $\tau$, in spite of the single point of overlap at $(8, 2)$.}
		\label{fig:line-defs}
\end{figure}

For an environment $f \in \scrD^\Z$ and a path $\pi$, define the \textbf{length} of $\pi$ with respect to $f$ by
$$
|\pi|_f =  \sum_{i=m}^{n} f_{i}(t_{i-1}) - f_i(t_{i}^-).
$$
Informally, this definition is chosen so that the all jumps of $f$ that lie along the path $\pi$ are collected.  
An equivalent definition of path length is given through the finitely additive signed measure $df$ on $\mathbb R\times \Z$. Define
$$
df\left([x,y]\times\{i\}\right)=f_{i}(y) - f_i(x^-).
$$ 
Finite additivity defines the measure $df$ of any finite union of finite intervals.  We can then write
$$
|\pi|_f =  df(\pi).
$$
For $u=(p;q)=(x,n;y,m)\in (\R \X \Z)^2$  define the \textbf{last passage value} of $f$ from $p$ to $q$ by
\begin{equation}
\label{E:lpp-def}
f[u]=f[p\to q]=f[(x, n) \to (y, m)] = \sup_{\pi} |\pi|_f,
\end{equation}
where the supremum is taken over all paths $\pi$ from $p$ to $q$. If no path from $p$ to $q$ exists, we set $f[u] = -\infty$. We call a path $\pi$ from $p$ to $q$ a \textbf{geodesic} if $|\pi|_f = f[p \to q]$. The following two lemmas about geodesics in last passage percolation are standard. 
\begin{lemma}
[special case of Lemma 2.6, \cite*{DNVcadlag}]
	\label{L:geod-existlpp} 
Let $f \in \scrD^\Z$, and consider a point $(p, q)\in (\R \X \Z)^2$ with $p \nearrow q$. Then there is a geodesic from $p$ to $q$. Moreover, there are unique geodesics $\pi^-, \pi^+$ from $p$ to $q$ such that $\pi^-$ is to the left of all geodesics from $p$ to $q$ and $\pi^+$ is to the right of all geodesics from $p$ to $q$. We refer to $\pi^-$ and $\pi^+$ as leftmost and rightmost geodesics.
\end{lemma}

The following simple fact is shown in the continuous setting in \cite*{DOV}, Proposition 3.7. The same proof works in the cadlag setting.
\begin{lemma}
	\label{L:mono-path}
	Let $f \in \scrD^\Z$, and for $(p, q) \in (\R \X \Z)^2$ with $p \nearrow q$, let $\pi^+[p; q]$ denote the rightmost geodesic  from $p$ to $q$. Then for $x \le x', y \le y',$ and $m \le n$, the path $\pi^+[x, n; y, m]$ is to the left of the path $\pi^+[x', n; y',m]$.
	Moreover, for any $(p, q), (p', q')$, the intersection $\pi^+[p; q] \cap \pi^+[p'; q']$ is also a path. Note that this intersection may be empty.
\end{lemma}

As discussed in the introduction, last passage percolation is best thought of as a directed metric on the underlying set (in this case, $\R \X \Z$). One manifestation of this is that last passage percolation exhibits the following \textbf{metric composition law}. The proof is immediate from the definition.

\begin{lemma}[Metric composition law]\label{L:metric}
	Let $(p, q) = (x, {\ell}; y, m)$ be such that $p \nearrow q$. Then for any $f \in \scrD^\Z$ and any $k \in \{m + 1, \dots, {\ell}\}$, we have
	\begin{align*}
	f[(x, \ell) \to (y, m)] = \max_{z \in [x, y]} {\bigl(}f[(x, \ell) \to (z, k)] + f[(z, k - 1) \to (y, m)]{\bigr)}
	\end{align*}
	and for any $k \in \{m, \dots, {\ell}\}$, we have
	\begin{align*}
	f[(x, \ell) \to (y, m)] = \max_{z \in [x, y]} {\bigl(}f[(x, \ell) \to (z, k)] + f[(z, k) \to (y, m)] - (f_k(z) - f_k(z^-)){\bigr)}.
	\end{align*}
\end{lemma}

The metric composition law yields a useful \textbf{triangle inequality} for last passage percolation. With all notation as above we have
\begin{equation}
\label{E:triangle-ineq}
f[(x, \ell) \to (y, m)] \ge f[(x, \ell) \to (z, k)] - f_k(z) + f_k(z^-) + f[(z, k) \to (y, m)].
\end{equation}

\subsection{Multi-point last passage values and the RSK correspondence}
\label{S:multipoint}
For two vectors $\mathbf{p} = (p_1, \dots, p_k), \mathbf{q} = (q_1, \dots, q_k) \in (\R \X \Z)^k$ with $p_i \nearrow q_i$ for $i = 1, \dots, k$, a \textbf{disjoint $k$-tuple (of paths)} from $\mathbf{p}$ to $\mathbf{q}$ is a vector $\pi = (\pi_1,\ldots ,\pi_k)$, where 
\begin{itemize}[nosep]
	\item $\pi_i$ is a path from $p_i$ to $q_i$,
	\item $\pi_i$ is to the left of $\pi_j$ for $i < j$,
	\item $\pi_i$ and $\pi_j$ are essentially disjoint for all $i \ne j$.
\end{itemize}
For $f \in \scrD^\Z$ and a disjoint $k$-tuple $\pi$, define its length with respect to $f$ by
$$
|\pi|_f=df(\pi_1\cup \dots\cup \pi_k).
$$
Importantly, with this definition, no jump of any $f_i$ can be counted more than once even if that jump lies on multiple paths $\pi_i$.
For $\mathbf u = (\mathbf p, \mathbf q) \in (\R \X \Z)^{2k}$, we can then define the multi-point last passage value
\begin{equation}\label{E:multipoint}
f[\mathbf u] = f[{\bf p \to \bf q}] = \sup_{\pi} |\pi|_f,
\end{equation}
where the supremum is over disjoint $k$-tuples from $\bf p$ to $\bf q$. Again, if no such $k$-tuples exist, we set $f[\mathbf u] = -\infty$.
We call a $k$-tuple $\pi$ satisfying $f[{\bf p \to \bf q}] = |\pi|_f$ a {\bf disjoint optimizer}. For any $\bp, \bq$ and $f \in \scrD^\Z$, there is a disjoint optimizer from $\bp$ to $\bq$ as long as the supremum \eqref{E:multipoint} is nonempty.

Multi-point last passage values allow us to define the RSK correspondence for cadlag paths. First, let $\scrD^n_0$ be the space of cadlag functions $f:[0, \infty) \X \{1, \dots, n\} \to \R$ with positive jumps and $f_i(0) \ge 0$. By convention, we again set $f_i(0^-) =0$. If $f_i(0) > 0$, we interpret this as $f_i$ having a jump at $0$. Define the \textbf{RSK/melon map} $W$ by the relationship
\begin{equation}
\label{E:melon-def}
\sum_{i=1}^k Wf_i(y) = f[(0, n)^k \to (y, 1)^k].
\end{equation}
Here and in the sequel we write $p^k$ for a vector $(p, \dots, p)$ consisting of $k$ copies of the point $p$. The melon map has the following important properties.

\begin{prop}[Proposition 3.12(i, iii) in \cite*{DNVcadlag}]
	\label{P:W-facts}
	Consider the melon map $W$ on $\scrD^n_0$.
	\begin{enumerate}[nosep, label=(\roman*)]
		\item (Isometry) For any $f \in \scrD^n_0$ and any pair $(\bp, \bq)$ with $p_i = (x_i, n)$ and $q_i = (y_i, 1)$, we have 
		$$
		f[\bp \to \bq] = Wf[\bp \to \bq].
		$$
		\item (Ordering) For any $f \in \scrD^n_0$, the lines in $Wf_i$ are ordered. More precisely, for any $i \ge 2$ and $y \ge 0$, we have $Wf_{i-1}(y^-) \ge Wf_i(y)$.
	\end{enumerate}
\end{prop} 

Proposition \ref{P:W-facts}(i) in the continuous path setting was crucially used in the construction of the Airy sheet in \cite*{DOV}. The proof in the cadlag setting is similar. However, the result goes back much further. Close relatives of Proposition \ref{P:W-facts}(i) were first found by \cite*{noumi2002tropical} in the fully discrete setting and by \cite*{biane2005littelmann} in the continuous setting. Proposition \ref{P:W-facts}(ii) is more classical, and can be shown with a path-crossing argument.

The ordering in Proposition \ref{P:W-facts}(ii) also gives the sequence $Wf_1, \dots,Wf_n$ the appearance of stripes on a watermelon; this is the reason for the name melon map. We call $Wf$ the \textbf{melon} of $f$.

In many natural settings for studying last passage percolation, the map $W$ is almost surely a bijection, see  \cite*{DNVcadlag}. Though we will not use bijectivity explicitly in this paper, this can be thought of as the reason that certain measures on $\scrD^n_0$ have tractable pushforwards under $W$, and hence the reason why studying RSK is interesting probabilistically. 

Moving forward, we will need the following consequence of Proposition \ref{P:W-facts}. Consider a disjoint $k$-tuple $\pi = (\pi_1, \dots, \pi_k)$ from $p^k=(x,n)^k$ to $q^k=(y,m)^k$ where $n-m=k$. Let $u = \pi_1 \cup \dots \cup \pi_k$, and consider the set $[x,y]\times\{m,\dots, n\}\setminus u$. This is given by
$$
[x,s_{m+1})\times \{m\}\,\,\cup\,\,  (s_{m+1},s_{m+2})\times\{m+1\}\,\,\cup \dots\cup \,\,(s_{n},y]\times\{n\}
$$
where $s_{m+1}\le s_{m+2}\le \cdots \le s_n$. If $s_{m+1} = x$ or $s_n = y$, then the first or last term above is omitted. We call such a set a {\bf complementary path} from $(x,m)$ to $(y,n)$. Informally, a complementary path is just a decreasing path that consists of open intervals except the first and last, which are half-closed. 
Define the \textbf{first passage value}
\begin{equation}
\label{E:ka-inf}
f[p\to_{\mathfrak{f}} q] = \inf_\ka df(\ka),
\end{equation}
where the infimum is over all complementary paths from $p$ to $q$. It is easy to check from the construction that
$$
f[p^k\to q^k]+f[(x, m)\to_{\mathfrak{f}} (y, n)]=df([x,y]\times{\{m,\dots,n\}}).
$$
Next, let $R_{z, n}$ be the rotation in $\R^2$ that sends $(z, 1)$ to $(0, n)$ and $(0, n)$ to $(z, 1)$. Note that $R_{z, n}$ is an involution on  $\R \X \{1, \dots, n\}$. For $f \in \scrD^\Z$,
 the pushforward of the measure $df$ under $R_{z, n}$ is again a finitely additive measure associated to a function $R_{z, n} f \in \scrD^n_0$. Define
$$
W_z^* f = W R_{z, n}f.
$$
\begin{lemma}
	\label{L:k-line-prior}
	Let $f \in \scrD^\Z$ be such that $f_i(0^-) = 0$ for all $i \in \{1,\dots, n\}$.  For every $0 \le x < z$ and $k \in {\{1,\dots,n\}}$, we have that
	\begin{align*}
	Wf[(x, n) \to (z, k)] = Wf_k(z) - W^*_zf[(z - x, 1)^+ \to_{\mathfrak{f}} (z, k)].
	\end{align*}
\end{lemma}

{The superscript $+$ means the limit as the first spatial coordinate decreases to the specified value.} Lemma \ref{L:k-line-prior} is stated and proven in the continuous setting in \cite*{DOV}, Lemma 5.3. There is no essential difference to the proof in the cadlag setting, so we omit it.

\subsection{Lattice last passage percolation}
\label{S:lattice-lpp}

Many of the models we study in this paper are lattice last passage percolation models. These can be recast as line last passage models across cadlag paths.

For two points $p = (x, n), q = (y, m) \in \Z^2$ with $x \le y$ and $n \ge m$, we say that $\pi = (\pi_1 = p, \dots, \pi_k = q)$ is a \textbf{path from $p$ to $q$} if $\pi_i - \pi_{i-1} \in \{(1, 0), {(0, -1)}\}$ for all $i$. For an array $G = \{G_u : u \in \Z^2\}$ of nonnegative numbers, we can define the weight of any path $\pi$ from $p$ to $q$ by
$$
|\pi|_G= \sum_{v\in \pi} G_v
$$
and the last passage value 
$$
G[p\to q]:= \max_{\pi} |\pi|_G,
$$
where the maximum is taken over all possible paths $\pi$ from $p$ to $q$. Again, we refer to a path $\pi \sset \Z^2$ {attaining this maximum} as a \textbf{geodesic} from $p$ to $q$. If $p$ and $q$ are not ordered as above (i.e. no paths from $p$ to $q$ exist), then we define $G[p \to q] = -\infty$. More generally, for vectors $\bp, \bq$ define the {\bf multi-point last passage value}
\begin{equation}
\label{D:LPP-k}
G[\bp \to \bq] :=
\max_{\pi_1,\ldots, \pi_k}{\bigl(}|\pi_1|_G+\cdots +|\pi_k|_G{\bigr)}
\end{equation}
where the maximum now is taken over all possible $k$-tuples of \emph{disjoint} paths, where each $\pi_i$ is a path from $p_i$ to $q_i$. If no disjoint $k$-tuples exist, set $G[\bp \to \bq] = -\infty$. As in the cadlag setting, we will be particularly concerned with multi-point last passage values with bunched endpoints. With this in mind, we introduce the shorthand $G[p^k \to q^k]$ for the $k$-point last passage value from 
$$
(p - (0, k-1), \dots, p - (0, 1), p) \to (q, q + (0, 1), \dots, q + (0,k-1)).
$$
The value $G[p^k \to q^k]$ is best thought of as a last passage value with $k$ disjoint paths from $p$ to $q$, hence the similar notation to the corresponding object in the cadlag setting. We are forced to stagger the start and end points of the paths to allow for disjoint paths.
We can associate the array $G$ to a cadlag environment $f^G \in \scrD^\Z$ by setting 
\begin{equation}
\label{E:fGkk}
f^G_k(0^-) = 0, \qquad \mathand \qquad f^G_k(t) - f^G_k(s) = \sum_{r \in {(s, t]\cap\Z}} G_{r, k}.
\end{equation}
Multi-point last passage values are the same in $G$ and $f^G$. 

\begin{prop}[Proposition 8.1, \cite*{DNVcadlag}]
	\label{P:melon-cor}
		Let $G:\Z^2 \to \R$ be a nonnegative array. We have
	$$
	f^G[(x, n)^k \to (y, 1)^k] = G[(\ceil{x}, n)^k \to (\floor{y}, 1)^{k}]
	$$
as long as there exists a disjoint $k$-tuple in $\Z^2$ from $(\ceil{x}, n)^k$ to $(\floor{y}, 1)^{k}$.
\end{prop}

\section{The Airy line ensemble and the Airy sheet}
\label{S:ALE-S}

We now turn our attention to last passage percolation in random environments. All the models we study in this paper can be recast as sequences of random cadlag environments $L_n \in \scrD^n_0$ whose melons $WL_n$ converge after rescaling to a universal scaling limit: the \textbf{parabolic Airy line ensemble}.

 The parabolic Airy line ensemble is a random continuous function $\scrA: \N \X \R \to \R, (i, x) \mapsto \scrA_i(x)$. Its finite dimensional distributions were first described in \cite*{prahofer2002scale} via a determinantal formula. \cite*{CH} showed that these finite dimensional distributions are associated to a unique continuous function whose lines are almost surely strictly ordered: $\scrA_1 > \scrA_2 > \dots$. 
 
The process $\scrA$ can be loosely viewed as an infinite system of nonintersecting Brownian motions. This intuition was made rigorous in \cite*{CH}, where the authors showed that $\scrA$ satisfies a useful Brownian Gibbs property. This property states that conditioned on the values of $\scrA$ on the complement of a region $\{\ell, \dots, k\} \X [a, b]$, inside that region $\scrA$ consists of independent Brownian bridges, conditioned so that the whole process remains nonintersecting and continuous.
While we will not need to appeal to the Brownian Gibbs property directly in this paper, we will use the following basic consequence of that property. For this proposition and in the sequel, we say that a Brownian motion or Brownian bridge has \textbf{variance $\al$} if its quadratic variation on any interval $[a, b]$ is equal to $\al(b-a)$.

\begin{prop}[\cite*{CH}, Proposition 4.1]
	\label{P:brownian-airy}
	Fix an interval $[a, a + b] \sset \R$ and $k \in \N$, and define $B_i(t) = \scrA_i(a + t) - \scrA_i(a)$ for $i \in \{1, \dots, k\}$. Then on the interval $[0, b]$ the sequence $(B_1, \dots, B_k)$ is absolutely continuous with respect to the law of $k$ independent Brownian motions with variance $2$.
\end{prop}

We also require a few symmetries of $\scrA$, which go back to \cite*{prahofer2002scale}.

\begin{prop}
	\label{P:sym}
	The parabolic Airy line ensemble possesses a flip symmetry, $\scrA(\cdot) \eqd \scrA(-\;\cdot)$, and the process $t \mapsto \scrA(t) + t^2$ is stationary. We refer to the top line $\scrA_1$ as the \textbf{parabolic Airy process}.
\end{prop}

If a sequence of melons $WL_n$ converges in distribution after rescaling to the parabolic Airy line ensemble, then the last passage processes $L_n[(0,n) \to (\cdot, 1)]$ converge in distribution after rescaling to the parabolic Airy process $\scrA_1$. 

Moreover, by Proposition \ref{P:W-facts}(i), last passage values of the form $L_n[(x,n) \to (y, 1)]$ for $x \ge 0$ are contained as last passage values in the melon $WL_n$. This suggests that the scaling limit of these values could be contained in terms of a last passage problem involving the Airy line ensemble. This was shown in \cite*{DOV} for \textbf{Brownian last passage percolation}, where each $L_n$ consists of independent two-sided Brownian motions. When combined with translation invariance, this uniquely identifies the scaling limit of the two-parameter processes $L_n[(\cdot,n) \to (\cdot, 1)]$, known as the Airy sheet.

\begin{definition}
\label{D:airy-sheet}
The \textbf{Airy sheet} is a random continuous function $\scrS:\mathbb R^2\to \mathbb R$ so that
\begin{enumerate}[label=(\roman*),nosep]
	\item
	$\scrS$ has the same law as  $\scrS(\cdot+t, \cdot+t)$ for all $t\in \mathbb R$.
	\item $\scrS$ can be coupled with an Airy line ensemble so that $\scrS(0,\cdot)=\scrA_1(\cdot)$ and for all $(x, y, z) \in \Q^+ \X \Q^2$ there exists a random variable $K_{x, y,z} \in \N$ so that for all $k \ge K_{x, y,z}$, almost surely
	$$
	\scrA[(-\sqrt{k/(2x)}, k) \to (z, 1)]- \scrA[(-\sqrt{k/(2x)}, k) \to (y, 1)] = \scrS(x, z) - \scrS(x, y).
	$$
\end{enumerate}
\end{definition} 

The Airy sheet exists and is unique in law, see Section 8 of \cite*{DOV}. Part (ii) of the above definition can also be replaced by the following Busemann function definition:
\begin{itemize}
	\item [(ii')] The Airy sheet $\scrS$ can be coupled with a parabolic Airy line ensemble so that almost surely, $\scrS(0,\cdot)=\scrA_1(\cdot)$ and for all $x > 0$ and $y, z \in \R$, we have that
	\begin{equation}
	\label{E:buse}
	\lim_{k \to \infty} \scrA[(-\sqrt{k/(2x)}, k) \to (z, 1)] - \scrA[(-\sqrt{k/(2x)}, k) \to (y, 1)]= \scrS(x, z) - \scrS(x, y).
	\end{equation}
\end{itemize}

Definition \ref{D:airy-sheet}(ii') essentially says that the Airy sheet value $\scrS(x, y)$ is the renormalized limit, as $k \to \infty$, of the last passage value in $\scrA$ from $(-\sqrt{k/(2x)}, k)$ to $(y, 1)$. Rigorously showing that such a renormalized limit exists is an open problem, see Conjecture 14.2 in \cite*{DOV}. Studying the differences $\scrS(x,y) - \scrS(x,z)$ instead gets around this issue.

While the proof in \cite*{DOV} of convergence to the Airy sheet for Brownian last passage percolation uses a few properties that are specific to that model, the most delicate parts of the proof only involve estimates about the Airy line ensemble and translation and reflection invariance properties of the underlying model. Because of this, we can adapt that proof to a much more general setting, see Theorem \ref{T:airy-sheet-gen} below. Our generalization relies on some results from \cite*{DOV}, but also simplifies some steps.

\section{Convergence to the Airy sheet}
\label{S:sheet}

In this section we prove Theorem \ref{T:intro-sheet}.
We first prove tightness of the prelimiting sheets.
Rather than appealing to the Kolmogorov-Centsov criterion as in \cite*{DOV} to prove tightness, we set up a general topological framework which requires fewer underlying assumptions on the model. This framework exploits a fundamental \textbf{quadrangle inequality} for last passage models.

\begin{lemma}[Quadrangle inequality, special case of Lemma 2.5 in \cite*{DNVcadlag}]
	\label{L:quadrangle}
	Let $f \in \scrD^\Z$. For any $x_1 \le x_2 \le y_1 \le y_2$, and $n \ge m$, we have that
	\begin{align}
	\nonumber
	f[(x_1, n) \to (y_1, m)] + &f[(x_2, n) \to (y_2, m)] \\ \ge& f[(x_1, n) \to (y_2, m)] + f[(x_2, n) \to (y_1, m)]. 	\label{E:f-y-1}
	\end{align}
\end{lemma}

Our Airy sheet prelimits will be of the form
\begin{equation}
\label{E:SnLn}
S_n(x, y) = L_n[(x + a_n, n)^+ \to (y, 1)] {+} c_n,
\end{equation}
for a sequence of environments $L_n \in \scrD^\Z$, and constants $c_n, a_n$ with $a_n \to -\infty$ with $n$. Here
$$
L_n[(x + a_n, n)^+ \to (y, 1)] := \lim_{z \downarrow x+ a_n} L_n[(z, n) \to (y, 1)]
$$ The functions $S_n$ are cadlag in each variable (this is the reason for the $+$ in the first coordinate in \eqref{E:SnLn}), and by Lemma \ref{L:quadrangle}, for $x_1 \le x_2$ and $y_1 \le y_2$ we have
\begin{equation}
\label{E:quad-re}
S_n(x_1, y_1) + S_n(x_2, y_2) - S_n(x_1, y_2) - S_n({x_2, y_1}) \ge 0
\end{equation}
as long as all points $(x_i, y_j)$ are in the {open half-plane}
$
\scrH_{a_n} = \{(x, y) \in \R^2 : x + a_n < y\}.
$
{Outside of this half-plane}, $S_n = -\infty$.
For $r > 0$, let $\CDF_r$ be the space of functions $s:[-r, r]^2 \to {\R}$ which are cadlag in each variable and satisfy the inequality \eqref{E:quad-re} with $s$ in place of $S_n$. Any such function $s$ is the cumulative distribution function of a Borel measure $\mu_s$ on $(-r, r]^2$: that is, {$\mu_s((a,b]\X(c,d])=s(a,c)+s(b,d)-s(a,d)-s(b,c)$}. 

Any $s \in \CDF_r$ is uniquely determined by the triple $(s(0, \cdot), s(\cdot, 0), \mu_s)$. We define the {\bf sheet topology} on $\CDF_r$ as the product topology given by Skorokhod topologies on cadlag functions in the first two coordinates, and the vague topology on measures in the third coordinate. {Vague convergence means convergence of integrals against continuous functions of compact support.} Background on the Skorokhod topology on cadlag functions can be found in \cite*{kallenberg2006foundations}, Chapter 14. For us, the main necessary facts are that the Skorokhod topology is Polish and that if $f_n \to f$ in the Skorokhod topology and $f$ is continuous, then $f_n \to f$ uniformly. 

{This identifies $\CDF_r$ with a subspace of a Polish space of triples, allowing locally finite measures in the third coordinate.} It is a partial analogue of the Skorokhod topology for two-variable functions. 
The following lemma is one reason for introducing this topology. We omit the straightforward proof {of part (1)}.

\begin{lemma}
	\label{L:sheet-topology}
	{Let $r\in\N$.}
	\begin{enumerate}
		\item If a sequence $s_n \in \CDF_r$ converges to a continuous limit $s$ in the sheet topology, then $s_n \to s$ uniformly on {compact subsets of $(-r, r]^2$}.
		\item Let $s \in \CDF_r$. If for all rational points $x \in [-r, r]$, the functions $s(x,\cdot), s(\cdot,x)$ are continuous, then $s$ is continuous. 
	\end{enumerate}
\end{lemma}

\begin{proof}[{Proof of (2)}]
{The two marginals of $\mu_s$ have continuous distribution functions, since these are differences of the continuous sections at $r$ and $-r$. Thus $\mu_s$ gives zero mass to every horizontal or vertical line, so $(x,y)\mapsto\mu_s((-r,x]\times(-r,y])$ is continuous. The identity
\[
s(x,y)=s(x,-r)+s(-r,y)-s(-r,-r)+\mu_s((-r,x]\times(-r,y])
\]
now gives continuity of $s$.}
\end{proof}

We can translate Lemma \ref{L:sheet-topology} (2) into a condition for tightness of random functions $S_n$.
\begin{lemma}
	\label{L:tight-sheet}
	Let $S_n$ be a sequence of random variables in $\CDF_r$ for some $r \in \N$.  Suppose that for every rational $x \in [-r, r]$,  $S_n(x, \cdot)$ and $S_n(\cdot, x)$ are tight in the Skorokhod topology and all limit points are supported on continuous functions. Then $S_n$ is tight in the sheet topology and all limit points are supported on continuous functions.
\end{lemma}
\begin{proof}
	Tightness of $\mu_{S_n}$ in the vague topology follows since
	$$
	\mu_{S_n}({(-r, r]^2}) = S_n(r, r) + S_n(-r, -r) - S_n(r, -r) - S_n(-r, r),
	$$
	and the random variables on the right side are all tight by assumption. {To identify limits, extend these measures by zero to $[-r,r]^2$ and take a joint subsequential limit in the weak topology, together with the sections at $-r,0,r$. The marginal distribution functions of the measures are differences of the boundary sections, so they converge uniformly to continuous functions vanishing at $-r$. Thus the limiting measure gives zero mass to every horizontal or vertical line, including the boundary. Together with the limiting sections at $0$, it defines a continuous element of $\CDF_r$. This proves tightness in the sheet topology and continuity of all limits.}
\end{proof}

 Since $a_n \to -\infty$, our Airy sheet prelimits $S_n|_{[-r, r]^2}$ are random elements of $\CDF_r$ for all large enough $n$. Moving forward, we will ignore the minor issue that $S_n|_{[-r, r]^2}$ may take on the value $-\infty$ when $n$ is small.

We can now state our main Airy sheet convergence theorem. For $k \ge 1$ we use the increment notation
$$
f[u \to_{\Delta_k} v ] = f[u^k \to v^k ] - f[u^{k-1} \to v^{k-1} ],
$$
where $f[u \to_0 v] :=0$. Because prelimits of the Airy line ensemble are not necessarily continuous functions, we will put the product of Skorokhod topologies on this function space. 

\begin{theorem}
	\label{T:airy-sheet-gen}
	Let $L_n \in \scrD^\Z$ be a sequence of environments, let $a_n, c_n \in \R$ be sequences with $a_n \to -\infty$. Assume that for every $z \in \R$, the sequences
	\begin{align*}
	A^z_{n, k}(\cdot) &= L_n[(z + a_n, n) \to_{\Delta_k} (z + \cdot, 1)] + c_n, \quad \mathand \\
	\tilde A^z_{n, k}(\cdot) &= L_n[(z + a_n - \cdot, n)  \to_{\Delta_k} (z, 1)] + c_n,
	\end{align*}
	{each} converge in distribution in the product-of-Skorokhod topologies {over $k$} as $n \to \infty$ to the parabolic Airy line ensemble $\scrA_{k}(\cdot)$. Define
	$$
	S_n(x, y) =L_n[(x+ a_n, n)^+ \to (y, 1)] + c_n.
	$$
	Then for any $r \in \N$, we have that $S_n|_{[-r, r]^2} \cvgd \scrS|_{[-r, r]^2}$ as random elements of $\CDF_r$. Moreover, there exists a coupling where $S_n \to \scrS$ compactly a.s.
\end{theorem}
Theorem \ref{T:intro-sheet} is an immediate consequence of Theorem \ref{T:airy-sheet-gen}.

\begin{remark}
	\label{R:shock-measure}
	The Airy sheet $\scrS$ is itself associated to a random measure $\mu_{\scrS}$ on $\R^2$ via the limiting analogue of \eqref{E:quad-re}. We call the measure $\mu_\scrS$ the \textbf{shock measure} of the Airy sheet. This measure has a beautiful structure that is related to exceptional geodesic behaviour in the directed landscape and Brownian local time. This measure has been previously studied by \cite*{basu2021fractal, bates2019hausdorff, ganguly2021local, dauvergne2021last}. We explore the shock measure more in  upcoming work.
\end{remark}

We prove Theorem \ref{T:airy-sheet-gen} by following the strategy from \cite*{DOV}. For the remainder of this section, all sequences will be as in Theorem \ref{T:airy-sheet-gen} with $c_n = 0$. The general $c_n$ case {can be reduced to} the $c_n = 0$ case by replacing the environment $L_n$ with an environment $L_n' = L_n + c_n f_n$, where $f_n$ is the CDF of uniform measure on the interval $[a_n/2, a_n/2 + 1]$.

Define $T_a:\scrD^\Z \to \scrD^\Z$ by $T_a f = f(a + \cdot) - f(a^-)$, and let 
$$
W_a f(x) =\begin{cases} 
(W T_a f)(x-a) & \text{ for } x\ge a\\
0 & \text{ for } x<a
\end{cases}
$$
be the {\bf melon of $f$ opened up at $a$}.
In the remainder of this section we write 
$$
W^n := W_{a_n} L_n.
$$
The main step in the proof of Theorem \ref{T:airy-sheet-gen} involves locating the rightmost geodesic from $(x + a_n, n)$ to $(y, 1)$ in melon $W^n$. This is the goal of the next few propositions.

In the remainder of this section, for a random array $\{R_{n, k} :n, k \in \N\}$ we will write
\begin{equation}\label{E:onotation}
R_{n,k} = \oo(r_{k})  \qquad \text{ if for all }\epsilon>0 \qquad  \;\; \sum_{k=1}^\infty \limsup_{n\to\infty} \prob(|R_{n,k}/r_{k}|>\epsilon) < \infty.
\end{equation}
The idea behind this notation is that if we can pass to a limit in $n$ to get a sequence $R_k$, then by the Borel-Cantelli lemma, $R_k/r_k \to 0$ almost surely.

\begin{prop}
	\label{P:fn-cvg}
	{Under the assumptions of Theorem \ref{T:airy-sheet-gen},} define
	$$
	F^n_k(x, z) = W^n[(x + a_n, n) \to (z, k)] - W^n_k(z).
	$$
	For any ${x>0}$, the function $F^n_k(x, z)$ is monotone increasing in $z$ {on $[x+a_n,\infty)$}, and for any interval $[b, c] \sset \R$ we have
	\begin{equation}
	\label{E:Fnk-uni}
	\sup_{a \in [b, c]} \lf|F^n_k(x, a \sqrt{k}) - 2\sqrt{kx}(\sqrt{2} +  a\sqrt{x}) \rg| = \oo(\sqrt{k}).
	\end{equation}
\end{prop}

Proposition \ref{P:fn-cvg} is the analogue of Proposition 6.1 in \cite*{DOV}. It follows from Lemma \ref{L:k-line-prior} and the following estimate for last passage percolation across the Airy line ensemble. 

\begin{theorem}[\cite*{DOV}, {proof of Theorem 6.3}]
	\label{T:airy-lp} There exists a constant $d \in \N$ such that for every $\ep > 0$ and $x > 0$, we have
	$$
	\sum_{{k\ge2}} \p \lf(\frac{{|}\scrA[(0, k) \to (x, 1)] -{2\sqrt{2kx}}{|}}{k^{9/21}\log^dk} > \ep\rg) < \infty.
	$$
\end{theorem}

\begin{proof}[Proof of Proposition \ref{P:fn-cvg}] 
	By explicitly writing out the definition of last passage percolation in \eqref{E:lpp-def}, we have
	\begin{equation}
	\label{E:sum-of-gaps}
F^n_k(x, z) = - W^n_n((x+ a_n)^-) + \sup \sum_{i=k}^{n-1} (W^n_{i+1}(t_i) - W^n_i(t_i^-)).
	\end{equation}
	Here the supremum is over all sequences $t_{n-1} \le \dots \le t_k \in [x + a_n, z]$. As we increase $z$, we are taking the supremum over a larger interval, so $F^n_k(x,z)$ can only increase.
	
	We now prove \eqref{E:Fnk-uni}. By Lemma \ref{L:k-line-prior}, we can write
	\begin{align}
	\label{E:FP-melon}
	F^n_k(x, z) &= -W^*_{z} L^n[(z-a_n - x, 1)^+ \to_{\mathfrak{f}} (z-a_n, k)].
	\end{align}
	Rephrased in terms of $W^*_{z} L^n$, the {convergence assumption on $\tilde A_{n, k}$} in Theorem \ref{T:airy-sheet-gen} states that
	$$
	W^*_{z} L^n(-a_n + \cdot)
	$$
	converges in the Skorokhod topology to $\scrA$. By the continuity of $\scrA$, the first passage value in \eqref{E:FP-melon} therefore converges to a first passage value across $\scrA$:
	$$
	-W^*_{z} L^n[(z-a_n - x, 1)^+ \to_{\mathfrak{f}} (z-a_n, k)] \cvgd -\scrA[ (z -  x, 1) \to_{\mathfrak{f}} (z, k)].
	$$
	By the flip symmetry in Proposition \ref{P:sym}, and the equivalence between the definitions of first and last passage for continuous functions, the right hand side above is equal in distribution to
	\begin{eqnarray}
	\label{E:Aykf2}
	\scrA[(-z, k) \to (-z + x, 1)].
	\end{eqnarray}
	By the stationarity of $\scrA(t) + t^2$ (Proposition \ref{P:sym}), \eqref{E:Aykf2} is in turn equal in distribution to
	$$
	\scrA[(0, k) \to (x, 1)] + 2xz.
	$$
	Plugging in $z = a \sqrt{k}$ for a fixed $a \in \R$, Theorem \ref{T:airy-lp} now implies that
	$$
	F^n_k(x, a \sqrt{k}) = 2\sqrt{kx}(\sqrt{2} + a \sqrt{x}) +  \oo(\sqrt{k}).
	$$
This gives \eqref{E:Fnk-uni} when $b = c$. The case when $b < c$ follows from this case by monotonicity of $F^n_k$.
\end{proof}

\begin{corollary}
	\label{C:Gn-upper}
	With $L^n, a_n$ as in Theorem \ref{T:airy-sheet-gen}, and $W^n=W_{a_n}L^n$, for $z\le y$ define
	$$
	G^n_k(z, y) = W^n_k(z^-) +W^n[(z, k) \to (y, 1)].
	$$
	Then $G^n_k$ is monotone decreasing in $z$, and for any fixed $y \in \R$ and $[b, c] \sset (-\infty, 0)$ we have
	\begin{align}
	\label{E:Gnk-easy}
		G^n_k(a_n, y) & \le \oo(\sqrt{k}) \qquad \mathand \\
		\label{E:supG1}
		\sup_{a \in [b, c]} {\biggl(}G^n_k(a \sqrt{k}, y) + \frac{\sqrt{k}}{|a|}{\biggr)}  &\le  \oo(\sqrt{k}).
	\end{align}
\end{corollary}

\begin{proof}
Just as in \eqref{E:sum-of-gaps}, we can rearrange the terms in the definition \eqref{E:lpp-def} to get that
\begin{equation}
\label{E:mono-G}
G^n_k(z, y) = W^n_1(y) +\sup \sum_{i=1}^{k-1} W^n_{i+1}(t_i) - W^n_i(t_i^-),
\end{equation}
where the supremum is over all sequences of times $t_{k-1} \le \dots \le t_1 \in [z, y]$. Since this supremum can only get smaller as we increase $z$, $G^n_k(\cdot, y)$ is monotone decreasing. Moreover, by Proposition \ref{P:W-facts}(ii), the sum under the supremum is always nonpositive. Therefore 
$$
G^n_k(a_n, y) \le W^n_1(y) = \oo(\sqrt{k}).
$$
The equality above follows from the much stronger fact that $W^n_1(y) = A^0_{n, 1}(y)$, and hence converges in distribution. It just remains to prove \eqref{E:supG1}. By the monotonicity of $G^n_k$, it is enough to prove \eqref{E:supG1} when $b = c$.
	
	With $F^n_k$ as in Proposition \ref{P:fn-cvg} and $x = 1/(2b^2)$, the triangle inequality \eqref{E:triangle-ineq} applied at the points $(a_n+x,n)$, $(b\sqrt{k},k)$, and $(y,1)$ gives
	\begin{equation}
	\label{E:Gn-Fn-Tri}
	G^n_k(b\sqrt{k}, y) \le W^n[(a_n + x, n) \to (y, 1)] - F^n_k(x, b \sqrt{k}).
	\end{equation}
	The first term on the right hand side of \eqref{E:Gn-Fn-Tri} equals $A^x_{n, 1}(y - x)$, and hence converges in distribution and is $\oo(\sqrt{k})$. The second term can be bounded by Proposition \ref{P:fn-cvg}, yielding the desired bound.
\end{proof}

We can use Proposition \ref{P:fn-cvg} and Corollary \ref{C:Gn-upper} to locate jump times for melon geodesics. For the remainder of the section, we let $Z^n_k(x, y)$ be the unique time when the rightmost geodesic in $W^n$ from $(a_n + x, n)$ to $(y, 1)$ intersects both line $k$ and line $k-1$. We note that $Z^n_k(x, y)$ is nonincreasing in $k$, and that for a fixed $k$, $Z^n_k(x,y)$ is nondecreasing in $x$ and $y$ by monotonicity of geodesics, Lemma \ref{L:mono-path}.

\begin{lemma}
	\label{L:zk-unif-bd}
	{Assume the hypotheses of Theorem \ref{T:airy-sheet-gen}.}
	For every compact set $K \sset (0, \infty) \X\R$ we have
	\begin{equation}
	\label{E:ZnkxyKxy}
		\sup_{(x, y) \in K} \lf|Z^n_k(x, y) + \sqrt{\frac{k}{2x}} \rg| = \oo(\sqrt{k}).
	\end{equation}
	Moreover, for any fixed $x > 0, y \in \R,$ and $k \in \N$, the sequence $Z^n_k(x, y)$ is tight in $n$.
\end{lemma}

This is the analogue of Lemma 7.1 in \cite*{DOV}.

\begin{proof}
	We prove \eqref{E:ZnkxyKxy} for a fixed $x, y$. The extension to compact sets follows by monotonicity of rightmost geodesics, Lemma \ref{L:mono-path}. {For each $\ep>0$, take a finite grid in a rectangle containing $K$, with $x$-mesh so small that $x\mapsto -1/\sqrt{2x}$ varies by at most $\ep/2$ on each cell. Monotonicity in both endpoints brackets $Z^n_k(x,y)$ by its values at the lower-left and upper-right corners. Thus a centered error exceeding $\ep\sqrt{k}$ forces a corner error exceeding $\ep\sqrt{k}/2$. The grid is finite for each $\ep$, so summing the pointwise bounds from \eqref{E:onotation} over its corners gives \eqref{E:ZnkxyKxy}.} With $F^n_k$ and $G^n_k$ as in Proposition \ref{P:fn-cvg} and Corollary \ref{C:Gn-upper}, by the triangle inequality \eqref{E:triangle-ineq} we have 
	\begin{equation}
	\label{E:Hn-Fn-Gn}
	W^n[(a_n + x, n) \to (y, 1)] \ge F^n_k(x, w) + G^n_k(w, y).
	\end{equation}
	{Equality holds in \eqref{E:Hn-Fn-Gn} at $w=Z^n_k(x,y)$. It therefore suffices to show a gap of order $\sqrt{k}$ whenever $|w+\sqrt{k/(2x)}|>\ep\sqrt{k}$, with errors controlled in the sense of \eqref{E:onotation}.} By Proposition \ref{P:fn-cvg} and Corollary \ref{C:Gn-upper}, for any compact interval $[b, c] \sset (-\infty, 0)$ we have the following bound:
	\begin{align}
	\label{E:abc}
	\sup_{a \in [b, c]} F^n_k(x, a \sqrt{k}) + G^n_k(a \sqrt{k}, y) +\sqrt{k} \lf(\sqrt{2|a|x} - \sqrt{1/|a|}\rg)^2 &\le \oo(\sqrt{k}).
	\end{align}
	{Equivalently, $F^n_k(x,a\sqrt{k})+G^n_k(a\sqrt{k},y)\le -\sqrt{k}\,q_x(a)^2+\oo(\sqrt{k})$ uniformly for $a\in[b,c]$, where $q_x(a)=\sqrt{2|a|x}-\sqrt{1/|a|}$. This square vanishes only at $a_*=-1/\sqrt{2x}$, giving a strict gap away from $a_*$.}

	Moreover, we can use the monotonicity of $F^n_k$ and $G^n_k$ with the bounds from Proposition \ref{P:fn-cvg} and Corollary \ref{C:Gn-upper} to get bounds outside of $[b, c]$:
	\begin{equation}
	\label{E:abc-2}
	\begin{split}
	\sup_{w\in[a_n,b\sqrt{k}]} F^n_k(x,w)+G^n_k(w,y)
	&\le F^n_k(x,b\sqrt{k})+G^n_k(a_n,y)\\
	&\le 2\sqrt{k}(\sqrt{2x}+bx)+\oo(\sqrt{k}),\\
	\sup_{w\in[c\sqrt{k},{y}]} F^n_k(x,w)+G^n_k(w,y)
	&\le F^n_k(x,{y})+G^n_k(c\sqrt{k},{y})\\
	&\le 2\sqrt{2kx}+\sqrt{k}/c+\oo(\sqrt{k}).
	\end{split}
	\end{equation}
	Given $x>0$, we pick $b,c<0$ so that $\sqrt{2x} + bx$ and $\sqrt{2x} + 1/(2c)$ are negative. 
	{Choose them with $b<a_*<c$. The bounds in \eqref{E:abc} and \eqref{E:abc-2} then give, for each $\ep>0$, a $\delta>0$ such that
	\[
	\sup_{\substack{w\in[a_n+x,y]\\ |w+\sqrt{k/(2x)}|>\ep\sqrt{k}}}
	\bigl(F^n_k(x,w)+G^n_k(w,y)\bigr)\le -\delta\sqrt{k}+\oo(\sqrt{k}).
	\]
	Let $H_n=W^n[(a_n+x,n)\to(y,1)]$, the left hand side of \eqref{E:Hn-Fn-Gn}, and let $H$ be its shifted Tracy--Widom limit. Continuity of the law and its finite fourth moment give, for every $\ep>0$,
	\[
	\sum_{k=1}^{\infty}\limsup_{n\to\infty}\mathbb{P}(|H_n|>\ep\sqrt{k})
	=\sum_{k=1}^{\infty}\mathbb{P}(|H|\ge\ep\sqrt{k})<\infty.
	\]
	Consequently, $H_n=\oo(\sqrt{k})$. If $|Z^n_k(x,y)+\sqrt{k/(2x)}|>\ep\sqrt{k}$, equality in \eqref{E:Hn-Fn-Gn} forces $H_n\le-\delta\sqrt{k}/2$ unless the error in the preceding bound exceeds $\delta\sqrt{k}/2$. The summable bounds give \eqref{E:ZnkxyKxy}.}
	
	For fixed $x, y,$ and $k$, the tightness of $Z^n_k(x, y)$ in $n$ follows from \eqref{E:ZnkxyKxy} and the monotonicity $y \ge Z^n_2(x, y) \ge Z^n_3(x, y) \ge \dots.$ {Indeed, for any $\eta>0$, choose $m>k$ large so that \eqref{E:ZnkxyKxy} gives a finite lower bound on $Z^n_m(x,y)$ with probability at least $1-\eta$ for all large $n$. Then $Z^n_k(x,y)\ge Z^n_m(x,y)$ and $Z^n_k(x,y)\le y$, which proves tightness.} 
\end{proof}

Let $\pi_n(x, y)$ be the rightmost geodesic in $W^n$ from $(a_n + x, n)$ to $(y, 1)$. We can use Lemma \ref{L:zk-unif-bd} to prove a result about disjointness of geodesics in $W^n$. 

\begin{lemma}
	\label{L:close-paths}
	For every $x> 0$ and $y \in \R$ we have
	\begin{equation}
	\label{E:limepsilon}
	\lim_{\epsilon \downarrow 0} \limsup_{n \to \infty} \mathbb{P}(\pi_n(0, y - \ep) \text{ and } \pi_n(x, y) \text{ are essentially disjoint}) = 0.
	\end{equation}
\end{lemma}

This and the next lemma are analogues of Lemma 7.2 in \cite*{DOV}. The proof is slightly different, since the prelimiting last passage percolation here is not necessarily stationary. 

\begin{proof}
	We will prove the lemma with the leftmost geodesic $\bar \pi_n(0, y -\ep)$ in place of the rightmost one $\pi_n(0, y-\ep).$ This is a stronger statement by monotonicity of last passage geodesics, Lemma \ref{L:mono-path}.
	
	The path $\bar \pi_n(0, y-\ep)$ simply follows the top line in $W^n$ since $W^n_i(0^-) = 0$ for all $i$ and $W^n$ is ordered, Proposition \ref{P:W-facts}(ii). Therefore it is {essentially disjoint} from $\pi_n(x, y)$ {if and only if} $\pi_n(x, y)$ has its final jump {at or after} time ${y-\ep}$. The final jump of $\pi_n(x, y)$ occurs at the time $Z^n_2(x, y)$, so the probability in \eqref{E:limepsilon} is bounded above by
	$$
	\mathbb{P}(Z^n_2(x, y) {\ge} y -\ep).
	$$
	Lemma \ref{L:zk-unif-bd} implies that the sequence 
	\begin{equation}
	\label{E:Zn2x}
	(Z^n_i(x, y), i = 2, 3, \dots)
	\end{equation}
	is tight in $n$. Moreover, $W_k^n(\cdot) = A^0_{n, k}(\cdot)$, where $A^0_{n, k}$ is as in the statement of Theorem \ref{T:airy-sheet-gen}. Therefore by the assumption in Theorem \ref{T:airy-sheet-gen}, the pair $(W^n_k : k \in \N, Z^n_i(x, y), i \ge 2)$ is jointly tight, where the underlying topology is the Skorokhod topology for the paths $W^n_k$. Subsequential limits are of the form $(\scrA_k, k \in \N , Z_i, i \ge 2)$, where $\scrA$ is a parabolic Airy line ensemble.
	To complete the proof of the lemma, it suffices to show that in such a subsequential limit, $Z_2 < y$ almost surely.
	
	Consider such a subsequential limit. By Skorokhod's representation theorem, we can find a coupling of the environments such that along some subsequence $(W^n_k : k \in \N, Z^n_i(x, y), i \ge 2) \to (\scrA_k, k \in \N , Z_i, i \ge 2)$ almost surely. Since Skorokhod convergence implies uniform convergence on compact sets when the limit is continuous, along this subsequence $W^n_k \to \scrA_k$ compactly for all $k \in \N$. Therefore for any $k \in \N$, almost surely
	$$
	W^n[(Z^n_k(x, y), k) \to (y, 1)] \to \scrA[(Z_k, k) \to (y, 1)].
	$$
	In particular, since the points $Z^n_i(x, y), i = 2, \dots, k$ are jump times on a geodesic in $W^n$ from $(Z^n_k(x, y), k)$ to $(y, 1)$, the points $\{Z_i : i \in [2, k]\}$ are the jump times on a geodesic in $\scrA$ from $(Z_k, k)$ to $(y, 1)$.
	
	The asymptotics of $Z^n_k(x, y)$ from Lemma \ref{L:zk-unif-bd} imply that almost surely, the infimum 
	$$
	K = \inf \{k \in \{2, 3, \dots\} : Z_k < y - 1\}
	$$
	is finite. If $K = 2$, then $Z_2 < y$ as desired. If not, then the points $\{Z_i : i \in \{2, \dots K-1\}\}$ are jump times on a geodesic in $\scrA$ from $(y-1, K-1)$ to $(y, 1)$. Now, by Proposition \ref{P:brownian-airy}, for any $k \in \N$ the top $k$ lines of $\scrA$ restricted to the interval $[y-1, y]$ are absolutely continuous with respect to $k$ independent Brownian motions. Therefore for any $k \in \N$, almost surely all jump times on any geodesic in $\scrA$ from $(y-1, k-1)$ to $(y, 1)$ are contained in the open interval $(y-1, y)$.  In particular, $Z_2 < y$ almost surely on the event $K > 2$. 
\end{proof}

\begin{lemma}
	\label{L:close-paths-2}
	For every $x > 0$ and $y < z \in \R$ we have
	$$
	\lim_{\epsilon \to 0} \limsup_{n \to \infty} \mathbb{P}(\pi_n(x, y) \text{ and } \pi_n(x + \epsilon, z) \text{ are essentially disjoint}) = 0.
	$$
\end{lemma}

\begin{proof}
	There exist essentially disjoint geodesics from $p_1\to q_1$ and $p_2 \to q_2$ in any environment $f \in \scrD^\Z$ if and only if 
	$$
	f[p_1 \to q_1] + f[p_2 \to q_2] = f[(p_1, p_2) \to (q_1, q_2)].
	$$
	Since last passage values for arbitrary disjoint paths from line $n$ to line $1$ are preserved by $W$ by Proposition \ref{P:W-facts}(i), this implies that there exist essentially disjoint geodesics from $(a_n + x, n) \to (y, 1)$ and $(a_n + x + \ep, n) \to (z, 1)$ in $W^n$ if and only if this holds in $L^n$. By the same argument, this holds if and only if there exist essentially disjoint geodesics in the reverse melon $W^*_{z} L^n$ from $(z-y,n) \to (-a_n + z - x,1)$ and $(0,n)\to(-a_n + z - x - \ep,1)$.

Since the environment reversed at $z$ satisfies the same assumptions as the original environment after a shift by $a_n$, we can apply Lemma \ref{L:close-paths} to $W^*_{z} L^n$ to complete the proof.
\end{proof}

Lemma \ref{L:close-paths-2} and Lemma \ref{L:zk-unif-bd} can be combined to give bounds on the branching points of geodesics in $W^n$.

\begin{corollary}
	\label{C:tree-struct}
Define
	$
	K_n(x; y, z) = \inf \{k \in \N : Z^n_k(x, y) = Z^n_k(x, z)\}.
	$
	For every triple $x > 0, y < z$, the sequence $K_n(x; y, z)$ is tight.
\end{corollary} 

\begin{proof} Recall the notation $\pi(x, y)$ for the rightmost geodesic in $W^n$ from $(a_n + x, n)$ to $(y, 1)$. We also let $\pi^*_n(x, y; x', y') = \pi(x, y) \cap \pi(x', y')$. This intersection is a (possibly empty)  path by Lemma \ref{L:mono-path}.
	
	Let $\ep > 0$. Observe that $\pi^*_n(x, z; x + \ep, y) = \pi^*_n(x, y; x + \ep, z)$ whenever the rightmost geodesics $\pi_n(x, y)$ and $\pi_n(x+ \ep, z)$ are not disjoint. Moreover, on this event, $\pi^*_n(x, y; x, z)$ is given by the union of $\pi^*_n(x, z; x + \ep, y)$ with the initial segment of the geodesic $\pi_n(x, z)$. That is
	$$
	\pi^*_n(x, y; x, z) = \pi^*_n(x, z; x + \ep, y) \cup \{(r, m) \in \pi_n(x, z) : \exists (r', m') \in \pi^*_n(x, z; x + \ep, y) \text{ s.t. } r \le r'\}.
	$$
	This uses both the monotonicity and tree structure of rightmost geodesics established in Lemma \ref{L:mono-path}.
	Therefore $K_n(x, y, z) \le k$ on the event where
	\begin{itemize}[nosep]
		\item $\pi^*_n(x, z; x + \ep, y)$ intersects the top $k-1$ lines in $W^n$, and 
		\item $\pi_n(x, y)$ and $\pi_n(x+ \ep, z)$ are not disjoint.
	\end{itemize}
	Now fix $\de > 0$. The second condition holds with probability at least $1 - \de$ for all large enough $n$ as long as $\ep$ is small enough by Lemma \ref{L:close-paths-2}. The first condition holds whenever 
	\begin{equation}
	\label{E:Znk1}
Z^n_{k-1}(x, z) < Z^n_k(x + \ep, y),
	\end{equation} as this implies that the paths $\pi(x, z)$ and $\pi(x + \ep, y)$ must cross in the top $k$ lines of $W^n$. For any fixed $\ep > 0$, the asymptotics in Lemma \ref{L:zk-unif-bd} imply that \eqref{E:Znk1} happens with probability at least $1 - \de$ for all large enough $n$ as long as $k$ is large enough. Since $\de > 0$ was arbitrary, $K_n(x, y, z)$ is tight.
\end{proof}

We are now ready to prove Theorem \ref{T:airy-sheet-gen}.

\begin{proof}[Proof of Theorem \ref{T:airy-sheet-gen}]
The assumptions of the theorem guarantee that
	$$
	S_n(a, \cdot) \cvgd \scrA_1(\cdot - a) \quad \mathand \quad S_n(\cdot, a) \cvgd \scrA_1(a - \cdot)
	$$
	for every $a \in \Q$.
	Since the process $\scrA_1$ is continuous, Lemma \ref{L:tight-sheet} then implies that $S_n|_{[-r, r]^2}$ is tight in $\CDF_r$ for every $r \in \N$, and all distributional limits of $S_n|_{[-r, r]^2}$ are continuous. Consider any joint distributional subsequential limit of the sequence $(S_n|_{[-r, r]^2}, r \in \N), n \in \N$. This limit must be of the form $(\tilde \scrS|_{[-r, r]^2}, r \in \N)$ for some continuous function $\tilde \scrS$.
	We will show that $\tilde \scrS$ is an Airy sheet. The compact convergence in Theorem \ref{T:airy-sheet-gen} then follows from Skorokhod's representation theorem{ in the ambient space of triples}, Lemma \ref{L:sheet-topology}{ applied with a larger $r$}, and continuity of the Airy sheet. 
	
	To prove that $\tilde \scrS$ is an Airy sheet, we just need to show that 
	$$
	\tilde \scrS|_{[z, \infty) \X \R} \eqd \scrS_{[z, \infty) \X\R}
	$$
	for every $z \in \R$ where $\scrS$ is an Airy sheet. Since the assumptions of Theorem \ref{T:airy-sheet-gen} are invariant with respect to shifting $z$, and since $\scrS$ is shift invariant by definition, it is enough to prove this for $z = 0$.
	
	By Skorokhod's representation theorem, we can find a subsequence of $\{S_n\}$ and a coupling of the corresponding environments $L_n$ for which the following convergences all hold almost surely:
	\begin{itemize}[nosep]
		\item The functions $A^0_{n, k}(\cdot) = L_n[(a_n, n) \to_{\Delta_k} (\cdot, 1)]$ converge compactly to the Airy line ensemble $\scrA_k(\cdot)$.
		\item $S_n \to \tilde \scrS$ compactly and $\tilde \scrS(x, x + \cdot)$ is a parabolic Airy process for  rational $x$.
		\item For every $x \in \Q^+$ and $y, z \in \Q$, each of the jump times $Z^n_k(x, y)$ converges to a limit $Z_k(x, y)$ and $K_n(x; y, z)$ converges to a limit $K(x; y, z)$. Moreover,
		\begin{equation}
		\label{E:Zk-limit}
		\lim_{k \to \infty} \frac{Z_k(x, y)}{\sqrt{k}} = \frac{-1}{\sqrt{2x}}.
		\end{equation}
		This uses Lemma \ref{L:zk-unif-bd} and Corollary \ref{C:tree-struct}.
	\end{itemize} 
	Let $\scrS$ be an Airy sheet coupled to $\scrA$ by the relationship \eqref{E:buse}. By continuity of $\tilde \scrS, \scrS$, it is enough to show that $\tilde \scrS(x, y) = \scrS(x, y)$ for all $(x, y) \in \Q^+ \X \R$. For $x \in \Q^+, y < z \in \Q$ the third condition above guarantees that for all $k \ge K(x; y, z)$ we have
	$$
	\tilde \scrS(x, y) - \tilde \scrS(x, z) = \scrA[(Z_k(x, z), k) \to ({y}, 1)] - \scrA[(Z_k(x, z), k) \to ({z}, 1)].
	$$
	By the asymptotics in \eqref{E:Zk-limit}, for any $\ep > 0$, for all large enough $k$
	we can apply the quadrangle inequality, Lemma \ref{L:quadrangle}, to the points $(Z_k(x,z),k), (-\sqrt{k/(2(x + \ep))}, k)$ and $(z,1)$,$(y,1)$ in the environment $\scrA$. This gives the lower bound
	$$ 
	\tilde \scrS(x, y) - \tilde \scrS(x, z) \ge \scrA[(-\sqrt{k/(2(x + \ep))}, k) \to ({y}, 1)] - \scrA[(-\sqrt{k/(2(x + \ep))}, k) \to ({z}, 1)].
	$$
	Formula \eqref{E:buse} for $\scrS$ then implies that
	$$
	\tilde \scrS(x, y) - \tilde \scrS(x, z) \ge \scrS(x + \ep, y) -  \scrS(x +\ep, z) 
	$$
	for all $\ep > 0$. By the same reasoning, the opposite inequality holds for all ${-x<\ep<0}$. The continuity of $\scrS, \tilde \scrS$ then implies that $\tilde \scrS(x, y) - \tilde \scrS(x, z) = \scrS(x, y) - \scrS(x, z)$ for rational $x, y, z$. This extends to all $x \in \Q^+, y, z \in \R$, since both $\tilde \scrS(x, \cdot)$ and $\scrS(x, \cdot)$ are continuous for rational $x$, and hence
	$$
	\scrS(x, \cdot) = \tilde \scrS(x, \cdot) + C_x
	$$
	for random constants $C_x, x \in \Q^+$. Since $\scrS(x, \cdot) \eqd \tilde \scrS(x, \cdot) \eqd \scrA_1(\;\cdot \;- x)$, ergodicity of the Airy process $\scrA_1(t) + t^2$ (equation (5.15) in \cite*{prahofer2002scale}), implies that $\scrS(x, \cdot) = \tilde \scrS(x, \cdot)$ for rational $x$, as desired.
\end{proof}

\begin{remark}
	\label{R:airy-sheet-coup}
	The proof of Theorem \ref{T:airy-sheet-gen} actually shows something stronger than just convergence of $S_n$ to $\scrS$. Namely, letting $A_{n, k}(\cdot) = L_n[(a_n, n) \to_{\Delta_k} (\cdot, 1)]$, it shows that there is a coupling of the environments $L_n$ such that the functions $(S_n,A_n)$ converge compactly to $(\scrS, \scrA)$, where $\scrS$ is an Airy sheet and $\scrA$ is a parabolic Airy line ensemble, coupled via the relationship in Definition \ref{D:airy-sheet}.
\end{remark}

\section{Directed metrics}
\label{S:directed-metrics}
In Sections \ref{S:directed-metrics} to \ref{S:graph-cvg}, we develop an abstract framework for deciding when a sequence of last passage percolation models that converges to Airy sheet also converges to the directed landscape. To study this question, we introduce directed metric spaces, which generalize metric spaces. In contrast with the metric property, the directed metric property is preserved under certain natural scaling operations. The directed landscape and all last passage percolation models are random directed metrics.

\begin{definition}
A \textbf{directed metric of positive sign} on a set $S$ is a function $d:S \X S \to \R \cup \{\infty\}$ satisfying: 
\begin{itemize}[nosep]
    \item $d(x,x)=0$ for all $x\in S$, 
    \item (Triangle Inequality) $d(x, z) \le d(x,y)+d(y,z)$ for all  $x,y,z\in S$.
\end{itemize}
We call the pair $(S, d)$ a \textbf{directed metric space}.
\end{definition}
A metric is always a directed metric of positive sign. A directed metric of positive sign is a generalization of a metric without the symmetry condition $d(x, y) = d(y, x)$ and the positivity condition $d(x, y) > 0$ whenever $x \ne y.$

A \textbf{directed metric of negative sign} is a function $d:S \X S \to \R\cup \{-\infty\}$ satisfying $d(x, x) = 0$ for all $x \in S$ and the reverse triangle inequality $d(x, z) \ge d(x, y) + d(y, z)$ for all $x, y, z \in S$. Equivalently, a function $d$ is a directed metric of negative sign if $-d$ is a directed metric of positive sign.

Directed metrics of negative sign will be important to us later on because last passage models can be viewed in this way. However, for the remainder of this section, we restrict our attention to directed metrics of positive sign since the two notions are equivalent up to a sign change, and directed metrics of positive sign can more naturally be thought of as distance functions.

There are a few standard methods for building new directed metrics from old ones. These methods are summarized by the following lemma, whose proof we leave as a straightforward exercise.

\begin{lemma}
\label{L:straightforward-dms}\ \\ \vspace{-1.5em}
\begin{enumerate}[nosep]
	\item Let $g$ be a function from a set $R$ to a directed metric space $(S,d)$. Then the {\bf pullback} of $d$, defined by 
	$$
	d'(x,y)=d(g(x),g(y))
	$$ 
	is a directed metric on $R$.
	\item If $d_1, \dots d_k$ are directed metrics on a space $S$ and $c_1, \dots, c_k \ge 0$, then $c_1 d_1 + \dots + c_k d_k$ is also a directed metric on $S$, {with the convention $0\cdot\infty=0$}.
	\item If $d_1, d_2$ are directed metrics on $S$, then $\max(d_1, d_2)$ is also a directed metric on $S$.
	\item If $\{d_n\}$ is a sequence of directed metrics on $S$ with a pointwise limit $d:S \X S \to \R \cup \{\infty\}$, then $d$ is a directed metric on $S$.
	\item If $I$ is a {nonempty} collection of directed metrics on $S$ then the function $d:S \X S \to \R \cup \{\infty\}$ given by $d(u) = \sup \{e(u) : e \in I\}$ is also a directed metric.
	\item For any function $h:S \to \R$, the function $d_h(x, y) = h(x) - h(y)$ is a directed metric of both positive and negative sign.
\end{enumerate}
\end{lemma}

We will combine points 2, 4, and 6 later to take limits of directed metrics after centering and rescaling.

The method of building directed metrics in Lemma \ref{L:straightforward-dms}.5 allows us to associate a canonical directed metric to any function from $S \X S$ to $\R$. 

\begin{definition}
\label{D:induced-dm}
Let $S$ be a set, $A \sset S \X S$, and $d_0:A \to \R \cup \{\infty\}$. If there is a directed metric $d$ so that  $d|_A \le d_0$, then we can define a directed metric on $S$ by
$$
e(d_0)(u) = \sup \{d(u)\}
$$
where the supremum is over all directed metrics satisfying $d|_A\le d_0$.
This \textbf{induced directed metric} $e(d_0)$ is the maximal directed metric on $S$ that is bounded above by $d_0$ on $A$.
\end{definition}

We will also want to {study paths attaining equality in the triangle inequality}. To do this, we introduce an abstract definition of geodesics. This definition has the advantage that it does not require any extra structure on the space $(S, d)$. 

\begin{definition}
	\label{D:geodesic}
Let $(S, d)$ be a directed metric space and let $x, y \in S$ be such that $d(x, y)$ is finite. A subset $A \sset S$ is a \textbf{geodesic set} from $x$ to $y$ if there exists a total order $\preceq$ on $A$ such that 
\begin{itemize}[nosep]
	\item $x \preceq z \preceq y$ for all $z \in A$. In other words, $x$ and $y$ are minimal and maximal elements in $(A, \prec)$.
	\item $d(a,c)=d(a, b)+d(b, c)$ for all triples $a \preceq b \preceq c\in A$.
\end{itemize}
We call $\preceq$ a compatible total order on $A$.
We can put a partial order on all geodesic sets from $x$ to $y$ by inclusion. Maximal elements in this partial order are called \textbf{geodesics}. 
\end{definition}

\begin{prop}
\label{P:geod-exist}
Let $(S, d)$ be a directed metric space, and let $x, y \in S$ be points with $d(x, y) < \infty$. Let $A$ be a geodesic set from $x$ to $y$ in $(S, d)$. Then $A$ is contained in a geodesic $\pi$ from $x$ to $y$.

Since $\{x, y\}$ is always a geodesic set from $x$ to $y$ in $(S, d)$, this implies that there is a geodesic between every pair of points $x, y \in S$ with $d(x, y) < \infty$. 
\end{prop}

To prove Proposition \ref{P:geod-exist}, we need to investigate how different compatible total orders on geodesic sets are related. For this, we introduce a fundamental equivalence relation.

\begin{definition}
	\label{D:dm-equivalence}
	Let $(S, d)$ be a directed metric space. We say that $x$ and $y$ are $d$-equivalent and write $x \sim y$ if $d(x, y) + d(y, x) = 0$.
\end{definition}

By the triangle inequality, $d(x, y) + d(y, x) {\ge} d(x, x) = 0$ for every pair $x, y$. If two points $x$ and $y$ are $d$-equivalent then this is an equality, and we can move back and forth between $x$ and $y$ at zero cost. Note that if $x \sim y$, then both $d(x, y)$ and $d(y, x)$ are finite.
If $d$ is a true metric, then $d$-equivalence is a trivial relation. The next lemma gives two important properties of $d$-equivalence.

\begin{lemma}
\label{L:dequiv-relate}
Let $(S, d)$ be a directed metric space. 
\begin{enumerate}[nosep]
	\item  If $x, y, z \in S$ and $y \sim z$, then
	\begin{equation}
	\label{E:forfree-moves}
	d(x, z) = d(x, y) + d(y, z) \quad \mathand \quad d(z, x) = d(z, y) + d(y, x).
	\end{equation}
	\item The relation $\sim$ is an equivalence relation.
\end{enumerate}
\end{lemma}

\begin{proof}
For any points $x, y,$ and $z$, by two triangle inequalities we have that
	\begin{equation}
	\label{E:2tri}
	d(x, z) \le d(x, y) + d(y, z) \le d(x, z) + d(z, y) + d(y, z).
	\end{equation}
If $y \sim z$, then the right and left hand sides of \eqref{E:2tri} are equal. Hence all inequalities in \eqref{E:2tri} are in fact equalities, yielding the first equation in \eqref{E:forfree-moves}. The second equation follows by symmetric reasoning.

All parts of checking that $\sim$ is an equivalence relation are self-evident except for transitivity. For transitivity, if $x \sim y$ and $y \sim z$, then both equalities in \eqref{E:forfree-moves} hold. Adding these two equations together and using that $x \sim y$ and $y \sim z$ gives that $d(x, z) + d(z, x) = 0$.
\end{proof}

\begin{lemma}
\label{L:geodesic-compat} Let $(S, d)$ be a directed metric space, and let $x, y\in S$ with $d(x, y) < \infty$.
Let $A$ be a geodesic set from $x$ to $y$ and let $q$ be the quotient map from $A$ to $\bigslant{A}{\sim}$. Then there exists a unique total order $\preceq$ on $\bigslant{A}{\sim}$ such that $\le$ is a compatible order on $A$ if and only if $x \le z \le y$ for all $z \in A$ and
\begin{equation}
\label{E:abiff}
a \le b \qquad \text{ implies } \qquad qa \preceq qb. 
\end{equation}

\end{lemma}

In other words, Lemma \ref{L:geodesic-compat} says that compatible total orders on a geodesic set $A$ from $x$ to $y$ are all equivalent up to rearranging points in the same equivalence class.
\begin{proof}
Consider points $a, b, c \in A$, and $a' \sim a, b' \sim b, c' \sim c$. By repeated applications of Lemma \ref{L:dequiv-relate}.1, we have
\begin{align}
d(a, c) &= d(a, a') + d(a', c') + d(c', c), \qquad \mathand \\
d(a, b) + d(b, c) &= d(a, a') + d(a', b') + d(b', c') + d(c', c). 
\end{align}
Moreover, since $a \sim a'$ and $c \sim c'$, both $d(a, a')$ and $d(c, c')$ are finite. Therefore 
\begin{align}
\label{E:three-prong}
d(a, c) = d(a, b) + d(b, c) \qquad \text{if and only if} \qquad d(a',c') = d(a', b') + d(b',c').
\end{align}
We use this to prove the lemma. 
Let $\le$ be a compatible total order on $A$. We can construct an induced order $\preceq$ on $\bigslant{A}{\sim}$ by setting $w \prec z$ for $w \ne z$ whenever $a < b$ for some $a \in q^{-1} w, b \in q^{-1} z$. We first check that $\prec$ is well-defined. For this, suppose that $a, a', b, b' \in A$ with $a < b, a' > b'$ with $a \sim a'$ and $b \sim b'$. We just need to check that $a' \sim b'$. Since $a < b$ and $a' > b'$, we have
\begin{equation}
\label{E:dxb}
d(x, b) = d(x, a) + d(a, b) \qquad \mathand \qquad d(x, a') = d(x, b') + d(b', a').
\end{equation}
Applying \eqref{E:three-prong} with the points $x, b, a$ and $x, b', a'$ to the first equation above, we get that
$$
d(x, b') = d(x, a') + d(a', b').
$$
Combining this with the second equation in \eqref{E:dxb} gives that $d(a', b') = - d(b', a')$, as long as the distances $d(x, b'), d(x, a')$ are both finite. This is guaranteed by the assumption that $d(x, y) < \infty$ and the second condition on geodesic sets applied to the triples $x, a', y$ and $x, b', y$.

We now show that $\preceq$ satisfies the if and only if statement in the theorem. First suppose that $\le'$ is a total order on $A$ with $x \le' {z} \le' {y}$ for all $z$, satisfying \eqref{E:abiff}. We show that $\le'$ is compatible.
 Let $a \le' b \le' c$. By the construction of $\preceq$, we can find $\hat a \in q^{-1}q a, \hat b \in q^{-1}q b, \hat c \in q^{-1}q c$ such that $\hat a \le \hat b \le \hat c$, and hence
$$
d(\hat a, \hat c) = d(\hat a, \hat b) + d(\hat b, \hat c).
$$
Applying \eqref{E:three-prong} to the points $a, b, c$ and $\hat a, \hat b, \hat c$ then shows that $\le'$ is compatible.

Now, suppose that $\le'$ is any total order on $A$. For $\le'$ to be compatible, we need $x$ and $y$ to be minimal and maximal elements in $\le'$. We check that $\le'$ must satisfy \eqref{E:abiff} by contradiction. Suppose that \eqref{E:abiff} fails. Then there exists $a <' b \in A$ with $qb \prec qa$. This implies that $b < a$ in the original order on $A$, and so 
$$
d(x, a) = d(x, b) + d(b, a).
$$
Moreover, since $qb \prec qa$, we have $a \not\sim b$ and hence $d(a, b) > -d(b, a)$. Therefore using that $d(x, a)$ and $d(x, b)$ are finite, we get that $d(x, b) < d(x, a) + d(a, b)$. Thus $<'$ is not compatible.
\end{proof}

We can use this structure of compatible orders in geodesic sets to find geodesics in directed metric spaces.

\begin{proof}[Proof of Proposition \ref{P:geod-exist}]
Let $I$ be the set of all geodesic sets from $x$ to $y$ containing $A$, and let 
$$
I' = \{(A, \prec) : A \in I, \prec \text{ compatible}\}.
$$
We can put a partial order on $I'$ by saying that $(A, \prec) \le (A', \prec')$ if $A \sset A'$ and if $\prec'|_{A \X A} = \; \prec$. Any totally ordered subset $J \sset I'$ has an upper bound in $I'$ given by the union of all sets in $J$ with the union of all the total orders. Hence by Zorn's lemma, $I'$ contains a maximal element. {Let $(A, \prec)$ be such an element. We show that $A$ is maximal in $I$, and hence is a geodesic.}

For this, suppose $A \sset A'${ with $A' \in I$, and let $\prec'$ be a compatible order on $A'$}. 
Let $\prec'_*$ be the total order induced on $ \bigslant{A'}{\sim}$ by $\prec'$ via the quotient map $q:A' \to \bigslant{A'}{\sim}$ and consider {a} total order $\triangleleft$ on $A'$ extending $\prec${, with $x$ its least element and $y$ its greatest element}. Define a new total order $\prec''$ on $A'$ {by
\[
a \prec'' b \quad \text{if and only if} \quad qa \prec'_* qb \quad \text{or} \quad (qa=qb \text{ and } a \triangleleft b).
\]} The `only if' part of Lemma \ref{L:geodesic-compat} implies that $\prec''$ is an extension of $\prec$, and the `if' part of Lemma \ref{L:geodesic-compat}, implies that $\prec''$ is compatible on $A'$. Hence $(A, \prec) \le (A', {\prec''})$ so by maximality, $A = A'$.
\end{proof}

By Proposition \ref{P:geod-exist}, for any pair of points $x, y \in S$ for which $d(x, y)$ is finite, there is at least one geodesic from $x$ to $y$ since $\{x, y\}$ is always a geodesic set from $x$ to $y$. Note that geodesics may not always yield any interesting information about the space (i.e. the pair $\{x, y\}$ could be the only geodesic from $x$ to $y$, or the entire space $S$ could be a geodesic between any pair of {distinct} points).

We can use Proposition \ref{P:geod-exist} to show that geodesics behave well under pullbacks.

\begin{lemma}
	\label{L:pullback-geod}
	Let $g$ be a surjective function from a set $R$ to a directed metric space $(S,d)$, and let $e$ be the pullback metric on $R$. Let $x, y \in R$ be such that $e(x, y) < \infty${ and $gx\ne gy$}. Then
	$$
	\{ A: A \sset R \text{ is an $e$-geodesic from $x$ to $y$} \} = \{g^{-1} B : B \sset S \text{ is a $d$-geodesic from $gx$ to $gy$} \}.
	$$
\end{lemma}

\begin{proof}
	The inverse image of a geodesic set $B$ {from $gx$ to $gy$} in $(S, d)$ is a geodesic set {from $x$ to $y$} in $(R, e)$. To see this, observe that if $\prec$ is a compatible total order on $B$, then by breaking ties {with $x$ least in its fibre and $y$ greatest in its fibre,} we can find a total order $\preceq'$ on $g^{-1}B$ such that {$a \prec' b$ whenever $ga \prec gb$.} The order $\prec'$ makes $g^{-1}B$ a geodesic set. 
	
	On the other hand, if $A$ is a geodesic set from $x$ to $y$, then we claim that $gA$ is a geodesic set from $gx$ to $gy$. To do this, we just need to find a compatible total order on $A$ whose pushforward onto $gA$ is also a well-defined total order. 
Let $q:S \to \bigslant{S}{\sim}$ be the quotient map on $S$ for $d$-equivalence. Then $q' = q g$ is the quotient map on $R$ for $e$-equivalence. Let $\preceq$ be the total order on $q' A$ identified via Lemma \ref{L:geodesic-compat}. Let $\preceq_S$ be any {total} order on $gA$ such that $a \preceq_S b$ implies $q a \preceq qb${, with $gx$ least and $gy$ greatest}, and let $\preceq_R$ be any {total} order on $A$ such that $a \preceq_R b$ implies $ga \preceq_S gb${, with $x$ least and $y$ greatest}. Then by Lemma \ref{L:geodesic-compat}, $\preceq_R$ is a compatible order on $A$. {The induced order on $gA$ is $\preceq_S$, and is compatible since $e(a,b)=d(ga,gb)$.}

	Now, let $A$ be a geodesic from $x$ to $y$ in $R$. Then $gA$ is a geodesic set from $gx$ to $gy$. Since $e(x, y) = d(gx, gy)$, we have that $d(gx, gy) < \infty$. Therefore by Proposition \ref{P:geod-exist}, there is a geodesic  $B$ from $gx$ to $gy$ so that $gA \sset B$. We want to show that $A = g^{-1} B$. For this, observe that $g^{-1} B$ is a geodesic set from $x$ to $y$ in {$R$}. Since $A\subset g^{-1}B$, maximality implies  $A = g^{-1} B.$
	
	On the other hand, let $B$ be a geodesic from $gx$ to $gy$ in $S$, and let $A$ be a geodesic between points $x$ to $y$ containing the geodesic set $g^{-1} B$. Again, $A$ exists by Proposition \ref{P:geod-exist}. Then $B \sset gA$, and $gA$ is a geodesic set from $gx$ to $gy$. Therefore $gA = B$ by maximality and thus $g^{-1} B = A$. 
\end{proof}

We end this section by noting that certain subsets of geodesics are geodesics. The proof of this lemma is straightforward and left for the reader.
\begin{lemma}
	\label{L:sub-geod}
	Let  $(S, d)$ be a directed metric and let $A$ be a geodesic in $(S, d)$ with compatible order $\le$. Let $[a, b] \sset A$ be an interval in the total order $\le$. Then $[a, b]$ is a geodesic from $a$ to $b$.
	\end{lemma}
\section{Examples of directed metrics}
\label{S:metric-examples}
In this section, we give some important examples of directed metrics, elaborating on Section \ref{SS:classical}. We start with general constructions, then move to constructions related to last passage percolation. Our most important and complex example is left until the end: the directed landscape. 

\begin{example}
\label{Ex:additive-metrics} Let $S$ be a set, and consider a function $h:S\to \mathbb R$. The {\bf additive metric of $h$} is the directed metric $d_h(x,y)=h(x)-h(y)$. Note that $d_h$ is a directed metric of both positive and negative sign. We will frequently just write $h$ in place of $d_h$ to refer to the additive metric when there is no ambiguity. 
\end{example}

Additive metrics are trivial in a certain sense. For example, any two points in $S$ are $d_h$-equivalent, so the whole space $S$ is a geodesic between any pair of {distinct} points. Moreover, if $d$ is another directed metric on $S$ (of either sign), then $d + d_h$ is a directed metric with the same geodesics as $d$, and $d$-equivalence is the same as $(d + d_h)$-equivalence.

Our next set of examples comes from generalizing the notion of graph distance.

\begin{example}
	\label{Ex:graphs} Let $G = (V, E)$ be a directed graph. We can construct the following directed metrics of positive sign from data associated to $G$. 
\begin{enumerate}[label=(\roman*)]
	\item  Define $d_0(x,y)=1$ for all $(x,y) \in E$. The induced directed metric $e(d_0)$ can be thought of as a directed graph distance on $V$. If $G$ is a directed version of an undirected graph (i.e $(x, y) \in E$ whenever $(y, x) \in E$), then this gives the ordinary graph distance. This construction can also be applied to weighted finite graphs, where we set $d_0(x, y)$ to be the weight of the edge $(x, y)$.
	\item Consider a function $h:V \to \R$.  We can consider a directed metric on the edge set $E$ induced by the function $d_0((x,y),(y,z))=h(y)$ for all $(x,y),(y,z)\in E$. Note that this induced metric may not be defined for certain graphs. {If $h>0$, this induced directed metric is well-defined.} 
	\item In the setting of (ii), we can also consider the metric on $V\cup E$ induced by $d_0((x,y),y)=h(y)$ and  $d_0(x,(x,y))=0$ for all $(x,y)\in E$. 
\end{enumerate}
\end{example}

All of the constructions in Example \ref{Ex:graphs} have clear analogues for directed metrics with negative sign. In particular, last passage percolation on $\Z^2$ can be interpreted in the context of (iii) above.

\begin{example}
\label{Ex:LPP}
Consider a directed graph on the vertex set $V = \Z^2$ where $(x, y) \in E$ whenever $y - x \in \{(0, -1), (1, 0)\}$. Then the negative-sign version of the construction in Example \ref{Ex:graphs}(iii) gives a directed metric $d$ on $\Z^2 \cup E$ corresponding to last passage percolation. We can explicitly write out distances as follows. Let $u, v \in \Z^2$ with $u_1 \le v_1$ and $u_2 \ge v_2$, and let
$u^- = u - (1,0), u_- = u - (0, -1)$, and $v^+ = v + (1,0), v_+ = v + (0, -1)$. Then
\begin{align*}
&d(u, x) + h(u) = d(y, x) = \max_\pi \sum_{w \in \pi} h(w), \\
&\quad \text{ for } \quad x \in \{v, (v, v^+), (v, v_+)\}, \quad y \in \{(u^-, u), (u_-, u)\}.
\end{align*}
Here the maximum is over all up-right lattice paths $\pi$ from $u$ to $v$. In the language of Section \ref{S:lattice-lpp}, the right hand side above is equal to the last passage value $h[u \to v]$. We call $d$ the \textbf{lattice last passage metric} defined by $h$. 
\end{example}

We have defined the lattice last passage metric on $\Z^2 \cup E$, rather than just $\Z^2$, in order to make it easier to embed into the plane later on.
We can similarly relate last passage across cadlag paths to an induced metric on $\R \X \Z$. We define this metric on $\R \X (\Z/2)$, where $\R \X (\Z + 1/2)$ will play the role of an edge set making this directed metric easier to embed into the plane.

\begin{example}
\label{Ex:LPP-lines}
Consider $f \in \scrD^\Z$, and let $d$ be the directed metric of negative sign on $\R \X (\Z/2)$ induced by the function
\begin{equation}
\label{E:line-edge-set}
\begin{array}{ll}
d_0((x,n),(y,n))&=f_n(y)-f_n(x)\\
d_0((x,n),(x,n -1/2))&= 0,\\
d_0((x,n+1/2),(x,n))&=f_n(x)-f_n(x^-),
\end{array}
\end{equation}
for $x  < y \in \R$ and $n \in {\Z}$.
The directed metric $d$ corresponds to last passage across functions as defined in Section \ref{S:lpp-cadlag}. Namely,
\begin{align*}
d((x, n+1/2), u) = f_n(x)-f_n(x^-) + d((x,n),u) &= f[(x,n)\to (y,m)],\\ \qquad \text{ for } \qquad u &\in \{(y, m), (y, m-1/2)\}.
\end{align*}
We call $d$ the \textbf{line last passage metric} defined by $f$.
\end{example}

\textbf{Important pullbacks:} \qquad When taking limits, we will have to embed last passage percolation on the lattice and across lines into a continuous setting. This can be done by looking at pullbacks of last passage metrics into the plane. One nice feature of this pullback approach is that we can easily pass between geodesics in the two models via Lemma \ref{L:pullback-geod}.

We first embed the lattice $\Z^2$ into $\R \X (\Z/2)$.
Define $f_{\R \X \Z \to \Z^2}: \R \X (\Z/2) \to \Z^2 \cup E$ by
\begin{equation}
\label{E:embed-1}
f_{\R \X \Z \to \Z^2}(x, n) = \begin{cases}
(x, n), \quad &x, n \in \Z, \\
((\floor{x}, n), (\ceil{x}, n)), \quad &x \notin \Z, n \in \Z, \\
((x, \ceil{n}), (x, \floor{n})), \quad &x \in \Z, n \in \Z + 1/2, \\
((\ceil{x}, \ceil{n}), (\ceil{x}, \floor{n})), \quad &x \notin \Z, n \in \Z + 1/2.
\end{cases}
\end{equation}
{For a metric $d$ defined from a nonnegative function $h$ as in Example \ref{Ex:LPP}, consider also the line last passage metric defined by the cadlag functions $h_n$ from \eqref{E:fGkk}:}
$$
h_n(0) = 0, \quad h_n(x) - h_n(y) = \sum_{u \in (y, x] \X \{n\}} h(u).
$$
{Represent horizontal edges by their midpoints and vertical edges by the corresponding points of $\R\X(\Z+1/2)$. At pairs joined by a directed lattice path, $d$ agrees with the line metric at these representatives: nonnegativity allows a line path to be moved onto the lattice without decreasing its weight. The full pullback under \eqref{E:embed-1} can identify distinct points and need not agree with the line metric at every pair.}
Next, we embed $\R\times \Z$ into the plane $\mathbb R^2$. Define $f_{\R^2 \to \R \X \Z}: \R^2 \to \R \X (\Z/2)$ by
\begin{equation}
\label{E:line-plane}
f_{\R^2 \to \R \X \Z}(x, y) = \begin{cases}
(x, -y), \quad &y \in \Z, \\
(x, \floor{-y} +1/2), \quad &y \notin \Z.
\end{cases}
\end{equation}
Note the sign change in \eqref{E:line-plane}. Our convention for last passage across lines means that finite distances can only be assigned for pairs of points $(u, v)$ where $u$ has a higher line index than $v$. However, in the directed landscape finite distances can only be assigned to a pair of points $(x, s; y, t)$ when $s < t$. Since line index corresponds to time in the directed landscape, we need to reverse the line ordering to prove convergence.
For future use in Section \ref{S:cvg-lattice}, we write out the composition
\begin{equation}
\label{E:lat-to-plane}
f_{\R^2 \to \Z^2} := f_{\R \X \Z \to \Z^2} \circ f_{\R^2 \to \R \X \Z}
\end{equation}
explicitly.
\begin{equation}
\label{E:embed-3}
f_{\R^2\to \Z^2}(x, y) = \begin{cases}
(x, -y), \quad &x, y \in \Z, \\
((\floor{x}, -y), (\ceil{x}, -y)), \quad &x \notin \Z, y \in \Z, \\
((x, \ceil{-y}), (x, \floor{-y})), \quad &x \in \Z, y \notin \Z, \\
((\ceil{x}, \ceil{-y}), (\ceil{x}, \floor{-y})), \quad &x \notin \Z, y \notin \Z.
\end{cases}
\end{equation}

\begin{remark}
	\label{R:planar-embedding} It may be helpful for the reader to keep the following construction in mind when thinking about lattice last passage metrics.
	If $d$ is a lattice last passage metric generated from a  \textbf{nonnegative} weight function $h:\Z^2 \to \R$, and $d'$ is the pullback of $d$ under the map $f_{\R^2\to \Z^2}$, then we can alternately describe $d'$ similarly to Example \ref{Ex:Poi}, as the negative sign metric on $\R^2$ induced by the function
	$
	d_0(a, c) = h(c)\indic(c \in \Z^2),
$
where $d_0$ is defined for all $a, c \in \R^2$ with $a \ne c, a_1 \le c_1, a_2 \le c_2$.
\end{remark}
\subsection*{The directed landscape}
\label{SS:DL}

Our most important example of a (random) directed metric is the directed landscape. To define the directed landscape, let $\scrS$ be an Airy sheet as in Definition \ref{D:airy-sheet}. The rescaled function
\begin{equation}
\label{E:airy-sheet-scale}
\scrS_s(x, y) = s\scrS(x/s^2, y/s^2)
\end{equation}
is an \textbf{Airy sheet of scale $s$}. We refer to the process ${\scrS_s}(0, \cdot)$ as a \textbf{parabolic Airy process} of scale $s$. The directed landscape is built out of independent Airy sheets of different scales.

\begin{definition}
	\label{D:DL-def} The directed landscape is a random function $\scrL:\R^2 \X \R^2 \to \R \cup \{-\infty\}$ satisfying the following properties:
	\begin{enumerate}[label=\Roman*.]
		\item (Domain) Let $\Rd = \{(x, s; y, t) \in \R^4 : s < t\}$ and  $\Delta = \{(u, u) \in \R^4 \}$. The function $\scrL$ is continuous and finite on $\Rd$, equal to zero on the diagonal $\Delta = \{(u, u) \in \R^4 \}$, and equal to $-\infty$ on $\R^4\smin (\Rd \cup \Delta)$.
		\item (Independent Airy sheet increments) For any disjoint time intervals $\{(t_i, s_i) : i \in \{1, \dots k\}\}$, the random functions
		$$
		\scrL(\cdot, t_i ; \cdot, s_i), \quad  i \in \{1, \dots, k \}
		$$
		are independent Airy sheets of scale ${(s_i-t_i)^{1/3}}$.
		\item (Metric composition law) Almost surely, for any $r<s<t$ and $x, y \in \R$ we have
		\begin{equation}
		\label{E:M-comp}
		\scrL(x,r;y,t)=\max_{z \in \mathbb R} [\scrL(x,r;z,s)+\scrL(z,s;y,t)].
		\end{equation}
	\end{enumerate}
\end{definition}
Formally, we need to define a space of functions for $\scrL$ to take values in. We will address this issue in detail in Section \ref{S:topologies}. For now, $\scrL$ can be viewed as a random element in the space of continuous functions from $\Rd$ to $\R$ with the topology of uniform convergence on compact sets, since $\scrL$ is deterministic off of $\Rd$. The directed landscape $\scrL$ is unique in law by Theorem 10.9 in \cite*{DOV}, and properties I and III imply that it is almost surely a directed metric of negative sign on $\R^2$.

We record two properties of $\scrL$ that will be necessary for establishing convergence.

\begin{prop}
\label{P:from-DOV} Almost surely, $\scrL$ satisfies the following.
\begin{enumerate}
	\item $\scrL$ is continuous on $\Rh := \R^4 \smin \Delta$.
	\item For any $(x, s; y, t) \in \Rd$, there exists a (random) compact set $[A, B] \sset \R$ such that all geodesics from $(x, s)$ to $(y, t)$ are contained in $[A, B] \X [s, t]$.
\end{enumerate}
\end{prop}

\begin{proof}
Both of these results follow straightforwardly from the estimate on $\scrL$ in Corollary 10.7 in \cite*{DOV}. This corollary states that there exists a random constant $C>0$ such that almost surely, for all $u = (x, {t}; y, t + s) \in \Rd$ we have
\begin{equation}
\label{E:LuCu}
\lf|\scrL(u) + \frac{(x - y)^2}{s} \rg|\le C s^{1/3} \log^{4/3}\lf(\frac{2(||u|| + 2)^{3/2}}{s}\rg)\log^{2/3}(||u|| + 2). \qedhere
\end{equation}
\end{proof}

We will also use well-known tail bounds on the one-point distributions of $\scrL$.

\begin{theorem}[\cite*{tracy1994level}]
\label{T:tracy-widom}
Let $u = (x, t; y, t + s^3) \in \Rd$. Then the random variable
$
Y = \scrL(u)/s + (x - y)^2/{s^4}
$
is a standard GUE Tracy-Widom random variable. In particular, it satisfies the tail bound
$$
\p\lf(|Y| \ge m\rg) \le ce^{-dm^{3/2}}
$$
for constants $c, d> 0$.
\end{theorem}

\section{Topologies for planar directed metrics}
\label{S:topologies}

The goal of this section is to define topologies which are suitable for proving convergence of lattice and line last passage metrics to the directed landscape.

In the previous section we saw that last passage metrics can be embedded in the plane in a natural way. This embedding yields discontinuous metrics. On the other hand, the directed landscape itself is continuous on $\Rh = \{(u, v) \in \R^4 : u \ne v\}$. One approach to proving convergence would be to convert all directed metrics into continuous ones, and then appeal to a standard topology on a space of continuous functions.

We do not pursue this approach, since the necessary modifications of discontinuous last passage metrics do not necessarily preserve the triangle inequality or the geodesic structure. Instead we will study directed metrics up to different notions of equivalence.
 
For a given function $f:S\to R$ between two metric spaces define the \textbf{ closed graph } of $f$ as
\begin{equation}\label{E:graph}
\mathfrak{g} f=\text{closure of }\{(x,f(x)):x\in S\}\subset S\times R.
\end{equation} If $R$ has total order, then the \textbf{closed hypograph} $\mathfrak{h} f$ of $f$ is the closure of the set $\{(x,y):{y\le f(x)}\}$ in $S\times R$.

We will study functions from $\Rh$ to $\bar \R := \R \cup \{\pm \infty\}$ up to equivalence of closed graphs or hypographs. The reason for working on $\Rh$, rather than all of $\R^4$, is twofold. First, planar directed metrics are equal to $0$ on the diagonal $\Delta = \R^4 \smin \Rh$, so removing $\Delta$ does not change what information we have about the metrics. Second, the directed landscape $\scrL$ is discontinuous on $\Delta$ and so notions of convergence become unnecessarily complicated if we start considering what happens there. Define sets
$$
\scrE_g = \{\mathfrak{g} f : f:\Rh\to \bar \R\}, \quad \mathand \quad \scrE_h = \{\mathfrak{h} f : f:\Rh \to \bar \R\}. 
$$
To put topologies on $\scrE_h$ and $\scrE_g$, we embed them in a larger space $\scrE_*$. 
 
 Let $\scrE_*$ be the set of all closed subsets $\Ga \sset \Rh \times \bar \R$ such that $\Ga \cap (\{x\} \X \bar \R)$ is nonempty for all $x \in \Rh$. For $(r, s) \in \Rh \X \bar \R$ with $r = (u, v)$, define
\begin{equation}
\label{E:map-E}
E(r, s) = \lf(r, \frac{|u- v|e^{-|r|}s}{1 + |s|} \rg).
\end{equation}
Here we use the convention that $E(r, \pm \infty) = (r, \pm |u-v|e^{-|r|})$. The image of $E$ is contained in the set $\Rh \X [-1, 1]$.
For two sets $e_1, e_2 \in \scrE_*$, let $d(e_1, e_2)$ be the Hausdorff distance between $E(e_1)$ and $E(e_2)$. That is, 
\begin{equation}
\label{E:hausdorff}
\begin{split}
d(e_1, e_2) &= \inf \{r > 0 : E(e_1) \sset E(e_2)_{+r}, \; E(e_2) \sset E(e_1)_{+r}\} \qquad \text{where}\\
 \quad A_{+r} &= \{x \in \Rh \X [-1, 1]:   \inf_{a \in A} |a - x|  < r\}.
\end{split} 
\end{equation}
The function $E$ maps closed sets to closed sets since it has a continuous inverse. 

\begin{lemma}
\label{L:D-polish}
The metric space $(\scrE_*, d)$ {is compact. In particular, it is a Polish space.}
\end{lemma}

\begin{proof}
Let $\Ga_n \in \scrE_*$ be a sequence. Since the Hausdorff topology on compact subsets of any compact space is compact, for any compact set $K \sset \Rh$, the sequence $E \Ga_n \cap (K \X [-1, 1])$ is precompact. Therefore letting 
$$
O_i = \{r = (u, v) \in \Rh : |u-v| > 1/i, |r| < i\} \X [-1, 1],
$$
by a diagonalization argument we can find a subsequence $N$ such that $\{E\Ga_n \cap \close{O_i} : n \in N\}$ converges to a limit $\Lambda^i$ for all $i$. Define
$$
\Lambda = \bigcup_{i=1}^\infty \Lambda^i \cap O_i.
$$
Note that $\Lambda = \Lambda^i$ on  $O_i$.
To complete the proof, we show that $E^{-1}\Lambda$ is the limit of
$\{\Ga_n : n \in N\}$. Recalling the fattening notation $\Lambda_{+r}$ from \eqref{E:hausdorff}, for any fixed $i$ we have
\begin{equation}
\label{E:Garrr}
E\Ga_n \sset \Lambda_{+r} \qquad \text{ if and only if } \qquad E\Ga_n \cap O_i \sset \Lambda_{+r} \quad \mathand \quad E\Ga_n \cap O_i^c \sset \Lambda_{+r}.
\end{equation}
The second containment on the right hand side of \eqref{E:Garrr} holds for all $n$ whenever $r \ge 2/i$ since $\Lambda_{+r}$ intersects every slice $\{x\} \X \bar \R$ and $$
{(E\Ga_n\cup\Lambda)\cap O_i^c}\sset\Rh\X(-1/i,1/i).
$$
We now analyze the first containment on the right hand side in \eqref{E:Garrr}. Hausdorff convergence of $E\Ga_n \cap \close{O_{i}}$ to $\Lambda^{i}$ implies that for any $r > 0$, for all large enough $n \in \N$ we have 
$$
E\Ga_n \cap O_i \sset \Lambda^{i}_{+r} \sset \Lambda_{+r}.
$$
Therefore for all large enough $n\in N$, we have
\begin{equation}
\label{E:Garrr2}
E\Ga_n \sset \Lambda_{+r} \quad \text{for all}  \quad r \ge 2/i.
\end{equation}
By a similar argument we can also show that for all large enough $n\in N$, we have $\Lambda \sset (E\Ga_n)_{+r}$ for all $r \ge 2/i$. Since $i$ was arbitrary, combining these two containment statements, and using the definition \eqref{E:hausdorff}, implies that $\Ga_n \to E^{-1}\Lambda$ along $N$.
\end{proof}

 Convergence statements about graphs and hypographs can be converted into more natural convergence statements about the underlying functions when the limit is continuous.
 
 For the first lemma, we will use a notion of uniform convergence for functions whose range is all of $\bar \R$, rather than just $\R$. For such functions, we say that $f_n \to f$ uniformly if $\frac{f_n}{1 + |f_n|} \to \frac{f}{1 + |f|}$ uniformly as functions with range in $[-1, 1]$. {We say that $f_n\to f$ compactly if this convergence is uniform on every compact set.} When the limit $f$ is continuous, {compact convergence} is equivalent to the statement that $f_n(x_n) \to f(x)$ for any $x_n \to x$.
 
 \begin{lemma}
 	\label{L:graph-equiv}
 	Consider a sequence of functions $f_n:\Rh \to \bar \R$ and a continuous function $f:\Rh \to \bar \R$. Then $\mathfrak{g}f_n \to \mathfrak{g}f$ in ${\scrE_*}$ if and only if
 	$$
 	f_n \to f
 	$$
 	compactly as functions from $\Rh \to \bar \R$.
 \end{lemma} 

We leave the straightforward proof of Lemma \ref{L:graph-equiv} as an exercise for the reader.
 
\begin{lemma}
	\label{L:hypo-equiv}
Consider a sequence $f_n:\Rh \to \bar \R$ and a continuous function $f:\Rh \to \bar \R$. 
\begin{enumerate}
	\item Suppose $\mathfrak{h}f_n \to \mathfrak{h} f$ in $\scrE_*$. Then for any sequence $x_n$ converging to a point $x \in \Rh$, we have
	\begin{equation}
	\label{E:limsupp}
	\limsup_{n \to \infty} f_n(x_n) \le f(x).
	\end{equation}
	In particular, if $f$ has range contained in $\R$ on some compact set $K \sset \Rh$, then
	$$
	(f_n - f)^+ \to 0 \text{ uniformly on } K.
	$$
	\item Suppose that there exists a dense set $D$ such that $f_n(y) \to f(y)$ for all $y \in D$, and that for every convergent sequence $x_n \to x \in \Rh$, \eqref{E:limsupp} holds.
	Then $\mathfrak{h}f_n \to \mathfrak{h} f$ in $\scrE_*$.
\end{enumerate}
\end{lemma}

\begin{proof} We first prove 1.
Let $x_n \to x$ in $\Rh$. The points $(x_n, f_n(x_n)) \in \mathfrak{h} f_n$ for all $n$, so the point $y = (x, \limsup f_n(x_n)) \in \mathfrak{h}f$. Moreover, the continuity of $f$ implies that
$$
\{x\} \X \bar \R \cap \mathfrak{h}f = \{(x, z) \in \{x\} \X \bar \R : z \le f(x) \},
$$
so $\limsup_n f_n(x_n) \le f(x)$. 

For 2, since $\scrE_*$ is compact, we can pick out a subsequential limit $\Ga$ in $\scrE_*$. For $x \in \Rh$, define
$$
g(x) = \sup \{ z \in \bar \R : (x, z) \in \Ga\}.
$$
The function $g$ is upper semicontinuous since $\Ga$ is closed, and $\mathfrak{h} g = \Ga$. We want to show that $g = f$. First, we can find a sequence $x_n$ converging to $x$ such that 
$$
\limsup_{n \to \infty} f_n(x_n) = g(x).
$$
Hence $f \ge g$ by our assumption that \eqref{E:limsupp} holds for all convergent sequences. On the other hand, we have that $g(x) \ge \limsup f_n(x)$ for all $x \in D$, so $g|_D \ge f|_D$. {Upper semicontinuity of $g$ and continuity of $f$ then imply} that $g \ge f$.
\end{proof}

\begin{remark}
	\label{R:other-domains}	
	In later sections, we will also consider graph convergence of functions defined on slices
	$$
	S_{s, t} = \{(x, s; y, t) \in \R^4 : x, y \in \R\}
	$$
	for fixed $s < t$, or on $\R^2$ or $\R$.
	We similarly embed the closed graphs/hypographs of $\bar \R$-valued functions $f$ with domain $S_{s, t}, \R^2,$ or $\R$ into spaces $\scrE_*(S_{s, t}), \scrE(\R^2), \scrE_*(\R)$, where $\scrE_*(C)$ consists of all closed subsets of $C \X \bar \R$ that intersect $\{x\} \X \bar \R$ for all $x \in C$. We say that $\Ga_n \to \Ga$ in one of these three spaces $\scrE_*(C)$ if $E' \Ga_n \to E' \Ga$ in the Hausdorff topology, where $E':C \X \bar \R \to \R \X [-1, 1]$ is defined similarly to $E$:
	$
	E'(x, s) = (x, e^{-|x|} s/(1 + |s|) 
	$.
	 It is easy to check that Lemmas \ref{L:D-polish}, \ref{L:graph-equiv}, and \ref{L:hypo-equiv} still hold in this context.
\end{remark}

\section{Convergence of geodesics} 
\label{S:geodesics}

Convergence of hypographs, and hence also convergence of graphs, is a strong enough notion to yield convergence of geodesics for planar directed metrics with a small amount of natural additional structure. 

We start with a proposition showing convergence of geodesic sets. While this proposition is stated for planar metrics, the proof does not require anything about the underlying set being $\R^2$. We work with directed metrics of negative sign with an eye to proving convergence of geodesics in last passage percolation; an analogous statement holds for directed metrics of positive sign.

\begin{prop}
\label{P:convergence}
Let $d_n$ be a sequence of directed metrics of negative sign on $\R^2$ whose hypographs converge in $\scrE_*$ to {the hypograph of} a limiting directed metric $d$ which is continuous on $\Rh$. Let $(x_n, y_n) \in \Rh$ be a sequence converging to $(x, y) \in \Rh$ such that 
$$
d_n(x_n, y_n) \to d(x, y) > -\infty.
$$
Let $A_n$ be a sequence of bounded $d_n$-geodesic sets from $x_n$ to $y_n$. Then all subsequential limits of $\bar A_n$ in the Hausdorff topology on closed subsets of $\R^2$ are $d$-geodesic sets from $x$ to $y$.
\end{prop}

In Proposition \ref{P:convergence}, we work with closures of geodesic sets rather than geodesic sets themselves so that convergence in the Hausdorff topology is well-defined. In the sequel, we will frequently make this exchange between sets and their closures.

\begin{proof}
Let $A$ be any subsequential limit of the sequence of closures $\bar A_n$. We will show that for every pair of points $a, b \in A$, that either
\begin{equation}
\label{E:d-sigsig}
d(x, y) = d(x, a) + d(a, b) + d(b, y),
\end{equation}
or else \eqref{E:d-sigsig} holds with the roles of $a$ and $b$ reversed. 

For this, fix $a, b \in A \smin \{x, y\}$. First assume $a \ne b$. There exists $a_n, b_n \in A_n$ such that $a_n \to a, b_n \to b$. For each $n$, either $a_n \le_n b_n$ or $b_n \le_n a_n$, where $\le_n$ is a compatible order on $A_n$. By possibly passing to a subsequence and relabelling the points, we may assume $a_n {\le_n} b_n$ for all $n$. Then Lemma \ref{L:hypo-equiv} implies 
\begin{equation}
\label{E:d-lim}
\limsup_{n \to \infty} d_n(x_n, a_n) + d_n(a_n, b_n) + d_n(b_n, y_n) \le d(x, a) + d(a, b) + d(b, y) \le d(x, y).
\end{equation}
The second inequality in \eqref{E:d-lim} follows from the triangle inequality for $d$. 
The same inequality chain holds when $a = b$ when we take $a_n = b_n$ since then $d(a_n, b_n) = d(a, b) = 0$. On the other hand, since $A_n$ was a geodesic set, the left hand side of \eqref{E:d-lim} equals 
$$
\limsup_{n \to \infty} d_n(x_n, y_n).
$$
This equals $d(x, y)$ by the assumption of the lemma, so the inequalities in \eqref{E:d-lim} must all be equalities. This proves \eqref{E:d-sigsig} when $a, b \notin \{x, y\}$. If $\{a, b\} = \{x, y\}$, then \eqref{E:d-sigsig} is trivially true. If only one of either $a$ or $b$ is in the set $\{x, y\}$, say, $a$, then \eqref{E:d-sigsig} for the pair $a, b$ follows from \eqref{E:d-sigsig} applied to the pair $b, b$.  

Now let $\preceq$ be the relation on $A$ where we say that $a \preceq b$ if \eqref{E:d-sigsig} holds. We want to modify $\preceq$ to create a compatible total order on $A$. For this, let $\le$ be an arbitrary total order on $A$ with least element $x$ and greatest element $y$. Define a new relation $\preceq'$ given by
\begin{equation}
\label{E:xprecy}
a \preceq' b \text{ if } \begin{cases}
&a \preceq  b \mathand {b \not\preceq a}, \text{ or} \\
&a \preceq b \mathand b \preceq a, \text{ and }a \le b.
\end{cases}
\end{equation}
We claim that the new relation $\preceq'$ is a compatible total order on $A$. Since for any $a, b$, \eqref{E:d-sigsig} either holds for $a, b$ or for $b, a$, we have $a \preceq' b$ or $b \preceq' a$ for any $a, b$. Moreover, the construction in \eqref{E:xprecy} ensures that if $a \preceq' b$ and $b \preceq' a$ then $a = b$. Next, we check transitivity. Suppose $a \preceq' b \preceq' c$. By \eqref{E:d-sigsig} for the pair $a, b$ and two triangle inequalities, we have that
\begin{equation}
\label{E:dxy-ex}
d(x, y) = d(x, a) + d(a, b) + d(b, y) \le d(x, b) + d(b, y) \le d(x, y).
\end{equation}
All the inequalities above must be equalities, implying that $d(x, a) + d(a, b) = d(x, b)$. Combining this with \eqref{E:d-sigsig} for the pair $b, c$ gives that
\begin{equation*}
\label{E:dxy-2ex}
d(x, y) = d(x, a) + d(a, b) + d(b, c) + d(c, y).
\end{equation*}
Again, by appealing to the triangle inequality as in \eqref{E:dxy-ex}, this implies that \eqref{E:d-sigsig} holds for the pair $a, c$. {If also $c\preceq a$, the same equalities give $d(a,b)+d(b,a)=d(b,c)+d(c,b)=0$. Thus $b\preceq a$ and $c\preceq b$, so $a\le b\le c$ and hence $a\preceq'c$. Otherwise $a\preceq'c$ follows directly from \eqref{E:xprecy}.} Thus $\preceq'$ is transitive, and hence a total order. This reasoning also implies that $d(a, c) = d(a, b) + d(b, c)$, so the total order satisfies the second condition of Definition \ref{D:geodesic}. Finally, for any $a$, \eqref{E:d-sigsig} holds for the pairs $x, a$ and $a, y$, so the construction \eqref{E:xprecy} implies that $x \preceq' a \preceq' y$. Hence $\preceq'$ is a compatible order on $A$ and so $A$ is a geodesic set.
\end{proof}

If we impose some structure on the limiting metric $d$ then we can improve Proposition \ref{P:convergence} to give convergence of geodesics.
The following structure is natural in light of the fact that we are interested in proving convergence to the directed landscape.

\begin{definition}
\label{D:spacetime}
Let $d$ be a directed metric of negative sign on $\R^2$. We say that $d$ is a \textbf{spacetime metric} if $d$ is continuous on $\Rh$ and satisfies
\begin{enumerate}
	\item (Domain) $d > -\infty$ on $\Rd$, $d = 0$ on $\Delta$, and $d = -\infty$ everywhere else.
	\item (Metric composition) For all $s < r < t$ and $x, y \in \R$, {the maximum below is attained and} we have 
	$$
	d(x, s; y, t) = \max_{z \in \R} {[}d(x, s; z, r) + d(z, r; y, t){]}.
	$$
	\item For any $(x, s; y, t) \in \Rd$, there exists a compact set $K = [a, b] \X [s, t] \sset \R^2$ such that all geodesics from $(x, s)$ to $(y, t)$ are contained in $K$.
\end{enumerate}
\end{definition}

By Proposition \ref{P:from-DOV} and Definition \ref{D:DL-def}, the directed landscape $\scrL$ is almost surely a spacetime metric. We start by proving two basic properties of spacetime metrics. 

\begin{prop}
\label{P:spacetime-properties}
Let $d$ be a spacetime metric on $\R^2$.
\begin{enumerate}[label=(\roman*), nosep]
	\item For every $(x, s; y, t) \in \Rd$, all geodesics from $(x, s)$ to $(y, t)$ are of the form 
	$$
	\Ga(\pi) = \{(\pi(r), r) : r\in [s, t]\}
	$$
	for some continuous function $\pi:[s, t] \to \R$ with $\pi(s) = x, \pi(t) = y$. Moreover, the only compatible order on each geodesic is the order induced by the time coordinate $r \in [s, t]$.
	\item For any $(u; v) \in \Rd$, a geodesic set from $u$ to $v$ is a geodesic if and only if it is a connected subset of $\R^2$. 
\end{enumerate}
\end{prop}

\begin{proof}
We first prove (i). Let $A$ be a geodesic from $u = (x, s)$ to $v = (y, t)$. Since $d = -\infty$ off of $\Rd \cup \Delta$, $A$ must be the graph of a function $\pi$ from some subset $D \sset [s, t]$ to $\R$ and the only compatible order on $A$ is the order induced by the time coordinate. Next, $D$ is closed by continuity of $d$ on $\Rh$ and maximality of $A$. Finally, suppose $(a, r_1), (b, r_2) \in A$ and $(r_1, r_2) \cap D = \emptyset$. Then by metric composition, there is some $r_3 \in (r_1, r_2)$ and some $c \in \R$ such that 
$$
d(a, r_1; b, r_2) = d(a, r_1; c, r_3) + d(c, r_3; b, r_2).
$$
We can add the point $(c, r_3)$ to the geodesic $A$, contradicting maximality. Therefore $D = [s, t]$. 

Next, we show that $\pi$ is right continuous. For this, we let $\bar \pi (r) = (\pi(r), r)$. Let $r_n \in [s, t)$ be any nonincreasing sequence converging to $r \in [s, t)$. By Definition \ref{D:spacetime}.3, $\bar \pi(r_n)$ is necessarily a bounded sequence in $\R^2$ so it has a subsequential limit $y$. Suppose that $y \ne \bar \pi(r)$. By continuity of $d$ on $\Rh$, we have 
$$
d(u, v) = d(u, \bar \pi(r)) + d(\bar \pi(r), \bar \pi(r_n)) + d(\bar \pi(r_n), v) \to d(u, \bar \pi(r)) + d(\bar \pi(r), y) + d(y, v). 
$$
Since $d$ is finite on $\Rd$ and equal to $-\infty$ on $\Rh \smin \Rd$, we have $d(\bar \pi(r), y) = -\infty$, whereas all other terms above are finite. This is a contradiction, so $y = \bar \pi(r)$ and therefore $\pi(r_n) \to \pi(r)$. We can similarly show that $\pi$ is left continuous.

Part (ii) follows from part (i), and the fact that any geodesic set is contained in a geodesic (Proposition \ref{P:geod-exist}).
\end{proof}

\begin{remark}
	\label{R:landscape-geodesics}
	In \cite*{DOV}, geodesics in $\scrL$ are defined differently. Namely, instead of a geodesic being a subset $\Ga(\pi) = \{(\pi(r), r) : r \in [s, t]\}$ of $\R^2$, a geodesic is simply the function $\pi$ itself. By Proposition \ref{P:spacetime-properties}(i) and the fact that $\scrL$ is almost surely a spacetime landscape, the two notions are equivalent, so all geodesic results from \cite*{DOV} apply here. For example, for any $(u; v) \in \Rd$, there is almost surely a unique $\scrL$-geodesic $\Ga$ from $u$ to $v$ by \cite*{DOV}, Theorem 1.7.
\end{remark}

We can use Proposition \ref{P:spacetime-properties} to give our first geodesic convergence theorem.

\begin{theorem}
\label{T:geod-cvg}
Let $d_n$ be a sequence of directed metrics of negative sign whose hypographs converge in $\scrE_*$ to the hypograph of a spacetime metric $d$. Let $(u_n; v_n) \to (u; v) \in \Rd$. Suppose that $\Ga_n$ is a sequence of $d_n$-geodesics from $u_n$ to $v_n$ which are connected subsets of $\R^2$, and that
$$
d_n(u_n, v_n) \to {d}(u, v).
$$
Then the sequence of closures $\bar \Ga_n$ is precompact in the Hausdorff topology and all subsequential limits of $\bar \Ga_n$ are $d$-geodesics from $u$ to $v$.
\end{theorem}

\begin{proof} Let $K$ be a compact set in $\R^2$ containing all $u_n$ and $v_n$, and all $d$-geodesic sets from $u$ to $v$ in its interior. Such a $K$ exists by Definition \ref{D:spacetime}.3. The sequence $\bar \Ga_n \cap K$ is a precompact sequence of sets in the Hausdorff topology. All subsequential limits of $\bar \Ga_n \cap K$ are geodesic sets from $u$ to $v$ by Proposition \ref{P:convergence}. In particular, any such limit is contained in $\Int(K)$, so $\bar \Ga_n \cap \del K = \emptyset$ for all large enough $n$. By the connectedness of $\bar \Ga_n$, for such $n$ this implies that either 
$$
\bar \Ga_n \cap \Int (K) \qquad \text{ or } \qquad \bar \Ga_n \cap K^c
$$
is empty. Since $(u_n, v_n) \in \bar \Ga_n \cap \Int (K)$ for large enough $n$, this implies that
$
\bar \Ga_n \cap K^c = \emptyset
$
for all large enough $n$. Therefore $\bar \Ga_n$ is a sequence of bounded $d_n$-geodesic sets, so all subsequential limits are $d$-geodesic sets by Proposition \ref{P:convergence}. Connectedness of the sets $\bar \Ga_n$ implies that all subsequential limits  are connected, and hence are geodesics by Proposition \ref{P:spacetime-properties} (ii). 
\end{proof}

We have the following straightforward corollary of Theorem \ref{T:geod-cvg} in the $d_n = d$ case.
\begin{corollary}
	\label{C:geod-cvg-2} 
	Let $d$ be a spacetime metric on $\R^2$.
	If $(u_n; v_n) \to (u; v)\in \Rd$ and $\Ga_n$ is a sequence of geodesics from $u_n$ to $v_n$, then $\Ga_n$ is precompact in the Hausdorff topology, and all subsequential limits of $\Ga_n$ are geodesics from $u$ to $v$. 
\end{corollary}

\begin{proof}
Continuity of $d$ on $\Rh$ implies that $d(u_n, v_n) \to d(u, v)$, and Proposition \ref{P:spacetime-properties} implies that each $\Ga_n$ is closed and connected. Therefore we can apply Theorem \ref{T:geod-cvg} to complete the proof.
\end{proof}

In some cases we will not be able to check the condition that $d_n(u_n, v_n) \to {d}(u, v)$. We can loosen this condition by introducing a bit more structure on the prelimiting directed metrics $d_n$. {For rounded embeddings, we will apply the same argument to the underlying lattice or line paths; see Step 2 of the proof of Theorem \ref{T:lattice-models}.}

\begin{theorem}
	\label{T:geod-cvg-stronger}Let $d_n$ {be a sequence of directed metrics} of negative sign in $\R^2$ satisfying the following conditions.
\begin{enumerate}[nosep]
	\item Let $\le$ be the total order on $\R^2$ given by $(x, s) \le (y, t)$ if $s < t$ or if $s =t$ and $x {\le} y$. Then $d_n(u; v) =-\infty$ for $v < u$. This means that the only compatible order on any $d_n$-geodesic is the order $\le$.
	\item Any $d_n$-geodesic is the image of a continuous curve $\pi:[0,1] \to \R^2$ satisfying $\pi(s) \le \pi(t)$ for $s \le t$.
	\item The sequence $\mathfrak{h} d_n$ converges to the closed hypograph of a spacetime metric $d$ in $\scrE_*$.
\end{enumerate}
	Suppose that $(u_n, v_n) \to (u, v) \in \Rd$ and that $d_n(u_n, v_n) > -\infty$ for all large enough $n$. Then any sequence of geodesic closures $\bar \Ga_n$ is precompact in the Hausdorff topology. Moreover, if there is a unique $d$-geodesic $\Ga$ from $u$ to $v$ then $\bar \Ga_n \to \Ga$.
\end{theorem}

To prove Theorem \ref{T:geod-cvg-stronger}, we first show the following lemma about geodesic ordering.
\begin{lemma}
	\label{L:ordering}
Let all notation and assumptions be as in Theorem \ref{T:geod-cvg-stronger}.
Let $(u', v') = (x', s'; y', t')$, $(u, v) = (x, s; y, t)$ and suppose that $s' \le s < t \le t'$. Suppose also that there is a {$d_n$-geodesic} $\Lambda$ from $u'$ to $v'$ and points  $x^- < x, y^- < y$ such that $({x^-}, s), ({y^-}, t) \in \Lambda$. 

Then for any $d_n$-geodesic $\Ga$ from $u$ to $v$ and any $(w, r) \in \Ga$, there exists a {$d_n$-geodesic} $\Ga^-$ from $u'$ to $v'$ with $(w^-, r) \in \Ga^-$ for some $w^- \le w$.  
\end{lemma} 

\begin{proof} {Fix a $d_n$-geodesic $\Ga$ from $u$ to $v$ and $(w,r)\in\Ga$.} Since $\Lambda$ is connected by assumption $2$ of Theorem \ref{T:geod-cvg-stronger}, there must exist a point $(\la, r) \in \Lambda$. If $\la \le w$, then we're done. Suppose $\la > w$. Since $x^- < x$ and $y^- < y$, Condition $2$ in Theorem \ref{T:geod-cvg-stronger} implies that there must exist points $p, q \in \Ga \cap \Lambda$ with $p \le (w, r) \le q$ in the order in Theorem \ref{T:geod-cvg-stronger}.1.

Letting $[a, b]$ denote an interval in the order $\le$, let $\Lambda'$ denote the union of the three sets $\Lambda \cap [u', p], \Ga \cap [p, q]$ and $\Lambda \cap [q, v']$. We claim that $\Lambda'$ is a geodesic set from $u'$ to $v'$. Let $a \le b \le c \in \Lambda'$.
Let $e_0 \le \dots \le e_6$ be the order statistics of the set $\{u', a, b, c, p, q, v'\}$. Then we have
\begin{equation}
\label{E:u-v-}
d_n(u', v') = d_n(u', p) + d_n(p,q) + d_n(q, v') = \sum_{i=1}^6 d_n(e_{i-1}, e_i).
\end{equation}
The first equality follows by the geodesic property of $\Lambda$, and the second equality follows by applying the geodesic property of $\Lambda$ and $\Ga$ to each of the intervals $[u', p], [p,q],$ and $[q, v']$. Equation \eqref{E:u-v-} and the triangle inequality for $d_n$ implies that $d_n(a, c) = d_n(a, b) + d_n(b, c)$. Hence $\Lambda'$ is a geodesic set. Moreover, $(w, r) \in \Lambda'$ since $(w, r) \in \Ga \cap [p, q]$. Since every geodesic set is contained in a geodesic (Proposition \ref{P:geod-exist}), there exists some geodesic $\Ga^-$ from $u'$ to $v'$ with $(w, r) \in \Ga^-$, as desired. 
\end{proof}

\begin{proof}[Proof of Theorem \ref{T:geod-cvg-stronger}]
Let $(u_n, v_n) = (x_n, s_n; y_n, t_n), (u, v) = (x, s; y, t)$ be as in the statement of the theorem. We will construct a sequence $(u_n^-, v_n^-) = (x^-_n, s^-_n; y^-_n, t^-_n)$ such that
\begin{enumerate}[nosep]
	\item $(u_n^-, v_n^-) \to (u,v)$ and $d_n(u_n^-, v_n^-) \to d(u,v)$.
	\item For all large enough $n$, $[s_n, t_n] \sset [s^-_n, t^-_n]$.
	\item For all large enough $n$, any $d_n$-geodesic $\Ga$ from $u_n^-$ to $v_n^-$ satisfies
	$$
	\Ga \cap {(\R \X \{s_n\})} \sset (-\infty, x_n) \X \{s_n\}, \quad \mathand \quad \Ga \cap {(\R \X \{t_n\})} \sset (-\infty, y_n) \X {\{t_n\}}. 
	$$
\end{enumerate}
Let $i \in \N$, and let $\Ga^m_i$ be any sequence of $d$-geodesics from $(x-1/i - 1/m, s - 1/m)$ to  $(y-1/i - 1/m, t + 1/m)$.
 By Corollary \ref{C:geod-cvg-2} and Proposition \ref{P:spacetime-properties}, the sequence $\Ga^m_i, m \in \N$ is precompact and subsequential limits are graphs of continuous paths from $(x - 1/i, s)$ to $(y - 1/i, t)$. In particular, there exists $m_i \ge i$ such that any $d$-geodesic $\Lambda$ from  $u^i := (x-1/i - 1/m_i, s - 1/m_i)$ to  $v^i := (y-1/i - 1/m_i, t + 1/m_i)$ satisfies 
\begin{equation}
\label{E:lambda}
\begin{split}
&\Lambda \cap {(\R \X [s-1/(2i), s])} \sset (-\infty, x - 1/(2i)] \X [s-1/(2i), s] \qquad \mathand \\
&\Lambda \cap {(\R \X [t, t + 1/(2i)])} \sset (-\infty, y - 1/(2i)] \X [t, t + 1/(2i)].
\end{split}
\end{equation}
Next, by hypograph convergence of $d_n \to d$, we can find a sequence of points $(u_n^i, v_n^i) \to (u^i, v^i)$ such that $d_n(u_n^i, v_n^i) \to d(u^i, v^i)$ as $n \to \infty$.
By Theorem \ref{T:geod-cvg}, there exists $n(i) \in \N$ such that for all $n \ge n(i)$, any $d_n$-geodesic $\Lambda$ from $u_n^i$ to $v_n^i$ satisfies \eqref{E:lambda} with $x -1/(2i), y - 1/(2i)$ replaced by $x - 1/(3i), y - 1/(3i)$. We can also choose $n(i)$ large enough so that for $n \ge n(i)$, we have
\begin{align*}
[s_n, t_n] \sset [s^i_n, t^i_n] \qquad \mathand \qquad |d_n(u_n^i, v_n^i) - d(u^i, v^i)| \le 1/i.
\end{align*}
We can then set $(u^-_n, v^-_n) = (u^{j(n)}_n, v^{j(n)}_n)$ for any increasing sequence $j(n)$ with $n(j(n)) \le n$ for all large enough $n$. To guarantee the condition (i) we have used that $d(u^i, v^i) \to d(u,v)$. This uses the fact that $m_i \ge i$ and the continuity of $d$.

Now let $\Ga_n$ be any sequence of $d_n$-geodesics from $u_n$ to $v_n$. We first check that $\bar \Ga_n$ is precompact in the Hausdorff topology. Consider any sequence $(w_n, r_n) \in \Ga_n$ with $r_n \in [s_n, t_n]$ converging to a point $r \in [s, t]$.

By conditions 2 and 3 satisfied by $u_n^-, v_n^-$ and Lemma \ref{L:ordering}, for large enough $n$ we can find geodesics $\Ga_n^-$ from $u_n^-$ to $v_n^-$ with $(w_n^-, r_n) \in \Ga_n^-$ for some $w_n^- \le w_n$. By condition 1 and Theorem \ref{T:geod-cvg}, {the sequence $\Ga_n^-$ is precompact and every subsequential limit is a geodesic from $u$ to $v$}, so all subsequential limits of $(w_n^-, r_n)$ must lie on a geodesic $\Lambda$ from $u$ to $v$. Therefore since $w_n^- \le w_n$, all subsequential limits of $w_n$ are bounded below by some $w \in \R$ with $(w, r)$ in a $d$-geodesic from $u$ to $v$. 

By a symmetric argument, all subsequential limits of $w_n$ are bounded above by some $w' \in \R$ with $(w', r)$ in a $d$-geodesic from $u$ to $v$. Since all $d$-geodesics from $u$ to $v$ are contained in a common compact set by Definition \ref{D:spacetime}.3, the sequence $\bar \Ga_n$ is precompact. Now suppose there is a unique $d$-geodesic $\Ga$ from $u$ to $v$. The set $\Ga$ is equal to $\{(\pi(r), r) :r \in [s, t]\}$ for some continuous function $\pi$ by Proposition \ref{P:spacetime-properties}, so the bounds above guarantee that all subsequential limits of $\bar \Ga_n$ are subsets of $\Ga$. The connectivity of each $\bar \Ga_n$ then guarantees that $\bar \Ga_n \to \Ga$.
\end{proof}

{The uniqueness condition in Theorem \ref{T:geod-cvg-stronger}} is almost surely satisfied for the directed landscape for any fixed $(u, v) \in \Rd$ by Remark \ref{R:landscape-geodesics}.

We finish this section with a proposition about convergence of directed metrics along geodesics. This will be used in Section \ref{S:consequences} to prove convergence of the longest increasing subsequence in a uniform permutation. 

\begin{prop}
\label{P:convergence-alond-geod}
Let $d_n$ be a sequence of directed metrics of negative sign whose graphs converge to a spacetime metric $d$ in $\scrE_*$. Let $(u_n;v_n) \to (u; v) \in \Rd$, and consider a sequence of $d_n$-geodesics $\Ga_n$ from $u_n$ to $v_n$ converging to a $d$-geodesic $\Ga$ from $u$ to $v$.
Then for any $x_n \in \Ga_n$ converging to a point $x \in \Ga$ we have $d_n(u_n, x_n) \to d(u, x).$ Moreover,
$$
\limsup_{n \to \infty} \sup_{x \in \Ga_n} |d_n(u_n, x)| \le \sup_{x \in \Ga} |d({u}, x)|< \infty.
$$
\end{prop}

\begin{proof}
If $(u, x) \in \Rh$, then $d_n(u_n, x_n) \to d(u, x)$ by Lemma \ref{L:graph-equiv} and the continuity of $d$ on $\Rh$. The only other possibility is $x = u$.
We can use the geodesic property of $\Ga_n$ to write
$$
d_n(u_n, x_n) = d_n(u_n, v_n) - d_n(x_n, v_n).
$$
By graph convergence of $d_n \to d$, both terms on the right hand side converge to $d(u, v)$, again by Lemma \ref{L:graph-equiv}. Since $d$ is finite on $\Rd$, this implies that $d_n(u_n, x_n) \to 0 = d(u, u)$. 

We turn to the second part of the proposition. The first inequality is immediate from the first part of the proposition. If the second inequality fails, then we could find a convergent sequence $x_n \in \Ga$ such that $d(u, x_n)$ would not converge, contradicting the first part of the proposition in the special case when $d_n = d$ for all $n$. 
\end{proof}

\section{Multi-time convergence to $\scrL$}
\label{S:landscape-cvgfdd}
In the next three sections, we give criteria for random last passage metrics  to converge to $\scrL$. While our main goal is to develop criteria for hypograph and graph convergence, we start with a weaker notion of multi-time convergence. 

We will discretize time while keeping space continuous. Recall from Remark \ref{R:other-domains} the spaces $S_{s, t}, \scrE_*(S_{s, t})$.
The results of Section \ref{S:sheet} imply that certain rescaled last passage metrics $\scrL^n|_{S_{t, t + s^3}}$ can be coupled to converge compactly a.s.\ to an Airy sheet of scale $s$. Since Airy sheets are continuous, this is equivalent to graph convergence in $\scrE_*(S_{s, t})$ by Lemma \ref{L:graph-equiv}. Because of this, graph convergence in each of the spaces $\scrE_*(S_{s, t})$ will be our main assumption in this section.

Next, for sequences of random functions $\scrK_n, \scrK:\Rd \to \bar \R$, we say that $\scrK_n$ converges to $\scrK$ in the \textbf{multi-time sense} if for any set $S= \bigcup_{i=1}^\infty S_{t_i, t_i + s_i^3}$ for $t_i \in \R, s_i > 0$,  we have
\begin{equation}
\mathfrak{g}\scrK_n|_S \cvgd \mathfrak{g}\scrK|_S
\end{equation}
in the product space $\prod_{i=1}^\infty \scrE_*(S_{t_i, t_i + s_i^3})$.

We say that a function $W:\R^4 \to \R$ is \textbf{temporally independent} if the restrictions
$
W|_{(\R \X I_1)^2},\dots, W|_{(\R \X I_k)^2}
$
are independent for disjoint time intervals $I_1, \dots, I_k$. The next theorem shows that Airy sheet convergence for temporally independent metrics implies multi-time convergence.

\begin{theorem}
	\label{T:rational-cvge}
	Let $W^n$ be a sequence of temporally independent directed metrics on $\R^2$ of negative sign.
	Suppose that for any $t\in \R, s > 0$, the random function $W^n|_{S_{t, t + s^3}}$ converges to an Airy sheet of scale $s$ in $\scrE_*(S_{t, t + s^3})$.
	Then $W^n$ converges to $\scrL$ in the multi-time sense.
\end{theorem}

\begin{proof}
Without loss of generality, we may assume that $S = \bigcup S_{s^i, s^j}$ where the union is over all $s^i < s^j$ in some countable dense subset of $\R$. First, $\mathfrak{g} W^n|_S$ is tight in the product space $\prod \scrE_*(S_{s^i, s^j})$ since it has convergent marginals. Let $\scrK$ be any subsequential limit of $W^n$ in this space.
 We first show that for any $s = s^i < t = s^j$, we have
\begin{equation}
\label{E:K-determined}
\scrK|_{S_{s, t}} = \limsup_{n \to \infty} \scrK|_{S_{s_n, t_n}},
\end{equation}
for certain sequences $s_n \cvgdown s, t_n \cvgup t$. First, $\scrK$ inherits the triangle inequality from $W^n$, so
\begin{equation}
\label{E:scrKn-tri}
\scrK(x,s; y, t) \ge \scrK(x,s; x, s_n) + \scrK(x,s_n; y, t_n) + \scrK(y,t_n; y, t).
\end{equation}
Now, $\scrK|_{S_{s^i, s^j}} \eqd \scrL|_{S_{s^i, s^j}}$ for all $i, j$, so all four random variables in \eqref{E:scrKn-tri} are rescaled and shifted Tracy-Widom random variables by Theorem \ref{T:tracy-widom}.
In particular, as long as {$s_n\downarrow s$ and $t_n\uparrow t$ sufficiently quickly that the tail probabilities below are summable}, the tail bounds on $\scrK(x,s; x, s_n), \scrK(y,t_n; y, t)$ in Theorem \ref{T:tracy-widom} imply that
$$
\scrK(x, s; y, t) \ge \limsup_{n \to \infty} \scrK(x,s_n; y, t_n)
$$
almost surely.
However, $\scrK(x, s_n; y, t_n) \eqd \scrL(x, s_n; y, t_n)$, which converges in distribution with $n$ to $\scrK(x, s;y, t) \eqd \scrL(x, s; y, t)$, so this inequality is almost surely an equality. 

Next, $\scrK|_S$ inherits temporal independence from $W^n$. Equation \eqref{E:K-determined} implies that this independence also holds when the intervals $I_1, \dots, I_k$ overlap at their endpoints. We use this to show that $\scrK|_S \eqd \scrL|_S$. By Kolmogorov's extension theorem, it suffices to prove that $\scrK|_{S'} \eqd \scrL|_{S'}$ for all subsets $S' \sset S$ of the form
$$
	S' = \bigcup_{i < j \in \{1, \dots, m\}} S_{t_i, t_j},
$$
	where $t_1 < t_2 < \dots < t_m$. Letting
	$$
	S^* = \bigcup_{i=1}^{m-1} S_{t_i, t_{i+1}},
	$$
	we have $\scrK|_{S^*} \eqd \scrL|_{S^*}$: both consist of $m-1$ independent Airy sheets of scale $(t_{i+1} - t_{i})^{1/3}$. Next, the triangle inequality for $\scrK$ gives that
	\begin{equation}
	\label{E:L-x-y}
	\scrK(x,t_i;y,t_k) \ge \sup_{z \in \mathbb R} [\scrK(x,t_i;z,t_j)+\scrK(z,t_j;y,t_k)]
	\end{equation}
	for all  $x, y \in \R$ and $i < j < k \in \{1, \dots, m\}$. By the assumption of Airy sheet convergence for $\scrK$, and the independence of $\scrK(x,t_i;z,t_j)$ and $\scrK(z,t_j;y,t_k)$, we have that
	\begin{equation}
	\label{E:KeqdL}
	\begin{split}
\scrK(x,t_i;y,t_k) &\eqd \scrL(x,t_i;y,t_k) \mathand  \\
\scrK(x,t_i;\cdot,t_j)+\scrK(\cdot,t_j;y,t_k) &\eqd \scrL(x,t_i;\cdot,t_j)+\scrL(\cdot,t_j;y,t_k).
\end{split}
	\end{equation} 
	By the metric composition law for $\scrL$, \eqref{E:KeqdL} implies that the right hand side of \eqref{E:L-x-y} is a maximum, rather than a supremum, and that the two sides of \eqref{E:L-x-y} are equal in distribution. Therefore by the inequality in \eqref{E:L-x-y}, they must be equal almost surely. Continuity of both sides of \eqref{E:L-x-y} as functions of $x$ and $y$ therefore implies that almost surely, $\scrK$ satisfies the metric composition law everywhere on $S'$. This, along with $\scrK|_{S^*}$, uniquely determines $\scrK|_{S'}$. The directed landscape $\scrL|_{S'}$ is determined from $\scrL|_{S^*}$ in the same way.
\end{proof}

\section{Hypograph convergence to $\scrL$}
\label{S:hypo}

The multi-time convergence in Theorem \ref{T:rational-cvge} does not immediately imply either hypograph or graph convergence. For example, we could modify any sequence of metrics $W^n$ satisfying the conditions of Theorem \ref{T:rational-cvge} in the following way.
 Let $Q$ be a Poisson point process on $\R$, and define the map $P_n:\R^2 \to \R^2$ by
$$
 P_n(u) = \begin{cases}
u \qquad u \notin \R \X Q, \\
0, \qquad u \in \R \X Q.
 \end{cases}
$$
Let $\tilde W^n$ be the pullback of $W^n$ under $P_n$. Then $\tilde W^n$ satisfies the conditions of Theorem \ref{T:rational-cvge} since almost surely $(\R \X Q)^2 \cap S = \emptyset$ for any fixed set $S = \bigcup_{i=1}^\infty S_{s_i, t_i}$, but $\mathfrak{h} \tilde W^n \not \cvgd \mathfrak{h} \scrL$ since, for example, for large enough $c$ the random variable
$$
\sup_{s < t \in [0, 1], x, -y \in [c, 2c]} {\tilde W^n}(x, s; y, t)
$$
will have an atom at $0$ whose size does not tend to $0$ with $n$ whereas the corresponding random variable for $\scrL$ is nonatomic.
To work around this issue, we will need to be able to control any unusually large values of $W^n$ on subsets of $\R^4$. For example, in the setting above, $\tilde W^n(u, v) = 0$ whenever $u, v \in \R \X Q$, whereas $\scrL(u, v)$ can be arbitrarily small anywhere. For a set $A \sset \R^4$, and a function $d:\R^4 \to \bar\R$, we write
$$
d(A) = \sup \{d(x) : x \in A\}.
$$
Typically $d$ will be a directed metric of negative sign, though we will prove convergence results in greater generality.
We now give conditions for hypograph convergence of functions that arise naturally in the study of last passage percolation. For this proposition, recall that $\Rh = \R^4 \smin \Delta$, where $\Delta = \{(u;u) \in \R^4\}$.

\begin{prop}
\label{P:hypograph-cvg}
Suppose that $\scrK_n$ is a sequence of random functions from $\Rh \to \bar \R$. Suppose that the following conditions hold:
\begin{enumerate}[nosep, label=(\roman*)]
	\item For some countable dense subset $D \sset \Rh$ we have $\scrK_n|_D \cvgd \scrK|_D$, where $\scrK:\Rh \to \bar \R$ is a random continuous function {which is finite on $\Rd$}.
	 \item {For every compact $K\subset\{(x,s;y,t):s>t\}$, $\sup_K\scrK_n\to-\infty$ in probability.}
	 \item For $p = (x, t)\in \R^2,$ define
	 $$
	 B^+(p, r) = [x - r^2, x + r^2] \X [t + r^3, t + 2 r^3], \quad B^-(p, r) = [x - r^2, x + r^2] \X [t - 2r^3, t - r^3].
	 $$
	 For every compact $K{\subset}\Rd$ and every $\eps>0$  we have
	 $$
	 \lim_{r \to 0^+} \frac{\sup_{(p,q)\in K}\limsup_{n\to\infty} \p\left( {\scrK_n}(B^+(p,r) \X B^-(q,r))-{\scrK_n}(p,q)>\eps\right)}{r^{10}} = 0.
	 $$
	\item For $A \sset \R^4$, let $A_{+r} = \{u \in \R^4 : \inf_{a \in A} |a - u| < r\}$. For any $m \in \R$ and any compact set $A \sset \{(x, s; y, s) : x, y, s \in \R, x \ne y\}$ we have  
	$$
	\lim_{r \to 0^+} \limsup_{n \to \infty} \p(\scrK_n(A_{+r}) > m) = 0. 
	$$
\end{enumerate}
Then $\mathfrak{h}\scrK_n \cvgd \mathfrak{h} \scrK$ as random elements of $\scrE_*$. Moreover, we have the joint distributional convergence $(\scrK_n|_D, \mathfrak{h}\scrK_n) \cvgd (\scrK|_D, \mathfrak{h} \scrK)$.
\end{prop}

We have built the KPZ scaling into the third condition above in order to make the conditions of Proposition \ref{P:hypograph-cvg} easier to check in the cases we care about. Condition (iii) handles unusually large values of $\scrK_n$ in the interior of $\Rd$, and condition (iv) ensures that $\scrK_n$ is close to $-\infty$ near the boundary of $\Rd$.

\begin{proof}
	{Tightness follows from (i) and compactness of $\scrE_*$.}
	Let $(\mathcal{J}, \Ga)$ be any joint weak subsequential limit of $(\scrK_n|_D, \mathfrak{h} \scrK_n)$. By (i), $\mathcal{J} \eqd \scrK|_D$. By Skorokhod's representation theorem, we can couple the functions $\scrK_n$ so that $(\scrK_n|_D, \mathfrak{h} \scrK_n) \to (\scrK|_D, \Ga)$ almost surely along the given subsequence.
	
	The set $\Ga$ is the closed hypograph of a unique upper semicontinuous function $\scrK':\Rh \to \bar \R$. We will show that $\scrK' = \scrK$. {For $u\in D$, convergence of $(u,\scrK_n(u))$ and of the hypographs gives $(u,\scrK(u))\in\Ga$, so $\scrK|_D\le\scrK'|_D$.} Therefore upper semicontinuity of $\scrK'$ and continuity of $\scrK$ implies that $\scrK \le \scrK'$. To establish that this is in fact an equality, we first show that for any {closed box $A\sset\Rd$ with rational endpoints and nonempty interior,}
	\begin{equation}
	\label{E:k'k}
	\scrK'(A) = \scrK(A \cap D) = \scrK(A)
	\end{equation}
	almost surely. In \eqref{E:k'k}, the second equality follows from the continuity of $\scrK$ {and density of $A\cap D$ in $A$}. Next, there exists an $A$-dependent constant $c > 0$ such that for every small enough $r > 0$, there exists a finite mesh $U_r \sset A_{+4r^2} \cap D$ with $|U_r| \le c {r^{-10}}$ and
	$$
	A \sset \bigcup_{(p, q) \in U_r} {\Int\bigl(B^+(p, r) \X B^-(q, r)\bigr)}.
	$$
If
	\begin{equation}
	\label{E:K'A}
	\scrK'(A) > \scrK(A \cap D) + \ep
	\end{equation}
	for some $\ep > 0$, then there exists $(p,q) \in U_r$ so that either 
	$$
	\scrK'({\Int\bigl(B^+(p, r) \X B^-(q, r)\bigr)}) > \scrK(p, q) + \ep/2 \qquad \text{ or } \qquad \scrK(p,q)>\scrK(A \cap D)+\ep/2.
	$$
	Thus by a union bound the probability of \eqref{E:K'A} is bounded above by
	\begin{align*}
	&\sum_{(p, q) \in U_r} \p\lf(\scrK'({\Int\bigl(B^+(p, r) \X B^-(q, r)\bigr)}) > \scrK(p, q) + \ep/2\rg) \\
	&\qquad + \p\lf(\scrK(A_{+{4}r^2} \cap D) - \scrK(A \cap D) > \ep/2\rg).
	\end{align*}
	{Hypograph convergence on each open box bounds its probability by the liminf of the corresponding prelimit probabilities on the closed box.} As $r\downarrow 0$, the first term above goes to $0$ by assumption (iii) and the bound on $|U_r|$, and the second term goes to $0$ by the continuity of $\scrK$, proving \eqref{E:k'k}. Now, suppose that $\scrK'(u) > \scrK(u)$ for some $u \in \Rd$. Then by the continuity of $\scrK$, for some {such box $A\sset\Rd$ containing $u$ in its interior,} we have $\scrK'(A) \ge \scrK'(u) > \scrK(A)$. {There are only countably many such boxes, so \eqref{E:k'k} holds for all of them on one event of probability one.} This contradicts \eqref{E:k'k}, implying that $\scrK = \scrK'$ on $\Rd$.

	 Moreover, $\scrK'(x, s; y, t) = \scrK(x, s; y, t) = -\infty$ for $s > t$ by assumption (ii). Finally, assumption (iv) guarantees that $\scrK' = -\infty$ on the set $E = \{ u = (x, s; y, s) \in \R^4: x \ne y \}$. Since $\scrK \le \scrK'$ everywhere, this implies $\scrK = \scrK'$ on $E$. Putting all this together gives that $\scrK = \scrK'$ on $\Rh$.
\end{proof}

\subsection{Independent models}
\label{SS:ind-hypo}
Conditions (iii) and (iv) of Proposition \ref{P:hypograph-cvg}, while technical, can be checked for sequences of last passage models with an independence structure whose one-point marginals converge to appropriately rescaled Tracy-Widom random variables. 

\begin{definition}
	\label{D:independent}
	A collection of random cadlag functions $F = \{F_j:\R\to \R, j \in \Z\}$ is an \textbf{independent environment} if for any disjoint sets $[a_1, b_1] \X \{j_1\}, \dots, [a_k, b_k] \X \{j_k\} \sset \R \X \Z$, the increments $F_{j_i}(b_i) - F_{j_i}(a_i)$ are independent for $i = 1, \dots, k$.
\end{definition}

Any independent environment $F$ gives rise to a directed metric $d_F$ on $\R \X (\Z/2)$ as in Example \ref{Ex:LPP-lines}. We will check the conditions of Proposition \ref{P:hypograph-cvg} for sequences of rescaled line last passage metrics arising from independent environments. Standard examples of lattice last passage percolation fit into this framework by the embedding \eqref{E:embed-1}.

 Consider sequences of independent environments $F^n$ on $\R \X \Z$, positive real numbers $m_n$, continuous additive metrics $h_n$ on $\R \X (\Z/2)$ (see Example \ref{Ex:additive-metrics}), and linear maps $L_n:\R^2 \to \R^2$ of the form
$$
L_n(x, y) = {(a_n x+b_n y,c_n y)}. 
$$
{Assume $a_n,c_n>0$ and $c_n\to\infty$.} Each $L_n$ is the composition of a horizontal shear and a diagonal transformation. Recalling the definition of $f_{\R^2 \to \R \X \Z}$ in \eqref{E:line-plane}, let $\scrK_n$ be the pullback under $f_{\R^2 \to \R \X \Z} \circ L_n$ of the directed metric
$$
m_n d_{F_n} + h_n = d_{m_n F_n} + h_n
$$
on $\R \X (\Z/2)$. 
By Lemma \ref{L:straightforward-dms}, each $\scrK_n$ is a directed metric of negative sign. Moreover, the rescaling by $m_n$ and the augmentation by the additive metric $h_n$ does not change the geodesic structure, so geodesics in $\scrK_n$ {whose endpoints have distinct images} are simply pullbacks under $f_{\R^2 \to \R \X \Z} \circ L_n$ of geodesics in $d_{F^n}$. In particular, if $F^n$ itself arose from a lattice environment, then these are pullbacks of lattice last passage geodesics in $\Z^2$.

Now, the independence structure of $F^n$ guarantees that each $d_{F^n}$ is a temporally independent metric in the sense of Theorem \ref{T:rational-cvge}. Since each $L_n$ is the composition of a diagonal transformation and a horizontal shear, each $\scrK_n$ is also temporally independent. Therefore if the two-time marginals of $\scrK_n$ converge to Airy sheets, then $\scrK_n$ converges to $\scrL$ in the multi-time sense, implying that $\scrK_n$ satisfies condition (i) of Proposition \ref{P:hypograph-cvg} where the limit is the directed landscape $\scrL$. {Moreover, $\scrK_n(x,s;y,t)=-\infty$ whenever $c_n(s-t)>1$. Thus condition (ii) holds uniformly on compact subsets of $\{s>t\}$.} Conditions (iii) and (iv) of Proposition \ref{P:hypograph-cvg} will require the following additional assumption on $\scrK_n$.

\begin{description}
	\item[Mobile Tracy-Widom convergence:] \qquad
	A sequence of random functions $\scrK_n: \Rd \to \R$ satisfies \textbf{mobile Tracy-Widom convergence} if for any sequence $u_n \to u =(x, t; y, t + s^3) \in \Rd$, we have 
	$$
	\scrK_n(u_n) \cvgd \scrL(u) \eqd sT -\frac{(x -y)^2}{s^3}.
$$
	Here $\scrL$ is the directed landscape, and $T$ is a standard GUE Tracy-Widom random variable.
\end{description}

We can now state the main hypograph convergence theorem for independent environments; this is a rephrasing of Theorem \ref{T:intro-hypo}.

\begin{theorem}
	\label{T:ind-main}
	Let $\scrK_n$ be a sequence of random directed metrics of negative sign arising from a sequence of independent environments on $\R \X \Z$ as above. Suppose that $\scrK_n$ satisfies mobile Tracy-Widom convergence and that $\scrK_n|_{S_{t, t + s^3}}$ converges in $\scrE_*(S_{t, t + s^3})$ to an Airy sheet of scale $s$ for all $t \in \R, s > 0$. Then $\mathfrak{h} \scrK_n \cvgd \mathfrak{h} \scrL$ in $\scrE_*$.
\end{theorem}

\begin{remark}
	\label{R:ind-main}
	In the setting of Theorem \ref{T:ind-main}, we actually have a stronger joint convergence result. For any countable union $S = \bigcup_{i=1}^\infty S_{t_i, t_i + s_i^3}$, the pair $(\mathfrak{h}\scrK_n, \scrK_n|_S)$ converges in distribution to $(\mathfrak{h}\scrL, \scrL|_S)$. Here the underlying function space is
	$\scrE_* \X \prod_{i=1}^\infty \scrE_*(S_{t_i, t_i + s_i^3})$.
\end{remark}

\subsection{A maximal inequality for last passage percolation}
\label{SS:maximal}

To pass from multi-time convergence to hypograph convergence in Theorem \ref{T:ind-main}, we need a maximal inequality for last passage percolation. This can be thought of as a last-passage version of Kolmogorov's maximal inequality for sums of independent random variables. We will work within the framework of line last passage percolation set up in Example \ref{Ex:LPP-lines}. A similar maximal inequality could be proven for other directed metrics with an ordered independence structure, e.g. planar Poisson last passage percolation. 

Through this subsection, we let $F$ be an independent cadlag environment on $\R \X \Z$, and let $d_F$ be the corresponding directed metric of negative sign on $\R \X (\Z/2)$ defined as in Example \ref{Ex:LPP-lines}. 

As in Proposition \ref{P:hypograph-cvg}, for $A \sset (\R \X (\Z/2))^2$ and a directed metric $d$ of negative sign on $\R \X (\Z/2)$, we will let
$
d(A) = \sup \{d(u) : u \in A\}.
$
We use the shorthand $d(B, C)$ for $d(B \X C)$.

\begin{lemma}[Set-to-point maximal inequality]
	\label{L:doob-lpp}
	Let $F$ be an independent environment on $\R \X \Z$, and let $d$ be a deterministic, continuous directed metric of negative sign on $\R \X (\Z/2)$. Let $d_F^* = d_F + d$. Let $B \sset \R \X (\Z/2)$ be a finite union of intervals $[s,t] \X \{\ell\}$ for $\ell \in \Z/2$. Then for any $p, q \in \R \X (\Z/2)$ and $c, \ep \in \R$ we have
	\begin{align}
	\label{E:doob-i}
	\p\lf(d_F^*(p, q) \le c - \ep, c < d_F^*(p, B) \rg)  &\le \sup_{b \in B} \; \p\lf(d_F^*(b, q) < - \ep \rg), \\
	\label{E:doob-ii}
	\p\lf(d_F^*(p, q) \le c - \ep, c < d_F^*(B, q) \rg)  &\le \sup_{b \in B} \; \p\lf(d_F^*(p, b) < - \ep \rg).
	\end{align}
\end{lemma}

Note that we could have considered more general sets $B$ in Lemma \ref{L:doob-lpp}; we work with finite unions of intervals for simplicity. The deterministic correction $d$ is necessary since we will apply Lemma \ref{L:doob-lpp} to shifted versions of line last passage metrics.

\begin{proof}[Proof of Lemma \ref{L:doob-lpp}]
	We only prove \eqref{E:doob-i}. The inequality \eqref{E:doob-ii} follows by a symmetric argument. Let $\del B$ be the boundary of $B$ in $\R \X (\Z/2)$. Note that $\del B$ is a finite set since $B$ is a finite union of intervals.
	For $k \in \N$, define the finite approximation
	$$
	B_k = \{(x, \ell) \in B : 2^kx \in \Z\} \cup \del B.
	$$
	{For any point $(x, \ell) \in B$, there is} a nonincreasing sequence $x_k$ with $(x_k, \ell) \in B_k$ such that $x_k \cvgdown x$. Since $F$ is a cadlag environment and $d$ is continuous, we always have that
	$$
	\limsup_{k \to\infty} d_F^*(p, (x_k, \ell)) \ge d_F^*(p, (x, \ell)).
	$$
	Therefore $d_F^*(p, B_k) \cvgup d_F^*(p, B)$ as $k \to \infty$.

	Now define a total order on $\R \X (\Z/2)$ by letting $(x, n) < (y, m)$ whenever $n > m$, or $n=m$ and $x < y$. 
	 For each $k$, define
	$$
	\tau_k = \min \{b \in B_k : d^*_F(p, b) > c\},
	$$
	where the minimum is with respect to this total order.
	If the set on the right hand side above is empty, set $\tau_k = \infty$. The triangle inequality for $d^*_F$ implies that on the event $\tau_k \ne \infty$, we have
	$$
	d_F^*(p, \tau_k) + d_F^*(\tau_k, q) \le d_F^*(p, q)
	$$
	Now, using that $d_F^*(p, \tau_k) > c$ on the event $\tau_k \ne \infty$, we have the containment of events
	\begin{equation}
	\label{E:E2k1}
	E_{1, k} = \{\tau_k \ne \infty, d^*_F(p, q) \le c - \ep \} \sset E_{2, k} = \{\tau_k \ne \infty, d_F^*(\tau_k, q) < - \ep \}.
	\end{equation}
	Now, since $d_F^*(p, B_k) \cvgup d_F^*(p, B)$, we have that 
	$$
	E_{1, k} \to E_1 = \{c < d^*_F(p,B), d^*_F(p, q) \le c - \ep\}. 
	$$
	The event $E_1$ is the event on the left hand side of \eqref{E:doob-i}. Therefore by \eqref{E:E2k1}, to complete the proof it suffices to show that
	\begin{equation}
	\label{E:supbbb}
	\p E_{2, k} \le \sup_{b \in B} \p (d_F^*(b, q) < - \ep)
	\end{equation}
	for all $k \in \N$. We have
	\begin{equation}
	\label{E:E2k}
	\p E_{2, k} \le \max_{b \in B_k} \p(d_F^*(b, q) < -\ep \; |\; \tau_k = b).
	\end{equation}
	Since $F$ is an independent environment, the random variable $d_F^*(b, q)$ is independent of the event $\tau_k = b$ by the order in which we explored the set $B$. Therefore the right hand side of \eqref{E:E2k} is equal to 
	$$
	\max_{b \in B_k} \p(d_F^*(b, q) < -\ep).
	$$
	This is bounded above by the right hand side of \eqref{E:supbbb}, as desired. 
	\end{proof}
	
The ideas behind Lemma \ref{L:doob-lpp} can be iterated to bound maximum last passage values between two sets.
\begin{lemma} [Set-to-set maximal inequality]
	\label{L:doob-lpp-boot}
	Let $F$ be an independent environment on $\R \X \Z$, and let $d$ be a deterministic continuous directed metric of negative sign on $\R \X (\Z/2)$. Let $d_F^* = d_F + d$. 
	Let $A, B \subset \R \X (\Z/2)$ be finite unions of sets of the form $[s, t] \X \{\ell\}$ for $\ell \in \Z/2$ and let $p, q \in \R \X (\Z/2)$. Then
	\begin{equation}
	\label{E:doop-boot}
	\begin{split}
	& \; \p\lf( d_F^*(p, q) \le c - \ep, c < d_F^*(A, B) \rg) \le \;\sup_{a \in A} \p(d_F^*(p, a) < -\ep/2) +  \sup_{b \in B} \p \lf( d_F^*(b, q) < -\ep/2\rg).
	\end{split}
	\end{equation}
\end{lemma}

Note that we do not impose any ordering constraints on $p, q, A, B$ in Lemma {\ref{L:doob-lpp-boot}}. Theorem \ref{T:intro-maxineq} is the special case of Lemma \ref{L:doob-lpp-boot} when $d = 0$ and $A, B \sset \R \X \Z$.

\begin{proof}
We define the same finite approximations $B_k \sset B$ as in the proof of Lemma \ref{L:doob-lpp} so that $d^*_F(A, B_k) \cvgup d^*_F(A, B)$ as $k \to \infty$. We set
	$$
	\tau_k = \min \{b \in B_k : d_F^*(A, b) > c \},
	$$
	where the minimum is with respect to the total order used in the proof of Lemma \ref{L:doob-lpp}. We set $\tau_k = \infty$ if the set on the right hand side above is empty. On the event $\tau_k \ne \infty$, the triangle inequality gives that
	$$
d_F^*(w, \tau_k) + d_F^*(\tau_k, q) \le d_F^*(w, q) 
	$$
	for any $w$.
	Taking a supremum in the above inequality over all points $w \in A$ gives that
		$$
	d_F^*(A, \tau_k) + d_F^*(\tau_k, q) \le d_F^*(A, q).
	$$
	Since $d_F^*(A, \tau_k) > c$ on the event $\tau_k \ne \infty$, this implies the containment of events
	\begin{equation}
	\label{E:Z-event-2}
	D_{1, k} = \{\tau_k \ne \infty, d_F^*(A, q) \le c - \ep/2 \} \sset D_{2, k} = \{\tau_k \ne \infty, d_F^*(\tau_k, q) < -\ep/2\}.
	\end{equation}Now set
	$$
	D_1 = \{c < d^*_F(A, B), d_F^*(A, q) \le c - \ep/2\}.
	$$
	Exactly as in the proof of Lemma \ref{L:doob-lpp}, we have that
	\begin{equation}
	\label{E:E1-bbd}
	\begin{split}
	\p D_1 = \lim_{k \to \infty} \p D_{1, k} \le \sup_k \p D_{2, k} \le \sup_{b \in B} \p \lf( d_F^*(b, q) < -\ep/2\rg).
	\end{split}
	\end{equation}
	Finally, the event on the left hand side of \eqref{E:doop-boot} is contained in the union
	$$
	D_1 \cup \{ d_F^*(p, q) \le c - \ep, c - \ep/2 < d_F^*(A, q) \}.
	$$
	The probability $\p D_1$ is bounded by \eqref{E:E1-bbd}. The probability of the second event above can be bounded by Lemma \ref{L:doob-lpp}. A union bound completes the proof.
\end{proof}

\begin{remark}
	\label{R:edges-included}
	Both Lemma \ref{L:doob-lpp} and Lemma \ref{L:doob-lpp-boot} also apply in the lattice setting, by virtue of the embedding from lattice to lines. In the lattice setting the proofs simplify since we no longer need to take a discrete approximation of the sets in the lemma. 
\end{remark}

\subsection{The proof of Theorem \ref{T:ind-main}}

We can now use the maximal inequality {to} check tightness conditions (iii) and (iv) in Proposition \ref{P:hypograph-cvg} about convergence of random sequences of functions to prove Theorem \ref{T:ind-main}.

We first translate the maximal inequality to the limiting setting. 

\begin{corollary}
	\label{C:limit-setting}
	Let $\scrK_n$ be as in Theorem \ref{T:ind-main}. For any compact convex sets $A, B \sset \R^2, p, q \in \R^2$ and $c, \ep \in \R$ we have 
	\begin{equation}
	\label{E:doop-boot-2}
	\begin{split}
	& \; \p\lf( \scrK_n(p, q) \le c - \ep, c < \scrK_n(A, B) \rg)
	\\ 
	&\le \;\sup_{a \in A} \p(\scrK_n(p, a) < -\ep/2) +  \sup_{b \in B} \p \lf( \scrK_n(b, q) < -\ep/2\rg).
	\end{split}
	\end{equation}
\end{corollary}

\begin{proof}
	We use the notation introduced in Section \ref{SS:ind-hypo}.
	Let $F^n$ be the sequence of independent environments that give rise to $\scrK_n$. Then $\scrK_n$ is the pullback under $f_{\R^2 \to \R \X \Z} \circ L_n$ of a metric $d_{F_n}^* = d_{m_n F_n} + h_n$ on $\R \X (\Z/2)$, see the discussion after Definition \ref{D:independent}. The environments $m_n F_n$ are independent, and each $h_n$ is a deterministic continuous directed metric, so $d_{F_n}^*$ satisfies all the conditions of Lemma \ref{L:doob-lpp-boot}. Moreover, since ${A}$ and ${B}$ are compact and convex, $f_{\R^2 \to \R \X \Z} L_n {A}$ and $f_{\R^2 \to \R \X \Z} L_n {B}$ are finite unions of closed intervals in $\R \X (\Z/2)$. Hence we can apply Lemma \ref{L:doob-lpp-boot} and take a pullback to prove the corollary.
\end{proof}

\begin{proof}[Proof of Theorem \ref{T:ind-main} and Remark \ref{R:ind-main}]
	We check the conditions of Proposition \ref{P:hypograph-cvg}.
As discussed prior to the statement of Theorem \ref{T:ind-main}, $\scrK_n$ satisfies conditions (i) and (ii) of that proposition. It remains to check conditions (iii) and (iv). Throughout the proof we use the shorthand 
$$
A(p, q, r) = B^+(p,r) \X B^-(q,r).
$$
See the statement of Proposition \ref{P:hypograph-cvg} for the definition of the right hand side above.

We first show (iii). Fix a compact set $K \sset \Rd$ and let $(p, q) \in K$. By a union bound, for any $m > 0$ we can write
\begin{equation}
\label{E:split-it}
	\begin{split}
	&\p\left( \mathcal K_n(A(p, q, r))-\mathcal K_n(p,q)>\eps\right) \\
	&\le  \p(|\scrK_n(p, q)| \ge m) + \sum_{k = - \floor{2m/\ep}}^{\ceil{2m/\ep}} \p(\scrK_n(p, q) \le k \ep/2, \;\; k \ep/2 + \ep /2 < \scrK_n(A(p, q, r)){)}.
	\end{split}
	\end{equation}
To see the above inequality, consider two numbers $a, b \in \R$ with $a - b > \ep$. Then either $|b| > m$, or else there exists some lattice point $x \in (\ep/2) \Z \cap [-m, m + \ep/2)$ such that $b \le x \le x + \ep/2 < a$. The above union bound uses this containment with $a =\mathcal K_n(A(p, q, r))$ and $b = \scrK_n(p,q)$.
	By Corollary \ref{C:limit-setting}, the right hand side of \eqref{E:split-it} is bounded above by
	\begin{equation}
	\label{E:Knprelim}
\p(|\scrK_n(p, q)| \ge m) + \sum_{k = - \floor{2m/\ep}}^{\ceil{2m/\ep}} {\biggl[}\sup_{z \in B^+(p, r)} \p(\scrK_n(p, z) < -\ep/{4}) + \sup_{z \in B^-(q, r)} \p \lf(\scrK_n(z, q) < -\ep/{4}\rg){\biggr]}.
	\end{equation} 
	We can pass to the limit $\scrL$ via mobile Tracy-Widom convergence as long as $r$ is small enough so that $A(p, q, r)$ has compact closure in $\Rd$ for all $(p, q) \in K$. This gives that the limsup of \eqref{E:Knprelim} is equal to
	\begin{equation}
	\label{E:Lpqp}
	\p(|\scrL(p, q)| \ge m) + \sum_{k = - \floor{2m/\ep}}^{\ceil{2m/\ep}} {\biggl[}\sup_{z \in B^+(p, r)} \p(\scrL(p, z) < -\ep/{4}) + \sup_{z \in B^-(q, r)} \p \lf(\scrL(z, q) < -\ep/{4}\rg){\biggr]}
	\end{equation}
	Using the Tracy-Widom tail bounds from Theorem \ref{T:tracy-widom}, \eqref{E:Lpqp} is bounded above by
	$$
	c_Ke ^{- d_K m^{3/2}} + (c m/\ep) e^{- d(\ep/({4}r) - 1)^{3/2}},
	$$
	where $c_K$ and $d_K$ are $K$-dependent constants, and $c$ and $d$ are constants. Setting $m = 1/r$, this expression is $o(r^{10})$ as $r \to 0^+$, yielding (iii).
	
	We now show (iv). Let $A$ be a compact subset of $\{(x, s; y, s) \in \R^4 : x \ne y\}$. We may assume that 
	$$
	A = \{(x, s; y, s) \in \R^4 : x \in I, y \in J, s \in T\}
	$$
where $I, J,$ and $T$ are closed intervals and $I$ and $J$ are disjoint. Any other compact subset of $\{(x, s; y, s) \in \R^4 : x \ne y\}$ can be covered by a finite union of sets of this form. For any fixed $\de > 0$, define
$$
\scrI_\de = \{(p, q) = (x, s - 3\de^3/2; y, s + 3\de^3/2) : x \in I \cap (\de^2/2) \Z, y \in J \cap (\de^2/2) \Z, s \in T \cap (\de^3/2) \Z\}.
$$
Then for all small enough $r > 0$ we have that 
	$$
	A_{+r} \sset \bigcup_{(p, q) \in \scrI_\de} A(p, q, \de).
	$$
	For all small enough $r > 0$ and $m > 0$, by a union bound we have
	\begin{align}
	\nonumber
	\p(\scrK_n(A_{+r}) > -m)
\le &\sum_{(p, q) \in \scrI_\de} \p(\scrK_n(A(p, q, \de)) > -m) \\
	\nonumber
	\le &\sum_{(p, q) \in \scrI_\de} {\bigl[}\p(\scrK_n(A(p, q, \de)) > -m, \scrK_n(p, q) \le -2m) \\
	\label{E:Knm}
	&\qquad + \p(\scrK_n(p, q) > -2m){\bigr]}.
	\end{align}
By Corollary \ref{C:limit-setting}, \eqref{E:Knm} is bounded above by
	\begin{equation}
	\label{E:sup-pq}
\begin{aligned}
\sum_{(p, q) \in \scrI_\de} {\biggl[}&\sup_{z \in B^+(p, \de)} \p(\scrK_n(p, z) < -m/2) \\
&+ \sup_{z \in B^-(q, \de)} \p(\scrK_n(z, q) < -m/2) + \p(\scrK_n(p, q) > -2m){\biggr]}.
\end{aligned}
	\end{equation}
	Passing to the limit via mobile Tracy-Widom convergence gives that the $\limsup$ of \eqref{E:sup-pq} is bounded above by
	\begin{equation}
	\label{E:sumiI}
\begin{aligned}
\sum_{(p, q) \in \scrI_\de} {\biggl[}&\sup_{z \in B^+(p, \de)} \p(\scrL(p, z) < -m/2) \\
&+ \sup_{z \in B^-(q, \de)} \p(\scrL(z, q) < -m/2) + \p(\scrL(p, q) > -2m){\biggr]}.
\end{aligned}
	\end{equation} Applying the Tracy-Widom tail-bounds in Theorem \ref{T:tracy-widom} and using that $|\scrI_\de| \le c_A \de^{-7}$ for an $A$-dependent constant $c_A$ gives that \eqref{E:sumiI} is bounded above by
	$$
	c_A\de^{-7} \lf(\exp(-d(m/\de)^{3/2}) + \exp \lf(-d[(1/(c_A\de^4) - m/\de)^+]^{3/2}\rg)\rg),
	$$
	for an $A$-dependent constant $c_A > 0$ and an absolute constant $d > 0$.
	Therefore for every fixed $m > 0$, we have
	\begin{align*}
	&\limsup_{r \to 0^+} \limsup_{n \to \infty} \p(\scrK_n(A_{+r}{)} > -m) \\
	&\qquad \le c_A\de^{-7} \lf(\exp(-d|m/\de|^{3/2}) + \exp \lf(-d[(1/(c_A\de^4) - m/\de)^+]^{3/2}\rg)\rg).
	\end{align*}
	Now, the left hand side above is independent of $\de$, and the right hand side above tends to $0$ as $\de \to 0$. Assumption (iv) of Proposition \ref{P:hypograph-cvg} follows.
	
	Now, let $S$ be as in the statement of the theorem, and let $D$ be any countable dense set in $\Rh$ such that $D \cap S$ is dense in $S$.
	The sequence $(\mathfrak{h} \scrK_n, {\scrK_n}|_D, \scrK_n|_S)$ is tight, and by Proposition \ref{P:hypograph-cvg} and Theorem \ref{T:rational-cvge} any distributional subsequential limit is of the form $(\mathfrak{h} \scrL, \scrL|_D, \scrL'|_S)$, where both $\scrL$ and $\scrL'$ are directed landscapes. Moreover, $\scrL|_{D \cap S} = \scrL'|_{D \cap S}$, so by continuity $\scrL'|_S = \scrL|_S$. This gives the joint convergence in Remark \ref{R:ind-main}.
\end{proof}

\begin{remark}
	\label{T:weaker-than-mobile-TW}
	We have stated Theorem \ref{T:ind-main} with the assumption of mobile Tracy-Widom convergence since it is natural and often easy to check in practice. However, the reader carefully following the proof of Theorem \ref{T:ind-main} will note that a weaker moment assumption would suffice. For example, {retaining one-point convergence $\scrK_n(p,q)\cvgd\scrL(p,q)$ for every fixed $(p,q)\in\Rd$, we can replace mobile Tracy-Widom convergence by} the following moment estimate.
		$$
		{\sup_{p\in\R^2}\limsup_{n\to\infty}\max\biggl\{
\sup_{z\in B^+(p,r)}\E[(\scrK_n(p,z)^-)^{11}],\;
\sup_{z\in B^-(p,r)}\E[(\scrK_n(z,p)^-)^{11}]\biggr\}=O(r^{11}).} 
		$$
{Take $m=r^{-1/2}$ in \eqref{E:Knprelim}. Markov's inequality bounds the sum there by $O(r^{21/2})=o(r^{10})$. For fixed $m>0$, the two connector terms in \eqref{E:sup-pq}, summed over $\scrI_\de$, are $O(\de^4)$. The remaining terms in both estimates are controlled by one-point convergence and the Tracy-Widom tails.}
\end{remark}

\section{Graph convergence to $\scrL$}
\label{S:graph-cvg}
Hypograph convergence is a strictly weaker notion than graph convergence. In fact, it is possible to construct sequences of i.i.d.\ lattice last passage percolation models whose hypographs converge to $\scrL$, but whose graphs do not, see Example \ref{Ex:elpp} in Section \ref{S:integrable}.

To move from hypograph convergence to graph convergence, we will not only need a way of controlling large values in prelimits of $\scrL$, but also a way of controlling small values. Specifically, the following will be enough. For this proposition, we work with both open balls $B(q, r)$ and their closures, denoted by $\bar B(q, r)$.

\begin{prop}
\label{P:graph-cvg}
Suppose that $\scrK_n:\Rh \to \bar \R$ is a sequence of random functions. Let $D$ be a countable dense subset of ${\Rd}$ and suppose that $(\mathfrak{h}\scrK_n, \scrK_n|_D) \cvgd (\mathfrak{h}\scrK, \scrK|_D)$ for some random continuous function $\scrK:\Rh \to \bar \R$ with $\scrK(x, s; y, t) = -\infty$ whenever $s \ge t$, and $\scrK$ is finite on $\Rd$. Suppose that
for every open ball $B(q, r) \sset \Rd$ whose closure is also contained in $\Rd$, there exists a sequence of finite subsets {$D_m$ of $D$} such that $D_m \to \bar B(q, r)$ in the Hausdorff metric, and for every $\ep > 0$ we have
\begin{equation}
\label{E:Kninf}
\lim_{m \to \infty} \limsup_{n \to \infty} \p\lf(\inf_{u \in B(q, r)} \scrK_n(u) < \inf_{u \in D_m} \scrK_n(u) - \ep \rg) = 0.
\end{equation}
Then $\mathfrak{g} \scrK_n \cvgd \mathfrak{g} \scrK$ in $\scrE_*$.
\end{prop}

\begin{proof}
{The assumed convergence and compactness of $\scrE_*$ give joint tightness.}
Let $(\Ga,\allowbreak \mathfrak{h} \scrK,\allowbreak \scrK|_D)$ be any joint subsequential distributional limit of $(\mathfrak{g} \scrK_n,\allowbreak \mathfrak{h}\scrK_n,\allowbreak \scrK_n|_D)$. Define
$$
\scrK'(u) = \inf \{x : (u, x) \in \Ga\}
$$ for all $u \in \Rh$. This is a lower semicontinuous function. Noting that 
$$
\scrK(u) = \sup \{x : (u, x) \in \Ga\},
$$
to prove the desired convergence it is enough to show that almost surely, $\scrK' \ge \scrK$.
Since $\scrK = -\infty$ on the set $\Rh \smin \Rd = \{(x, s; y, t) \in \Rh : s \ge t\}$, this inequality holds on that set. Now suppose that $\scrK'(u) < \scrK(u)$ for some $u \in \Rd$. Let $\bar B(q_i, r_i) \sset \Rd$ be a sequence of closed balls with rational center and rational radius $r_i \to 0$ such that $u \in B(q_i, r_i)$ for all $i$. By the continuity of $\scrK$, we have
$$
\inf_{v \in B(q_i, r_i)} \scrK(v) \to \scrK(u).
$$
Hence for large enough $i$, we have
\begin{equation}
\label{E:vqr'}
\scrK'(u) < \inf_{v \in B(q_i, r_i)} \scrK(v),
\end{equation}
and so
\begin{equation}
\label{E:vqr}
\inf_{v \in B(q_i, r_i)} \scrK'(v) < \inf_{v \in B(q_i, r_i)} \scrK(v).
\end{equation}
Since there are only countably many choices of $q_i, r_i$, it is enough to show that the probability that \eqref{E:vqr} holds is zero for every fixed rational pair $q, r$. Letting $D_m$ be as in the statement of the {proposition} for the set $B(q, r)$, by the continuity of $\scrK$ the event in \eqref{E:vqr} is equal to the union over $k \in \N$ of the events
\begin{equation*}
B_k := {\bigcup_{M=1}^\infty\bigcap_{m=M}^\infty} A_{m, k}, \qquad \text{ where } \qquad 
A_{m, k} = \lf\{\inf_{v \in B(q, r)} \scrK'(v) < \inf_{v \in D_m} \scrK(v) - 1/k\rg\}.
\end{equation*}
We will show that
\begin{equation}
\label{E:Amkk}
\p A_{m, k} \le \limsup_{n \to \infty} \p A_{m, k, n}, \quad \text{ where } \quad A_{m, k, n} = \lf\{\inf_{v \in B(q, r)} \scrK_n(v) < \inf_{v \in D_m} \scrK_n(v) - 1/k\rg\}.
\end{equation}
By the assumption of the {proposition}, the right hand side of the inequality in \eqref{E:Amkk} goes to $0$ as $m \to \infty$. Hence \eqref{E:Amkk} implies that $\p B_k = 0$, giving that the probability of \eqref{E:vqr} is also $0$.

It remains to show \eqref{E:Amkk}. It is perhaps easiest to see this with Skorokhod's representation theorem. Consider a coupling of the $\scrK_n$s along a subsequence $N$ such that almost surely, $\mathfrak{g} \scrK_n \to \Ga$ and $\scrK_n|_D \to \scrK|_D$. Graph convergence implies that for any $(v, \scrK'(v)) \in \Ga$, there is a sequence $v_n \to v$ such that $\scrK_n(v_n) \to \scrK'(v)$. Then for $v$ in the open ball $B(q, r)$, we have
$$
\limsup_{n \in N} \inf_{w \in B(q, r)} \scrK_n(w) \le \lim_{n \to \infty} \scrK_n(v_n) = \scrK'(v).
$$
Taking the infimum over $v \in B(q, r)$ we get
$$
\limsup_{n \in N} \inf_{v \in B(q, r)} \scrK_n(v) \le \inf_{v \in B(q, r)} \scrK'(v). 
$$
Since $D_m$ is finite and $\scrK_n|_D \to \scrK|_D$ almost surely, we also have
$
\inf_{v \in D_m} \scrK_n(v) \to \inf_{v \in D_m} \scrK(v)
$
almost surely.
Hence in this coupling, almost surely
$$
\om \in A_{m, k}\qquad \implies \qquad \om \in A_{m, k, n} \text{ for all large enough $n \in N$},
$$
and so $\indic (A_{m, k}) \le \liminf_{n \in N} \indic(A_{m, k, n})$ almost surely. Applying Fatou's lemma and replacing the liminf over $n \in N$ with a limsup over all $n \in \N$ yields \eqref{E:Amkk}.
\end{proof}

For directed metrics $\scrK_n$ in the form of Theorem \ref{T:ind-main}, the conditions of Proposition \ref{P:graph-cvg} can be shown by a chaining argument as long as we have moderate deviation bounds on the one-point distributions of $\scrK_n$. This type of chaining argument was pioneered by \cite*{talagrand2006generic}, and  was used previously in this setting by \cite*{basu2014last} to prove lower bounds on the infimum of Poisson last passage percolation. We adapt the chaining framework for general planar directed metrics. The KPZ scaling is built into our setup to give an optimal result in the cases we care about.

We first set up a lemma needed for the chaining argument.  
For each $p = (p_1, p_2) \in \R^2$, we will set
$
p_k^+ = (p_{k,1}^+, p_{k,2}^+)
$
to be a point in $2^{-2k/3} \Z \X 2^{-k} \Z $ which is close to $p$, while still lying ahead of $p$ in the time direction. More precisely, define $p_k^+$ so that  $|p- p_k^+|$ is minimal subject to the condition 
\begin{equation}
\label{E:p-separation}
p_2 + 2^{-k} \le p^+_{k, 2} < p_2 + 2^{-k + 1}.
\end{equation}
Similarly let $p_k^-$ denote the point in $2^{-2k/3} \Z \X 2^{-k} \Z$ which minimizes $|p- p_k^-|$ subject to the condition $p_2 - 2^{-k} \ge p^-_{k, 2} > p_2 - 2^{-k + 1}$. If there are multiple closest points, choose an arbitrary method of breaking ties. We note for later use that with these definitions we have
\begin{equation}
\label{E:p-cone}
|p_1 - p_{k, 1}^+| \le |p_2 - p_{k, 2}^+|^{2/3} \quad \mathand \quad |p_{k, 1}^+ - p_{k-1, 1}^+| \le 2|p_{k, 2}^+ - p_{k-1, 2}^+|^{2/3},
\end{equation}
as well as the corresponding inequalities for $p_k^-$. Now consider a directed metric $d$ on $\R^2$. We will control the infimum of $d$ on a set $A$ in terms of its values on a grid. Define
$$
\eta^+_k(d, A) = \inf_{p \in A} d(p, p_k^+) \qquad \mathand \qquad \de^+_k (d, A) = \inf_{p \in A} d(p_{k+1}^+, p_k^+)
$$
and similarly let
$$
\eta^-_k(d, A) = \inf_{p \in A} d({p_k^-}, p) \quad \mathand \qquad \de^-_k (d, A) = \inf_{p \in A} d(p_{k}^-, p_{k+1}^-).
$$
Finally, for a set $A \sset \R^2$, we let $A_\ell^+ = \{p_\ell^+ : p \in A\}$ and $A_\ell^- = \{p_\ell^- : p \in A\}$. 
\begin{lemma}
\label{L:chaining-data}
Let $d$ be a directed metric of negative sign and let $K = A \X B \sset \R^4$. Then for every $\ell < k \in \N$ we have
\begin{equation}
\label{E:chaining}
\inf_{u \in K} d(u) \ge \inf_{u \in A^+_\ell \X B^-_\ell} d(u) + \sum_{i=\ell}^{k-1} \lf( \de^+_i(d, A) + \de^-_i(d, B) \rg) + \eta^+_k(d, A) + \eta^-_k(d, B). 
\end{equation}
\end{lemma}

\begin{proof}
We have the iterated triangle inequality
$$
d(p, q) \ge d(p, p^+_k) + \sum_{i=\ell}^{k-1} d(p^+_{i+1}, p^+_i) + d(p^+_\ell, q^-_\ell) + \sum_{i=\ell}^{k-1} d(q^-_i, q^-_{i+1}) + d(q^-_k, q).
$$
Taking infima over $A \X B$ on both sides and using that $\inf {(R+T)} \ge \inf R + \inf T$ for sets $R, T \sset \R$ yields \eqref{E:chaining}.
\end{proof}

We can use Lemma \ref{L:chaining-data} to get a probabilistic condition for when a sequence of directed metrics satisfies the conditions of Proposition \ref{P:graph-cvg}.

\begin{prop}
\label{P:prob-cond}
Let $\scrK_n$ be a sequence of random directed metrics of negative sign on $\R^2$ and let $K = A \X B \sset \Rd$ {be compact}. Suppose that there exists a sequence $k(n) \to \infty$ such that for every $\ep > 0$ we have
\begin{itemize}
	\item $\p(\eta_{k(n)}^+(\scrK_n, A) < -\ep), \p(\eta_{k(n)}^-(\scrK_n, B) < -\ep) \to 0$ as $n \to \infty$.
	\item There exist positive numbers $\be_{i, n}(\ep)$ such that 
\begin{align}
\nonumber
\sup_{p \in A} \p(\scrK_n(p^+_{i+1}, p^+_i) < -\ep/i^2) &\le 2^{-5i/3} \be_{i, n}(\ep),\\
\nonumber
\sup_{q \in B} \p(\scrK_n(q^-_i, q^-_{i+1}) < -\ep/i^2) &\le 2^{-5i/3} \be_{i, n}(\ep), \qquad \mathand \\
\label{E:elln-infty}
\lim_{\ell \to \infty} \limsup_{n \to \infty} \sum_{i=\ell}^{k(n)} \be_{i, n}(\ep) &= 0.
\end{align} 
\end{itemize}
Then for any $\ep > 0$,
$$
\lim_{m \to \infty} \limsup_{n \to \infty} \p\lf(\inf_{u \in K} \scrK_n(u) < \inf_{u \in A^+_m \X B^-_m} \scrK_n(u) - \ep \rg) = 0.
$$
\end{prop}

\begin{proof} 
By Lemma \ref{L:chaining-data}, we just need to show that
\begin{equation}
\label{E:Wnell}
\lim_{\ell \to \infty} \limsup_{n \to \infty} \p\lf(\sum_{i=\ell}^{k(n)-1} \lf( \de^+_i(\scrK_n, A) + \de^-_i(\scrK_n, B) \rg) + \eta^+_{k(n)}(\scrK_n, A) + \eta^-_{k(n)}(\scrK_n, B) < -\ep\rg) = 0
\end{equation}
for all $\ep > 0$. First, for $\ell \ge 2$ we have the union bound
\begin{equation}
\label{E:sumsum}
\p \lf(\sum_{i=\ell}^{k(n) -1} \de_i^+(\scrK_n, A) < - \ep/4 \rg) \le \sum_{i=\ell}^{k(n) - 1} \p \lf(\de_i^+(\scrK_n, A) < - \ep/(4i^2) \rg)
\end{equation}
By the construction of the points $p_i^+$, there is a constant $c_A > 0$ such that the cardinality of the set
$$
\{ (p_{i+1}^+, p_i^+) : p \in A\}
$$
is bounded above by $c_A 2^{5i/3}$ for all $i$. Therefore by a second union bound, the right hand side of \eqref{E:sumsum} is bounded above by
$$
c_A \sum_{i=\ell}^{k(n)} \be_{i, n}(\ep/4).
$$
The probability that {$\sum \de_i^-(\scrK_n, B)$ is less than $-\ep/4$} can be similarly bounded.
Finally, $\p(\eta_{k(n)}^+(\scrK_n, A) {<} -\ep/4)$ and $\p(\eta_{k(n)}^-(\scrK_n, B) {<} -\ep/4)$ tend to $0$ with $n$ by assumption. Putting these four bounds together and using the asymptotics \eqref{E:elln-infty} on $\be_{i, n}(\ep/4)$ proves \eqref{E:Wnell}.
\end{proof}

We also record a quantitative version of Proposition \ref{P:prob-cond}. This will be used later on when establishing tightness of exponential moments of last passage models.
\begin{corollary}
\label{C:true-bound-for-holes}
In the setting of Proposition \ref{P:prob-cond}, for $m \ge 2, n \in \N,$ and $\ep > 0$, we have
\begin{equation}
\label{E:infuK-prob}
\begin{split}
\p\Bigg(\inf_{u \in K} \scrK_n(u) &< \inf_{u \in A^+_m \X B^-_m} \scrK_n(u) - \ep \Bigg) \\
&\le \p(\eta_{k(n)}^+(\scrK_n, A) < -\ep/4) + \p(\eta_{k(n)}^-(\scrK_n, B) < -\ep/4) + c_K \sum_{i=m}^{k(n)} \be_{i, n}(\ep/4),
\end{split}
\end{equation}
where $c_K>0$ is a $K$-dependent constant.
\end{corollary}

Corollary \ref{C:true-bound-for-holes} follows by simply keeping track of the bound we get on the probability of the right side of \eqref{E:infuK-prob} when proving Proposition \ref{P:prob-cond}.

 Putting Proposition \ref{P:prob-cond} together with Proposition \ref{P:graph-cvg}, we obtain the following corollary. We state this for convergence to the directed landscape.
 
\begin{corollary}
\label{C:Kn}
Suppose that $\scrK_n$ is a sequence of random directed metrics of negative sign, and let 
$$
D = \Rd \cap \lf(\bigcup_{k \in \N} (2^{-2k/3} \Z \X 2^{-k} \Z)^2 \rg).
$$
Suppose that $(\mathfrak h {\scrK_n}, {\scrK_n}|_{D}) \cvgd (\mathfrak h \scrL, \scrL|_{D})$, where $\scrL$ is the directed landscape. Suppose that for every  compact set of the form $A \X B \sset \Rd$, that $\scrK_n$ satisfies the conditions of Proposition \ref{P:prob-cond} on $A \X B$. Then $\mathfrak g \scrK_n \cvgd \mathfrak g \scrL$.
\end{corollary}

Corollary \ref{C:Kn} follows from Propositions \ref{P:graph-cvg} and \ref{P:prob-cond}. {For a compact set $K\sset\Rd$, cover $K$ by finitely many compact products contained in its $1/j$-neighbourhood. Apply Proposition \ref{P:prob-cond} to these products and take the union of their grids. Choosing the grid level sufficiently large for each $j$ gives finite sets converging to $K$ in Hausdorff distance and satisfying \eqref{E:Kninf}.}

\section{Convergence for lattice last {passage} percolation models}
\label{S:cvg-lattice}

In this section, we combine all of the results of the prior sections to give convergence criteria for lattice last passage percolation models with independent, identically distributed weights. 

To make the statements more concrete and avoid topological issues, all theorems in this section and the next are given in terms of couplings for sequences of directed metrics converging to the directed landscape. The first type of coupling we consider will satisfy the following three conditions on a set $\Om$ of probability one. To state these conditions, let $\scrK_n$ be a sequence of random directed metrics of negative sign on $\R^2$.

\medskip

\noindent {\bf Multi-time convergence (on $S$).}  Let $S$ be a finite or countable union of sets of the form
		$$
		S_{s, t} = \{(x, s, y, t) : x, y \in \R\},
		$$
		where $s < t$. Then $\scrK_n \to \scrL$ compactly on $S_{s, t}$ for any $S_{s, t} \sset S$. 

\medskip		

\noindent{\bf Hypograph convergence.} The hypographs $\mathfrak{h} \scrK_n \to \mathfrak{h} \scrL$ in ${\scrE_*}$. By Lemma \ref{L:hypo-equiv}, this implies that 
		\begin{itemize}[nosep]
		    \item $(\scrK_n - \scrL)^+ \to 0$ compactly on $\Rd$,
		    \item for any $u_n \to u \in \Rh$, we have $\limsup_{n \to \infty} \scrK_n(u_n) \le \scrL(u)$,
		    \item for any $u \in \Rd$ there exists a random sequence $U_n \to u$ such that $\scrK_n(U_n) \to \scrL(u)$.
		\end{itemize}

\medskip		

\noindent{\bf Geodesic convergence off holes.} Let $(p_n, q_n) \to (p, q) \in \Rd$, and let $\pi_n$ be a sequence of $\scrK_n$-geodesics from $p_n$ to $q_n$. Then $\pi_n$ is precompact in the Hausdorff topology. If $\om \in \Om$ is in the set
		$
		\{\scrK_n(p_n, q_n) \to \scrL(p, q) \}
		$
		then all subsequential limits of $\pi_n(\om)$ are $\scrL$-geodesics from $p$ to $q$. If
		$$
		\om \in \{\text{There exists a unique $\scrL$-geodesic $\pi$ from $p$ to $q$}\},
		$$
		then $\pi_n(\om) \to \pi$.

\medskip

A second, stronger coupling encodes graph convergence and the resulting strengthened geodesic convergence. This coupling will satisfy the following two conditions almost surely.

\medskip		

\noindent{\bf Graph convergence.} $\mathfrak g \scrK_n \to \mathfrak g \scrL$ as graphs in $\Rh$. In particular, $\scrK_n \to \scrL$ compactly on $\Rd$.

\medskip		

\noindent{\bf Geodesic convergence everywhere.} For any sequence $(u_n, v_n) \to (u, v) \in \Rd$ and any $\scrK_n$-geodesics $\pi_n$ from $u_n$ to $v_n$, the sequence $\pi_n$ is precompact in the Hausdorff topology and any subsequential limit $\pi$ of $\pi_n$ is an $\scrL$-geodesic from $u$ to $v$. In particular, if there is a unique $\scrL$-geodesic $\pi$ from $u$ to $v$, then $\pi_n \to \pi$. 

\medskip
	
The first theorem in this section gives conditions for when a sequence of potentially changing i.i.d.\ lattice last passage models can be coupled with the directed landscape to satisfy multi-time convergence, hypograph convergence, and geodesic convergence off holes. We do not consider graph convergence or geodesic convergence everywhere in this setting. By results of the previous sections, our main input is convergence to the Airy line ensemble.

Recall that for a lattice environment $X$,
$$
X[p^{k} \to q^{k}] = \sup_{\pi_1, \dots, \pi_k} \sum_{i=1}^k \sum_{v \in \pi_i} X_v,
$$
where the supremum is over all sets of $k$ disjoint minimal length lattice paths $\pi_1, \dots, \pi_k$, where $\pi_i$ goes from $p - (0, k -i)$ to $q + (0, i-1)$. We also write 
$$
X[p \to_{\Delta_k} q] = X[p^{k} \to q^{k}] - X[p^{{k-1}} \to q^{{k-1}}],
$$
where $X[p^{0} \to q^{0}] = 0$.

A random function $\scrA^r:\N \X \R \to \R$ is a \textbf{parabolic Airy line ensemble of scale $r$} if
$$
\scrA^r_{{k}}(\cdot) = r \scrA_{{k}}(\cdot/r^2),
$$
where $\scrA$ is a parabolic  Airy line ensemble. A parabolic Airy line ensemble of scale $r$
is related to an Airy sheet of scale $r$ in the same way that the parabolic Airy line ensemble is related to the Airy sheet.

\begin{theorem}
	\label{T:lattice-models} Let $X_n = \{X_n(u) : u \in \Z^2\}$ be a sequence of arrays of i.i.d.\ random variables with distribution $\mu_n$ supported on $[0, \infty)$. Let $m_n \in \N$ be a sequence converging to $\infty$ with $n$. Suppose that there exist sequences $\al_n, \be_n, \chi_n, \tau_n$, {with $\chi_n,\tau_n>0$,} such that for every sequence $r_n$ converging to some $r > 0$, the functions
		\begin{equation}
		\label{E:Arn-k}
		A^{r_n}_{n, k}(y) =  \frac{X_n[(0, \floor{r_n n}) \to_{\Delta k} (\floor{r_n m_n} + \floor{\tau_n y}, 1)] - \al_n \floor{r_n n} - \tau_n y \be_n}{\chi_n}
		\end{equation}
		converge jointly in distribution as $n \to \infty$ to an Airy line ensemble of scale $r^{1/3}$ in the product-of-Skorokhod topologies on functions from $\R \to \R$. Suppose also that as $n \to \infty$ we have
		\begin{equation}
		\label{E:scaling}
\frac{\be_n}{\chi_n} \to 0, \qquad \frac{\be_n m_n - \al_n n}{n \chi_n} \to 0, \qquad \frac{1}{\tau_n} \to 0, \qquad \frac{m_n}{n \tau_n} \to 0.
		\end{equation}
	Then letting $d_{X_n}$ denote the lattice last passage metric defined from $X_n$, after rescaling and centering $d_{X_n}$ can be coupled with $\scrL$ so that for any finite or countable union $S = \bigcup S_{s_i, t_i}$, we have multi-time convergence on $S$, hypograph convergence, and geodesic convergence off holes.
	More precisely, we define a sequence of additive metrics by setting $h_n:\Z^2 \to \R$ by
	$$
	h_n(m_n t/n+ x, - t) = \al_n t + \be_n x,
	$$
	and extending $h_n$ to the directed edge set $E$ for $\Z^2$ by letting $h_n((e_1, e_2)) = {h_n}(e_1)$. 
Let $\scrK_n$ be the pullback onto $\R^2$ of the metric
$
\chi_n^{-1}(d_{X_n} {+} h_n)
$
under the map $
f_{\R^2 \to \Z^2} \circ L_n :\R^2 \to \Z^2 \cup E,
$
where $L_n:\R^2 \to \R^2$ is the linear map with matrix
$$
A = \lf[\begin{array}{cc}
\tau_n & m_n \\
0 & n 
\end{array}\rg],
$$
and $f_{\R^2 \to \Z^2}$ is as in \eqref{E:lat-to-plane}. Then for any finite or countable union $S = \bigcup S_{s_i, t_i}$, there is a coupling of $\scrK_n$ and $\scrL$ satisfying multi-time convergence on $S$, hypograph convergence, and geodesic convergence off holes.
\end{theorem}

The proof of Theorem \ref{T:lattice-models} is straightforward but tedious, and consists of carefully checking the assumptions of the convergence theorems in the previous sections. We remark that the scaling assumptions in \eqref{E:scaling} are quite natural. They are there purely to handle lattice effects. The first two conditions amount to the statement that as $n \to \infty$,
\begin{equation}
\label{E:hn-scaling-cond}
\frac{|h_n(v) - h_n(v')|}{\chi_n} \to 0 
\end{equation}
for adjacent vertices $v$ and $v'$. The final two scaling conditions are equivalent to the statement that if $v$ and $v'$ are adjacent vertices in $\Z^2$ connected in $\Z^2$ by an edge, then $\|L_n^{-1} v - L_n^{-1} v' \|_2 \to 0$ as $n \to \infty$. In particular, the third scaling condition actually follows from the Airy line ensemble convergence, \eqref{E:Arn-k}. Indeed, if $\tau_n^{-1}$ did not converge to $0$, then subsequential limits of the line ensembles $A^{r_n}$ would be piecewise linear.

\begin{proof}
\textbf{Step 1: Converting to a line environment.} \qquad
In order to apply the theorems of previous sections, we first convert $X_n$ into a line environment. As in \eqref{E:fGkk}, define sequences of nondecreasing cadlag environments $\bar X_n:\R \X \Z \to \R$ by letting
$$
\bar X_{n, i}(0) = 0, \qquad \mathand \qquad \bar X_{n, i}(y) - \bar X_{n, i}(x) = \sum_{p \in (x, y] \X \{i\}} X_n(p).
$$
{By the discussion following \eqref{E:embed-1}, $d_{\bar X_n}$ agrees with finite lattice passage values at the indicated representatives of vertices and edges.} We also define the continuous additive metric $d_{\bar h_n}$ on ${\R \X (\Z/2)}$ from the functions $\bar h_n:\R \X (\Z/2) \to \R$ given by
\begin{equation}
\label{E:hn-def}
\begin{split}
\bar h_n(m_n t/n + x, -t) &= \al_n t + \be_n x, \qquad (m_n t/n  + x, -t) \in \R \X \Z \quad \mathand \\
\bar h_n(m_n t/n + x, -t) &= h_n(m_n t/n + x, -t - 1/2), \qquad (m_n t/n + x, -t) \in \R \X (\Z + 1/2).
\end{split}
\end{equation}
Let $\scrL_n$ be the pullback of $\chi_n^{-1}(d_{\bar X_n} {+} \bar h_n)$ under the map $f_{\R^2 \to \R \X \Z} \circ L_n$. 
{To compare passage values, move each endpoint horizontally by at most $1/\tau_n$, keeping its time coordinate fixed, so that its image is the representative of its lattice vertex or edge. Whenever $\scrK_n$ is finite, it equals $\scrL_n$ at these moved endpoints up to a deterministic error tending uniformly to zero by \eqref{E:hn-scaling-cond}. Otherwise $\scrK_n=-\infty$, so the corresponding upper bound still holds. On compact subsets of $\Rd$, all such lattice pairs are eventually connected, since $m_n/\tau_n\to\infty$, as shown in Step 4 below. After adjoining rational time slices to $S$, multi-time and hypograph convergence for $\scrL_n$ therefore imply the same for $\scrK_n$, using Lemma \ref{L:hypo-equiv}. We treat geodesics separately in Step 2. To prove convergence of the passage values,} we check the conditions of Theorem \ref{T:ind-main}. Before checking these conditions, we show why this theorem implies the results.

\textbf{Step 2: The proof assuming Theorem \ref{T:ind-main}.} \qquad  If Theorem \ref{T:ind-main} applies (or more precisely, Remark \ref{R:ind-main}), then for any countable union $S = \bigcup_{i=1}^\infty S_{s_i, t_i}$, we have that $(\mathfrak{h} \scrL_n, \scrL_n|_S) \cvgd (\mathfrak{h} \scrL, \scrL|_S)$. By Skorokhod's representation theorem, there exists a coupling of $\scrL_n, \scrL$ such that $(\mathfrak{h} \scrL_n, \scrL_n|_S) \to (\mathfrak{h} \scrL, \scrL|_S)$ almost surely. This immediately implies multi-time convergence on $S$ and hypograph convergence in this coupling. 

{Finally, we prove geodesic convergence off holes by the path-ordering argument of Theorem \ref{T:geod-cvg-stronger}. Rounding identifies distinct points, so its first two hypotheses need not hold for the pullback metrics. For endpoints converging to $(x,s;y,t)\in\Rd$, their rounded images are distinct for all large $n$. By Lemma \ref{L:pullback-geod}, a pullback geodesic $\Ga_n$ contains the full fibres over a maximizing lattice or line path. Let $\gamma_n$ be the polygonal representation of that path, with edge endpoints represented by their midpoints. Then $\Ga_n$ is connected, its endpoints are within $\eps_n$ of those of $\gamma_n$, and
$$
d_{\mathrm H}(\bar\Ga_n,\gamma_n)\le\eps_n,
\qquad
\eps_n=\frac1{\tau_n}+\frac{m_n}{n\tau_n}+\frac1n\longrightarrow0.
$$
Here the bound follows by applying $L_n^{-1}$ to a unit cell, and convergence follows from \eqref{E:scaling}. For the line embedding alone the term $1/\tau_n$ can be omitted. Additive centering does not change the geodesics. Since $\scrL$ is almost surely a spacetime metric, Theorem \ref{T:geod-cvg} now applies whenever the endpoint passage values converge.}

{For arbitrary endpoints, choose the auxiliary endpoints in the proof of Theorem \ref{T:geod-cvg-stronger}, taking the diagonal sequence slowly enough that $\eps_n$ is smaller than the strict spatial and temporal separations used there. Apply the crossing argument of Lemma \ref{L:ordering} to the paths $\gamma_n$. Intersecting lattice paths can be spliced at common vertices or edges, and line paths at common points. Replacing the segment between two intersections preserves its passage value, since both segments are maximizing between the same endpoints. Thus, for any $(w_n,r_n)\in\Ga_n$, the lower comparison gives a point $(w_n^-,r_n^-)$ in the closure of a geodesic between the auxiliary endpoints with
$$
w_n^-\le w_n+2\eps_n,\qquad |r_n^--r_n|\le2\eps_n.
$$
The same holds for the upper comparison with the spatial inequality reversed. Theorem \ref{T:geod-cvg} applies to these comparison geodesics. Letting $n\to\infty$ gives the same bounds as in the proof of Theorem \ref{T:geod-cvg-stronger}, hence precompactness and convergence whenever the limiting geodesic is unique. This argument applies to both $\scrK_n$ and $\scrL_n$.}

\textbf{Step 3: Reducing Theorem \ref{T:ind-main} to Theorem \ref{T:airy-sheet-gen}.} \qquad We now check the two conditions {of} Theorem \ref{T:ind-main}. To do this, we will reduce to a statement that can be checked via Theorem \ref{T:airy-sheet-gen}.

The mobile Tracy-Widom convergence and Airy sheet convergence required for Theorem \ref{T:ind-main} are both implied by the statement that for any $s_n \to s \in \R$ and $r_n \to r > 0$ we have
\begin{equation}
\label{E:Lnnn}
\mathfrak{g} \scrL_n(\cdot, s_n; \cdot, s_n + r_n) \cvgd \mathfrak{g}\scrS_{r^{1/3}}
\end{equation}
Here the underlying topology is graph convergence in $\scrE_*(\R^2)$ (see Remark \ref{R:other-domains}), and $\scrS_{r^{1/3}}$ denotes an Airy sheet of scale $r^{1/3}$. 
We first deal with the case when $n s_n \notin \Z$ for all $n$; if $n s_n \in \Z$ for infinitely many $n$, there is an extra complication that we leave to Step 5.
Set 
$$
\tilde r_n = \frac{\floor{-ns_n} - \ceil{-n(s_n + r_n)} + 1}n.
$$
We will show that the convergence in \eqref{E:Lnnn} is equivalent to graph convergence of the functions
\begin{equation}
\label{E:want-after-trans}
S_n(x, y) = \chi_n^{-1} \lf(\bar X_n[(\tau_n x, \floor{\tilde r_n n}) \to (\floor{\tilde r_n m_n} + \tau_n y, 1)] - (\al_n \floor{\tilde r_n n} + \tau_n (y-x) \be_n ) \rg)
\end{equation}
to $\scrS_{r^{1/3}}$. This we will be able to check with Theorem \ref{T:airy-sheet-gen}. The deterministic second term in \eqref{E:want-after-trans} is equal to
\begin{equation}
\label{E:Hdef}
H_n(y-x) := \bar h_n(\tau_n(y-x) +\tilde r_n m_n, -\tilde r_n n).
\end{equation} 
 Letting $g_n = \bar h_n \circ f_{\R^2 \to \R \X \Z} \circ L_n$, we can write 
\begin{equation}
\label{E:split-Ln}
\begin{split}
\scrL_n(x, s_n; y, s_n + r_n)  &=\chi_n^{-1} \bigg(\bar X_n[(\tau_n x + m_n s_n, \floor{-ns_n}) \to{} \\
&\qquad {(}\tau_n y + m_n (s_n + r_n), \ceil{-n(s_n + r_n)})] \\
&\qquad - (g_n(y, s_n + r_n) - g_n(x, s_n)) \bigg).
\end{split}
\end{equation}
We have
\begin{equation*}
\label{E:gncalc}
\begin{split}
g_n(y, s_n + r_n) - g_n(x, s_n) = \;&\bar h_n(\tau_n y + m_n (s_n + r_n), \floor{-n (s_n + r_n)}) \\
&- \bar h_n(\tau_n x + m_n s_n, \floor{-n s_n}) \\
= \; & \bar h_n(\tau_n(y - x) + m_nr_n, -\tilde{r_n} n + O(1) ).
\end{split}
\end{equation*}
Here the first equality is by definition and {the second equality is by the linearity} of $\bar h_n$. By \eqref{E:hn-scaling-cond} and the definition \eqref{E:Hdef}, this is equal to 
\begin{equation}
\label{E:Hndeep}
 H_n(y - x + \ep_n) + o(\chi_n), \qquad \text{ where } \qquad \ep_n = \frac{m_n r_n - {\floor{m_n \tilde r_n}}}{\tau_n}.
\end{equation}
Next, by the translation invariance of $X_n$, the first term on the right hand side of \eqref{E:split-Ln} is equal in distribution, jointly in $x$ and $y$, to 
\begin{equation}
\label{E:splitt}
\begin{split}
\bar X_n[(\tau_n (x + \de_n), \floor{\tilde r_n n}) \to {(}\tau_n (y + \de_n+ \ep_n) + {\floor{m_n \tilde r_n}} , 1)], \quad \text{where} \quad \de_n = \frac{m_n s_n - \floor{m_n s_n}}{\tau_n}.
\end{split}
\end{equation}
Combining \eqref{E:splitt} and \eqref{E:Hndeep}, we get that
$$
\scrL_n(x, s_n; y, s_n + r_n) \eqd S_n(x + \de_n, y + \de_n + \ep_n) + o(1),
$$
where the equality in distribution is joint in all $x$ and $y$. Finally, $\de_n = O(\tau_n^{-1})$ and $\ep_n = O({m_n/(n\tau_n)+\tau_n^{-1}})$, so both terms converge to $0$ by \eqref{E:scaling}. Hence graph convergence of $\scrL_n(\cdot, s_n; \cdot, s_n + r_n)$ is equivalent to graph convergence of $S_n$ when $ns_n \notin \Z$.

\textbf{Step 4: Checking Theorem \ref{T:airy-sheet-gen} when $ns_n \notin \Z$.} \qquad Define the rescaled and linearly shifted line environment $Z^n$ by
$$
Z^n_i(y) = \chi_n^{-1} \lf(\bar X_{n, i}(\tau_n y + \floor{\tilde r_n m_n} ) {-} \be_n \tau_n y\rg).
$$
By the fact that last passage values across cadlag paths commute with linear shifts, we have that
\begin{equation}
\label{E:Sndef}
\begin{split}
&S_n(x, y) = Z^n[(x + a_n, \floor{\tilde r_n n}) \to (y, 1)] + c_n, \qquad \text{ where } \\
&a_n = -\frac{\floor{\tilde r_n m_n}}{\tau_n} \quad \mathand \quad c_n = \frac{{-\al_n \floor{\tilde r_n n} + \be_n \floor{\tilde r_n m_n}}}{\chi_n}.
\end{split}
\end{equation}
We will check that $Z^n$ satisfies the conditions of Theorem \ref{T:airy-sheet-gen}. {The convergence in \eqref{E:Arn-k} to finite values for every $y$ forces $m_n/\tau_n\to\infty$, so $a_n\to-\infty$.} First note that $Z^n$ is a cadlag line environment with positive jumps since $X_n$ was supported on $[0, \infty)$. Next, letting $A^{\tilde r_n}_{n, k}$ be as in \eqref{E:Arn-k}, we have
\begin{align}
\label{E:AY-trans}
{A^{\tilde r_n}_{n,k}}(y) = Z^n[(a_n, \floor{\tilde r_n n}) \to_{\Delta_k} (y, 1)] + c_n.
\end{align}
Here the $\to_{\Delta k}$ notation for line models is as in Theorem \ref{T:airy-sheet-gen}. This requires the fact that multi-point last passage values in a lattice environment are equivalent to multi-point last passage values in the corresponding line environment, Proposition \ref{P:melon-cor}.

{The increments of $Z^n$ are invariant under translations by $z\in\tau_n^{-1}\Z$. Rounding a starting coordinate upwards to this grid changes only the linear centering, by at most $|\be_n|/\chi_n$. Thus, for each fixed $z\in\R$,}
\begin{align*}
{Z^n[(a_n+z,\floor{\tilde r_n n})\to_{\Delta k}(z+y,1)]+c_n}
{\;\eqd\; A^{\tilde r_n}_{n,k}(y+\epsilon_n(z))+O(|\be_n|/\chi_n),}
\end{align*}
{jointly in $y,k$, where $|\epsilon_n(z)|\le\tau_n^{-1}$. The error is deterministic and bounded uniformly in $y$. Both errors tend to zero by \eqref{E:scaling}, so \eqref{E:Arn-k} gives the required Airy line ensemble convergence. Rotating the i.i.d.\ array by $\pi$ gives the same conclusion for the reversed last passage values, with the same rounding errors. Finally, taking the right limit in the starting coordinate changes its rounded value by at most one grid step. Monotonicity of the uncentered last passage values and continuity of the limiting ensemble therefore give the same convergence for these right limits.}

Therefore by Theorem \ref{T:airy-sheet-gen} and \eqref{E:Sndef}, the sequences ${S_n(\cdot^+,\cdot)}$ can be coupled to an Airy sheet ${\scrS_{r^{1/3}}}$ of scale $r^{1/3}$ so that $S_n \to {\scrS_{r^{1/3}}}$ compactly a.s.\ {The same monotonicity and one-grid-step comparison gives this convergence for $S_n$ itself.} Therefore $\mathfrak{g} S_n \cvgd \mathfrak{g}{\scrS_{r^{1/3}}}$, as desired.

\textbf{Step 5: The extension on the lattice.} \qquad 
 Finally, we extend \eqref{E:Lnnn} to the case when $ns_n \in \Z$ for some $n$. {At a lattice starting point, the line last passage value includes the initial weight, whereas the directed metric omits it.} When 
 $(\tau_n x + m_n s_n, -ns_n) \in \Z^2$, then the right hand side of \eqref{E:split-Ln} is equal to
$$
\scrL_n(x, s_n; y, s_n + r_n) + \chi_n^{-1} X_n(\tau_n x + m_n s_n, - n s_n).
$$
Therefore since all entries in the environment $X_n$ are nonnegative, the right hand side of \eqref{E:split-Ln} is always an upper bound for $\scrL_n(x, s_n; y, s_n + r_n)$ and a lower bound for
\begin{align*}
&\scrL_n(x-\tau_n^{-1}, s_n; y, s_n + r_n) + \chi_n^{-1}\lf(g_n(x, s_n) - g_n(x - \tau_n^{-1}, s_n)\rg) \\
&\qquad = \scrL_n(x-\tau_n^{-1}, s_n; y, s_n + r_n) + \chi_n^{-1}\be_n.
\end{align*}
The equality above follows from writing $g_n$ in terms of $\bar h_n$, and then evaluating $\bar h_n$. The term $\chi_n^{-1} \be_n \to 0$ and the shift $\tau_n^{-1} \to 0$ by the first and third conditions in \eqref{E:scaling}. 
Therefore the sequence ${\scrL_n}(x, s_n; y, s_n + r_n)$ is stochastically dominated above and below by sequences of random functions whose graphs converge to $\mathfrak{g}{\scrS_{r^{1/3}}}$, and so $\mathfrak{g}{\scrL_n}(x, s_n; y, s_n + r_n) {\cvgd} \mathfrak{g}{\scrS_{r^{1/3}}}$, as desired. 
\end{proof}

Theorem \ref{T:lattice-models} simplifies greatly if $X_n = X$ for all ${n}$ and $m_n = \floor{\rho n}$ for all $n$ for some fixed constant $\rho$. In this case, the sequences $\chi_n, \tau_n, \al_n, \be_n$ can be given by $\chi n^{1/3}, \tau n^{2/3}, \al, \beta$ for fixed numbers $\chi, \tau, \al, \beta >0$. The $n^{1/3}, n^{2/3}$ scaling for $\chi_n$ and $\tau_n$ is forced upon us by the $1-2-3$ scale invariance of the directed landscape. Note that when the sequences $X_n$ change with $n$ and $m_n$ is not proportional to $n$, this no longer needs to be the case. We will see an example of this in Section \ref{S:integrable} for sequences of environments of i.i.d.\ geometric random variables. We also address graph convergence and geodesic convergence everywhere in this simpler setting.

\begin{theorem}
	\label{T:lattice-models-i}
	Let $X = \{X(u) : u \in \Z^2\}$ be a collection of i.i.d.\ random variables in $[0, \infty)$ and let $\rho > 0$. Suppose that there exists $\al, \be, \chi, \tau > 0$ such that
	$$
	A^n_k(y) = \chi^{-1} n^{-1/3} \lf(X[(0, n) \to_{\Delta k} (\floor{\rho n} + \lfloor \tau n^{2/3} y \rfloor, 1)] - (\al n + \be \tau n^{2/3} y  ) \rg)
	$$ 
	converge jointly in distribution as $n \to \infty$, in the product-of-Skorokhod topologies on functions from $\R \to \R$, to the Airy line ensemble. 
	Let $h$ be the additive metric on $\Z^2 \cup E$ defined from the function $h:\Z^2 \to \R$ with
	\begin{align*}
	h(\rho t + x, -t) &= \al t + \be x.
	\end{align*}
	Set $h(e_1, e_2) = {h(e_1)}$ for all edges.
	Let $\scrK_n$ be the pullback onto $\R^2$ of the metric
	$
	\chi^{-1} n^{-1/3}(d_X {+} h)
	$
	under the map 
	$
	f_{\R^2 \to \Z^2} \circ L_n :\R^2 \to \Z^2 \cup E,
	$
	where $L_n:\R^2 \to \R^2$ is the linear map with matrix
	$$
	A = \lf[\begin{array}{cc}
	\tau n^{2/3} & \rho n \\
	0 & n 
	\end{array}\rg].
	$$
	Then for any finite or countable union $S = \bigcup S_{s_i, t_i}$, there is a coupling of $\scrK_n$ and $\scrL$ satisfying multi-time convergence on $S$, hypograph convergence, and geodesic convergence off holes.
	If, in addition, for some $\de > 0$ we have that 
	\begin{equation}
	\label{E:liminf-5mom}
	\limsup_{n \to \infty} \sup_{y \in [-3, 3]} {\expt [((A^n_1(y))^-)^{5 + \de}]} < \infty,
	\end{equation}
	then there is a coupling of $\scrK_n$ and $\scrL$ satisfying graph convergence and geodesic convergence everywhere.
\end{theorem}

Note that the constant $3$ in \eqref{E:liminf-5mom} is arbitrary; a variant of the same proof would work with $3$ replaced by some $\ep > 0$. However, in practice we do not expect that there is any difference between the verifiability of the two conditions.

{Under the corresponding Airy line ensemble convergence assumption, we expect that $\expt [(X(0,0)^-)^{5 + \de}]<\infty$ for some $\de>0$ is sufficient for graph convergence for general i.i.d.\ weights in $\R$.} {Graph convergence in distribution to a limit finite and continuous on $\Rd$ requires
$$
\lim_{r\to\infty} r^5\p\bigl(X(0,0)^->r\bigr)=0.
$$
Indeed, a fixed rescaled spacetime box contains order $n^{5/3}$ lattice vertices. Moving the starting point from an incoming edge to its incident vertex changes the passage value by that vertex weight. Continuity therefore requires the largest negative weight in the box, divided by $n^{1/3}$, to tend to zero in probability. By independence, this is equivalent to the displayed condition.}

Just as with the proof of Theorem \ref{T:lattice-models}, the proof of Theorem \ref{T:lattice-models-i} is a straightforward but tedious check on the assumptions of convergence theorems in previous sections.

\begin{proof}
	It is straightforward to check that $X$ satisfies the Airy line ensemble convergence condition in Theorem \ref{T:lattice-models}. This follows since we are using the same environment for all $n$, and since the scaling factors satisfy the $1-2-3$ KPZ scaling. Moreover, all of the scaling conditions in \eqref{E:scaling} are immediate, so Theorem \ref{T:lattice-models} applies.
	
	Therefore we just need to check graph convergence and geodesic convergence everywhere. For this, it suffices to show that we can couple $\scrK_n, \scrL$ so that $\mathfrak{g} \scrK_n \to \mathfrak{g} \scrL$. 
	If this is the case, then for any sequence $(u_n, v_n) \to (u, v) \in \Rd$, we have $\scrK_n(u_n, v_n) \to \scrL(u, v)$. Since all geodesics in $\scrK_n$ are connected, we can apply Theorem \ref{T:geod-cvg} to conclude that geodesic convergence everywhere holds in this coupling.
	
	Rather than showing graph convergence of $\scrK_n$ to $\scrL$, we will show graph convergence of $\tilde \scrK_n$ to $\scrL$, where $\tilde \scrK_n$ is the pullback of $\scrK_n$ under a map $g_n:\R^2 \to \R^2$ such that
	\begin{itemize}[nosep]
		\item  $g_n(u) = u + \xi_n$ for a sequence $\xi_n \to 0$,
		\item {$L_ng_nG\cap((\Z\X\R)\cup(\R\X\Z))=\emptyset$ for all $n$, where $G=\bigcup_{k\in\N}(2^{-2k/3}\Z\X2^{-k}\Z)$}.
	\end{itemize} 
{Such translations exist since the excluded shifts form a countable union of lines. Write $\xi_n=(\xi_{n,1},\xi_{n,2})$.} The first condition above guarantees that the conclusion of Theorem \ref{T:lattice-models} holds for $\tilde \scrK_n$, and that graph convergence of $\tilde \scrK_n$ is equivalent to graph convergence of $\scrK_n$. The second condition will allow us to avoid lattice difficulties when applying the results of Section \ref{S:graph-cvg}.
	
	By Corollary \ref{C:Kn}, it is enough to check the conditions of Proposition \ref{P:prob-cond} for $\tilde \scrK_n$ for any fixed compact set $A \X B \sset \Rd$.
	
	Let $k(n) = \ceil{3\log_2 (n)/4}$. By the bounds in \eqref{E:p-separation} and \eqref{E:p-cone} relating $p$ and $p_{k(n)}^+$, we can check that 
	$$
	L_n(p_{k(n)}^+) - L_n(p) = (\Theta(n^{1/4}) + O(n^{1/6}), \Theta(n^{1/4}))
	$$
	for all $p \in \R^2$. Therefore for large enough $n$, ${\tilde\scrK_n}(p, p_{k(n)}^+)$ is finite for all $p$. Moreover, since $X(u)$ is supported on $[0, \infty)$, ${\tilde\scrK_n}(p, p_{k(n)}^+)$ is bounded below by
	\begin{equation}
	\label{E:detlower}
-{\chi^{-1}} n^{-1/3} \lf({h} \circ f_{\R^2 \to \Z^2} \circ L_n{\circ g_n}(p_{k(n)}^+) - {h} \circ f_{\R^2 \to \Z^2} \circ L_n{\circ g_n}(p) \rg) = -{\chi^{-1}} n^{-1/3} \Theta(n^{1/4}). 
\end{equation}
A similar deterministic lower bound holds on ${\tilde\scrK_n}(p_{k(n)}^-, p)$. The right hand side of \eqref{E:detlower} tends to $0$ with $n$, giving the first condition of Proposition \ref{P:prob-cond}.

{Write $\hat p_{i,j}^+=p_{i,j}^++\xi_{n,j}$.} Now, for all $i \le k(n) - 1$ and $p \in \R^2$, {by the choice of $g_n$}, we can explicitly write out ${\tilde\scrK_n}(p_{i+1}^+, p_i^+)$ as in \eqref{E:split-Ln}:
\begin{equation}
\label{E:XXKn}
\begin{split}
&{\tilde\scrK_n}(p_{i+1}^+, p_i^+) \\
=\; &{\chi^{-1}} n^{-1/3} \bigg( {X}[(\ceil{\tau n^{2/3}{\hat p_{i+1,1}^+}+\rho n{\hat p_{i+1,2}^+}}, \floor{-n{\hat p_{i+1,2}^+}}) \to{}
\\ &\hspace{2em}(\floor{\tau n^{2/3}{\hat p_{i,1}^+}+\rho n{\hat p_{i,2}^+}}, \ceil{-n{\hat p_{i,2}^+}})] \\
- \; &\lf({h} \circ f_{\R^2 \to \Z^2} \circ L_n{\circ g_n}(p_i^+) - {h} \circ f_{\R^2 \to \Z^2} \circ L_n{\circ g_n}(p_{i+1}^+) \rg) \bigg).
\end{split}
\end{equation}
Now let
\begin{equation}
\label{E:pin}
p_{i, n} = \frac{1}n \lf(\floor{-n{\hat p_{i+1,2}^+}} - \ceil{-n{\hat p_{i,2}^+}} + 1 \rg) \qquad \mathand \qquad q_{i, n} = \frac{p_{i, 1}^+ - {p_{i+1,1}^+}}{p_{i, n}^{2/3}}.
\end{equation}
By translation invariance of ${X}$, \eqref{E:XXKn} is equal in distribution to
\begin{equation}
\label{E:piiii}
\begin{split}
&{\chi^{-1}} n^{-1/3} \bigg( X[(0, n p_{i, n}) \to ( \floor{\rho n p_{i, n}} + \floor{\tau (n p_{i, n})^{2/3} (q_{i, n} + \ep_n)}, 1)] \\
-\;&\al n p_{i, n} {-} \be \tau (n p_{i, n})^{2/3}(q_{i, n} + \ep_n) + \de_n \bigg),
\end{split}
\end{equation}
where $\ep_n$ and $\de_n$ are error terms satisfying ${|\ep_n (n p_{i, n})^{2/3}|}, |\de_n| \le c$ for some constant $c$ that is independent of the initial choice of $p$. We can then rewrite \eqref{E:piiii} in terms of ${A}^{np_{i, n}}_1$ to get that 
	\begin{equation}
	\label{E:KnA1-explicit}
	\begin{split}
	{\tilde\scrK_n}(p_{i+1}^+, p_i^+) &\eqd O(n^{-1/3}) + p_{i, n}^{1/3}{A}_{1}^{np_{i, n}}\lf(q_{i, n} + \ep_n\rg).
	\end{split}
	\end{equation}
	By \eqref{E:p-cone} and the error bound on $\ep_n$, we have that $q_{i, n} + \ep_n \in [-3,3]$ for large enough $n$ and $i \le k(n)$. Moreover, from the bound in \eqref{E:p-separation} we have that $p_{i, 2}^+ - p_{i+1, 2}^+ \ge {2^{-i-1}}$ for all $i$. Therefore for any $n \in \N, p \in \R^2, i \le k(n)$, we have that $n p_{i, n} \ge c n^{1/4} - 1$ for some universal $c > 0$. These two facts will allow us to bound \eqref{E:KnA1-explicit} via \eqref{E:liminf-5mom} and Markov's inequality.
	
	More precisely, there exists $c', d > 0$ and $n_0 \in \N$ such that for all $n \ge n_0, \ep \ge c' n^{-1/3}, p \in \R^2,$ and $i \le k(n)$, we have
	\begin{equation}
		\label{E:be-ready}
	\begin{split}
	\p\lf({\tilde\scrK_n}(p_{i+1}^+, p_i^+) \le -\ep \rg) &\le \sup_{y \in [-3, 3]} \p( {A}^{n p_{i, n}}_1(y) \le - p_{i, n}^{-1/3}\ep/2) \\
	&\le d \ep^{-(5 + \de)} p_{i, n}^{(5 + \de)/3} \le  d \ep^{-(5 + \de)} {2}^{-(5 + \de)(i-1)/3}.
	\end{split}
	\end{equation}
	The third inequality follows from the bound in \eqref{E:p-separation}, which ensures that $p_{i, 2}^+ - p_{i+1, 2}^+ \le 2^{-i + 1}$ for all $i$. The same bound holds for ${\tilde\scrK_n}(q_i^-, q_{i+1}^-)$ for all $q \in \R^2$. Given \eqref{E:be-ready}, we can satisfy the second condition of Proposition \ref{P:prob-cond} for all $\ep > 0$ by taking $\be_{i, n}(\ep) = i^{-2}$ for all $n \ge n_0, i \ge i_0$ for some $\ep$-dependent natural numbers $n_0, i_0$.
\end{proof}

\begin{remark}
	\label{R:intro-cct}
	 Theorem \ref{T:intro-compact} can be extracted from the proof of Theorem \ref{T:lattice-models-i}. {For a fixed i.i.d.\ environment with $1$-$2$-$3$ scaling, translation invariance and rescaling $n$ to $nr_n+O(1)$ for a time increment $r_n\to r>0$ reduce mobile Tracy-Widom convergence to multi-time convergence on $S_{0,1}$. The rounding is handled as in Steps 3 and 5 of the proof of Theorem \ref{T:lattice-models}. Theorem \ref{T:ind-main} and Remark \ref{R:ind-main} then give the joint hypograph and multi-time convergence required by Corollary \ref{C:Kn}; the remaining chaining argument above uses only the moment bound.}
\end{remark}

All statements about convergence of last passage metrics to the directed landscape can be converted into statements that do not use the directed metric framework or the topological language of Theorems \ref{T:lattice-models}, \ref{T:lattice-models-i}. The statements become cleaner in this context, since we need not be concerned with preserving the metric structure of the prelimit when we embed. We demonstrate this here by translating the conclusion of Theorem \ref{T:lattice-models-i}. We have simplified the conclusions slightly here. We also include a statement about joint convergence to the Airy line ensemble here that will be of use in Section \ref{S:integrable}.

\begin{corollary}
	\label{C:Xiidmaster}
Let $X$ be an i.i.d.\ environment of nonnegative random variables and let $\rho, \chi, \tau, \al, \beta$ be positive numbers. Let $(x, s)_n = (\floor{\rho n s + \tau n^{2/3} x}, \floor{-ns})$, and define
\begin{equation*}
\scrL_{n, \Delta k} (x, s; y, t) = \frac{{X}[(x, s)_n \to_{\Delta k} (y, t)_n] - {\al}(t-s)n - {\be}\tau(y - x) n^{2/3}}{\chi n^{1/3}},
\end{equation*}
and write $\scrL_n := \scrL_{n, \Delta 1}$.
Suppose that $\{\scrL_{n, \Delta k}(0, 0; y, 1) : (y, k) \in {\R \X \N}\}$ converges in distribution in the product-of-Skorokhod topologies to the Airy line ensemble. Then $\scrL_{n}$ can be coupled with the directed landscape $\scrL$ so that on a set $\Om$ of probability $1$, the following hold:
$$
\scrL_n(\cdot, s; \cdot, t) \to \scrL (\cdot, s; \cdot, t)
$$
compactly for all rational pairs $s, t$, and for any $u_n \to u \in \Rh$ we have
$$
\limsup_{n \to \infty} \scrL_n(u_n) \le \scrL(u).
$$
Now fix $(x, s; y, t) \in \Rd$, let $\pi_n$ be any sequence of last passage geodesics from $(x, s)_n$ to $(y, t)_n$, and let $\Om' \sset \Om$ be the set of probability one where $\scrL$ has a unique geodesic $\pi$ from $(x, s)$ to $(y, t)$. Then on $\Om'$, the sequence of rescaled geodesics
$$
\tilde \pi_n := \lf[\begin{array}{cc}
\tau^{-1} n^{-2/3} & {+}\rho \tau^{-1} n^{-2/3} \\
0 & -n^{-1} 
\end{array}\rg] \pi_n
$$
converges to $\pi$ in the Hausdorff topology. For this convergence, we view the geodesics $\pi_n$ as subsets of $\Z^2$, which are then in turn subsets of $\R^2$. Moreover, we can additionally couple $A^n_k(\cdot) = \scrL_{n, \Delta k}(0,0; \cdot, 1)$ so that $A^n_k(\cdot)$ converges to an Airy line ensemble $\scrA$, and $\scrA$ and the Airy sheet $\scrL(\cdot, 0; \cdot, 1)$ are coupled as in Remark \ref{R:airy-sheet-coup}.
 Finally, if in addition there exists $\de > 0$ such that
\begin{equation}
\label{E:Lnkexpt}
\limsup_{n \to \infty} \sup_{y \in [-3,3]} {\expt[(\scrL_{n}(0, 0; y, 1)^-)^{5 + \de}]} <\infty, 
\end{equation}
then the coupling can be set up so that $\scrL_n \to \scrL$ compactly on $\Rd$ on $\Om$.
\end{corollary}

\begin{proof}
This is essentially just a rephrasing of Theorem \ref{T:lattice-models-i}. 
Note that there is only a slight difference between the definitions of $\scrL_{n, \Delta k}(0, 0; y, 1)$ and $A^n_k(y)$ in Theorem \ref{T:lattice-models-i}, and between $\scrL_{n}$ and $\scrK_n$ from Theorem \ref{T:lattice-models-i}. In particular, we have that
$$
{A}^n_k(y) = \scrL_{n, {\Delta k}}(0, 0; y + O(n^{-2/3}), 1) + O(n^{-1/3}),
$$
so convergence of $\scrL_{n, \Delta k}(0, 0; y, 1)$ to the Airy line ensemble $\scrA$ implies convergence of ${A}^n_k(y)$ to the Airy line ensemble and \eqref{E:Lnkexpt} implies \eqref{E:liminf-5mom}. Also,
\begin{equation}
\label{E:LK-equivalence}
\scrL_{n}(u) = \scrK_n(u + O(n^{-2/3})) + O(n^{-1/3}),
\end{equation}
so all convergence statements about $\scrK_n$ pass over to convergence statements about $\scrL_{n, {\Delta 1}}$. The fact that we can couple the limiting Airy line ensemble $\scrA$ with the Airy sheet $\scrL(0, 0; \cdot, 1)$ follows from Remark \ref{R:airy-sheet-coup}.
Finally, it is straightforward to check that each of the geodesics $\tilde \pi_n$ is Hausdorff distance $O(n^{-2/3})$ from a geodesic in $\scrK_n$ from $(x,s)$ to $(y, t)$. The geodesic convergence in the corollary then follows from geodesic convergence off holes for $\scrK_n$.
\end{proof}

\begin{remark}
	\label{R:Xn-notin}
The results of this section only apply to nonnegative random variables as written. This is due to the technical constraints on the combinatorics of cadlag last passage percolation from Section \ref{S:lpp-cadlag}. Though we do not pursue this here, the results of this section {could be modified to include} more general random variables whose support is potentially all of $\R$. 

For example, for an i.i.d.\ environment of general random variables $X_u \in \R$ satisfying all the assumptions of Theorem \ref{T:lattice-models-i} except nonnegativity, for each $n$ we could replace the environment $X = \{X_u : u \in {\Z^2}\}$ with a sequence of nonnegative environments $X^n = \{(X_u + c_n)^+ : u \in {\Z^2}\}$ for a sequence of constants $c_n$ satisfying $\p(X_u + c_n \le 0) = o(n^{-2})$. {On a box of size} $O(n) \X O(n)$, the two environments are equal with high probability up to a constant shift. Hence by shifting the second environment by shifted versions of original function $h$, we could still ensure the Airy line ensemble convergence in Theorem \ref{T:lattice-models}.
\end{remark}

\subsection{Exponential tightness}
\label{SS:exponential-tight}
If we strengthen certain tail bound assumptions in the setup of Theorem \ref{T:lattice-models-i} and Corollary \ref{C:Xiidmaster}, then we can move from distributional convergence to convergence of moments. 
We give an example of how to do this in the setting of Theorem \ref{T:lattice-models-i} with additional exponential control on one-point distributions.

\begin{theorem}[Upper exponential tightness]
\label{T:upper-exponential}
Consider the setup of Theorem \ref{T:lattice-models-i} without the additional assumption \eqref{E:liminf-5mom}. {Suppose that there exists $\mathfrak p>0$ such that for every compact set $K\sset\Rd$ there exists $a>1$ with}
\begin{equation}
\label{E:LS-infty}
\limsup_{n \to \infty} \sup_{u \in K} \expt a^{|\scrK_{n}(u)|^\mathfrak{p}} <\infty.
\end{equation}
{Then for every compact $K\subset\Rd$ there exist $b>1$ and $c>0$ such that for all large enough $n$} 
we have 
\[
\expt b^{\sup_K (\scrK_{n}^+)^\mathfrak{p}}<c.
\]
\end{theorem}
\begin{proof}
It suffices to prove this for compact product sets $A \X B \sset \Rd$. 
Fix points $p, q$ with $\{p\} \X A, B \X \{q\} \sset \Rd$. Let $m > 0$. We apply Corollary \ref{C:limit-setting} to $\scrK_n$ with $c=m$, $\ep = 2m/3$, and the given points $p, q$. This gives 
\begin{align*}
    \p\lf(m < \sup_{u \in A \X B} \scrK_n(u) \rg)
    &\le \p({\scrK_n}(p, q) > m/3) + \sup_{a \in A} \p(\scrK_n(p, a) < -m/3) \\
    &\qquad + \sup_{b \in B} \p \lf( \scrK_n(b, q) < -m/3\rg).
\end{align*}
By Markov's inequality and assumption \eqref{E:LS-infty}, there exists $c > 0$ such that for all large enough $n$, the right hand side of this expression is bounded above by
$
c a^{- m^\mathfrak{p}/3^\mathfrak{p}},
$
yielding the result for every $1 < b < a^{1/3^\mathfrak{p}}$.
\end{proof}

\begin{theorem}[Lower exponential tightness]
\label{T:lower-exponential}
In the setup of Theorem \ref{T:lattice-models-i}, {suppose additionally that there exists $\mathfrak l>0$ such that for every compact set $K\sset\Rd$ there exists $a>1$ with}
\begin{equation}
\label{E:LSn-infty}
\limsup_{n \to \infty} \sup_{u \in K} \expt a^{({\scrK}_{n}(u)^-)^\mathfrak{l}} <\infty.
\end{equation}
{Then for every compact $K\subset\Rd$ there exist $b>1$ and $c>0$ such that for all large enough $n$} 
we have 
\[
\expt b^{\sup_K ({\scrK}_{n}^-)^\mathfrak{l}}<c.
\]
\end{theorem}

\begin{proof}
As in the proof of Theorem \ref{T:upper-exponential}, we let our compact set $K$ be of the form $A \X B \sset \Rd$.
We will use the notation of Proposition \ref{P:prob-cond}. Our aim is to appeal to Corollary \ref{C:true-bound-for-holes}. 

We follow the proof of Theorem \ref{T:lattice-models-i}. {As there, we first translate the metric to avoid lattice boundaries and suppress the tilde. Bounds for the original metric follow by slightly enlarging the compact set.} With the choice $k(n) = \floor{3 \log_2(n)/4}$ as in that proof, condition \eqref{E:detlower} shows that for large enough $n$,
\begin{equation}
\label{E:eta-kn}
\p(\eta_{k(n)}^+(\scrK_n, A) < -1) + \p(\eta_{k(n)}^-(\scrK_n, B) < -1) = 0.
\end{equation}
Next, recalling \eqref{E:be-ready}, there exists $c' > 0$ and $n_0 \in \N$ such that for all $n \ge n_0, \ep \ge c' n^{-1/3}, p \in \R^2,$ and $i \le k(n)$, we have
\begin{equation}
		\label{E:be-ready-2}
	\p\lf(\scrK_n(p_{i+1}^+, p_i^+) \le -\ep \rg) \le \sup_{y \in [-3, 3]} \p( {A}^{n p_{i, n}}_1(y) \le - p_{i, n}^{-1/3}\ep/2).
	\end{equation}
As in the proof of Theorem \ref{T:lattice-models-i}, the quantity $n p_{i, n} \to \infty$ uniformly over $p \in \R^2$ and $i \le k(n)$, so by Markov's inequality and the assumptions of the theorem, there exists $d, c_\mathfrak{l} > 0$ such that for all $n \ge n_0, \ep \ge c' n^{-1/3}, p \in \R^2,$ and $i \le k(n)$, the {right-hand side} of \eqref{E:be-ready-2} is bounded above by
$$
d \exp ( -\log (a) {2^{-\mathfrak{l}}}p_{i, n}^{-\mathfrak{l}/3}\ep^\mathfrak{l}) \le  d \exp ( -c_\mathfrak{l} 2^{\mathfrak{l} i/3}\ep^\mathfrak{l}).
$$
The same bound holds in the same parameter range for $	\p\lf(\scrK_n(p_{i}^-, p_{i+1}^-) \le -\ep \rg)$. Therefore the assumptions of Corollary \ref{C:true-bound-for-holes} hold with 
$\be_{i, n}(\ep) = d \exp ( -c_\mathfrak{l} 2^{\mathfrak{l} i/3}\ep^\mathfrak{l} i^{-2\mathfrak{l}}) 2^{5 i /3}$ for all $n \ge n_0$ and $\ep \ge c' n^{-1/3} k(n)^2$. In particular, as long as $n_0$ is large enough we can use this value of $\be_{i, n}(\ep)$ for all $\ep > 1$.

Next, let $m \ge 2$ be the smallest integer so that $A_m^+ \X B_m^- \sset \Rd$ and let $b > {8}$. Then with notation as in Section \ref{S:graph-cvg},
\begin{equation}
\label{E:sup-double-bd}
\p( \sup_K \scrK^-_n > b) \le \p\lf( \sup_{A_m^+ \X B_m^-} \scrK^-_n > b/2\rg) + \p\lf( \inf_K \scrK_n - \inf_{A_m^+ \X B_m^-} \scrK_n < - b/2 \rg)
\end{equation}
By \eqref{E:LSn-infty}, a union bound, and Markov's inequality,  the first probability on the right side of \eqref{E:sup-double-bd} is bounded above by $c_{K} a^{-b^\mathfrak{l}/2^\mathfrak{l}}$, where $c_K > 0$ is a $K$-dependent constant. By Corollary \ref{C:true-bound-for-holes}, the second term in \eqref{E:sup-double-bd} is bounded above by $c_K \sum_{i=m}^{k(n)} \be_{i, n}({b/8})$. By the bounds on $\be_{i, n}$, this is bounded above by $c_K \exp (-c_\mathfrak{l} b^\mathfrak{l})$, and so \eqref{E:sup-double-bd} is also bounded above by $c_K \exp (-c_\mathfrak{l} b^\mathfrak{l})$, yielding the result.
\end{proof}

\section{Convergence of integrable last passage models}
\label{S:integrable}

We now use the results of the previous sections to prove convergence of specific models of last passage percolation to the Airy sheet and the directed landscape. This section will cover the proofs of Theorems \ref{T:intro-main} and \ref{T:intro-geod}. The language of convergence in our main theorems here is slightly different than in the introduction, {we give the translation in the proof of Theorems \ref{T:intro-main} and \ref{T:intro-geod} below}.

\subsection{Integrable zoo}
\label{S:zoo}

We first prove multi-time convergence, hypograph convergence, and geodesic convergence off holes for geometric last passage percolation in all feasible parameter directions by checking the conditions of Theorem \ref{T:lattice-models}. We then appeal to coupling arguments to extend this to other integrable last passage models.

\subsubsection{Geometric environment}

\begin{theorem}
	\label{T:geom-lpp}
	For each $\ga \in (0, \infty)$, let $X_{\ga} = \{X_{\ga}(u) : u \in \Z^2\}$ be an array of i.i.d.\ geometric random variables on $\{0, 1, 2, \dots\}$ of mean $\ga$. Let $n, m_n, \ga_n$ be any sequences satisfying
	\begin{equation*}
n \to \infty, \qquad m_n\to\infty, \qquad n m_n \ga_n\to\infty.
	\end{equation*}
	Then there exist parameters $\chi_n, \tau_n, \al_n, \be_n$ such that with the sequence of environments $\scrK_n$ defined as in Theorem \ref{T:lattice-models} from the i.i.d.\ environments $X_{\ga_n}$, the following holds: for any finite or countable union $S = \bigcup S_{s_i, t_i}$, there is a coupling of $\scrK_n$ and $\scrL$ satisfying multi-time convergence on $S$, hypograph convergence, and geodesic convergence off holes. See the beginning of Section \ref{S:cvg-lattice} for definitions.
	
Define $\chi_n, \tau_n, \al_n, \be_n$ as follows. For each $n,m\in \mathbb{N}$ and $\ga \in (0,\infty)$, define the arctic curve:
\begin{equation}\label{E:arctic1}
g_{n,\ga}(m) = (m + n)\ga + 2 \sqrt{mn \ga(1 + \ga)},
\end{equation}
which is the deterministic approximation of the last passage value $X_\ga[(0,n) \to (m, 1)]$.
	We define all scaling parameters in terms of the value of the arctic curve $g=g_{n,\ga_n}$ and its derivatives $g',g''$ evaluated at $m_n$:
	\begin{equation}
	\label{E:tau-xi-geom}
n\al_n = g, \qquad \be_n = g', \qquad \tau_n^3 = \frac{2 g'(1 + g')}{g''^2}, \qquad \chi_n^3 = \frac{ (g'(1+g'))^{2}}{ -2g''}.
	\end{equation}
\end{theorem}

In this and all remaining theorems in this section, we have matched our theorem presentation above with that of \cite*{DNV}, so that the interested reader can easily refer back to the results we use in that paper. One can easily check that setting $\rho = m_n/n$ in all theorems, that the scaling parameters match those in the table prior to Theorem \ref{T:intro-main}. 

The scaling assumptions in Theorem \ref{T:geom-lpp} are optimal. The condition that $n m_n \ga_n \to \infty$ ensures that the mean total weight of the random variables in the box $[0, m_n] \X [1, n]$ tends to infinity. If this is not the case, then $X_n[(0, n) \to (m_n, 1)]$ could not rescale to a continuous random variable whose support is all of $\R$.

\begin{proof}
We first check the four scaling conditions in \eqref{E:scaling}. For this, we drop the subscripts $n$ in order to avoid cluttered notation. Because the scaling relationships in \eqref{E:tau-xi-geom} are expressed in terms of cubes, it will be easier to check that
 \begin{equation}
 \label{E:scaling2}
 \frac{\be^3}{\chi^3} \to 0, \qquad \frac{|\be m - \al n|^3}{(n \chi)^3} \to 0, \qquad \frac{1}{\tau^3} \to 0, \qquad \frac{m^3}{n^3 \tau^3} \to 0.
 \end{equation}
 Letting $\xi = \sqrt{\ga(1 + \ga)n/m}$, we can calculate that
 $$
 \frac{\be_n^3}{\chi_n^3} = \frac{(\ga + \xi)\xi}{(1 + \ga + \xi)^2 m} \le \frac{1}m.
 $$
 This tends to zero since $m \to \infty$ with $n$. Next, we calculate that
 \begin{align}
\nonumber
 \frac{|\be m - \al n|^3}{(n \chi)^3} &= \frac{(n \ga + \xi m)^3\xi}{n^3(1 + \ga + \xi)^2(\ga + \xi)^2 m} \\
  \label{E:beall}
 & \le \frac{4\ga^3\xi}{(1 + \ga + \xi)^2(\ga + \xi)^2 m} + \frac{4\xi^4 m^2}{n^3(1 + \ga + \xi)^2(\ga + \xi)^2}
 \end{align}
The first term in \eqref{E:beall} is always bounded above by $4/m$, which tends to $0$ as $n \to \infty$. The second term in \eqref{E:beall} is bounded above by
\begin{equation*}
\frac{4}n \frac{\xi^4 m^{2}}{n^{2} \ga^2(1 + \ga)^2} = \frac{4}n,
\end{equation*}
which tends to $0$ with $n$. Finally, we can write
\begin{equation*}
\frac{1}{\tau^3} = \frac{\xi^2}{8(1 + \ga + \xi){(\ga + \xi)}m^2} \le \frac{1}{8m^2}
\end{equation*}
and
\begin{equation*}
\frac{m^3}{n^3 \tau^3} = \frac{\xi^2 m}{8n^3 (1 + \ga + \xi){(\ga + \xi)}} \le \frac{\xi^2 m}{8\ga(1 + \ga)n^3} = \frac{1}{8n^2},
\end{equation*}
both of which tend to $0$ with $n$.
The assumption of Airy line ensemble convergence for each of the sequences \eqref{E:Arn-k} in Theorem \ref{T:lattice-models} is a translation of Theorem 1.1 from \cite*{DNV}. Note that the definition of last passage percolation in that paper follows lattice paths with edges in $\{(0, 1), (1, 0)\}$, rather than in $\{(0, -1), (1, 0)\}$. Also, in that paper Airy line ensemble convergence is stated with a slightly different deterministic term subtracted from the main last passage term. The difference between the two deterministic terms is $O(\be_n/\chi_n)$, which tends to $0$ with $n$ by \eqref{E:scaling2}.
\end{proof}

\cite*{DNV} used the fact that geometric last passage percolation converges to the parabolic Airy line ensemble in any feasible direction in the parameter space to pass to other integrable last passage models via coupling. These couplings are strong enough that we can immediately get versions of Theorem \ref{T:geom-lpp} for these models.

\subsubsection{Exponential environment}

We can pass from a sequence i.i.d.\ geometric environments to an i.i.d.\ exponential environment via the usual geometric-to-exponential limiting transition. The details are given in \cite*{DNV}, Corollary {6.2}, and are applicable for the present convergence as well.  

\begin{corollary}\label{C:exponential}
	Let $X$ be an environment of independent exponential random variables of mean $1$. For any $m,n$ define the arctic curve
	$$
	g_n(m) = n+m+2\sqrt{nm},
	$$
	which is the deterministic approximation of the last passage value $X[(0,n) \to (m, 1)]$ in this model.
	
	Let $m_n\to \infty$ be a sequence of natural numbers. Denoting by $g, g', g''$ the value of $g_n$ and its derivatives evaluated at $m_n$, we set the scaling parameters for this model:
	\begin{align}\nonumber
	\label{E:tau-xi-exponential}
	n\al_n = g, \quad \be_n =g', \quad \tau_n^3 &= \frac{2  g'^{2}}{g''^2} = \frac{8 m_n^2 (\sqrt{m_n} + \sqrt{n})^2}n, \quad &\chi_n^3 = \frac{ g'^{4}}{ -2g''} = \frac{(\sqrt{m_n} + \sqrt{n})^4}{\sqrt{m_n n}}.
	\end{align}
	Then with the sequence of environments $\scrK_n$, defined as in Theorem \ref{T:lattice-models}, the following holds: for any finite or countable union $S = \bigcup S_{s_i, t_i}$, there is a coupling of $\scrK_n$ and $\scrL$ satisfying multi-time convergence on $S$, hypograph convergence, and geodesic convergence off holes.
\end{corollary}

\subsubsection{Poisson lines environment}
\label{SS:Poisson lines}

There are also two degenerations of the geometric environment to line last passage environments. The first is a Poisson transition which keeps the discrete nature of the weights. The details of this transition are given in \cite*{DNV}, Corollary {6.5}. 

\begin{corollary} 
	\label{C:Poisson-lines} 
	Let $F = \{F_i: i \in \Z\}$ be a collection of {independent two-sided} rate-$1$ Poisson processes. Define the Poisson lines arctic curve
	$$
	g_{n}(t) = t+ 2 \sqrt{tn}.
	$$
	This is the deterministic approximation of the last passage value
	$F[(0, n) \to (t, 1)]$. Let $t_n$ be any sequence of {positive} numbers such that $n t_n \to \infty$.
	We define scaling parameters in terms of the arctic curve $g=g_{n}$ and its derivatives $g',g''$ taken in the variable $t$ at the value $t_n$:
	\begin{equation*}
	n\al_n = g, \qquad \be_n=g', \qquad \tau_n^3 = \frac{2 g'}{g''^2}, \qquad \chi_n^3 = \frac{ g'^{2}}{ -2g''}.
	\end{equation*}
Define $h_n:\R \X (\Z/2) \to \R$ by $h_n({t_n y/n} +x, {-y}) = {-\al_n y - \be_n x}$, and let $\scrK_n$ be the pullback onto $\R^2$ of the directed metric ${\chi_n^{-1}}(d_F - h_n)$ under the map $f_{\R^2 \to \R \X \Z} \circ L_n:\R^2 \to \R \X \Z$, where $L_n$ is the linear map with matrix
$$
A = \lf[\begin{array}{cc}
	\tau_n & {t_n} \\
	0 & n 
\end{array}\rg].
$$
Then for any finite or countable union $S = \bigcup S_{s_i, t_i}$, there is a coupling of $\scrK_n$ and $\scrL$ satisfying multi-time convergence on $S$, hypograph convergence, and geodesic convergence off holes.
\end{corollary}

\subsubsection{Brownian last passage percolation}
\label{SS:Brownian-LP}

The second degeneration from a lattice to a line environment yields Brownian last passage percolation. Convergence of Brownian last passage percolation to the directed landscape was proven in \cite*{DOV}. We include the result here for completeness. 

\begin{corollary}\label{C:Brownian}
	Let $\{B_i : i \in \Z\}$ be a collection of independent copies of a two-sided Brownian motion of variance $1$. Define the Brownian arctic curve
	$$
	g_{n}(t) = 2 \sqrt{tn},
	$$
	which is the deterministic approximation of the last passage value  $B[(0, n) \to (t, 1)]$. We define scaling parameters in terms of the arctic curve $g=g_{n}$ and its derivatives $g',g''$ taken in the variable $t$ at the value $1$:
	\begin{equation*}
	n\al_n = g, \qquad \be_n=g', \qquad \tau_n^3 = \frac{2}{g''^2}, \qquad \chi_n^3 = \frac{1}{ -2g''}.
	\end{equation*}
	Let $\scrK_n$ be defined from $d_B$ and $\al_n, \be_n, \tau_n, {\chi_n}$ as in Corollary \ref{C:Poisson-lines}{, with $t_n=1$}. Then {for any} finite or countable union $S = \bigcup S_{s_i, t_i}$, there is a coupling of $\scrK_n$ and $\scrL$ satisfying multi-time convergence on $S$, hypograph convergence, and geodesic convergence off holes.
\end{corollary}

\subsubsection{Poisson last passage percolation in the plane}

Finally, geometric last passage percolation degenerates to a planar Poisson last passage percolation as we decrease the mean of the geometric random variables and make the grid $\Z^2$ finer and finer. For this definition, we write $\R^4_\nearrow = \{(a, c) \in \R^4 : a_1 \le c_1, a_2 \le c_2\}$. 

\begin{definition}
\label{D:planar-poisson}
Let $P \sset \R^2$ be a collection of points such that $|P \cap K|$ is finite for any compact set $K \sset \R^2$. Let $d_P$ be the directed metric of negative sign induced on $\R^2$ by the function $f:\R^4_\nearrow \to \R$ given by
$$
f(a, c) = \begin{cases} 1, \qquad 
P \cap \lf([a_1, c_1] \X [a_2, c_2] \smin \{a\} \rg) \ne \emptyset \\
0, \qquad \text{otherwise}.
\end{cases} 
$$
\end{definition}

Observe that for $(a, c) \in {\R^4_\nearrow}$, $d_P(a, c) + |P \cap \{a\}|$ is a last passage value in the usual sense of Poisson last passage percolation. That is,
$$
d_P(a, c) + |P \cap \{a\}| = \sup_{\pi} |P \cap \pi|,
$$
where the supremum is over all continuous up-right paths from $a$ to $c$. The details of the transition from discrete geometric environments to a Poisson environment are given in \cite*{DNV}, Corollary {6.9}.

\begin{corollary}
	\label{C:Poisson-box}
	Let $P$ be a {rate-$1$} Poisson process on $\R^2$.
	To match with the main theorem, for all ${t>0}$ we define the arctic curve
	$$
	g_t(r) = 2\sqrt{rt},
	$$
	which is the deterministic approximation of $d_P(0,0; r, t)$. 
	We now define scaling parameters in terms of the arctic curve $g = g_t$ and its derivatives $g',g''$ taken in the variable $r$ at the point $r = t$:
	\begin{equation*}
	\tau_t^3 = \frac{2g'}{g''^2} = 8t^2, \qquad \chi_t^3 = \frac{g'^2}{ -2g''} = t.
	\end{equation*}
	Using these definitions, let $h:\R^2 \to \R$ be given by $h(x + y, y) = {-2y - x}$ and let $\scrK_t$ be the pullback of $\chi_t^{-1}(d_P - h)$ under the linear map $L_t:\R^2 \to \R^2$ with matrix
	$$
A_t = \lf[\begin{array}{cc}
	\tau_t & t \\
	0 & t 
\end{array}\rg].
	$$
	Then for any sequence $t_n \to \infty$ and any finite or countable union $S = \bigcup S_{s_i, t_i}$, there is a coupling of $\scrK_{t_n}$ and $\scrL$ satisfying multi-time convergence on $S$, hypograph convergence, and geodesic convergence off holes.
\end{corollary}

We have stated Corollary \ref{C:Poisson-box} for Poisson last passage percolation on square boxes of increasing area. Convergence results for more general rectangular boxes of increasing area follow from scale invariance of Poisson processes.

\subsubsection{The Sepp\"al\"ainen-Johansson model}
\label{SS:SJ}

The Sepp\"al\"ainen-Johansson model is the only model we consider that is not a degeneration of geometric last passage percolation.
We recall the setup of Example \ref{Ex:SJ}. Let $\Z^2$ be the lattice where $(x, y) \in \Z^2 \X \Z^2$ is a directed edge if $y-x \in \{(1, 0), (0, -1)\}$. Fix a parameter $p \in (0, 1)$, and let $d_{p}$ be the directed metric of positive sign on $\Z^2$ induced by the collection of independent weights $M^p$, where
\begin{itemize}[nosep]
    \item $M^p(x, y) = 0$ if $y - x = {(0, -1)}$,
    \item $M^p(x, y)$ is a Bernoulli random variable of mean $p$ if $y - x = (1, 0)$.
\end{itemize}
The Sepp\"al\"ainen-Johansson model $d_{p}$ can be reinterpreted as a line last passage metric. Given the weight function $M^p$, define functions $M^p_i:\Z \to \Z, i \in \Z$ by setting
$$
M^p_i(0) = 0, \qquad M^p_i(n) - M^p_i(m) = n - m - \sum_{j=m}^{n-1} M^p((j, i), (j+1, i)).
$$
Extend each $M^p_{i}$ to a function from $\R \to {\R}$ by letting it be linear on every interval $[n, n+1], n \in \Z$. Note that each function $M_i$ is an independent random walk with Bernoulli steps of mean $1- p$.
 Let $D_{1-p}$ be the line last passage metric associated to $M^p$.
By construction, for any $(x, n), (y, m) \in \Z^2$ we have
\begin{equation}
\label{E:Dp-dp-equivalence}
d_p((x, n), (y, m)) = (y - x) - D_{1-p}((x, n), (y, m)).
\end{equation}
We can use this formula to extend $d_p$ to a metric on all of $\R \X \Z/2$, the space where $D_{1-p}$ is naturally defined.

\begin{theorem}
	\label{T:seppalainen}
	For $p \in (0, 1)$ and $m \in \N$ set $\la = {p/(1-p)}$ and $\rho = m/n$.
Define the arctic curve
	$$
	g_{n, \la}(m) = \frac{(\sqrt{m \la} -\sqrt{n})^2}{1 + \la}\mathbf{1}(\rho > 1/\la),
	$$
	which is the deterministic approximation of the value $d_p((0, n), (m, 1))$. Consider sequences $p_n \in (0, 1), m_n \in \N$ such that $\rho_n >1/\la_n$ for all $n$. 
	We define scaling parameters in terms of $g =g_{n, \la_n}$ and its derivatives $g'$, $g''$ evaluated at the point $m_n$:
	\begin{equation*}
	n \al_n = g \qquad \be_n = g', \qquad \tau_n^3 = \frac{2g'(1 - g')}{(g'')^2}, \qquad \chi_n^3 = \frac{{[g'(1 - g')]^2}}{-2g''}.
	\end{equation*}
 Define $h_n:\R \X (\Z/2) \to \R$ by $h_n({\rho_n}y +x, -y) = {-\al_n y - \be_n x}$, and let $\scrK_n$ be the pullback onto $\R^2$ of the directed metric $\chi_n^{-1}(d_{p_n} - h_n)$ under the map $f_{\R^2 \to \R \X \Z} \circ L_n:\R^2 \to \R \X \Z$, where $L_n$ is the linear map with matrix
	$$
	A = \lf[\begin{array}{cc}
	\tau_n & {m_n} \\
	0 & n 
	\end{array}\rg].
	$$
	Note that unlike $d_{{p_n}}$, $\scrK_n$ is a directed metric of \textbf{negative sign} since $\chi_n < 0$.
	
	Now suppose that $\chi_n \to -\infty$ with $n$. Then for any finite or countable union $S = \bigcup S_{s_i, t_i}$, there is a coupling of $\scrK_n$ and $\scrL$ satisfying multi-time convergence on $S$, hypograph convergence, and geodesic convergence off holes. 
\end{theorem}

Note that the condition that $\chi_n \to-\infty$ is always satisfied if $\rho, \la$ are both fixed. The proof of Theorem \ref{T:seppalainen} is similar to the proof of the main theorem for general models, Theorem \ref{T:lattice-models}.

\begin{proof}
We can equivalently work with the last passage metrics $D_{1-p_n}$ defined by the formula \eqref{E:Dp-dp-equivalence}. {We use the maximal inequality directly on the independent Bernoulli edges, as explained below.}

We can alternately think of $\scrK_n$ as the pullback onto $\R^2$ of the directed metric $|\chi_n|^{-1}(D_{1-p_n} - H_n)$ under the map $f_{\R^2 \to \R \X \Z} \circ L_n$, where $H_n(a, b) = {-a} - h_n(a, b)$. We will use {the proof of Theorem \ref{T:ind-main} and Remark \ref{R:ind-main}, with a vanishing rounding error in the maximal inequality,} which will imply multi-time convergence on $S$ and hypograph convergence. {Geodesic convergence off holes then follows from the path comparison in Step 2 of the proof of Theorem \ref{T:lattice-models}. The required rounding error is $n^{-1}+\rho_n/\tau_n\to0$: the scaling formulas give $\rho_n/\tau_n\le n^{-2/3}+(n|\chi_n|)^{-1/2}$.}

In the line environments $M^{p_n}$, the multi-point last passage differences
$$
A^n_k(y) := |\chi_n|^{-1} \Big(M^{p_n}[(0, n) \to_{\Delta k} (\rho_n n + \tau_n y, 1)] - (\rho_n - \al_n) n - (1-\be_n) \tau_n y \Big)
$$ 
converge compactly to the Airy line ensemble $\scrA$ for any sequence of parameters $p_n, \rho_n$ such that the corresponding $\chi_n \to -\infty$. This is Corollary {6.11} in \cite*{DNV}. Therefore by Theorem \ref{T:airy-sheet-gen}, for any $t \in \R, s > 0$, the function $\scrK_n(\cdot, t; \cdot, t+s^3)$ converges in $\scrE_*(S_{t, t + s^3})$ to an Airy sheet of scale $s$. Moreover, Corollary {6.11} in \cite*{DNV} also implies that this same convergence holds if $t, s$ are replaced by convergent sequences $t_n \to t, s_n \to s > 0$. In particular, the sequence $\scrK_n$ also satisfies the mobile Tracy-Widom convergence condition in Theorem \ref{T:ind-main}, {so it remains to justify the maximal inequality used in its proof. Linear interpolation does not give independent increments on arbitrary real intervals. At lattice vertices, however, the proofs of Lemmas \ref{L:doob-lpp} and \ref{L:doob-lpp-boot} apply directly to the independent Bernoulli edges: the exploration reveals only edges preceding the stopping vertex. Rounding the two spatial endpoints to adjacent integers changes $D_{1-p_n}$ by at most $2$, since every interpolated line has slope in $[0,1]$. The centered, scaled values therefore change by at most $4/|\chi_n|$, since $0\le\be_n\le1$. Consequently the maximal inequality holds with an error $O(|\chi_n|^{-1})$ in its thresholds. This error tends to zero, so the proof of Theorem \ref{T:ind-main} applies.}
\end{proof}

\subsection{Graph convergence, exponential tightness and geodesic convergence everywhere}

The moment bounds necessary to prove graph convergence and geodesic convergence everywhere are not known for all of the models discussed above. However, they are available for geometric last passage percolation with a fixed height-to-width ratio and parameter, and for exponential, Brownian, and planar Poisson last passage percolation. In these models, stronger bounds also allow us to prove exponential tightness

\begin{theorem}
\label{T:geom-lpp-emph}
\begin{enumerate}
	\item (Geometric) Let the sequence of environments $\scrK_n$ be either as in Theorem \ref{T:geom-lpp} where $\ga = \ga_n$ and $m_n = \floor{\rho n}$ for some fixed constants $\rho$ and $\ga$ that do not depend on $n$. 
	\item (Exponential) Let $\scrK_n$ be as in Corollary \ref{C:exponential} for some sequence $m_n = \floor{\rho n}$. 
	\item (Poisson) Let $\scrK_n := \scrK_{t_n}$ be as in Corollary \ref{C:Poisson-box}. 
	\item (Brownian) Let $\scrK_n$ be as in Corollary \ref{C:Brownian}. 
\end{enumerate}
In all four cases, there is a coupling of $\scrK_n$ and $\scrL$ satisfying graph convergence and geodesic convergence everywhere. Moreover, in cases $1$ and $2$ {for every compact set $K\sset\Rd$ there exists $a>1$ with}
\begin{equation}
\label{E:Moment-exponential}
    {\limsup_{n \to \infty} \left(\E a^{[\sup_K \scrK_n^-]^3} + \E a^{\sup_K \scrK_n^+}\right)<\infty.}
\end{equation}
In case $3$, {for every compact $K\sset\Rd$ there exists $a>1$ with}
\begin{equation}
\label{E:Moment-Poisson}
    {\limsup_{n \to \infty} \left(\E a^{[\sup_K \scrK_n^-]^3} + \E a^{\sup_K \scrK_n^+}\right)<\infty,}
\end{equation}
and in case $4$, {for every compact $K\sset\Rd$ there exists $a>1$ with} 
\begin{equation}
\label{E:Moment-Brownian}
    {\limsup_{n \to \infty} \left(\E a^{[\sup_K \scrK_n^-]^{3/4}} + \E a^{[\sup_K \scrK_n^+]^{3/2}}\right)<\infty.}
\end{equation}
{In each case, with $\mathfrak l,\mathfrak r$ as in Theorem \ref{T:intro-main}, the coupling can be chosen so that for every compact $K\sset\Rd$ there exists $a>1$ with}
\begin{equation}
\label{E:Moment-error}
    {\lim_{n\to\infty}\left(\E a^{[\sup_K(\scrK_n-\scrL)^-]^{\mathfrak l\wedge(3/2)}}+\E a^{[\sup_K(\scrK_n-\scrL)^+]^{\mathfrak r}}\right)=2.}
\end{equation}
\end{theorem} 

\begin{proof}[Proof of Theorem \ref{T:geom-lpp-emph}.1]
When $X$ is a geometric environment with fixed $\ga, \rho$, by the proof of Theorem \ref{T:geom-lpp}, we are in the setting of Theorem \ref{T:lattice-models-i}. Given this, to prove the uniform tail bounds in \eqref{E:Moment-exponential}, it suffices to check the conditions of Theorems \ref{T:upper-exponential} and \ref{T:lower-exponential} with $\mathfrak{p} = 1$ and $\mathfrak{l} = 3$
and then appeal to those theorems. Indeed, if the condition in Theorem \ref{T:lower-exponential} holds then so does the condition \eqref{E:Lnkexpt}, so in some coupling $\scrK_n \to \scrL$ in the sense of graph convergence and geodesic convergence everywhere. Moreover, in this coupling
\[
\sup_K (\scrK_n - \scrL)^+, \sup_K (\scrK_n - \scrL)^- \to 0
\]
almost surely. {To obtain \eqref{E:Moment-error}, use
\[
\sup_K(\scrK_n-\scrL)^\pm\le\sup_K\scrK_n^\pm+\sup_K\scrL^\mp.
\]
By Corollary 10.7 in \cite*{DOV}, $\sup_K|\scrL|$ has an exponential moment of order $3/2$. Together with the bounds in \eqref{E:Moment-exponential}, Cauchy--Schwarz gives uniform integrability of the exponentials in \eqref{E:Moment-error} for $a>1$ sufficiently close to $1$. Almost sure convergence gives the limit. The same argument applies to the other three cases using their respective marginal bounds.}

{For any compact set $K\sset\Rd$ there exists $a>1$ such that for all large enough $n\in\N$ we have} the upper tail bound
\begin{equation}
\label{E:Knu-upper}
\p\lf(\scrK_n(u) > m \rg) \le c_K a^{-m},
\end{equation}
for all $m > 0$. See, for example, Proposition 2.2 in \cite*{ledoux2007deviation} for this bound.

By equation (1.21) in \cite*{baik2001optimal}, {for any compact set $K\sset\Rd$ there exist $a>1$ and $\de>0$ such that} for large enough $n$ and all $m \le \de n^{2/3}$ and $u \in K$ we have
\begin{equation}
\label{E:plfAn1}
\p\lf(\scrK_n(u) \le - m \rg) \le c_K a^{-m^3}.
\end{equation}
Since $\scrK_n(u)$ is deterministically bounded below on $K$ by ${-}d n^{2/3}$ for some constant $d{>0}$, \eqref{E:plfAn1} extends to all $m$, by possibly changing the constant $a > 1$.  Combining \eqref{E:Knu-upper} and \eqref{E:plfAn1} shows that the condition in Theorem {\ref{T:upper-exponential}} holds with $\mathfrak{p} = 1$. The inequality \eqref{E:plfAn1} shows that Theorem \ref{T:lower-exponential} holds with $\mathfrak{l} = 3$.
\end{proof}

\begin{proof}[Proof of Theorem \ref{T:geom-lpp-emph}.2] The proof here is similar. Again, it suffices to check the conditions of Theorem \ref{T:upper-exponential} and \ref{T:lower-exponential} with $\mathfrak{p} = 1$ and $\mathfrak{l} = 3$. By symmetry, we can assume $\rho \ge 1$, i.e. $m_n \ge n$ for all $n$. The bound \eqref{E:Knu-upper} also holds for the exponential environment. This is again shown in \cite*{ledoux2007deviation}, see the discussion immediately following Proposition 2.2 therein.

Next, by Theorem 2 in \cite*{ledoux2010small}, there are constants $c, c' > 0$ such that for every $\ep \in (0, 1]$ and $m, n \in \N$ we have
\begin{equation}
\label{E:exponential-tail}
\p\lf(X[(1, n) \to (m, 1)] \le (\sqrt{m} + \sqrt{n})^2(1 - \ep) \rg) \le c \exp \lf(-c' \ep^3 m n \lf(\frac{1}{\ep} \wedge \frac{m^{1/2}}{n^{1/2}} \rg)\rg).
\end{equation}
After a change of variables to $\scrK_n(u)$, this gives that {for every compact set $K\sset\Rd$ there exists $a>1$ such that} for all large enough $n$ the bound in \eqref{E:plfAn1} holds.
\end{proof}

\begin{proof}[Proof of Theorem \ref{T:geom-lpp-emph}.3]
 While we cannot appeal directly to Theorems \ref{T:lattice-models-i}, \ref{T:upper-exponential}, or \ref{T:lower-exponential} directly since the Poisson environment is not defined on a lattice, the proofs go through almost verbatim with the obvious changes. We sketch what is needed here.
 
Starting from Corollary \ref{C:Poisson-box}, we can conclude graph convergence and the lower tail in Theorem \ref{T:lower-exponential} via the chaining statements in Proposition \ref{P:prob-cond} and Corollary \ref{C:true-bound-for-holes}. Since the convergence of Poisson last passage percolation in Corollary \ref{C:Poisson-box} is following a standard $1-2-3$ scaling, i.e. $\tau_t = {2} t^{2/3}, \chi_t = t^{1/3}$, and Poisson last passage values are always nonnegative, we can again apply Proposition \ref{P:prob-cond} and Corollary \ref{C:true-bound-for-holes} with $k(n) = \floor{3 \log_2({t_n})/4}$. {Indeed, $t_n2^{-k(n)}\to\infty$ ensures that the endpoints are ordered, while the centering gives a uniform deterministic lower bound $-O(t_n^{2/3}2^{-k(n)})=-O(t_n^{-1/12})$ for the final increments.} Just as in the lattice case, here the bound on the $\be_{n, i}$ functions follows from the fact that {for every compact $K\sset\Rd$ there exists $a>1$ such that}
\begin{equation}
\label{E:LSS}
\limsup_{n \to \infty} \sup_{u \in K} \E a^{(\scrK_n(u)^-)^3} < \infty.
\end{equation}
 Equation \eqref{E:LSS} is an immediate corollary of Theorem 1.2 in \cite*{lowe2002moderate}. 
 
 For the upper tail, again, we cannot immediately appeal to the maximal inequality in Lemma \ref{L:doob-lpp-boot} since $d_P$ is not a line last passage metric. {We first apply the inequality to finite sets and approximate $d_P$ by vertically compressed line last passage metrics.} Indeed, starting with the Poisson process $P$ on $\R^2$ and a scaling parameter $s > 0$, form a collection of Poisson processes $P_{s, i}, i \in \Z$ on $\R$, by setting
 $$
 P_{s, i}[a, b] = P({[a,b]\X[-s(i+1),-si)})
 $$
 for all $i$. Let $d^*_{P_s}$ be the pullback of the line last passage metric formed from $P_{s, i}, i \in \Z$ under the map $(x, y) \mapsto (x, -y/s)$. {For any fixed finite collection of endpoints, the last passage values converge almost surely as $s\downarrow0$: on a bounded rectangle there are finitely many Poisson points, and for all small enough $s$ their vertical order is preserved. Thus the maximal inequality passes to $d_P$ for finite sets. Approximating the deterministic compact rectangles by finite sets, as in Lemma \ref{L:doob-lpp}, gives the inequality in the form needed here.} Therefore Corollary \ref{C:limit-setting}
 holds for $\scrK_n$, and so we can pass through the proof of Theorem \ref{T:upper-exponential}. {Write $L_t=d_P((0,0),(t,t))$. Theorem 1.3 in \cite*{lowe2001moderate} gives constants $c_1,c_2>0$ such that, for $t,z\ge1$,
\[
\p(L_t\ge2t+zt^{1/3})\le c_1\exp\left(-\frac{c_2z^{3/2}}{(1+z/t^{2/3})^{1/2}}\right).
\]
Since $z^{3/2}/(1+z/t^{2/3})^{1/2}\ge z/\sqrt2$ for $t,z\ge1$, this bound, together with \eqref{E:LSS}, gives the required exponential tail condition with $\mathfrak p=1$. In particular,}
 {for every compact $K\sset\Rd$ there exists $a>1$ such that}
\[
\limsup_{n \to \infty} \sup_{u \in K} \E a^{{\scrK_n(u)^+}} < \infty. \qedhere
\]
 \end{proof}
 
 \begin{proof}
 [Proof of Theorem \ref{T:geom-lpp-emph}.4]
 The graph convergence in the Brownian case is the main result of \cite*{DOV}, see Theorem 1.5 in that paper. Geodesic convergence everywhere is stronger than the type of geodesic convergence proven in Theorem 1.8 of \cite*{DOV}. It follows deterministically from graph convergence by Theorem \ref{T:geod-cvg}, as in the proof of Theorem \ref{T:lattice-models-i}. 
 
 {The lower-tail bound on $\scrK_n$ follows from the coupling estimate in Theorem 1.5 and the landscape bound in Corollary 10.7 of \cite*{DOV}.} For the upper bound, as in the proof of Theorem \ref{T:upper-exponential} we can just appeal to Corollary \ref{C:limit-setting}, since this corollary is stated in the generality of line last passage models. The required compressed exponential moment bound on $|\scrK_n(u)|$ follows in this case from Theorem 1 in \cite*{ledoux2010small}, which together with standard large deviation bounds on $|\scrK_n(u)|$ (e.g. see display (1.4) {in \cite*{ledoux2010small}}), gives that {for every compact set $K\sset\Rd$ there exists $a>1$ such that}
 \[
 \limsup_{n \to \infty} {\sup_{u\in K}} \E a^{{|\scrK_n(u)|^{3/2}}} < \infty. \qedhere
 \]
 \end{proof}
 
 \begin{remark}
{The exponents in the marginal bounds \eqref{E:Moment-exponential}--\eqref{E:Moment-Brownian} are optimal, except for the lower-tail exponent in the Brownian case. For Poisson last passage, the right-tail exponent $1$ cannot be increased: the event that there are exactly $k$ points in $[0,t]^2$ and they form an increasing sequence gives
\[
\p(L_t\ge k)\ge e^{-t^2}\frac{t^{2k}}{(k!)^2},
\]
which rules out any positive exponential moment of $L_t^p$ for $p>1$.}
Optimal {Brownian lower-tail} estimates could be obtained by applying the ideas of Theorem \ref{T:lower-exponential} and tail bounds from \cite*{ledoux2010small}. A few modifications would need to be made to resolve the fact that the Brownian landscapes $\scrK_n$ are not deterministically bounded below at any scale. For brevity, we do not pursue this here.
 \end{remark}

\begin{remark}
Since the bound \eqref{E:exponential-tail} has no restriction on the ratio of $m/n$, it is strong enough to give graph convergence and geodesic convergence everywhere for any sequence $\scrK_n$ as in Corollary \ref{C:exponential} with $m_n$ tending to $\infty$. The proof of Theorem \ref{T:lattice-models-i} can be modified to work in this setting. 

The bound on $\scrK_n(p^+_{i+1}, p^+_{i})$ goes through in the same way. However, the deterministic bound on $\scrK_n(p, p_{k(n)}^+)$ fails when the ratio $m_n/n$ tends to $0$ or $\infty$ quickly enough. To deal with this, we must replace this deterministic bound with a union bound over last passage values for all possible starting lattice points in a compact set. At a fine scale, the cost of taking the union bound is outweighed by the exponential tails in \eqref{E:plfAn1}.
\end{remark}

\stepcounter{theorem}
\begin{proof}[{Proof of Theorems \ref{T:intro-main} and \ref{T:intro-geod}}]
	In order to stay in the world of directed metrics after rescaling, the setup in this section looks slightly different than the simpler setup in Theorems \ref{T:intro-main} and \ref{T:intro-geod}. It is nonetheless easy to see that those statements are equivalent to the ones in this section. We give a detailed example of how to see this for the case of the geometric environment with fixed slope $\rho$ and mean $\ga$. Let $\sig_n = n^{1/3}$, set $v = (\rho, -1), u = ({\tau}, 0)$, and let $\scrK_n, \tau_n, \chi_n, \al, \be$ be as in Theorem \ref{T:geom-lpp} for $m_n := \rho n$. Note that it is not important that $m_n$ be an integer for the purposes of applying Theorem \ref{T:geom-lpp}, that $\al, \be$ do not change with $n$, and that $\chi_n = \chi n^{1/3}, \tau_n = \tau n^{2/3}$ for fixed $\chi, \tau$.
	
Let $d$ denote the last passage metric on $\Z^2 \cup E$ coming from an array of i.i.d.\ geometric random variables $X_\ga$, and let $h$ be the additive metric on $\Z^2 \cup E$ satisfying $h(\rho t + x, -t) = \al t + \be x$ for $(\rho t + x, -t) \in \Z^2$ and $h(e_1, e_2) = {h(e_1)}$ on edges. For $(x, s) \in \R^2$, {define $M_n(x,s)=x\sig_n^2u+s\sig_n^3v$, so $M_n(x,s)=((L_n(x,s))_1,-(L_n(x,s))_2)$.} Now let
\begin{align*}
(p)_n &=(\floor{({M_n}p)_1}, \floor{({M_n}p)_2}), \quad \floor{p}_n = {M_n^{-1}}(\floor{({M_n}p)_1}, \floor{({M_n}p)_2}), \quad \mathand \\
d_n(x, s; y, t) &= d((x, s)_n, (y, t)_n) - \lf({\al} \sig_n^3(t-s) + {\be\tau} \sig_n^2(y-x)\rg).
\end{align*}
If we restrict $d_n$ to $\floor{\R^2}_n \X \floor{\R^2}_n$, then $d_n$ is exactly the pullback of the metric $d + h$ under the map $f_{\R^2 \to \Z^2} \circ L_n$. That is, $\chi_n^{-1}  d_n = \scrK_n$ on this set. Moreover, from the scaling it is easy to see that $d_n(p, q) - d_n(\floor{p}_n, \floor{q}_n) = O(1)$, uniformly in $p, q,$ and $n$. Combining these facts with Theorem \ref{T:geom-lpp-emph} gives that $\chi_n^{-1} d_n \to \scrL$ compactly in some coupling, and that rescaled geodesics converge everywhere, as is necessary for Theorem \ref{T:intro-geod}.
\end{proof}

We finish this subsection with a different way of rescaling planar Poisson last passage percolation where we first rotate the plane by $\pi/4$ and then scale by a diagonal transformation. Rotated planar Poisson last passage percolation is precisely the polynuclear growth model of \cite*{prahofer2002scale}. Note that \cite*{prahofer2002scale} work with Poisson processes of intensity $2$, whereas we work with Poisson processes of intensity $1$.
{The prelimiting models satisfy} a natural reflection invariance that passes to $\scrL$, see Proposition \ref{P:landscape-sym}.

\begin{corollary}[Convergence of polynuclear growth]
	\label{C:rotated-Poisson}
	With all notation and assumptions as in Corollary \ref{C:Poisson-box}, let
	$\scrL_t$ be the pullback of $\chi_t^{-1}(d_P - d_h)$ under the linear map $M_t:\R^2 \to \R^2$ with matrix
	\begin{equation*}
	B_t = \lf[\begin{array}{cc}
	1/\sqrt{2} & 1/\sqrt{2} \\
	-1/\sqrt{2} & 1/\sqrt{2}  
	\end{array}\rg] 
	\lf[\begin{array}{cc}
	{\tau_t/\sqrt{2}}& 0 \\
	0 & \sqrt{2} t
	\end{array}\rg] = \lf[\begin{array}{cc}
	{\tau_t/2} & t \\
	{-\tau_t/2} & t
	\end{array}\rg].
	\end{equation*}
	Then for any sequence $t_n \to \infty$, there is a coupling of $\scrL_{t_n}$ and $\scrL$ satisfying graph convergence and geodesic convergence everywhere.
\end{corollary}

\begin{proof}
Let $\scrK_t, A_t$ be as in Corollary \ref{C:Poisson-box}. By Theorem \ref{T:geom-lpp-emph}, we consider a coupling of $\scrK_{t_n}, \scrL$ satisfying graph convergence and geodesic convergence {everywhere}. The environment $\scrL_{t}$ is the pullback of $\scrK_{t}$ under the map
	$$
	A_t^{-1}B_t = \lf[\begin{array}{cc}
	1 & 0 \\
	{-\tau_t/(2t)} & 1
	\end{array}\rg].
	$$
	Since $\tau_t/t \to 0$ as $t \to \infty$, $A_t^{-1}B_t$ converges to the identity as $t \to \infty$. Since $\scrL_{t_n}$ is the pullback of $\scrK_{t_n}$ under a sequence of maps converging to the identity, we have that $\mathfrak{g} \scrL_{t_n} \to \mathfrak{g} \scrL$ in the specified coupling. Moreover, geodesics in $\scrL_{t_n}$ are simply pullbacks of geodesics in $\scrK_{t_n}$ by Lemma \ref{L:pullback-geod}, so $\scrL_{t_n}$ inherits geodesic convergence everywhere from $\scrK_{t_n}$.
\end{proof}

\subsection{A distinction between hypograph and graph convergence}

We end this section by modifying exponential last passage percolation to give an example of a sequence of i.i.d.\ lattice environments $\tilde \scrK_n$ that can be coupled with $\scrL$ to satisfy {hypograph convergence and geodesic convergence off holes, but not graph convergence}.

\begin{example}
\label{Ex:elpp}
Let $X, \scrK_n$ be as in Corollary \ref{C:exponential} with $m_n = n$. Define a new sequence of i.i.d.\ environments $\tilde X_n$ {using independent marks, independent of $X$, so that}
$$
\tilde X_n(u) = \begin{cases}
X(u), \qquad&\text{ with probability } 1 - n^{-5/3} \\
- n^{1/3}, \qquad &\text{ with probability } n^{-5/3},
\end{cases} 
$$
and define $\tilde \scrK_n$ from $\tilde X_n$ in the same way that $\scrK_n$ was defined from $X$.
{There is a coupling of $\tilde \scrK_n$ and $\scrL$ satisfying hypograph convergence and geodesic convergence off holes.} However, $\mathfrak{g} \tilde \scrK_n$ does not converge in distribution to $\mathfrak{g} \scrL$.
\end{example}

To see why Example \ref{Ex:elpp} works, first observe that there exists a constant $c > 0$ such that {for all large enough $n$,}
\begin{equation}
\label{E:lfinf}
\p\lf(\inf_{{(x,s)\in[0,1]^2}} {\tilde\scrK_n(x,s;0,2)-\scrK_n(x,s;0,2)} \le - 2^{-4/3} \rg) > c.
\end{equation}

This is because there are {order $n^{5/3}$} squares in $\Z^2$ that get sent to the box $[0,1]^2$ after rescaling to get $\scrK_n$ and $\tilde \scrK_n$, and so with nonvanishing probability, there is some box $u$ in this set with $\tilde X_n(u) = -n^{1/3} < 0 \le X(u)$. Since $X$ dominates $\tilde X_n$ everywhere the bound on ${\tilde\scrK_n-\scrK_n}$ above follows since the scaling parameter $\chi_n = 2^{4/3} n^{1/3}$ for this model.

Since $\mathfrak{g}\scrK_n$ converges in distribution to $\mathfrak{g}\scrL$ by Theorem \ref{T:geom-lpp-emph} and $\scrK_n \ge \tilde \scrK_n$, \eqref{E:lfinf} implies that $\mathfrak{g} \tilde \scrK_n$ cannot also converge in distribution to $\mathfrak{g}\scrL$.

{For fixed $u\in\Rd$, choose an $X$-geodesic between the corresponding lattice endpoints. It has at most $cn$ vertices, so, conditionally on $X$, the probability that any of its weights is modified is at most $cn^{-2/3}$. Since $\tilde\scrK_n\le\scrK_n$, it follows that
$$
\p\bigl(\tilde\scrK_n(u)=\scrK_n(u)\bigr)\longrightarrow1.
$$
Applying this on a countable dense subset of $\Rd$, domination and graph convergence of $\scrK_n$ imply hypograph convergence of $\tilde\scrK_n$, by Lemma \ref{L:hypo-equiv} applied to joint subsequential limits. Skorokhod's representation theorem gives a coupling with almost sure hypograph convergence. Geodesic convergence off holes then follows from the lattice-path argument in Step 2 of the proof of Theorem \ref{T:lattice-models}, which does not require nonnegative weights.}

\section{Consequences}
\label{S:consequences}

In this section, we prove consequences of the convergence statements in Section \ref{S:integrable}.

\subsection{Symmetries of $\scrL$ and $\scrS$} 
\label{SS:symmetries}
The next proposition is a restatement of Proposition \ref{P:intro-sym}.
	\begin{prop}
		\label{P:landscape-sym}
	As a function of $x, s; y,$ and $t$, the directed landscape $\scrL$ satisfies
	\begin{equation}
	\label{E:Lxyyx}
	\scrL(x, s; y, t) \eqd \scrL(-x, s; -y, t) \eqd \scrL(y, -t; x, -s),
	\end{equation}
	and as a function of $x$ and $y$, the Airy sheet $\scrS$ satisfies
	\begin{equation}
	\label{E:Sxyyx}
	\scrS(x, y) \eqd \scrS(-x, -y) \eqd \scrS(y, x).
	\end{equation}
\end{prop}
\begin{proof}
Let $\scrL_t$ be as in Corollary \ref{C:rotated-Poisson}. We have the symmetries
\begin{equation*}
\scrL_t(x, s; y, t) \eqd \scrL_t(-x, s; -y, t) \eqd \scrL_t(y, -t; x, -s).
\end{equation*}
These follow from reflection invariance of a planar rate-$1$ Poisson process across the horizontal and vertical axes. The distributional equalities in \eqref{E:Lxyyx} follow immediately from these by Corollary \ref{C:rotated-Poisson}. The distributional equalities in \eqref{E:Sxyyx} are equivalent to the distributional equalities in \eqref{E:Lxyyx} when we fix $s = 0, t = 1$.
\end{proof} 

The next theorem is a restatement of Theorem \ref{T:sheet-map}. 

\begin{theorem}
\label{T:airy-sheet-mble}
Let $\scrA$ be an Airy line ensemble, and let $\scrA^*(x) = \scrA(-x).$ For $(x, y, z) \in \Q^+ \X \Q^2$ define
	\begin{equation}
	\label{E:S'-form}
	\begin{split}
	\scrS'(x, y, z) &= \scrA[(-\sqrt{k/(2x)}, k) \to (y, 1)]- \scrA[(-\sqrt{k/(2x)}, k) \to (z, 1)],\\
	\scrS'(-x, y, z) &= \scrA^*[(-\sqrt{k/(2x)}, k) \to (-y, 1)]- \scrA^*[(-\sqrt{k/(2x)}, k) \to (-z, 1)],
	\end{split}
	\end{equation}
for all sufficiently large $k$. This is well-defined since the right hand sides above stabilize almost surely as $k \to \infty$. For $(x, y) \in {(\Q \setminus \{0\})} \X \Q$, define
\begin{equation}
\label{E:Sxy-lim}
    \scrS(x, y) = \E \scrA_1(0) +\lim_{n \to \infty} \frac{1}{2n + 1} \sum_{z=-n}^{n}{[}\scrS'(x, y, z) - (z - x)^2{]},
\end{equation} Finally, extend $\scrS$ to all of $\R^2$ by continuity.
Then $\scrS$ is an Airy sheet and $\scrS(0,\cdot)=\scrA_1(\cdot)$.
\end{theorem}

\begin{proof} To prove Theorem \ref{T:airy-sheet-mble}, we will first construct a coupling between an Airy sheet $\scrS$ and $\scrA$  satisfying \eqref{E:S'-form} with $\scrS'(x, y, z)$ replaced by $\scrS(x, y) - \scrS(x, z)$.  

We work in the context of Theorem \ref{T:geom-lpp}. Let $X$ be an environment of i.i.d.\ geometric random variables of mean $\ga_n = 1$ and $m_n = n$. With these choices, there exist $\al, \be, \tau, \chi$ such that $\al =\al_n, \be = \be_n, \tau_n = \tau n^{2/3}$ and $\chi_n = \chi n^{1/3}$ for all $n$. Rather than work with the directed metrics $\scrK_n$ as in Theorem \ref{T:geom-lpp}, we will work with the prelimit $\scrL_{n, {\Delta k}}$ as defined in Corollary \ref{C:Xiidmaster}. That is, letting $(x, s)_n = (\floor{n s + \tau n^{2/3} x}, \floor{-ns})$, we set
	\begin{equation}
	\label{E:Lnk}
	\scrL_{n, \Delta k} (x, s; y, t) = \frac{X[(x, s)_n \to_{\Delta k} (y, t)_n] - {\al}(t-s)n - {\be}\tau(y - x) n^{2/3}}{\chi n^{1/3}}.
	\end{equation}
Also let $\tilde X(a, b) = X(-b, -a)$, and let $\tilde \scrL_{n, {\Delta k}}$ be defined from $\tilde X$ in the same way that $\scrL_{n, {\Delta k}}$ was defined from $X$.  All the assumptions of Corollary \ref{C:Xiidmaster} hold for the environments $X$ and $\tilde X$.

 By Corollary \ref{C:Xiidmaster}, $\scrL_{n, {\Delta 1}}, \scrL_{n, {\Delta\cdot}}(0, 0; \cdot, 1)$ converge jointly to $\scrL, \scrA$, where convergence to $\scrL$ is in the topology of convergence of graphs, and convergence to $\scrA$ is in the product of Skorokhod topologies. Moreover, $\scrL(\cdot, 0; \cdot, 1)$ is coupled to $\scrA$ as in Definition \ref{D:airy-sheet}. The same convergence holds for $\tilde \scrL, \tilde \scrL_{n, {\Delta\cdot}}(0, 0; \cdot, 1)$, so we can find a joint subsequential limit
 $\scrL, \scrA, \tilde \scrL, \tilde \scrA$ of 
$$
\scrL_{n, {\Delta 1}},  \qquad \tilde \scrL_{n, {\Delta 1}}, \qquad \scrL_{n, {\Delta\cdot}}(0, 0; \cdot, 1), \qquad \tilde \scrL_{n, {\Delta\cdot}}(0, 0; \cdot, 1).
$$
The coupling between $\scrL_{n, {\Delta 1}}$ and $\tilde \scrL_{n, {\Delta 1}}$ and the scaling in \eqref{E:Lnk} guarantees that
\begin{equation}
\label{E:LtildeL}
\scrL(x, s; y, t) = \tilde \scrL(-x, s; -y, t).
\end{equation}
Moreover, the coupling between $\scrL, \scrA,$ and  $\tilde \scrL, \tilde \scrA$ guarantees that with  $\scrS(\cdot, \cdot) = \scrL(\cdot,0;{\,}\cdot, 1)$, for all $(x, y, z) \in \Q^+ \X \Q^2$, we have
	\begin{equation}
	\label{E:S-form}
	\begin{split}
	\scrS(x, y) - \scrS(x, z) &= \scrA[(-\sqrt{k/(2x)}, k) \to (y, 1)]- \scrA[(-\sqrt{k/(2x)}, k) \to (z, 1)],\\
	\scrS(-x, y) - \scrS(-x, z) &= \tilde \scrA[(-\sqrt{k/(2x)}, k) \to (-y, 1)]- \tilde \scrA[(-\sqrt{k/(2x)}, k) \to (-z, 1)],
	\end{split}
	\end{equation}
for all sufficiently large $k$. Next we show that $\scrA^*(x) := \scrA(-x)$ is equal to $\tilde \scrA(x)$ for all $x \in \R$.

By continuity of $\scrA$ and $\tilde \scrA$, and the fact that $\scrA(-x) \eqd \tilde \scrA(x)$ for all $x$, to prove this it is enough to show that for any fixed $x > 0$ and $k \in \N$, we have
\begin{equation}
\label{E:Aiii}
\sum_{i=1}^k \tilde \scrA_i(-x) \le \sum_{i=1}^k \scrA_i(x) \qquad \mathand \qquad \sum_{i=1}^k \tilde \scrA_i(x) \le \sum_{i=1}^k \scrA_i(-x) \qquad \text{ almost surely}.
\end{equation}
We will prove the first of these inequalities. The second follows by a symmetric argument.
{In the prelimit comparison below, replace $x$ by $\floor{\tau n^{2/3}x}/(\tau n^{2/3})$, and keep the notation $x$.} In the prelimit, the right hand side of the first inequality in \eqref{E:Aiii} is a rescaled and centered multi-point last passage value in $X$ from
\begin{align*}
&\bp = ((0, 1-k), \dots, (0, 0))
\end{align*}
to
\begin{align*}
&\bq = ((n + \floor{\tau n^{2/3} x}, -n), \dots, (n + \floor{\tau n^{2/3} x}, - n +  k - 1){)},
\end{align*}
and the left hand side of the first inequality in \eqref{E:Aiii} is a rescaled and centered multi-point last passage value in $X$ from $\bp$ to
$$
\bq' = ((n, -n + \floor{\tau n^{2/3} x}), \dots (n, -n + \floor{\tau n^{2/3} x} +  k-1)).
$$
Also let $\bq''$ denote the vector $\bq'$ with all components shifted by $(1, 0)$. Since $x > 0$, we have the triangle inequality
\begin{equation}
\label{E:XXX}
X[\bp \to \bq'] + X[\bq'' \to \bq] \le X[\bp \to \bq],
\end{equation}
as any $k$-tuples of disjoint paths from $\bp \to \bq'$ and $\bq'' \to {\bq}$ can be concatenated to give a $k$-tuple of disjoint paths from $\bp$ to $\bq$. We can rewrite \eqref{E:XXX} in terms of $\scrL_{n, {\Delta k}}, \tilde \scrL_{n, {\Delta k}}$ to get that
\begin{equation}
\label{E:Ynk-sum}
Y_{n, k} + \sum_{i=1}^k \tilde \scrL_{n, \Delta {i}}(0,0; -x, 1)  \le \sum_{i=1}^k \scrL_{n, \Delta {i}}(0,0; x, 1)
\end{equation}
for some error term $Y_{n, k}$ given by
$$
Y_{n, k} = \sum_{i=1}^k \scrL_{n, \Delta {i}} (x + O(k n^{-2/3}), 1 - \tau n^{{-1/3}} x + O(k/n) ; x, 1){+O_k(n^{-1/3})}.
$$
The inequality \eqref{E:Ynk-sum} implies \eqref{E:Aiii} if we can show that the error term $Y_{n, k}$ converges to $0$. By translation invariance of $X$, we can write
$$
Y_{n, k} \eqd {(\ell_n/n)^{1/3}} \sum_{i=1}^k \scrL_{\ell_n, \Delta {i}}(0,0; z_n, 1){+O_k(n^{-1/3})},
$$
where {$\ell_n = \tau n^{2/3}x+k-1=\Theta(n^{2/3})$ and $z_n=-k/(\tau\ell_n^{2/3})\to 0$}. Since $\scrL_{\ell_n, {\Delta k}}(0,0; \cdot, 1)$ converges to the parabolic Airy line ensemble and $\ell_n/n \to 0$, $Y_{n, k} \to 0$ in distribution. This yields \eqref{E:Aiii}, and hence $\tilde \scrA = \scrA^*$.

We can now use \eqref{E:S'-form} to show that the construction in the theorem is well-defined and returns the Airy sheet $\scrS$. {For each fixed $x$, the process $y\mapsto\scrS(x,y)+(y-x)^2$ has the same law as $y\mapsto\scrA_1(y)+y^2$, which is strongly mixing}, see equation (5.15) in \cite*{prahofer2002scale} and \cite*{corwin2014ergodicity}, Proposition 1.13. Therefore with $\scrS'$ defined as in \eqref{E:S'-form} from $\scrA$, by \eqref{E:S-form} and the pointwise ergodic theorem, the right hand side of \eqref{E:Sxy-lim} returns the Airy sheet value $\scrS(x, y)$ almost surely. Extending continuously to $\R^2$ returns the entire Airy sheet $\scrS$. This Airy sheet satisfies $\scrS(0, \cdot) = \scrA_1(\cdot)$ by Definition \ref{D:airy-sheet}(ii).
\end{proof}

\subsection{Directed geodesics and longest increasing subsequences}
\label{SS:LIS}

Recall from Remark \ref{R:landscape-geodesics} that there is almost surely a unique $\scrL$-geodesic $\Ga$ from $(0,0)$ to $(0, 1)$. We can represent $\Ga$ as $\{(\Pi(t), t) : t \in [0, 1]\}$ for some random continuous function $\Pi:[0, 1] \to \R$. We call $\Pi$ a \textbf{directed geodesic}. The almost surely unique $\scrL$-geodesic between any other pair of points in the plane is equal in distribution to a rescaled, sheared, version of $\Pi$. In this section, we show that $\Pi$ is the scaling limit of the longest increasing subsequence in a uniform permutation.

Let $\sig$ be an element of the symmetric group $S_n$. An \textbf{increasing subsequence} in $\sig$ is a subsequence $I = \{i_1 < \dots < i_k\} \sset \{1, \dots n\}$ such that
$$
\sig(i_1) < \sig(i_2) < \dots < \sig(i_k).
$$
We think of each subsequence $I$ as a function from $\{1, \dots, k\}$ to $\N$ with $I(m) = i_m$.
A basic observation is that the length of the \textbf{longest increasing subsequence} can be related to a last passage value: if we let 
\begin{equation}
\label{E:PP-perm}
P=\{(i, \sig(i)): i \in \{1, \dots, n\}\},
\end{equation}
and define $d_P$ as in Definition \ref{D:planar-poisson}, then
$
d_P((0,0); (n, n))
$
is the largest length of any increasing subsequence in $\sig$.

When $\sig$ is a uniformly chosen permutation, then the point process $P$ is closely related to a Poisson process in the plane, and so we can use Corollary \ref{C:Poisson-box} and the connection above to show that the scaling limit of the longest increasing subsequence in a uniform permutation is the directed geodesic.

\begin{theorem}
	\label{T:subsequence}
	Let $\sig_n \in S_n$ be a uniformly chosen permutation. Let $I_n,J_n$ be arbitrary maximal length increasing subsequences in $\sig_n$. Then 
	$$
	\tilde I_n(x) :=\frac{I_n(\lfloor 2x\sqrt{n}\rfloor ) - nx}{{n^{5/6}}} \cvgd \Pi(x) \qquad \mathand \qquad \max_{1\le i \le 2\sqrt{n}}\frac{|I_n(i) - J_n(i) | }{n^{5/6}} \cvgp 0,
	$$
	where $\Pi$ is the directed geodesic, the underlying topology in the first convergence is the Skorokhod topology on cadlag functions from $[0, 1] \to \R$, {$|I_n|$ denotes the length of $I_n$, $I_n(0):=0$, and $I_n(i):=n$ for $i>|I_n|$}.
\end{theorem}

We will prove Theorem \ref{T:subsequence} in tandem with a closely related theorem. This can be thought of as a Kolmogorov-Smirnov test for longest increasing subsequences.

\begin{theorem}
	\label{T:kstest}	
	Let $\{X(i) : i \in \N\}$ be a sequence of i.i.d.\ random variables with a continuous CDF $F$. For each $n$, let $I_n = \{I_{n, 1} < \dots < I_{n, K_n}\} \sset \{1, \dots, n\}$ be a subsequence of maximal length satisfying
	$$
	X(I_{n, 1}) < \dots < X(I_{n, K_n}).
	$$
	Define the empirical distribution function $F_n:\R \to [0, 1]$ by
	$$
	F_n(x) = \frac{1}{K_n} \sum_{j=1}^{K_n} \indic(X(I_{n, j}) \le x).
	$$
	Then 
	$$
	\tilde F_n := n^{1/6} \lf(F_n - F\rg) \cvgd \Pi \circ F
	$$
	where $\Pi$ is the directed geodesic. Here the underlying topology is the Skorokhod topology on cadlag functions from $\R \to \R$. Moreover, if $\tilde F_n'$ is defined analogously from a different maximal length subsequence $J_n$, then
	\[
	\sup_{x \in \R} |\tilde F_n(x) - \tilde F_n'(x)| \cvgp 0 \qquad \mathas n \to \infty.
	\]
\end{theorem}

Note that the Skorokhod convergence in Theorems \ref{T:subsequence} and \ref{T:kstest} {is equivalent to the existence} of couplings such that $\tilde I_n \to \Pi$ uniformly, and $\tilde F_n \to \Pi \circ F$ uniformly. The key step in proving Theorems \ref{T:subsequence} and \ref{T:kstest} is the following convergence result for planar Poisson last passage percolation.

\begin{prop}
	\label{P:planar-poisson}
	Let $P$ be a {rate-$1$} Poisson process on $\R^2$, and define
	$$
	T_n = \inf \{t \ge 0 : {P([0,t]\X[0,\sqrt n])} \ge n \}.
	$$
	Let $\Ga_n \sset\R^2$ be any sequence of $d_P$-geodesics from $(0,0)$ to $(T_n, \sqrt{n})$ such that 
	$$
	\Ga_n = \{(E_n(x), x) : x \in [0, \sqrt{n}]\}
	$$
	for some continuous {nondecreasing surjection} $E_n:[0, \sqrt{n}] \to [0, T_n]$. Also define
	$$
	G_n(x) = d_P((0,0), (E_n(x), x)) \quad \mathand \quad G_n^{-1}(m) = \inf \{x \in [0, \sqrt{n}] : G_n(x) = m \}.
	$$
	Then
	\begin{equation}
	\label{E:GnnN}
	\tilde G_n(x) :=\frac{G_n(\sqrt{n}x) - 2\sqrt{n}x}{2n^{1/3}} \cvgd \Pi(x), \;\mathand\; \tilde E_n(x) := \frac{E_n \circ G_n^{-1}(\floor{2\sqrt{n}x}) - \sqrt{n}x}{{n^{1/3}}}  \cvgd \Pi(x),
	\end{equation}
	where $\Pi$ is the directed geodesic. Here the underlying topology is the Skorokhod topology on cadlag functions from $[0, 1] \to \R$. 
	To ensure that $\tilde E_n$ is well-defined for all $x$, we set $G_n^{-1}(y) = \sqrt{n}$ for $y > d_P((0,0), {(T_n,\sqrt n)})$.
	
	Moreover, if $\Ga_n'$ is a different sequence of geodesics from $(0,0)$ to ${(T_n,\sqrt n)}$, and $\tilde G_n', \tilde E_n'$ were formed from $\Ga_n'$ in the same way that $\tilde G_n, \tilde E_n$ were formed from $\Ga_n$, then
	\begin{equation}
	\label{E:GnnM}
	\sup_{x \in [0, 1]} {\bigl(}|\tilde G_n(x) - \tilde G_n'(x)| + |\tilde E_n(x) - \tilde E_n'(x)|{\bigr)} \cvgp 0
	\end{equation}
	as $n \to \infty$.
	
\end{prop}

\begin{proof}
	In the proof, we let $\scrK_{\sqrt{n}}$ and all related notation be as in Corollary \ref{C:Poisson-box}. Consider a new coupling of the environments $\scrK_{\sqrt{n}}$ where almost surely,
	\begin{enumerate}
		\item $\mathfrak{g} \scrK_{\sqrt{n}} \to \mathfrak{g} \scrL$ in $\scrE_*$.
		\item There is a unique $\scrL$-geodesic $\Ga$ from $(0,0) \to (0, 1)$. Moreover, for any $p_n \to (0, 1)$, any sequence of $\scrK_{\sqrt{n}}$-geodesics $\Lambda_n$ from $(0,0) \to p_n$ converges in the Hausdorff metric to $\Ga$.
		\item $(T_n - \sqrt{n})/n^{1/3} \to 0$.
	\end{enumerate}
	The fact that such a coupling exists follows from Theorem \ref{T:geom-lpp-emph}.3 for the first two points. The third point follows from elementary tail estimates for Poisson random variables. Now, in this coupling, with the matrices $A_t$ as in Corollary \ref{C:Poisson-box}, condition 3 above guarantees that
	\begin{equation*}
	A_{\sqrt{n}}^{-1} (T_n, \sqrt{n}) \to (0, 1).
	\end{equation*}
	This convergence, along with condition 2 above implies that
	\begin{equation}
	\label{E:Ga-n-x}
	A_{\sqrt{n}}^{-1}\Ga_n = \lf\{\lf(\frac{E_n(\sqrt{n}x) - \sqrt{n}x}{2 n^{1/3}}, x \rg) : x \in [0, 1] \rg\} \to \Ga.
	\end{equation}
	Moreover, Proposition \ref{P:convergence-alond-geod} implies that there exists a random constant $C$ such that
	$$
	\limsup_{n \to\infty} \sup_{u \in A_{\sqrt{n}}^{-1} \Ga_n} |\scrK_{\sqrt{n}}(0, 0; u)| \le C.
	$$
	Using that $\scrK_{\sqrt{n}}$ is the pullback under $A_{\sqrt{n}}$ of $n^{-1/6}(d_P - h)$, this gives that
	\begin{align}
	\label{E:GnnQ}
	G_n(x) &= h(0,0; E_n(x), x) + O(n^{1/6}) = x + E_n(x) + O(n^{1/6}),
	\end{align}
	Here the $O(n^{1/6})$ is bounded in absolute value by $C n^{1/6}$ for all $x \in [0, \sqrt{n}]$. Therefore
	$$
	\tilde G_n(x) = \frac{E_n(\sqrt{n}x) - \sqrt{n}x}{2 n^{1/3}} + O(n^{-1/6}).
	$$
	Therefore by \eqref{E:Ga-n-x} and the equivalence of Hausdorff convergence and uniform convergence to continuous functions, $\tilde G_n \to \Pi$ uniformly a.s. This yields the first convergence in \eqref{E:GnnN}. We turn to the second convergence statement in \eqref{E:GnnN}. In the chosen coupling, the first convergence in \eqref{E:GnnN} yields
	$$
	G_n(\sqrt{n}x) = 2\sqrt{n}x + 2n^{1/3}\Pi(x) + o(n^{1/3}),
	$$
	where $o(n^{1/3})/n^{1/3} \to 0$ uniformly a.s.
	This implies that $G_n^{-1}(\floor{2\sqrt{n}x}) = \sqrt{n}y_n(x)$ for some $y_n(x)$ satisfying 
	\begin{equation}
	\label{E:nxx}
	2\sqrt{n}x = 2\sqrt{n} y_n(x) + 2\Pi(y_n(x)) n^{1/3} + o(n^{1/3}). 
	\end{equation}
	Now, \eqref{E:Ga-n-x} also implies that
	\begin{equation}
	\label{E:nFn}
	E_n(\sqrt{n}x) = 2 \Pi(x) n^{1/3} + \sqrt{n}x + o(n^{1/3}).
	\end{equation}
	Plugging \eqref{E:nFn} and $G_n^{-1}(\floor{2\sqrt{n}x}) = \sqrt{n}y_n(x)$ into the formula for $\tilde E_n$ gives that
	\begin{equation}
	\label{E:Piyn}
	\tilde E_n(x) = \Pi(y_n(x)) + o(1).
	\end{equation}
	Finally, equation \eqref{E:nxx} implies that $|y_n(x) - x| = O(n^{-1/6})$, so \eqref{E:Piyn} and the continuity of $\Pi$ proves that $\tilde E_n \to \Pi$ uniformly a.s. 
	
	Finally, in this coupling the same uniform a.s.\ convergences hold with $\tilde G_n', \tilde E_n'$ in place of $\tilde G_n, \tilde E_n$, yielding \eqref{E:GnnM}.
\end{proof}

\begin{proof}[Proof of Theorem \ref{T:subsequence}]
	Let $P, T_n$ be as in the statement of Proposition \ref{P:planar-poisson}. Let
	$p_1, \dots, p_n$ be the Poisson points in $[0, T_n] \X [0, \sqrt{n}]$, listed in order of increasing first coordinate, and {let $\sig(i)$ be the rank of the second coordinate of $p_i$ among the $n$ second coordinates}.
	
	Then $\sig$ is a uniform random permutation in $S_n$. Now let $I_n \sset \sig$ be a subsequence of maximal length $L_n$, and define $\Ga_n \sset [0, T_n] \X [0, \sqrt{n}]$ by connecting the sequence of points 
	$$
	(0,0), p_{I_n(1)}, \dots, p_{I_n(L_n)}, (T_n, \sqrt{n})
	$$
	with straight lines between adjacent points. Then $\Ga_n$ is a $d_P$-geodesic from $(0,0)$ to $(T_n, \sqrt{n})$. Let $\tilde E_n$ be formed from $\Ga_n$ as in Proposition \ref{P:planar-poisson}. Then
	\begin{equation}
	\label{E:FNIN}
	\tilde E_n(x) = \frac{p_{{I_n}(\floor{2 \sqrt{n} x}), 1} - \sqrt{n} x}{{n^{1/3}}}.
	\end{equation}
	Here $p_{i, 1}$ denotes the first coordinate of $p_i$, {and for the endpoint conventions we set $p_{0,1}=0$ and use $p_{n,1}=T_n$ beyond the last selected point.}
	Now, $\{p_{i, 1} : i = 1, \dots, n\}$ is an exponential random walk, where each exponential random variable has mean $1/\sqrt{n}$. Therefore by Donsker's theorem, the sequence of random variables
	$$
	\max_{i \in [1, n]} |p_{i, 1} - i/\sqrt{n}|
	$$
	is tight as $n \to \infty$. Therefore by \eqref{E:FNIN},
	$$
	\tilde E_n(x) = \frac{{I_n}(\floor{2 \sqrt{n} x}) - nx}{{n^{5/6}}} + {O_{\p}(n^{-1/3})},
	$$
	so the first convergence in Theorem \ref{T:subsequence} follows from the second convergence in \eqref{E:GnnN}. Similarly, the second convergence in Theorem \ref{T:subsequence} follows from \eqref{E:GnnM}.
\end{proof}

\begin{proof}[Proof of Theorem \ref{T:kstest}]
	We first deal with the case when the $X(i)$ are uniformly distributed.
	Again, let $P, T_n$ be as in the statement of Proposition \ref{P:planar-poisson}, and let $p_1, \dots, p_n$ be the Poisson points in $[0, T_n] \X [0, \sqrt{n}]$, listed in order of increasing first coordinate. Letting $X(i) = p_{i, 2}/\sqrt{n}$, we have that $X(1), \dots, X(n)$ is a sequence of i.i.d.\ uniform random variables on $[0, 1].$
	
	As in the proof of Theorem \ref{T:subsequence}, if we let $I_n$ be a subsequence of maximal length satisfying 
	$$
	p_{I_n(1), 2} < \dots < p_{I_n(|I_n|), 2},
	$$
	then the set $\Ga_n$ formed as in the proof of Theorem \ref{T:subsequence} is a $d_P$-geodesic from $(0,0)$ to $(T_n, \sqrt{n})$. Moreover, letting $F_n, K_n$ be formed in Theorem \ref{T:kstest} from $X(1), \dots, X(n)$, and $G_n$ be as in Proposition \ref{P:planar-poisson}, we have that
	$
	G_n(\sqrt{n} x) = K_n F_n(x)
	$
	for all $x$. In particular,
	$$
	\tilde G_n(x) - \tilde F_n(x) = G_n(\sqrt{n} x) \lf( \frac{K_n - 2\sqrt{n}}{2 n^{1/3} K_n} \rg) = \frac{G_n(\sqrt{n} x)}{K_n}  O(n^{-1/6}),
	$$
	where in the final equality we have used that the length $K_n$ of the longest increasing subsequence is $2\sqrt{n} + O(n^{1/6})$. Since $G_n(x) \le K_n$ for all $x$, the above converges to $0$ in probability with $n$, yielding the convergence of $\tilde F_n$ for uniform random variables. The convergence of $|\tilde F_n - \tilde F_n'|$ in uniform norm follows from the analogous statement for $\tilde G_n, \tilde G_n'$ in \eqref{E:GnnM}.
	
	Now consider $X(i)$ with a general continuous CDF $F$. Define $F^{-1}:(0, 1) \to \R$ by
	$$
	F^{-1}(x) = \min \{y \in \R : F(y) = x\}.
	$$
	Since $F$ is continuous, this function is well-defined and satisfies $F \circ F^{-1}(x) = x$ for all $x \in (0, 1)$ and $F^{-1} \circ F(y) \le y$ for all $y \in \R$ {with $0<F(y)<1$}. We can write $X(i) = F^{-1} (U(i))$ for a sequence of i.i.d.\ uniform random variables $U(i)$ and for every $x \in \R$ and $i \in \N$ we have that
	\begin{equation}
	\label{E:Xiiii}
	X(i) = F^{-1}(U(i)) \le x \qquad \text{ if and only if } \qquad U(i) \le F(x).
	\end{equation}
	Now, with notation as in the statement of Theorem \ref{T:kstest}, let
	$$
	H_n(x) = \frac{1}{K_n} \sum_{i=1}^{K_n} \indic(U(I_{n, i}) \le x), \qquad \mathand \qquad \tilde H_n(x) = n^{1/6}(H_n(x) - x).
	$$
	By \eqref{E:Xiiii}, we have that 
	$
	\tilde F_n = \tilde H_n \circ F.
	$
	By Theorem \ref{T:kstest} for uniform random variables on $[0, 1]$, we can couple $G_n$ and $\Pi$ so that $\tilde H_n \to \Pi$ uniformly a.s., and hence $\tilde F_n = \tilde H_n \circ F \to \Pi \circ F$ uniformly a.s.\ as well, yielding the distributional convergence in the theorem. The convergence of $|\tilde F_n - \tilde F_n'|$ from the corresponding statement for uniform random variables follows similarly.
\end{proof}
\section{Interfaces}
\label{S:interfaces}
The goal of the remainder of the paper is to move from convergence of last passage percolation to convergence of TASEP. TASEP and exponential last passage percolation are related through a well-known transformation, which describes TASEP in terms of interfaces in the corresponding last passage metric. Because of this, we first prove general convergence theorems for interfaces. Though interfaces can be defined abstractly, we restrict our attention to planar directed metrics of negative sign with an asymptotic spacetime structure.
 
First, for a directed metric $d$ of negative sign on $\R^2$, {$S\subset\R^2$} and $t \in \R$, define the \textbf{backwards ball of radius $t$ centered at a set $S$} by
\begin{equation}
\label{E:BdSx-union}
B_d(S, t) = \{y \in \R^2 : d(y, x) > t \text{ for some }x \in S\} = \bigcup_{x \in S} B_d(x, t).
\end{equation}
We write $B_d(x, t) = B_d(\{x\}, t)$. Note that for $t\ge  0$, $x \notin B_d(x, t)$: in usual metric spaces, the analogous balls would have nonpositive radius, and so would be empty. Now, suppose that $d$ satisfies the positivity condition
\begin{equation}
\label{E:positivity}
d(x, s; x, t) \ge 0 \qquad \text{ for all } x \in \R, s \le t.
\end{equation}
For such a metric $d$, boundaries of backwards balls are graphs of functions. Indeed, letting $\bar{B}$ denote the Euclidean closure, define 
$$
I_d(S, t)(x) = \sup \{y \in \R : (x, y) \in \bar{B}_d(S, t)\},
$$
the \textbf{interface of radius $t$ centered at $S$}. {By \eqref{E:positivity} and the triangle inequality, every vertical section of $\bar B$ is a downward closed interval, possibly empty or all of $\R$. Thus $I_d(S,t)$ is upper semicontinuous, may take the values $\pm\infty$, and} $\bar{B}$ is the closed hypograph (see Section \ref{S:topologies}) of $I_d(S,t)$, with the infinities removed: 
$$
\bar{B}_d(S, t) = {\mathfrak h }I_d(S,t)\cap \R^2.
$$
We now move on to proving convergence theorems for interfaces. For the remaining theorems in this section, we will work with the following objects and assumptions.

\begin{assumptions} \label{A:interface} We have \\ 
\vspace{-2.2em}
\begin{itemize}
 \itemsep0em 
	\item a sequence of directed metrics $d_n$ of negative sign on $\R^2$ satisfying $d_n(x, s; y, t) = -\infty$ for $s > t$,
	\item a limiting metric $d$ which is continuous on $\Rh$ such that $\mathfrak{g} d_n \to \mathfrak{g} d$ in $\scrE_*$, 
	\item scaling parameters $b_n > 0$ and $a_n \in \R$ such that
\begin{equation}
\label{E:bnn}
b_n \to \infty, \quad a_n/b_n \to 0,
\end{equation}
and an additive metric $\ell_n$ given by 
\begin{equation*}
\ell_n(p,q)={\ell_n(q)-\ell_n(p)}, \qquad \text{ where } \qquad \ell_n(x, t) = a_n x + b_n t
\end{equation*}
such that $d_n + \ell_n$ satisfies the positivity condition \eqref{E:positivity}.
\end{itemize}
\end{assumptions}

\begin{theorem}
\label{T:abstract-dm}
With the Assumptions \ref{A:interface}, define a function $i_n:\Rd \to \R$ by
$$
i_n(y, t; x, s) = b_n I_{d_n+\ell_n}(x, s; b_n (s-t))(y) - b_n t - a_n(x-y).
$$
Then $i_n\to d$ compactly on $\R^4_\uparrow$.
Moreover, letting
$$
\tilde h_n({t,y};x,g)=b_n I_{d_n+\ell_n}(x, (g-a_nx)/b_n; b_nt)(y) + b_n t + a_ny,
$$
we have that
$
\tilde h_n({t,y};x, g) \to d(y,-t;x,0) + g
$
compactly on $(y, t; x, g) \in \R \X (0, \infty) \X \R^2$.
\end{theorem}

It helps with understanding how the scaling works to check that if $d_n=0$  then $i_n=0$. In Section \ref{S:TASEP}, $\tilde h_n$ will correspond to a rescaled and shifted version of a TASEP height function $h_n$.

\begin{proof} By continuity of $d$ on $\Rd$, it suffices to show that for any $u_n = (y_n, t_n; x_n, s_n) \to u = (y, t; x, s) \in \Rd$, that $i_n(u_n) \to d(u)$. To simplify notation in the proof, we let
$$
w_n := I_{d_n+\ell_n}(x_n, s_n; b_n (s_n- t_n))(y_n).
$$
Graph convergence of $d_n \to d$ and continuity of $d$ implies that $d_n \to d$ compactly on $\Rd$, see Lemma \ref{L:graph-equiv}. Therefore
$$
\limsup_{n \to \infty} \sup_{r \in [t - 1, (s + t)/2]} |d_n(y_n, r; x_n, s_n)| < \infty. 
$$
By the asymptotics in \eqref{E:bnn}, this implies that
$$
\frac{(d_n + \ell_n)(y_n, r; x_n, s_n)}{b_n} \to s-r
$$
uniformly for $r \in [t - 1, (s + t)/2]$.
Therefore since $s_n \to s$ and $t_n \to t$, for any $\ep > 0$ we have that
$$
(d_n + \ell_n)(y_n, t + \ep; x_n, s_n) < b_n(s_n-t_n)<(d_n + \ell_n)(y_n, t - \ep; x_n, s_n)
$$
for large enough $n$. This implies that
\begin{equation}
\label{E:wntt}
w_n \to t \quad \mathas \quad n \to \infty.
\end{equation}
Now, for any positive sequence $\ep_n$, condition \eqref{E:positivity} implies that
\begin{equation}
\label{E:dnpn}
(d_n + \ell_n)(y_n, w_n +\ep_n; x_n, s_n) \le b_n (s_n-t_n)\le (d_n + \ell_n)(y_n, w_n -\ep_n; x_n, s_n).
\end{equation}
Let $\ep_n>0$ be such that $b_n\ep_n \to 0$.  Since $d_n \to d$ compactly on $\Rd$ and $d$ is continuous we have
$$
(d_n + \ell_n)(y_n, w_n +\ep_n; x_n, s_n) - (d_n + \ell_n)(y_n, w_n -\ep_n; x_n, s_n) \to 0. 
$$
Therefore \eqref{E:dnpn} and the monotonicity of $(d_n + \ell_n)(y_n, \cdot\; ; x_n, s_n)$ implies that
\begin{equation}
\label{E:wnt}
(d_n + \ell_n)(y_n, w_n; x_n, s_n) - b_n (s_n - t_n) \to 0,
\end{equation}
or equivalently, 
\begin{equation}
\label{E:wnt2}
d_n (y_n, w_n; x_n, s_n) + a_n(x_n-y_n) + b_n(t_n -w_n)\to 0.
\end{equation}
Now, \eqref{E:wntt} implies that $(y_n, w_n; x_n, s_n) \to u = (y, t; x, s)$, so $d_n(y_n,w_n;x_n,s_n) \to d(u)$. Subtracting \eqref{E:wnt2} from this we get the required convergence 
\[
i_n(u_n)=b_nw_n - b_n t_n  -a_n(x_n - y_n) \to d(u) \quad \mathas n \to \infty.
\]
For the `Moreover' part of the theorem, we can rewrite $\tilde h_n$ in terms of $i_n$ as
\begin{equation}
\label{E:inhn}
\tilde h_n({t,y}; x, g) = i_n\lf(y, - t + \frac{g - a_n x}{b_n}; x, \frac{g - a_n x}{b_n}\rg) + g.
\end{equation}
Since $a_n/b_n \to 0$ with $n$, $(g-a_nx)/b_n \to 0$ compactly on $(x, g) \in \R^2$. Therefore compact convergence of $i_n \to d$ implies the compact convergence of $\tilde h_n$.
\end{proof}

When we apply Theorem \ref{T:abstract-dm} in Section \ref{S:TASEP}, it will give joint convergence of TASEP from a field of narrow wedge initial conditions. To give convergence of TASEP from general initial conditions, we will need a few variations.
The first handles initial conditions with compact support. 

\begin{theorem}
\label{T:set-cvg-1}
With the Assumptions \ref{A:interface}, fix a compact interval $J \sset \R$, and let $f_n:\R\to \R \cup \{-\infty\}$ be a sequence of upper semicontinuous functions. Assume that $f_n=-\infty$ outside $J$, and that $\mathfrak h f_n \to \mathfrak h f$ for some upper semicontinuous limit $f: \R \to \R \cup \{-\infty\}$ which is not equal to $-\infty$ everywhere.
Define
$$
\tilde h_n(t,y;f_n)=\sup_{x : f_n(x) \ne -\infty}  \tilde h_n(t,y;x,f_n(x)),\qquad 
h(t, y;f) = \sup_{x \in \R} {[}f(x)+d(y, -t; x, 0){]}.
 $$
Then 
$
\tilde h_n \to h$, compactly on $(t, y) \in (0, \infty) \X \R$.
\end{theorem}

Here recall from Remark \ref{R:other-domains} what it means for $\mathfrak h f_n \to \mathfrak h f$.
Note that for the directed metrics we study in Section \ref{S:TASEP}, we will be able to relate $\tilde h_n(t,y;f_n)$ to a set interface, using that for nicely behaved metrics $d$ and sets $S$, we have
$$
I_d(S, t) = \sup_{y \in S} I_d(y, t).
$$
This formula does not hold in general.

\begin{proof}
When all of the sets $f_n(\R) \smin \{-\infty\}$ lie in a common compact set $A$, then the claim follows by the uniform convergence in the `Moreover' part of Theorem \ref{T:abstract-dm}. 

Otherwise, we need to apply a truncation argument. For this, first observe that since $\mathfrak h f_n \to \mathfrak h f $ and $f$ is upper semicontinuous and never takes on the value $\infty$, and equals $-\infty$ off of $J$, that there exists some $a \in \R$ such that $f_n(x) \le a$ for all $n \in \N, x \in \R$.
Next, fix $b>0$ {large enough that $f(x)>1-b$ for some $x$,} and for any function $f$ define
$$
f_b^- = \begin{cases}f, \qquad &f(x) \ge - b, \\
-\infty, \qquad &{f}(x) < -b,
\end{cases} \qquad \text{ and } \qquad
f_b^+ = \begin{cases} f \vee (-b), \qquad &x \in J, \\
-\infty, \qquad &x \notin J.
\end{cases}
$$
With these definitions, all of the sets $(f_n)_b^-(\R)\smin \{-\infty\}$ and $(f_n)_b^+(\R)\smin \{-\infty\}$ lie in the compact set $[-b, a]$, so we can apply the theorem to the sequences $(f_n)_b^-, (f_n)_b^+$.

Now, recall that Hausdorff convergence is continuous under finite unions: if $A_n, B_n$ are any sequences of closed sets such that $A_n \to A$ and $B_n \to B$ in the Hausdorff topology, then $A_n \cup B_n \to A \cup B$. Therefore as $n \to \infty$,
$$
 \mathfrak{h} (f_n)_b^+ = \mathfrak{h} f_n \cup J \X [-\infty, -b] \to  \mathfrak{h} f \cup J \X [-\infty, -b] =  \mathfrak{h} f_b^+,
$$
so
\begin{equation}
\label{E:hnfb1}
\tilde h_n(\;\cdot\;; (f_n)_b^+) \to h(\;\cdot\;; f_b^+)
\end{equation}
compactly on $(0, \infty) \X \R$. Handling the lower truncations $(f_n)^-_b$ is more subtle. Consider any function $g_{b, N}$ such that $\mathfrak{h} g_{b, N}$ is a subsequential limit of $\mathfrak{h} (f_n)_b^-$ along some subsequence $N$. Then 
\begin{equation}
\label{E:hnfb2}
\tilde h_n(\;\cdot\;; (f_n)_b^-) \to h(\;\cdot\; ; g_{b,N})
\end{equation}
compactly on $(0, \infty) \X \R$ along $N$. Next, observe that $f_{b-1}^- \le g_{b, N} \le f$. Now,
\begin{equation}
\label{E:hnfb3}
h(\;\cdot\;; f_{b-1}^-) \le h(\;\cdot\;; g_{b,N})\le h (\;\cdot\;; f)\le  h(\;\cdot\;; f_b^+).
\end{equation}
Moreover, for any compact set $K \sset (0, \infty) \X \R$, we claim that
\begin{equation}
\label{E:hnfb4}
h(p; f_b^+) = h(p; f_{b-1}^-)
\end{equation}
on $K$ for all large enough $b$. Indeed, \eqref{E:hnfb4} holds if there exists any $s \in \R$ with $f(s) > - b + 1$, such that for all $p \in K$, we have
$$
f(s) + d(p; s, 0) > \sup_{x \in J} {[1-b} + d(p; x, 0){]}.
$$
This holds for all large enough $b$ for any $s$ such that $f(s) > -\infty$ by the continuity of $d$. The existence of one such $s$ follows since $f$ is not identically $-\infty$.

Finally, \eqref{E:positivity} implies that for any $b$, 
\begin{equation}
\label{E:hnfb5}
\tilde h_n(\;\cdot\;; (f_n)_b^-) \le \tilde h_n(\;\cdot\;; f_n)\le  \tilde h_n(\;\cdot\;; (f_n)_b^+).
\end{equation}
Combining \eqref{E:hnfb1}, \eqref{E:hnfb2}, \eqref{E:hnfb3}, \eqref{E:hnfb4}, and \eqref{E:hnfb5} implies that $\tilde h_n(\;\cdot \; ; f_n) \to h(\; \cdot \; ; f)$ compactly.
\end{proof}

To handle initial conditions $f_n, f$ with potentially noncompact support, we need an additional tightness condition on $d_n$.

\begin{tightness-condition}
	\label{tightness}
	Fix $t> 0$. For any compact set $K\subset \mathbb R$,  slope and height $s',h_*\in \mathbb R$ there exists $a \in \R$ and $g_n:\R \to \R$ with $g_n(x)\ge s'|x|-a$ for all $x$ and $\sup_{y\in K} \tilde h_n({t,y};g_n)\le h_*$ for all large enough $n$.
\end{tightness-condition}

 This tightness condition is tailored to be used later in the context of TASEP.

\begin{theorem}
	\label{T:noncompact-interfaces} 
	With the Assumptions \ref{A:interface}, suppose that the tightness condition \ref{tightness} holds for some fixed $t > 0$. Consider any sequence of initial conditions $f_n:\R \to \R \cup \{-\infty\}$ for which there exists $s, a \in  \R$ such that $f_n(x)<a+s|x|$ for all large enough $n$. Suppose that $\mathfrak h f_n\to \mathfrak h f$ for some upper semicontinuous function $f:\R \to \R \cup \{-\infty\}$ which is not equal to $-\infty$ everywhere. Define
$$
\tilde h_n(t,y;f_n)=\sup_{x : f_n(x) \ne -\infty}  \tilde h_n(t,y;x,f_n(x)),\qquad 
h(t, y;f) = \sup_{x \in \R} {[}f(x)+d(y, -t; x, 0){]}.
 $$
Assume that for every $s'> 0$ and every compact set $K \sset \R$, 
\begin{equation}
\label{E:supxy}
\sup_{(x,y)\in \mathbb R \times K}  s'|x|+d(y,-t;x,0)<\infty.
\end{equation}
Then 
$
\tilde h_n(t, \;\cdot\; ; f_n) \to h(t, \;\cdot\; ; {f})$, compactly on $\R$.
\end{theorem}

\begin{proof}
Fix a compact set $K \sset \R$.
First, $f(x) \le a + s|x|$ for all $x$ since each $f_n$ satisfies the same inequality. Therefore by \eqref{E:supxy} with $s' = {|s|+1}$, there exists a compact set $D \sset \R$ such that
\begin{equation}
\label{E:htyfD}
   h(t,y;f) = h(t, y, f_D), \qquad y\in K
\end{equation}
where $f_D$ is equal to $f$ on $D$ and $-\infty$ elsewhere. Also, let $x_0$ be such that $f(x_0) > -\infty$, and consider a sequence of points $x_n \to x_0$ such that $f_n(x_n) \to f(x_0)$. 
We have
$$
\liminf_{n \to \infty} \inf_{{y \in K}} \tilde h_n(t, y; f_n) \ge \lim_{n \to \infty} \inf_{{y \in K}} \tilde h_n(t, y; x_n, f_n(x_n)) = f(x_0)+\inf_{y\in K} d(y,-t,x_0,0),
$$
where the final equality follows from the `Moreover' in Theorem \ref{T:abstract-dm}. Letting $b$ denote the right hand side above, for large enough $n$ we have
$$
\tilde h_n(t, y; f_n) \ge b -1
$$
on $K$. Now let $g_n$ be as in the tightness condition with $K$ as above, $s' = {|s|+1}$, and ${h_*} = b - 2$. Then
$$
\tilde h_n(t, y; f_n) > \tilde h_n(t, y; g_n)
$$
on $K$ for all large enough $n$, and there exists a compact set $E$ containing $D$ in its interior such that $g_n > f_n$ on $E^c$. Therefore by condition \eqref{E:positivity}, on $K$
\begin{equation}
\label{E:trunc}
    \tilde h_n(t, y; f_n) = \tilde h_n(t, y; f_{n,E}),
\end{equation}
where $f_{n, E}$ is equal to $f_n$ on $E$ and $-\infty$ elsewhere.
Now let $f'$ be any subsequential limit in the hypograph topology of $f_{n, E}$. By Theorem \ref{T:set-cvg-1}, $\tilde h_n(t, y; f_{n, E}) \to h(t, y; f')$ uniformly on $K$ along this subsequence. Since $D$ is contained in the interior of $E$, we have $f_D \le f' \le f$, and so by \eqref{E:htyfD},
\begin{equation}
\label{E:3eqs}
    h(t,y;f) = h(t,y;f') =h(t, y, f_D), \qquad y\in K.
\end{equation}
Combining \eqref{E:trunc} and \eqref{E:3eqs} yields the theorem.
\end{proof}

\section{TASEP and the KPZ fixed point}
\label{S:TASEP}

In this section we prove Theorem \ref{T:intro-TASEP}, and a few other convergence results relating TASEP and the directed landscape. Let $\mathcal H$ be the set of functions $h:\mathbb Z \to \mathbb Z$ so that $h(0)$ is even and $|h(i)-h(i-1)|=1$ for all $i$.

{\bf Continuous time TASEP} is a continuous-time Markov process with state space $\mathcal H$. We write $h(t, y; f)$ for continuous time TASEP at time $t \in [0, \infty)$ and location $y \in \R$ started from an initial condition $f \in \scrH$. A continuous time TASEP $h$ decreases by $2$ at each of its local maxima at rate 1. The movements at each of the local maxima are independent.

Continuous time TASEP is connected to exponential last passage percolation via the well-known coupling given in the introduction. This coupling can be expressed using the language of interfaces introduced in Section \ref{S:interfaces}. Let
$E^2 = \{(x, y) \in \Z^2 : x + y \in 2 \Z\}$ denote the even lattice points. Let $X$ be an array of independent mean 1 exponential random variables indexed by $E^2$. Let $d_X$ be the directed metric of negative sign on $\R^2$ induced by the function $d_0(a,c)={X(c)\indic(c\in E^2)}$, {with value zero when $c\notin E^2$,} where $d_0$ is defined for all points $a\in\R^2,c\in{\R^2}$ with 
$
a \ne c, \quad |a_1 - c_1| \le {c_2-a_2}.
$
This directed metric corresponds to a rotated version of the exponential last passage metric in the plane, see Remark \ref{R:planar-embedding}.

Next, extend all functions in $\mathcal H$ to functions from $\mathbb R\mapsto\mathbb R$ by making them linear on intervals $[i,i+1], i\in \mathbb Z$. Recall the definition of the interface $I_{d_X}(p, t)$ from Section \ref{S:interfaces}. 
By the definition, for any $p \in E^2$ and $t \ge 0$, we have
$
I_{d_X}(p, t) \in \scrH
$,
and the function 
\begin{equation}
\label{E:one-NW}
h(t, y;p) := I_{d_X}(p, t)(y)
\end{equation}
evolves as continuous time TASEP started from the {\bf narrow wedge }
$$
  N^p(x) = p_2 - |x- p_1|,
$$
so we may write $h(t,y;N^p):=h(t,y;p)$.
We call this collection of coupled processes $h(\cdot; p), p \in \R^2$ the \textbf{narrow wedge field} for TASEP. For a general initial condition $f \in \scrH$, the standard variational formula implies that  
\begin{equation}
\label{E:var-form}
h(\cdot; f) = \max_{x
\in \Z} h(\cdot; (x,f(x))).
\end{equation}
evolves as TASEP started from $f$. This couples TASEP from all initial conditions. In the proofs, it will be convenient to use the formula \eqref{E:var-form} with $\R$ in place of $\Z$. This changes all values by at most $2$, so will not affect any convergence statements where space is rescaled. We can also use the set $M_f$ of local maxima of $f$ in place of $\Z$ without changing the right side of \eqref{E:var-form}.

\subsection{The KPZ fixed point}
\label{SS:KPZfixed}

The KPZ fixed point is a Markov process that describes the scaling limit of TASEP from an arbitrary single initial condition. It was first constructed in \cite*{matetski2016kpz} as the scaling limit of TASEP, and later shown in \cite*{nica2020one} to be a function of the directed landscape. The \textbf{KPZ fixed point} associated to a directed landscape $\scrL$ is the random function given by
$$
h_\scrL(t, y; f) := \max_{x \in \R} {[}\scrL(x, 0; y, t) + f(x){]}. 
$$
The domain of $h_\scrL$ is every $t > 0, y \in \R$ and any {upper semicontinuous} function $f:\R \to \R \cup \{-\infty\}$ that is not identically equal to $-\infty$, and satisfies
$$
f(x) \le s |x| +a
$$
for some $s, a > 0$ and all $x$. While weaker conditions on $f$ are possible, see \cite*{sarkar2020brownian}, we restrict our attention here to functions satisfying the linear growth bound above. By the bound in \eqref{E:LuCu}, $h_\scrL$ is almost surely continuous and finite for all valid choices of $t, y, f$.

The main theorem of \cite*{matetski2016kpz} shows that TASEP started from a given initial condition {converges after rescaling to the KPZ fixed point, see Theorem} \ref{T:matetski-kpz} of Section \ref{SS:TASEP-intro}. 

Their theorem concerns convergence of hypographs of random upper semicontinuous functions from $\R$ to $\R \cup \{-\infty\}$. However, Theorem 3.9 in \cite*{matetski2016kpz} implies that the same distributional convergence holds with the stronger topology of uniform convergence on compact sets. 

In the context of TASEP, the theorems in Section \ref{S:interfaces} allow for variations and strengthenings of Theorem \ref{T:matetski-kpz}. In particular, we can prove joint convergence to the KPZ fixed point from multiple initial conditions. 

\subsection{Convergence of TASEP}
\label{SS:conv-TASEP}

We first note that rotated exponential last passage still converges to the directed landscape and specify the scaling constants. The following is an application of Corollary \ref{C:exponential}.

\begin{corollary}
\label{C:rotated-ELPP}
Let $d_n'$ be the pullback of $d_{X}$ under the linear transformation $K_n(x, t) = (2 n^{2/3} x, nt)$,
let  
$
\ell_n(p, q) = \ell_n(q) - \ell_n(p)$, 
$\ell_n(x, t) = n^{2/3} t,
$
and let $d_n =  n^{-1/3} d_n' /2- \ell_n$. Then $\mathfrak{g} d_n \cvgd \mathfrak{g} \scrL$ in $\scrE_*$.
\end{corollary}

\begin{proof}
 This is a rotated version of the convergence of exponential last passage, Corollary \ref{C:exponential}, with  parameters $N = n/2$,  $m_N = N$, $\al_N = 4$, $\be_N = 2$, $\chi_N = 2 n^{1/3}$, $\tau_N = 2 n^{2/3}.$
{The map $(a,b)\mapsto(a+b,a-b)$ sends the lattice with steps $(1,0),(0,-1)$ to $E^2$. Under this map, the scaled point $(\tau_Nx+Nt,-Nt)$ becomes $(\tau_Nx,nt+\tau_Nx)$. Thus $K_n(x,t)$ corresponds to the original scaling evaluated at $(x,t-\tau_Nx/n)$. This shear tends to the identity. Moreover, the centering becomes
\[
4N\bigl(t-s-\tau_N(y-x)/n\bigr)+2\tau_N(y-x)=2n(t-s),
\]
which, divided by $\chi_N$, is exactly $\ell_n$. Rounding the transformed endpoints inward or outward gives upper and lower bounds with coordinate errors $O(1)$. At lattice endpoints, the difference between including and excluding the initial weight is an endpoint weight. On bounded rescaled sets, the largest exponential weight divided by $n^{1/3}$ tends to zero in probability by a union bound. Graph convergence therefore follows from Theorem \ref{T:geom-lpp-emph}.2 and the continuity of $\scrL$ on $\Rh$.} 
\end{proof}

Our goal is now to translate the three convergence theorems of Section \ref{S:interfaces} to TASEP. We start with a lemma that checks that the assumptions of Section \ref{S:interfaces} can be guaranteed to hold on a particular coupling.

\begin{lemma}
\label{L:assumption-check}

Let $d_n, \ell_n$ be as in Corollary \ref{C:rotated-ELPP}, let $a_n = 0$ and $b_n = n^{2/3}$, and let $d = \hat \scrL$ be a directed landscape.
\begin{enumerate}[label=(\roman*)] 
	\item 
	With these choices, {there exists a coupling} of the exponential environments $X_n$ for different $n$ where $\mathfrak{g} d_n \to \mathfrak{g} \hat \scrL$ in $\scrE_*$ almost surely. In this coupling, Assumptions \ref{A:interface} hold almost surely. Moreover, letting $h_n(t, y; f)$ denote the collection of coupled TASEPs governed by $X_n$, see \eqref{E:one-NW} and \eqref{E:var-form},
	 the functions $i_n, \tilde h_n$ defined in Theorem \ref{T:abstract-dm} are given by
	\begin{equation}
	\label{E:in-jn}
	\begin{split}
i_n(y, t; x, s) &= n^{-1/3} h_n(2n(s - t), 2n^{2/3} y; 2n^{2/3} x, ns) - n^{2/3} t \\
\tilde h_n(t, y; x, g) &= n^{-1/3} h_n(2n t, 2n^{2/3} y; 2n^{2/3} x, n^{1/3} g) + n^{2/3} t.
	\end{split}
	\end{equation}
	\item For any finite or countable set $T \sset (0, \infty)$ and any subsequence $Y \sset \N$, we can find a further subsequence $Y' \sset Y$ and a coupling of the exponential environments for different $n \in Y'$ such that both the conclusion of (i) holds a.s.\ along $Y'$ and Tightness condition \ref{tightness} holds a.s.\ along $Y'$ for all $t \in T$.
\end{enumerate}

\end{lemma}

\begin{proof}
For (i), by Corollary \ref{C:rotated-ELPP} and Skorokhod's representation theorem there exists a coupling where $\mathfrak{g} d_n \to \mathfrak{g} \hat \scrL$ in $\scrE_*$ almost surely. Let $X_n$ denote the  exponential environment at stage $n$ in this coupling. In this coupling, $d_n(x, s; y, t) = -\infty$ for $s {>} t$ since this property holds for $d_X$, $b_n \to \infty$ and $a_n/b_n = 0$ for all $n$, and $d_n + \ell_n$ satisfies \eqref{E:positivity} since it is a rescaled version of the $d_X$. We now check \eqref{E:in-jn}. With $d_n'$ as in Corollary  \ref{C:rotated-ELPP}, {a direct translation of the formula} for $i_n$ in Theorem \ref{T:abstract-dm} gives
\begin{equation}
\begin{split}
i_n(y, t; x, s) = n^{2/3} I_{2^{-1} n^{-1/3} d_{n}'}(x, s; n^{2/3}(s - t))(y) - n^{2/3} t.
\end{split}
\end{equation}
We move the scaling factor of $2^{-1} n^{-1/3}$ in front of $d_n'$ into the $n^{2/3}(s-t)$ to get that 
\begin{equation}
\label{E:modified}
i_n(y, t; x, s) = n^{2/3} I_{d_n'}(x, s; 2n(s - t))(y) - n^{2/3} t.
\end{equation}
To relate this to an interface in $d_{X_n}$, we can observe the following simple fact about backwards balls and pullbacks: if $d$ is a directed metric of negative sign on $\R^2$, and $g_*d$ is the pullback of $d$ under a function $g:\R^2 \to \R^2$, then for  $p \in \R^2$ and $t \in \R$, we have
\begin{equation}
\label{E:pullback}
B_{g_* d}(p, t) = g^{-1} B_d(gp, t) := \{x \in \R^2: gx \in B_d(gp, t) \}.
\end{equation}
Using \eqref{E:pullback} and the fact that $d_n'$ is the pullback of $d_{X_n}$ under the map $K_n(x, t) {=} (2n^{2/3} x, nt)$, the right side of \eqref{E:modified} equals 
$$
n^{2/3} n^{-1} I_{d_{X_n}}(2n^{2/3} x, n s; 2n(s - t))(2 n^{2/3} y) - n^{2/3} t.
$$
By \eqref{E:one-NW}, this is equal to the right hand side of \eqref{E:in-jn}. 
Moreover, $i_n$ is related to $h_n$ by the formula \eqref{E:inhn}; the same relationship holds between the two right sides of \eqref{E:in-jn}. 

We move to (ii). Fix $Y$. For every $s', a \in \N$, we can find a sequence of functions $g^{s', a}_n:\R \to \R$ such that
\begin{enumerate}
	\item $\mathfrak{h} g^{s', a}_n \to \mathfrak{h} g^{s', a}$, where $g^{s', a}(x) = (s'+1)|x| - a + 1.$
	\item $(s' + 2)|x| - a + 2 \ge g^{s', a}_n(x) \ge s'|x| - a$ for all large enough $n$ and all $x$, and $J_n g^{s', a}_n \in \scrH$ for all $n$, where $J_n f(x) = n^{1/3}f(x/(2n^{2/3}))$.
\end{enumerate}
Fix $T$, and for each $n$, define a function $D_n:\N^3 \X T \to \R$ by
$$
D_n(m, s', a, t) := \sup_{y\in [-m, m]} \tilde h_n(t, y ; g^{s', a}_n).
$$
By the formulas \eqref{E:var-form} and \eqref{E:in-jn}, $\tilde h_n(\cdot ; g^{s', a}_n)$ is a rescaled and centered TASEP from the initial condition $J_n g^{s', a}_n$ up to an error of size at most $2 n^{-1/3}$ (the error comes from going from a maximum over $\Z$ to a maximum over $\R$, see the discussion after \eqref{E:var-form}). More precisely, for all $t > 0, y \in \R$ we have
\begin{equation}
\label{E:tildehn}
\lf|\tilde h_n(t, y; g_n) - n^{-1/3}(h_n(2nt, 2n^{2/3}y; J_n g^{s', a}_n) {+} nt)\rg| \le 2n^{-1/3}.
\end{equation}
By \eqref{E:tildehn} and conditions 1 and 2 above, for any fixed $m, s', a, t$, the random variables $D_n(m, s', a, t)$ converge in distribution with limits determined by Theorem \ref{T:matetski-kpz}. Therefore $D_n$ is tight as a function in the pointwise topology, so we can find a subsequence $Y' \sset Y$ where $D_n, \mathfrak{g}d_n$ converge jointly in distribution to $D, \mathfrak{g}\hat \scrL$, where $\hat \scrL$ is a directed landscape. By Skorokhod's representation theorem, we can realize this as almost sure convergence. Now we check Tightness condition \ref{tightness}. Fix $t \in T$, a compact set $K \sset \R$, and a slope and height $s', h_* \in \R$ as in the statement of the tightness condition. By possibly increasing $s'$ and the size of $K$ and decreasing $h_*$, we may assume $s', -h_* \in \N$ and $K = [-m, m]$ for some $m \in \N$. For any $a \in \N$, by Theorem \ref{T:matetski-kpz} we have
$$
\p(D(m, s', a, t) \ge h_*) = \p\lf( \max_{x \in \R, y \in [-m, m]} \scrL(x, 0; y, t) + (s'+1)|x| - a + 1 \ge h_*\rg).
$$
By \eqref{E:LuCu}, the right hand side above converges to $0$ as $a \to \infty$. Therefore there exists a random $A \in \N$ such that $D(m, s', A, t) < h_*$ almost surely, and so almost surely,
$$
\sup_{y \in [-m, m]} \tilde h_n({t,y}; g^{s', A}_n) = D_n(m, s', A, t) < h_*
$$
for all large enough $n$. The functions $g^{s', A}_n$ satisfy the criteria of Tightness condition \ref{tightness}, so the condition holds almost surely for the given $s', -h_*, m \in \N, t \in T$. Since we only have countably many choices for these parameters, the tightness condition holds almost surely along $Y'$ for all $t \in T$.  
\end{proof}

We can now translate the three interface convergence theorems of Section \ref{S:interfaces}. 

\begin{theorem}[Convergence of the narrow wedge field]
\label{T:TASEP-1}
In the coupling in Lemma \ref{L:assumption-check} (i), $i_n \to \hat \scrL$ compactly on $\Rd$ almost surely. Moreover, letting 
\begin{equation}
\label{E:L-hat-L}
\scrL(x, s; y, t) = \hat \scrL(y, -t; x, -s),
\end{equation}
we have $
{\tilde h_n}(t, y; x, g) \to \scrL(x, 0; y, t) + g 
$
compactly on $(y, t; x, g) \in \R \X (0, \infty) \X \R^2$ almost surely, and $\scrL$ is a directed landscape.
\end{theorem}

\begin{proof}
This follows from Theorem \ref{T:abstract-dm} and Lemma \ref{L:assumption-check}(i). The fact that $\scrL$ is also a directed landscape follows from Proposition \ref{P:landscape-sym}.
\end{proof}

For these next two theorems, we use the notation $\iota_n f(x) = n^{-1/3} f(2 n^{2/3} x)$ from the introduction. The first theorem gives a strong convergence which is uniform in time, but only deals with compactly supported initial conditions.
\begin{theorem}[Compactly supported initial conditions]
	\label{T:TASEP-2}
	Let $f_n \in \scrH$. Suppose that there exists a compact set $I$, such that for every $n$, the set of local maxima $M_{\iota_n f_n}$ of $\iota_n f_n$  is contained in $I$.
	Assume that $\mathfrak h \iota_n f_n \to \mathfrak h f $ for some upper semicontinuous limit $f: \R \to \R \cup \{-\infty\}$ with ${f} \not\equiv -\infty$. Then in the coupling of Lemma \ref{L:assumption-check} (i) with $\scrL$ as in \eqref{E:L-hat-L}, almost surely we have
	\begin{equation}
	\label{E:Hnn}
	n^{-1/3}(h_n(2nt,2n^{2/3} y;f_n) + nt)  \to h_\scrL(t, y, f)
	\end{equation}
compactly on $(t, y) \in (0, \infty) \X \R$.
\end{theorem}

\begin{proof}
	 Using the notation $\tilde h_n$ from Lemma \ref{L:assumption-check}, since all the local maxima of $\iota_n f_n$ are contained in $I$ the left side of \eqref{E:Hnn} equals
$$
\max_{x \in M_{\iota_n f_n}} \tilde h_n(t,y; x, \iota_n f_n(x)) =  \max_{x \in I} \tilde h_n(t,y; x, \iota_n f_n(x)) + O(n^{-1/3}).
$$
Here the $O(n^{-1/3})$ correction comes from replacing the maximum over points of local maxima with a larger set, see the discussion after \eqref{E:var-form}. In the notation of Theorem \ref{T:set-cvg-1}, the right side above equals $\tilde h_n(t,y; g_n) + O(n^{-1/3})$, where $g_n$ equals $\iota_n f_n$ on $I$ and $-\infty$ elsewhere. Now, $\mathfrak h g_n \to \mathfrak h f$ since $\mathfrak h \iota_n f_n \to \mathfrak h f$. Therefore applying Theorem \ref{T:set-cvg-1} and Lemma \ref{L:assumption-check}(i) yields the theorem.
\end{proof}

The next theorem is a restatement of Theorem \ref{T:intro-TASEP}.

\begin{theorem}[Noncompact initial conditions]
	\label{T:TASEP-3}
		Fix a finite or countable set $T \sset (0, \infty)$. Then there is a coupling of a sequence $h_n$ of coupled TASEPs and $\scrL$, and an event $\Om$ of probability $1$ with the following property. Consider any collection of functions $f_n \in \scrH$ satisfying
	\begin{enumerate}
		\item $\mathfrak{h} \iota_n f_n \to \mathfrak{h} f$ for some $f: \R \to \R \cup \{-\infty\}$ with $f \not \equiv -\infty$, and 
		\item there exists $s, a \in \R$ such that for all $x$, ${\iota_n f_n}(x) < a + s |x|$ for all large enough $n$.
	\end{enumerate}
Then for any $t \in T$, on $\Om$ we have 
	$$
	n^{-1/3}(h_n(2nt,2n^{2/3} y;f_n) + nt) \to h_\scrL(t, y ; f)
	$$
compactly over $y \in \R$.
\end{theorem}

\begin{proof}
First, for any subsequence $Y \sset \N$, there exists a further subsequence $Y' \sset Y$ and a coupling of $h_n, n \in Y', \scrL$ such that the result holds along this subsequence. Indeed, this follows from Theorem \ref{T:noncompact-interfaces} and Lemma \ref{L:assumption-check}(ii). The fact that condition \eqref{E:supxy}, which is required for Theorem \ref{T:noncompact-interfaces}, holds almost surely for $\scrL$ follows from \eqref{E:LuCu}. 

Now, in the coupling in Lemma \ref{L:assumption-check}(ii), the sequence of rescaled exponential last passage metrics $d_n$ satisfies $\mathfrak{g} d_n \to \mathfrak{g} \hat \scrL$ almost surely, where $\scrL, \hat \scrL$ are related via \eqref{E:L-hat-L}. Also, with $D_n$ as in the proof of Lemma \ref{L:assumption-check}, in this coupling we have that almost surely,
$$
D_n(m, s', a, t) \to D(m, s', a, t) := \max_{x \in \R, y \in [-m, m]} \scrL(x, 0; y, t) + (s'+1)|x| - a + 1.
$$
for all $m, s', a \in \N, t \in T$. Recall also that $d_n \to \hat \scrL$ along this subsequence. Since the joint distribution of $D, \hat \scrL$ is independent of $Y, Y'$, this implies that the whole sequence $(D_n, \mathfrak{g}d_n), n \in \N$ converges in distribution to $(D, \mathfrak{g} \hat \scrL)$. Therefore we can apply the argument of the proof of Lemma \ref{L:assumption-check}(ii) without first passing to a subsequence, to get that Lemma \ref{L:assumption-check} (ii) holds with $Y' = Y = \N$, and hence so does Theorem \ref{T:TASEP-3}.
\end{proof}
\begin{remark}
\label{R:discrete}
Here we have focused on convergence of continuous-time TASEP to the KPZ fixed point by coupling with exponential last passage percolation. We can also consider discrete-time TASEP.

\textbf{Discrete-time TASEP} with parameter $p \in (0, 1)$ is a discrete time Markov chain $D(t, \cdot; h_0)$ defined on the same state space $\scrH$ as continuous-time TASEP. At each time, $D(t+1, \cdot; h_0)$ is obtained from $D(t, \cdot; h_0)$ by independently decreasing all local maxima by $2$ with probability $p$. Discrete-time TASEP can be coupled to last passage percolation with geometric weights in the same way that continuous-time TASEP can be coupled to last passage percolation with exponential weights. The correct geometric weights have mean $p^{-1}$ and take values in $\{1, 2, \dots\}$. This corresponds to the $\ga = p^{-1} - 1$ and $m_n = n$ case of Theorem \ref{T:geom-lpp}; note that in that theorem we worked with geometric random variables on $\{0, 1, 2, \dots\}$.

 We can then use Theorem \ref{T:geom-lpp} to conclude convergence of discrete-time TASEP to the KPZ fixed point $h_\scrL$, just as in Theorems \ref{T:TASEP-1}, \ref{T:TASEP-2}, and \ref{T:TASEP-3}. The rescaled version that converges to ${h_\scrL(t,y;f)}$ is
$$
D_n(t, y; f) := c_1 n^{-1/3} D\lf( c_3 n t, c_2 n^{2/3} y; J_n f\rg) + c_1 n^{2/3}t.
$$
Here $J_n f(x) := c_1^{-1} n^{1/3} f(x/(c_2 n^{2/3}))$, and $c_1, c_2, c_3$ are $p$-dependent constants given by
\begin{equation*}
c_1 = \frac{2^{1/3} (1-\sqrt{q})^{{1/3}}}{p^{1/3} q^{1/6}}, \qquad c_2 = \frac{2^{1/3}(1 + \sqrt{q})^{2/3}}{q^{1/6}}, \qquad c_3 = \frac{1 + \sqrt{q}}p,
\end{equation*}
where $q = 1- p$.
The rescaled narrow wedge field that converges to $\scrL$ (i.e. the analogue of $i_n$ in \eqref{E:in-jn}) is
$$
i_n(y, t; x, s) := c_1 n^{-1/3} {D}\lf( c_3 n (s-t), c_2 n^{2/3} y; c_2 n^{2/3} x, n s\rg) - c_1 n^{2/3}t.
$$
\end{remark}

\begin{remark}
	\label{R:slopes}
	We have only considered TASEP with unsloped initial conditions (i.e. initial conditions $f_n(x)$ that converge after rescaling without subtracting an $x$-dependent linear term). One can also consider both discrete and continuous-time TASEP with initial conditions with a nonzero slope $\rho \in (-1, 1)$. TASEP with such initial conditions also converges to the KPZ fixed point. The same methods apply for proving convergence, though one needs to be more careful about tracking certain linear transformations. In this context, the right last passage percolation convergence result to apply is Theorem \ref{T:geom-lpp} or Corollary \ref{C:exponential} with $(1+ \rho)m_n = (1- \rho)n$. 
\end{remark}

\smallskip

\noindent{\bf Acknowledgments.} We thank Timo Sepp\"al\"ainen for suggesting references. The authors are both supported by the NSERC discovery grants. 

\bibliographystyle{dcu}

\bibliography{MasterBib}

\end{document}